\documentclass{article}
\usepackage{amsmath,amssymb,amsthm}
\usepackage{mathrsfs}
\usepackage[colorlinks=true]{hyperref}
\usepackage{pdfsync,float}
\usepackage{enumitem}

\allowdisplaybreaks

\def\R{\mathbb{R}}
\def\N{\mathbb{N}}
\def\Q{\mathbb{Q}}
\def\Z{\mathbb{Z}}
\def\C{\mathbb{C}} 

\def\S{\mathscr{S}} 

\newcommand{\Op}{\mathrm{Op}}
\newcommand{\OpW}{\mathrm{Op^{\mathcal{W}}}}

\newcommand{\ord}{\mathrm{ord}}
\newcommand{\Tr}{\mathrm{Tr}}

\newcommand{\dsR}{\mathrm{d}_{\mathrm{sR}}}
\newcommand{\BsR}{\mathrm{B}_{\mathrm{sR}}}

\newcommand{\hatdsR}{\widehat{\mathrm{d}}_{\mathrm{sR}}}
\newcommand{\hatBsR}{\widehat{\mathrm{B}}_{\mathrm{sR}}}
\newcommand{\supp}{\mathrm{supp}}
\newcommand{\Cst}{\mathrm{Cst}}

\renewcommand{\geq}{\geqslant}
\renewcommand{\leq}{\leqslant}

\newtheorem{theorem}{Theorem}[section]

\newtheorem{proposition}{Proposition}[section]
\newtheorem{corollary}{Corollary}[section]
\newtheorem{definition}{Definition}[section]
\newtheorem{lemma}{Lemma}[section]
\theoremstyle{definition}\newtheorem{example}{Example}[section]
\theoremstyle{definition}\newtheorem{remark}{Remark}[section]

\newtheorem{innercustomthm}{Theorem}
\newenvironment{customthm}[1]
{\renewcommand\theinnercustomthm{#1}\innercustomthm}
{\endinnercustomthm}

\newcommand{\Qeq}{{\mathcal{Q}^{M\setminus\S}}} 

\title{Spectral asymptotics for sub-Riemannian Laplacians} 

\author{Yves Colin de Verdi\`ere\footnote{Institut Fourier, Universit\'e Grenoble Alpes, 100 rue des Math\'ematiques, 38610 Gi\`eres, France (\texttt{yves.colin-de-verdiere@univ-grenoble-alpes.fr}).}
\and
Luc Hillairet\footnote{Institut Denis Poisson, Universit\'e d'Orl\'eans, route de Chartres, 45067 Orl\'eans Cedex 2, France (\texttt{luc.hillairet@univ-orleans.fr}).}
\and
Emmanuel Tr\'elat\footnote{Sorbonne Universit\'e, CNRS, Universit\'e de Paris, Inria, Laboratoire Jacques-Louis Lions (LJLL), F-75005 Paris, France (\texttt{emmanuel.trelat@sorbonne-universite.fr}).}
}

\date{}

\begin{document}

\maketitle


\tableofcontents

\newpage

\section{Introduction}
This section is devoted to provide a short overview of our main results. More general and complete statements are given in the subsequent sections. We refer to Appendix \ref{sec_sR_geom} for reminders on sub-Riemannian geometry and details on all notions used hereafter.

\paragraph{Setting.}
Let $(M,D,g)$ be a \emph{sub-Riemannian} (in short, sR) structure, where $M$ is a smooth (i.e., $C^\infty$) connected compact manifold of dimension $n\in\N^*$,  
$D$ is a subsheaf of $TM$ called \emph{horizontal distribution} and $g$ is the sR metric (defined in Appendix). 
In particular, when $D$ is a subbundle of $TM$, $g$ is a Riemannian metric on $D$, but the rank of $D$ may vary on $M$. 
When $D=TM$, we are in the Riemannian case. The situation of interest in this article is when $D\neq TM$. This includes the almost-Riemannian case.
Throughout the article, we assume that $\mathrm{Lie}(D) = TM$ (\emph{H\"ormander condition}). Endowed with the sR distance $\dsR$, $M$ is a complete metric space.

The sR structure $(M,D,g)$ is said to be \emph{equiregular} if the \emph{sR flag} of $D$ is everywhere regular, which means that the subsheafs $D^k$ of $TM$, generated by the Lie brackets of length $k$ of smooth sections of $D$, keep everywhere on $M$ the same dimension. 
Otherwise, the sR structure is said to be \emph{singular} and we denote by $\S$ the \emph{singular set}, that is the closed subset of $M$ where the sR flag of $D$ is not regular. For instance, the Heisenberg (contact) sR structures and the Engel case are equiregular, while the Baouendi-Grushin and Martinet cases are singular.

Let $\mu$ be an arbitrary smooth Borel measure on $M$. Let $\triangle$ be the sR Laplacian associated with the metric $g$ and with the smooth measure $\mu$. The operator $\triangle$ is defined globally (see Appendix \ref{app_sRLaplacian}) and is locally defined as
\begin{equation}\label{deftriangle}
\triangle = - \sum_{i=1}^m X_i^*X_i = \sum_{i=1}^m \left(  X_i^2 + \mathrm{div}_\mu(X_i) X_i \right) 
\end{equation}
where the star is the transpose in $L^2(M,\mu)$ and where $D$ is locally spanned by $m$ vector fields $X_1,\ldots,X_m$.
The H\"ormander condition implies that the operator $\triangle$, defined on $\mathcal{D}(\triangle)=\{f\in L^2(M,\mu)\ \mid\ \triangle f\in L^2(M,\mu)\}$ is subelliptic, nonpositive, selfadjoint, $-\triangle$ has a discrete spectrum $0=\lambda_0<\lambda_1\leq\cdots\leq\lambda_k\leq\cdots$ with $\lambda_k\rightarrow+\infty$ as $k\rightarrow+\infty$, and $\triangle$ generates a strongly continuous contraction semigroup $(e^{t\triangle})_{t\geq 0}$ and a smooth positive symmetric heat kernel, denoted by $e=e_{\triangle,\mu}$.

The set $\Sigma=D^\perp$ (annihilator of $D$) is called the \emph{characteristic manifold} of the sR structure. Outside of $\Sigma$, $\triangle$ is elliptic. 

\medskip

In this article, our main objective is to establish small-time asymptotics of \emph{local and microlocal Weyl laws} in general sR cases, to identify the main terms of the expansions in a geometric, intrinsic way, and in particular to infer the Weyl law by the Karamata tauberian theorem.

\paragraph{Weyl law in sR geometry.}
The \emph{Weyl law} consists of describing the asymptotics of the spectral counting function
\begin{equation}\label{def_spectral_counting_function}
N(\lambda) = \# \{k\in\N \mid \lambda_k \leq \lambda \} \qquad \forall \lambda\in\R
\end{equation}
as $\lambda\rightarrow+\infty$.
In the Riemannian case, the Weyl law stipulates that $\lambda^{-n/2} N(\lambda)\rightarrow \frac{\omega_n\mathrm{Vol}(M)}{(2\pi)^n}$ as $\lambda\rightarrow+\infty$, where $\omega_n$ is the volume of the $n$-dimensional Euclidean ball and $\mathrm{Vol}$ is the Riemannian volume (see, e.g., \cite{Berger_2003, Ho-68} for this very classical result).
In contrast, if the codimension of $D$ in $TM$ is everywhere positive, then 
$\lambda^{-n/2} N(\lambda )\rightarrow+\infty$ as $\lambda\rightarrow+\infty$ (see \cite[Proposition 4.2]{CHT-I}). 
The Weyl law has been investigated in a tremendous number of papers, in the more general setting of hypoelliptic H\"ormander operators with multiple characteristics, in the 70's and 80's. Some of those achievements cover several classes of sR cases. We just cite few of them. 
The Weyl law is established in equiregular sR cases in \cite{Metivier1976}: 
$N(\lambda)\sim\Cst\,\lambda^{\mathcal{Q}^M/2}$ where $\mathcal{Q}^M$ is the Hausdorff dimension of $M$.
The results of \cite{Me-Sj-78, Mohamed_CPDE1983} cover singular sR cases where $\Sigma$ is a smooth conic submanifold of $T^*M\setminus 0$, which is either symplectic, like in the Heisenberg case or in the Baouendi-Grushin case without tangency point, or is not symplectic but the operator $\triangle$ is transversally elliptic along $\Sigma$ (as defined in \cite{BoutetGrigisHelffer_JEDP1976}), meaning that the microlocal quotient $\triangle$ along $\Sigma$ is elliptic, like in the Engel case. The recent paper \cite{ChitourPrandiRizzi_2022} covers some particular classes of almost-Riemannian cases having nice singularities.
These results show that $N(\lambda)\sim\Cst\,\lambda^r\ln\lambda$ for some well identified positive rational number $r$, where the constant is defined as a volume. 
But on the one part, there is no geometric interpretation of the constant, and on the other, for instance the Martinet case where the characteristic manifold is not symplectic is not covered.

Besides, according to the exponential estimates for sR heat kernels (recalled in Appendix \ref{app_sR_kernel}, see \cite{CoulhonSikora_PLMS2008, Je-Sa-86, KusuokaStroock, Saloff-Coste_IMRN1992, Sac-84, Varopoulos}), 
since $M$ is compact, there exist $C_1,C_2>0$ such that, along the diagonal,
$$
\frac{C_1}{\mu( \BsR(q,\sqrt{t}) )} \leq  \, e(t,q,q) \leq \frac{C_2}{\mu( \BsR(q,\sqrt{t}) )} \qquad \forall q\in M\qquad \forall t\in(0,1]
$$
where $\BsR(q,\sqrt{t})$ is the sR ball of center $q$ and radius $\sqrt{t}$.
Integrating over $M$ and applying the Karamata tauberian theorem 
then yields the well known result
\begin{equation}\label{FeffermanPhong}
C_1 \int_M \frac{1}{ \mu( \BsR(q,1/\sqrt{\lambda}) ) } \, d\mu(q) \leq N(\lambda) \leq C_2 \int_M \frac{1}{ \mu( \BsR(q,1/\sqrt{\lambda}) ) } \, d\mu(q) 
\end{equation}
obtained earlier in \cite[Theorem 2]{FeffermanPhong_1981}, which we will refer to as the \emph{Fefferman-Phong estimate}. 
It implies that $N(\lambda)$ is bounded above and below, up to scaling, by $\int_M \lambda^{\mathcal{Q}(q)/2}\, d\mu(q)$, where $\mathcal{Q}^M(q)$, which depends on $q$ in the singular case, is the homogeneous dimension at $q$ (see Appendix \ref{app_sRflag}).
The double inequality \eqref{FeffermanPhong} is general and does not require any assumption on the sR structure $(M,D,g)$. Anyway it does not give an equivalent of $N(\lambda)$ as $\lambda\rightarrow+\infty$.

\medskip

We go much further in the present article. As a consequence of our main results, we compute explicitly the equivalent of $N(\lambda)$ and we also characterize the constants appearing in the equivalent in a geometric way, under the sole assumption that the singular set $\S$ be Whitney stratifiable.
With respect to the above-mentioned existing results, the main novelty is that we deeply exploit the sR context in order to identify the main terms geometrically, using in particular the known concept of \emph{nilpotentization} (recalled in Appendix \ref{app_nilp}) and more generally using the new concept of multiple nilpotentization.

The geometric identification of spectral invariants in sR geometry is a recent subject. We mention however the paper \cite{BealsGreinerStanton_JDG1984} for some early results in this direction, but the paper was written before the real birth of sR geometry (certainly impulsed by the article \cite{Strichartz_JDG1986}). Then, the first results really established in a sR context are due to \cite{Ba-13} and have then been followed by a number of works aiming at exploiting intrinsic concepts in the sR framework, like good notions of curvatures (see \cite{AgrachevBarilariBoscain_book2019} for a recent textbook), in spectral developments.

Our present study is however not restricted to computing the asymptotics of $N(\lambda)$.

Much more generally, we give in this article small-time asymptotic expansions at any order of the local and microlocal Weyl laws.

\paragraph{Spectral asymptotics.}
The \emph{local Weyl law} consists of computing the small-time asymptotics of the function 
$$
t\mapsto\Tr(\mathcal{M}_f \, e^{t\triangle}) = \int_M f(q)\, e(t,q,q) \, d\mu (q) 
= \sum_{j=0}^{+\infty}e^{-\lambda_j t}\int_M f\phi_j^2\, d\mu
$$
for an arbitrary function $f$ on $M$, where $\mathcal{M}_f$ is the operator on $L^2(M,\mu)$ of multiplication by $f$, and the \emph{microlocal Weyl law} consists of computing the function 
$$
t\mapsto\Tr(A \, e^{t\triangle}) = \sum_{j=0}^{+\infty}e^{-\lambda_j t} \langle A\phi_j,\phi_j\rangle_{L^2(M,\mu)}
$$
for an arbitrary pseudo-differential operator $A$ of order $0$ on $M$. 
Here, $(\phi_j)_{j\in\N}$ is an arbitrary orthonormal eigenbasis of $L^2(M,\mu)$ corresponding to the ordered eigenvalues $(\lambda_j)_{j\in\N}$ (i.e., $\triangle\phi_j=-\lambda_j\phi_j$ for every $j\in\N$).

To reach this objective, an instrumental tool is the main result of \cite{CHT_AHL},\footnote{Actually, we have written this paper as a preliminary to the present one.}
recalled for sR heat kernels in Theorem \ref{lemfondamental} in Appendix \ref{app_lemfondam}, establishing a small-time asymptotic expansion for the heat kernel $e$ \emph{in an asymptotic neighborhood of} the diagonal (and not only \emph{along} the diagonal like it was done in \cite{BenArous_AIF1989}). This is the key to treat singular sR cases or to compute the microlocal Weyl law. 

In turn, exploiting Theorem \ref{lemfondamental}, we provide in Theorem \ref{thm_green} in Section \ref{sec_Green} an asymptotic expansion at any order of the Green kernel near the diagonal, thus recovering and improving results of \cite{FeffermanSanchezCalle_AnnMath1986, Sac-84}. As a consequence, we prove that, given any $q_0\in M$ (regular or not), the sR Laplacian $\triangle$ on $C^\infty(M\setminus\{q_0\})$ is essentially selfadjoint if and only if $\mathcal{Q}(q_0)\geq 4$. 

\medskip

The main terms of our small-time asymptotic expansions are identified as geometric objects attached to the sR structure, related to the concept of \emph{nilpotentization}. 
Given any $q\in M$, the nilpotentization $(\widehat{M}^{q},\widehat{D}^{q},\widehat{g}^{q})$ of the sR structure $(M,D,g)$ at $q$ is the Gromov-Hausdorff metric tangent space at $q$ of the complete metric space $(M,\dsR)$. It is a nilpotent homogeneous sR structure. In general, $\widehat{M}^{q}$ is a quotient of a Carnot group. Note that, in contrast to Riemannian geometry where all tangent spaces are isometric, in sR geometry the nilpotentizations of the sR structure $(M,D,g)$ at two different points are not sR isometric in general. This makes the geometric picture much more complex.
In what follows, we denote by $\widehat{\triangle}^{q}$ the selfadjoint sR Laplacian associated with the metric $\widehat{g}^{q}$ and with the nilpotentized measure $\widehat{\mu}^{q}$, and by $\widehat{e}^{q} = e_{\widehat{\triangle}^{q},\widehat{\mu}^{q}}$ the heat kernel associated with $\widehat{\triangle}^{q}$.

\paragraph{Weyl measures.}
Our results put in evidence the role of a new intrinsic measure in sR geometry, that we call the \emph{Weyl measure}, of which there exists a local and a microlocal version. We underline that these measures 
do not depend on the choice of the smooth measure $\mu$. Actually, although the definition \eqref{deftriangle} of the operator $\triangle$ (and thus so does its spectrum) depends on $\mu$, as noticed in \cite{CHT-I}, any self-adjoint second-order differential operator whose principal symbol is the cometric $g^\star$ of the sR structure, whose sub-principal symbol vanishes, and whose first eigenvalue is $\lambda_1=0$, is equal to $\triangle_{g,\mu}$ for some smooth measure $\mu$. 
Hence, in the sequel, all our results depend only on the sR structure (in particular, on the metric $g$) but not on the smooth measure $\mu$, whose choice has thus no importance.

\medskip


Borel measures on $M$ are identified with positive densities on $M$. 
Throughout the paper, given any $f\in L^1(M,\mu)$, we denote by $\int_M f\, d\mu$ the integral of $f$ on $M$ with respect to the smooth measure $\mu$.

The \emph{local Weyl measure} $w_{\triangle}$ is the probability measure on $M$ defined by
$$
\int _M f \, dw_{\triangle} = \lim _{t \rightarrow 0^+} \frac{\int_M f(q)\, e(t,q,q) \, d\mu(q)}{\int_M e(t,q,q) \, d\mu(q)}
$$
for every continuous function on $M$, whenever the limit exists for all such functions. In all sR cases investigated in this article 
(equiregular and stratified singular cases), 
the measure $w_{\triangle}$ exists and is also the weak limit of the sequence of probability measures $\frac{1}{N(\lambda)}\sum _{\lambda_j \leq \lambda } |\phi_j |^2 \, \mu $ (Ces\`aro mean) as $\lambda \rightarrow +\infty$.

The local Weyl measure $w_{\triangle}$ happens to be a \emph{canonical} measure in sR geometry, enjoying the same nice properties as the already known \emph{Popp measure}: like Popp, the local Weyl measure is ``doubly intrinsic" in the equiregular case in the sense that it commutes with nilpotentization, and it is invariant under sR isometries of $M$ (see Section \ref{sec_weyl_intrinsic}).

The \emph{microlocal Weyl law} $W_\triangle$ is the measure defined on the co-sphere bundle $S^\star M$ by
$$
\int _{S^\star M} a \, dW_{\triangle} = \lim_{t\rightarrow 0^+} \frac{\Tr\left(\Op(a) e^{t\triangle}\right)}{\Tr\left(e^{t\triangle}\right)}
$$
for every classical symbol $a$ of order $0$, whenever the limit exists for all such symbols (here, $\Op$ denotes any quantization operator).

General definitions and properties for local and microlocal Weyl measures are provided in Section \ref{sec_Weyl_measures}.

\paragraph{Equiregular sR structures.}
The equiregular case is the subject of Part \ref{part_equiregular} of this article.

\begin{customthm}{I}\label{thmI_equireg} 
{\it 
In the equiregular case,
for every $f\in C^\infty(M)$ there exists $F\in C^\infty(\R)$ such that
$$
\Tr(\mathcal{M}_f \, e^{t\triangle}) = \frac{1}{t^{\mathcal{Q}^M/2}} F(t) 
= \frac{1}{t^{\mathcal{Q}^M/2}} \int_M f(q)\, \widehat{e}^q(1,0,0) \, d\mu(q) + \mathrm{o}\left( \frac{1}{t^{\mathcal{Q}^M/2}}\right)\quad\textrm{as}\ t\rightarrow 0^+
$$
(note that $\widehat{e}^q \, d\mu(q)=e_{\widehat{\triangle}^q,\widehat{\mu}^q} \, d\mu(q)$ does not depend on the smooth measure $\mu$),
where $\mathcal{Q}^M$ is the Hausdorff dimension of the metric space $(M,\dsR)$.
As a consequence, the local Weyl measure $w_\triangle$ exists, is a smooth measure on $M$ and its density with respect to $\mu$ at any point $q\in M$ is 
$$
\frac{dw_{\triangle}}{d\mu}(q) = \frac{\widehat{e}^{q}(1,0,0)}{\int_M \widehat{e}^{q'}(1,0,0)\, d\mu(q')} ,
$$
and the spectral counting function has the asymptotics
$$
N(\lambda) \underset{\ \lambda\rightarrow+\infty}{\sim} \frac{\int _M \widehat{e}^{q}(1,0,0) \, d\mu(q)}{\Gamma(\mathcal{Q}^M/2+1)} \lambda^{\mathcal{Q}^M/2}  .
$$
}
\end{customthm}

The general statement is provided in Theorem \ref{thm_local_weyl_equiregular} in Section \ref{sec_local_weyl_equiregular}, with subsequent remarks.
It can be noticed that, in contrast to the Riemannian case, we always have $\mathcal{Q}^M>n$ as soon as $\mathrm{rank}(D)<n$.
Except the sR interpretation, this theorem was essentially stated in \cite{Metivier1976}.
%

The local Weyl measure coincides with the Popp measure (up to constant scaling) for free nilpotent sR structures and for nilpotent equiregular sR structures of dimension $\leq 5$ except for the bi-Heisenberg case (see Section \ref{sec_compare_weyl_popp}).

In turn, we prove in Section \ref{sec_regdet} that, as a consequence of Theorem \ref{thmI_equireg}, the regularized determinant (defined thanks to the zeta function) exists and is a global spectral invariant.

The general microlocal Weyl law in the equiregular case is provided in Theorem \ref{thm_microlocal_weyl_equiregular} in Section \ref{sec_equiregular_microlocal}, generalizing to a wide extent the result established in \cite{Taylor_CPDE2020} for contact closed manifolds.
We do not give the statement here, but we mention that, in particular, the microlocal Weyl measure $W_\triangle$ exists and is supported on $S\Sigma^{r-1}=S(D^{r-1})^\perp$ where $r$ is the degree of nonholonomy of the sR structure. Its explicit expression is given in Theorem \ref{thm_microlocal_weyl_equiregular}.


\paragraph{Singular sR structures.}
The singular case is the subject of Part \ref{part_singular} of this article. 

The horizontal distribution $D$ is said to be \emph{$\S$-nilpotentizable} if $D$ is locally diffeomorphic to its nilpotentization $\widehat{D}^q$ at every point $q$ of the singular set $\S$.
A smooth submanifold $N$ of $M$ is said to be \emph{equisingular} if the sR flag of $D$ along $N$ and the sR flag of $D$ restricted to $N$ are regular (see Appendix \ref{app_sRflag}).

\begin{customthm}{II.1}\label{thmBII1_sing_nilp} 
{\it 
We assume that the singular set $\S$ (and thus $M$) is Whitney stratified by equisingular smooth submanifolds and that $D$ is $\S$-nilpotentizable. 
Denoting by $\mathcal{Q}^1 < \cdots < \mathcal{Q}^s$ the Hausdorff dimensions of the strata of $M$ (including the equiregular region $M\setminus\S$) and by $m_1, \ldots, m_s$ their respective maximal multiplicities (see Section \ref{sec_thm_multistrates_nilp} for the definition), there exists a Borel measure $\nu$ on $M$ such that, for every $f\in C^\infty(M)$, we have
$$
\Tr(\mathcal{M}_f \, e^{t\triangle}) = \sum_{i=1}^s \sum_{j=0}^{m_i-1} \frac{\vert\ln t\vert^j}{t^{\mathcal{Q}^i/2}}  F_{i,j}(\sqrt{t})
= \left(\int_{\mathcal{N}_s} f \, d\nu\right) \frac{\vert\ln t\vert^{m_s-1}}{t^{\mathcal{Q}^s/2}}  + \mathrm{o}\left( \frac{\vert\ln t\vert^{m_s-1}}{t^{\mathcal{Q}^s/2}}  \right)\quad\textrm{as}\ t\rightarrow 0^+
$$
for some functions $F_{i,j}\in C^\infty(\R)$.
The support $\mathcal{N}_s$ of $\nu$ is the closure of a union of equisingular strata of $M$ of maximal Hausdorff dimension $\mathcal{Q}^s$. The density of $\nu$ is smooth on $\mathcal{N}_s$ and is expressed in terms of multiple nilpotentizations (along the various strata) of the heat kernel $e$. As a consequence, the local Weyl measure exists and is $w_\triangle=\nu/\nu(\mathcal{N}_s)$, and
$$
N(\lambda) \underset{\ \lambda\rightarrow+\infty}{\sim} \frac{\nu(\mathcal{N}_s)}{\Gamma(\mathcal{Q}^s/2+1)}\lambda^{\mathcal{Q}^s}(\ln\lambda)^{m_s-1} .
$$

In the particular case where $\S$ is an equisingular smooth submanifold of $M$ (i.e., $\S$ has a single stratum), denoting by $\mathcal{Q}^\S$ (resp., by $\Qeq$) the Hausdorff dimension of $\S$ (resp., of $M\setminus\S$):
\begin{itemize}
\item If $\mathcal{Q}^\S > \Qeq$ then $\mathcal{N}_s=\S$ and the density of $\nu$ on $\S$ is a ``transverse trace" of the nilpotentized heat kernel along $\S$.
\item If $\mathcal{Q}^\S = \Qeq$ then $\mathcal{N}_s=\S$, the density of $\nu$ on $\S$ is given in terms of a double nilpotentization of the heat kernel (one along $\S$ and the other along $M\setminus\S$), and actually we have an intrinsic two-terms small-time asymptotic expansion of $\Tr(\mathcal{M}_f \, e^{t\triangle})$ (i.e., the two first terms of the asymptotic expansion can be identified geometrically), the dominating term being in $\frac{\vert\ln t\vert}{t^{\mathcal{Q}^\S/2}}$.
\item If $\mathcal{Q}^\S < \Qeq$ then $\mathcal{N}_s=M$: the equiregular part dominates and $\frac{d\nu}{d\mu}(q)=\widehat{e}^q(1,0,0)$ at any $q\in M$ (as in the equiregular case).
\end{itemize}
}
\end{customthm}

The complete statements are given in Section \ref{sec_equisingular_nilp} (see Theorem \ref{thm_onestratum} for the case of one single stratum) and in Section \ref{sec_equisingular_stratified_nilp} (see Theorem \ref{thm_multistrates_nilp} for the case of multiple strata), as well as some examples. 

We also elaborate in Sections \ref{sec_Baouendi-Grushin} and \ref{sec_Martinet} on the Baouendi-Grushin and Martinet cases, giving more details on their two-terms small-time asymptotic expansions.
In turn, we establish in Section \ref{sec_QE} a \emph{Quantum Ergodicity} (QE) result in the Baouendi-Grushin case when $\S$ is connected with at most one tangency point: there exists a density-one subsequence of probability measures $|\phi_{j_k}|^2 \mu$ converging weakly to $w_\triangle = \nu/\nu(\S)$ (see Theorem \ref{thm_QE_Grushin}). This is the first QE result in sR geometry where the limit measure is singular (see \cite{CHT-I} for QE in the 3D contact case).

\medskip

Theorem \ref{thmBII1_sing_nilp} uses the concept of \emph{multiple nilpotentization} (defined in Section \ref{sec_doublenilp}) along \emph{chains of strata} of increasing topological dimensions. Roughly speaking, the double nilpotentization $\widehat{D}^{q_1,q_2}=\widehat{\widehat{D}^{q_1}}^{q_2}$ is the horizontal distribution obtained by first nilpotentizing $D$ at some point $q_1\in\S_1$ and then by nilpotentizing $\widehat{D}^{q_1}$ at some neighbor point $q_2\in\S_2$, where $\S_1$ and $\S_2$ are two equisingular strata of $M$ such that $\dim\S_1<\dim\S_2$ and $\S_1\subset\overline{\S_2}$.

This is thanks to the nilpotentizability assumption, which is defined and commented in \ref{sec_nilpotentizability}, that we can identify geometrically the main terms of the spectral asymptotics.
When nilpotentizability fails, the situation becomes more complex.

\begin{customthm}{II.2}\label{thmII.2_sing_nonnilp} 
{\it 
We assume that the sR structure $(M,D,g)$ is real analytic. 
There exist $k\in\{0,\ldots,n\}$ and a rational number $\gamma\in\Q$, 
only depending on $D$ (not on $g$), satisfying $\gamma\geq\frac{1}{2}\mathcal{Q}^s$ and if $\gamma=\frac{1}{2}\mathcal{Q}^s$ then $k\geq m_s-1$,
and there exist $\ell\in\N^*$
and a Borel measure $\nu$ on $M$ such that, for every $f\in C^\infty(M)$, we have
$$
\mathrm{Tr}(f \, e^{t\triangle}) = \frac{1}{t^\gamma} \sum_{i=0}^k F_i(t^{1/\ell}) \vert\ln t\vert^i  
= \left( \int_{\mathcal{N}} f \, d\nu\right) \frac{\vert\ln t\vert^k}{t^\gamma} + \mathrm{o}\left( \frac{\vert\ln t\vert^k}{t^\gamma} \right)\quad\textrm{as}\ t\rightarrow 0^+
$$
for some $F_0,\ldots,F_k\in C^\infty(\R)$.
The support $\mathcal{N}$ of $\nu$ is an equisingular stratified submanifold of $M$.
As a consequence, the local Weyl measure exists and is $w_\triangle=\nu/\nu(\mathcal{N})$, and 
$$
N(\lambda) \underset{\ \lambda\rightarrow+\infty}{\sim}  \frac{\nu(\mathcal{N})}{\Gamma(\gamma+1)} \lambda^\gamma \ln^k\lambda .
$$
}
\end{customthm}

The dominating term of the small-time asymptotics of the local Weyl law is always greater than or equal to that obtained in Theorem \ref{thmBII1_sing_nilp} in the equisingular stratified nilpotentizable case.

This result corresponds to Theorem \ref{thm_analyticsingularities} in Section \ref{sec_nonnilp_analytic}, where we also give examples showing that $\gamma$ is not an integer in general. In contrast to Theorem \ref{thmBII1_sing_nilp}, here the geometric role of the Hausdorff dimensions of the strata is lost, due to the lack of nilpotentizability. 

\medskip

Interestingly, in Theorems \ref{thmI_equireg}, \ref{thmBII1_sing_nilp} and \ref{thmII.2_sing_nonnilp}, the maximal complexity of the small-time asymptotics is a positive rational power of $\frac{1}{t}$ times an integer power of $\vert\ln t\vert$. There is no $\ln\vert\ln t\vert$ for instance.
To complete the picture, we give in Section \ref{sec_nonnilp_flat} some classes of examples where the sR structure is not real analytic, i.e., the vector fields defining it may involve flat terms. We obtain more ``exotic" Weyl laws, which may 
have an arbitrarily complex transcendence (see Proposition \ref{prop_nonanalytic}).

\medskip

As concerns the microlocal Weyl law, we do not compute it in the singular case, but the method that we develop certainly allows for its estimation. We just mention the general fact that, if $D$ is of codimension $1$ in $TM$, then the microlocal Weyl measure $W_\triangle$ is equal to half of the pullback of $w_\triangle$ by the double covering $S\Sigma\rightarrow M$ which is the restriction to $S\Sigma$ (with $\Sigma=D^\perp$) of the canonical projection of $T^\star M$ onto $M$ (see Section \ref{sec_def_weyl_measures}).

\paragraph{The $(J+K)$-decomposition.}
Let us comment shortly on the main basic technique used for estimating the small-time heat trace asymptotics in singular cases. While the local Weyl law straightforwardly follows from the Lebesgue dominated theorem in the equiregular case, the required domination property fails in general in the singular case (see Section \ref{sec_weyl_integrable}). In order to estimate the trace, we split the ``singular" integral defining the trace as the sum of two integrals:
$$
\Tr(\mathcal{M}_f \, e^{t\triangle}) = \int_M f(q)\, e(t,q,q) \, d\mu (q) = J(t)+K(t)
$$
(called the \emph{$(J+K)$-decomposition}) with
$$
J(t) = \int_{\mathcal{B}(\S,\sqrt{t})} f(q')\, e(t,q',q')\, d\mu(q') \qquad\textrm{and}\qquad
K(t) = \int_{M\setminus \mathcal{B}(\S,\sqrt{t})} f(q')\, e(t,q',q')\, d\mu(q') .
$$
The scaling $\sqrt{t}$ is the right one to use properties of the heat kernel (recall that the heat propagation at small
times is, up to exponentially small terms, located inside balls of radius of the order of $\sqrt{t}$: this follows from the finite speed propagation of singularities for the sR wave equation, see \cite{Me-84}), in particular the fact that the function $t\mapsto(\sqrt{t})^{\mathcal{Q}^M(q)}\, e (t, \delta_{\sqrt{t}}^q(x), \delta_{\sqrt{t}}^q(x') )$ has an asymptotic expansion at any order at $t=0$, whose first term is the nilpotentization $\widehat{e}^q(1,x,x')$ (see Theorem \ref{lemfondamental} in Appendix \ref{app_lemfondam}). This fact immediately yields an asymptotic expansion for $J(t)$. Expanding $K(t)$ is much more difficult and requires to perform multiple nilpotentizations of $e$ along equisingular strata of increasing topological dimensions.
Nilpotentizability ensures that the multiple limits are well defined.
In the analytic non-nilpotentizable case, we use finer stratifications related to subanalytic preparation theorems.

\section{Weyl measures}\label{sec_Weyl_measures}
In this section, we define Weyl measures. They had already been introduced in \cite{CHT-I}, in a different but equivalent way (see Section \ref{sec_karamata}). In Section \ref{sec_def_weyl_measures}, we define local and microlocal Weyl measures in terms of an appropriate limit of the heat kernel. In Section \ref{sec_karamata}, using a classical argument (tauberian theorem), we provide alternative expressions for those measures in terms of an appropriate limit of Ces\`aro means involving an orthonormal basis of eigenfunctions, and we give the asymptotics of the spectral counting function.

Recall that $\triangle$ is defined by \eqref{deftriangle} and that $e=e_{\triangle,\mu}$ is the sR heat kernel, where $\mu$ is an arbitrary smooth measure on $M$.

Given any $f\in C^0(M)$, we denote by $\mathcal{M}_f$ the operator on $L^2(M,\mu)$ of multiplication by $f$. The operator $e^{t\triangle}$ is of trace class and thus the operators $\mathcal{M}_f \, e^{t\triangle}$ and more generally $A\, e^{t\triangle}$, for any bounded operator $A$ on $L^2(M,\mu)$, are of trace class.

\subsection{Preliminary remark on the smooth measure}\label{sec_Weyl_measures_prelim}
%
Let us consider two sub-Riemannian Laplacians $\triangle_\mu$ and $\triangle_\nu$ associated with two different smooth measures $\mu$ and $\nu$ on $M$, but with the same metric $g$.
Assuming that $\nu= h^2\mu$ with $h$ a positive smooth function on $M$, we have $\mathrm{div}_{\nu}(X) = \mathrm{div}_{\mu}(X) + \frac{X(h^2)}{h^2}$, for every vector field $X$, and it follows that 
$$
h\triangle_\nu(\phi) = \triangle_\mu(h\phi) - \phi\triangle_\mu(h) = \triangle_\mu(h\phi) + h^2 \phi\triangle_\nu(h^{-1}) \qquad \forall \phi\in C^\infty(M) .
$$
Defining the isometric bijection $J:L^2(M,\nu)\rightarrow L^2(M,\mu)$ by $J\phi=h\phi$ (gauge transform), we have
\begin{equation}\label{jauge}
J\triangle_\nu J^{-1} = \triangle_\mu - \frac{1}{h} \triangle_\mu(h) \,\mathrm{id} = \triangle_\mu + h \triangle_\nu(h^{-1})
 \,\mathrm{id} = \triangle_\nu +W ,
\end{equation}
i.e., $\triangle_\nu$ is unitarily equivalent to $\triangle_\mu +W$, where $W$ is a bounded operator.

\begin{lemma}\label{lem_prelim_measure}
Given any $f\in C^0(M)$ and any (classical) pseudo-differential operator $A$ of order $0$, we have, as $t\rightarrow 0^+$,
\begin{align}
\Tr\left( e^{t\triangle_\nu} \right) &= \Tr\left( e^{t\triangle_\mu} \right) \left(1+\mathrm{O}(t)\right)  ,  \label{asympttr} \\
\Tr\left( \mathcal{M}_f\, e^{t\triangle_\nu}\right) &= \Tr\left( \mathcal{M}_f\, e^{t\triangle_\mu}\right) +  \mathrm{O}\big( t\, \Tr\big( e^{\frac{t}{2}\triangle_\mu} \big) \big)  ,  \label{traceMf}  \\
\Tr\left( A\, e^{t\triangle_\nu}\right) &= \Tr\left( A\, e^{t\triangle_\mu}\right) + \mathrm{o}\big( \Tr\big( e^{t\triangle_\mu} \big) \big) + \mathrm{O}\big( t\, \Tr\big( e^{\frac{t}{2}\triangle_\mu} \big) \big) .  \label{traceA}
\end{align}
\end{lemma}

\begin{proof}
Denoting by $(\lambda_j(\mu))_{j\in\N}$ (resp., by $(\lambda_j(\nu))_{j\in\N}$) the spectrum of $\triangle_\mu$ (resp., of $\triangle_\nu$), it follows from the Courant-Fischer min-max theorem that
$$
\vert \lambda_j(\nu) - \lambda_j(\mu) \vert \leq \Vert W\Vert_{L(L^2(M,\mu))} \qquad\forall j\in\N ,
$$
and thus $e^{-\lambda_j(\nu)t} = e^{-\lambda_j(\mu)t} (1+\mathrm{O}(t))$ as $t\rightarrow 0^+$, for every $j\in\N$, where the remainder $\mathrm{O}(t)$ is uniform with respect to $j$, and then \eqref{asympttr} follows.

Let us now consider an arbitrary bounded operator $A$ on $L^2(M,\mu)=L^2(M,\nu)$ (because $M$ is compact).
From the relation $e^{-\lambda_j(\nu)t} = e^{-\lambda_j(\mu)t} (1+\mathrm{O}(t))$, we infer \eqref{asympttr}.
Note that $e^{t J\triangle_\nu J^{-1}} = Je^{t\triangle_\nu}J^{-1}$ and, using \eqref{jauge}, we have $\Tr\left( e^{t(\triangle_\mu+W)} \right) = \Tr\left( e^{t\triangle_\mu} \right) \left(1+\mathrm{O}(t)\right)$. 

Now, let $A\in L(L^2(M,\mu))$ be arbitrary. By the Duhamel formula 
$$
Ae^{t(\triangle_\mu+W)} = Ae^{t\triangle_\mu} + \int_0^t Ae^{(t-s)(\triangle_\mu+W)} W e^{s\triangle_\mu}\, ds ,
$$
splitting the latter integral at $t/2$, using the inequality $\Vert BC\Vert_1 \leq \Vert B\Vert_{L(L^2(M,\mu))}\Vert C\Vert_1$ for all operators $B,C\in L(L^2(M,\mu))$ with $C$ of trace class, where $\Vert\ \Vert_1$ is the trace class norm, we obtain
$$
\Tr\left( A \, e^{t (\triangle_\mu+W)} \right) = \Tr\left( A \, e^{t\triangle_\mu} \right) + \mathrm{O}\Big( t\, \Tr\big( e^{\frac{t}{2}\triangle_\mu} \big) + t\, \Tr\big( e^{\frac{t}{2}(\triangle_\mu+W)} \big) \Big)
$$
as $t\rightarrow 0^+$, and thus, using \eqref{asympttr}, $\Tr\left( A e^{t (\triangle_\mu+W)} \right) = \Tr\left( A e^{t\triangle_\mu} \right) + \mathrm{O}\big( t\, \Tr\big( e^{\frac{t}{2}\triangle_\mu} \big) \big)$ as $t\rightarrow 0^+$.
Besides, by \eqref{jauge}, we have $\Tr\left( A e^{t (\triangle_\mu+W)} \right) = \Tr\left( J^{-1}AJ\, e^{t\triangle_\nu}\right)$.
Taking $A=\mathcal{M}_f$, the operator of multiplication by $f$, we have $\mathcal{M}_fJ=J\mathcal{M}_f$, and then \eqref{traceMf}.
Now, taking $A$ an arbitrary pseudo-differential operator of order $0$, $A$ does not a priori commute with $J$, but we have $J^{-1}AJ=A+J^{-1}C$, with $C=[A,J]$ that is a pseudo-differential operator of order $-1$ and thus compact. Let us prove that
\begin{equation}\label{traceC}
\Tr\left( C \, e^{t\triangle_\nu}\right) = \mathrm{o}\big( \Tr\big( e^{t\triangle_\nu} \big) \big)
\end{equation}
as $t\rightarrow 0^+$ (which gives \eqref{traceA}).
Considering an orthonormal eigenbasis $(\phi_j^\nu)_{j\in\N}$ of $L^2(M,\nu)$ corresponding to the ordered eigenvalues $(\lambda_j(\nu))_{j\in\N}$ of $\triangle_\nu$ (i.e., $\triangle_\nu\phi_j^\nu=-\lambda_j(\nu)\phi_j\nu$ for every $j\in\N$, with $\lambda_j(\nu)\leq\lambda_{j+1}(\nu)$), we have
$$
\frac{\Tr\left(C\, e^{t\triangle_\nu}\right)}{\Tr\left(e^{t\triangle_\nu}\right)}
= \frac{ \sum_{j=0}^{+\infty} e^{-\lambda_j(\nu) t} \langle C\phi_j^\nu,\phi_j^\nu\rangle_{L^2(M,\nu)} } { \sum_{j=0}^{+\infty} e^{-\lambda_j(\nu) t} }
$$
Since $C$ is compact, given any $\varepsilon>0$ there exists $N\in\N^*$ such that $\vert \langle C\phi_j^\nu,\phi_j^\nu\rangle_{L^2(M,\nu)} \vert \leq \varepsilon$ for every $j>N$, hence
$$
\frac{\Tr\left(C \, e^{t\triangle_\nu}\right)}{\Tr\left(e^{t\triangle_\nu}\right)} \leq \frac{ \sum_{j=0}^{N} e^{-\lambda_j(\nu) t} \langle C\phi_j^\nu,\phi_j^\nu\rangle_{L^2(M,\nu)} } { \sum_{j=0}^{+\infty} e^{-\lambda_j(\nu) t} } + \varepsilon
$$
and thus $\limsup \frac{\Tr\left(C \, e^{t\triangle_\nu}\right)}{\Tr\left(e^{t\triangle_\nu}\right)} \leq \varepsilon$ as $t\rightarrow 0$. Since $\varepsilon>0$ is arbitrary, \eqref{traceC} follows.
\end{proof}

\subsection{Local and microlocal Weyl measures}\label{sec_def_weyl_measures}


\begin{definition}\label{def_weyl}
The \emph{local Weyl measure} $w_{\triangle}$ is the probability measure on $M$ defined by
\begin{equation}\label{def_local_weyl_measure}
\int_M f\, dw_{\triangle} = \lim_{t\rightarrow 0^+} \frac{\Tr\left(\mathcal{M}_f\, e^{t\triangle}\right)}{\Tr\left(e^{t\triangle}\right)} = \lim_{t\rightarrow 0^+} \frac{\int_M f(q)\, e(t,q,q) \, d\mu(q)}{\int_M e(t,q,q) \, d\mu(q)}
\end{equation}
for every continuous function $f$ on $M$, whenever the limit exists for all such functions.

The \emph{microlocal Weyl measure} $W_{\triangle}$ is the probability measure on $S^\star M$ defined by
\begin{equation}\label{def_microlocal_weyl_measure}
\int _{S^\star M} a \, dW_{\triangle} = \lim_{t\rightarrow 0^+} \frac{\Tr\left(\Op(a)\, e^{t\triangle}\right)}{\Tr\left(e^{t\triangle}\right)}
\end{equation}
for every classical symbol $a$ of order $0$, whenever the limit exists
for all such symbols. Here, $\Op$ denotes any quantization operator (see Lemma \ref{lem_weyl_properties} hereafter).
\end{definition}

The fact that the limits defining the microlocal Weyl measures are indeed probabilities measures follows from the use of the Weyl quantization (see the next lemma).

\begin{lemma}\label{lem_weyl_properties}
When they exist, the Weyl measures have the following properties:
\begin{enumerate}
\item The local Weyl measure $w_{\triangle}$ is the pushforward of the microlocal Weyl measure $W_{\triangle}$ under the canonical projection $\pi:S^*M\rightarrow M$, i.e., $\pi_* W_{\triangle} = w_{\triangle} $.
\item The microlocal Weyl measure $W_{\triangle}$ does not depend on the quantization, is even with respect to the canonical involution of $S^*M$, and $\mathrm{supp}(W_\triangle)\subset S \Sigma $ where $\Sigma=D^\perp$. 

\item The measures $w_{\triangle}$ and $W_{\triangle}$ only depend on the sR structure (in particular, on the metric $g$), but they do not depend on the measure $\mu$. Moreover, they are invariant under sR isometries of $M$. 
\end{enumerate}
\end{lemma}

\begin{proof}
\begin{enumerate}
\item This is obvious.
\item Two quantizations of a bounded operator $B$ on $L^2(M,\mu)$ differ by a compact operator $C$ (more precisely, a pseudo-differential operator of order $-1$) and we know by \eqref{traceC} that $\frac{\Tr\left(C e^{t\triangle}\right)}{\Tr\left(e^{t\triangle}\right)} \rightarrow 0$ as $t\rightarrow 0^+$. 

To prove evenness with respect to involution, we use the Weyl quantization $\OpW$.
Considering a real-valued classical symbol $a=a(x,p)$ of order $0$, compactly supported with respect to $x$ in some chart $V\subset S^*M$,  
the density of the Schwartz kernel of $\OpW(a)$ with respect to the Lebesgue measure $m$ (see Appendix \ref{appendix_Schwartz}) is given in this chart by
$$
[ \OpW(a) ]_m(x,y) = \frac{1}{(2\pi)^n}\int_{\R^n\times\R^n} e^{i(x-y).p} a\left(\frac{x+y}{2},p\right)\, dy\, dp ,
$$
and thus
$$
\Tr\left(\OpW(a) \, e^{t\triangle}\right) = \frac{1}{(2\pi)^n}\int_{\R^n\times\R^n\times\R^n} e^{i(x-y).p} a\left(\frac{x+y}{2},p\right) e_{\triangle,m}(t,x,y)\, dx\, dy\, dp
$$
Since $a$ is real-valued, $\OpW(a)$ is selfadjoint and thus $\Tr\left(\OpW(a) \, e^{t\triangle}\right)$ is a real number and is therefore equal to its complex conjugate. Hence the above integral is also equal to 
$$
\int_{\R^n\times\R^n\times\R^n} e^{-i(x-y).p} a\big(\frac{x+y}{2},p\big) e_{\triangle,m}(t,x,y)\, dx\, dy\, dp
$$
because $a$ and $e_{\triangle,m}$ are real-valued. Making the change of variable $p\mapsto -p$, we conclude that $\Tr\left(\OpW(a) \, e^{t\triangle}\right)=\Tr\left(\OpW(\tilde a) \, e^{t\triangle}\right)$ where $\tilde a(x,p)=a(x,-p)$. This gives the second claim.

The third claim has been proved in \cite[Proposition 4.3]{CHT-I}.

\item 
The first fact follows from Lemma \ref{lem_prelim_measure}.

An sR isometry of $M$ preserves the sR metric, but not the volume $\mu$ in general; however, $w_{\triangle}$ and $w_{\triangle}$ do not depend on $\mu$. The second claim follows.
\end{enumerate}
\end{proof}

\begin{remark}\label{rem_double_covering}
In particular, as observed in \cite[Corollary 4.2]{CHT-I}, if the horizontal distribution $D$ is of codimension $1$ in $TM$, and if the local Weyl measure $w_\triangle$ exists, then the microlocal Weyl measure $W_\triangle$ exists and is equal to half of the pullback of $w_\triangle$ by the double covering $S\Sigma \rightarrow M$ which is the restriction of the canonical projection of $T^\star M $ onto $M$.
\end{remark}

\paragraph{Generalization: non compact manifold, manifold with boundary.}
We can generalize the definition of Weyl measures to the case where $M$ is not compact and/or $M$ has a boundary.
In all cases, we choose a nonpositive selfadjoint extension $(\triangle,\mathcal{D}(\triangle))$ of the second-order operator $\triangle$ defined by \eqref{deftriangle} (see \cite[Section 3.1]{CHT_AHL} for the geometric treatment of Dirichlet and Neumann boundary conditions).


The local and microlocal Weyl measures can always be defined locally. 
Indeed, by the hypoelliptic Kac's principle (see \cite[Theorem 3.1]{CHT_AHL}), the small-time asymptotics of hypoelliptic heat kernels is purely local, i.e., ``does not feel the boundary". This gives a way to define the Weyl measures (far from the boundary when $M$ has a boundary): take any relatively compact open subset $\Omega$ of $M$ with a smooth boundary. We consider the heat kernel $e_\Omega$ corresponding to the Dirichlet operator $\triangle_\Omega$ defined as the restriction of $\triangle$ to $\Omega$ with Dirichlet boundary conditions. Then, according to \cite[Theorem 3.1]{CHT_AHL}, we have $e_\Omega(t,q,q') = e(t,q,q')+\mathrm{O}(t^\infty)$ as $t\rightarrow 0^+$, uniformly with respect to $(q,q')\in\Omega^2$. Now, on $\Omega$ we can define the Weyl measure.

Making it global is a matter of choosing a normalization. There exists a canonical normalization when $w_\triangle(M)<+\infty$ (equivalently, $W_\triangle(M)<+\infty$), and the above definition of probability Weyl measures still makes sense.
When the total Weyl volume is not finite, a canonical normalization may not exist, however since any positive distribution is a Radon measure, the following construction always makes sense: assuming that 
$$
\Tr\left(\mathcal{M}_f e^{t\triangle}\right) =  \int_M f(q)\,  e(t,q,q) \, d\mu(q) = \varphi(t) Lf
$$
as $t\rightarrow 0^+$, where $L$ is a linear bounded form on the set of compactly supported continuous functions on $M$, we define $w_\triangle(f)=Lf$ (see also Theorem \ref{thm_weyl_measures_equiv} hereafter). The microlocal Weyl measure is defined similarly. Of course, in this more general case, they are not probability measures. Also, the main difficulty is in establishing such asymptotics.

\subsection{Equivalent expression and asymptotics of the spectral counting function}\label{sec_karamata}
Let $(\phi_k)_{k\in\N}$ be an orthonormal eigenbasis of $L^2(M,\mu)$ corresponding to the ordered eigenvalues $(\lambda_k)_{k\in\N}$, i.e., $\triangle\phi_k=-\lambda_k\phi_k$ for every $k\in\N$.

By a well known argument using the Karamata tauberian theorem, we have the following alternative expressions for the Weyl measures and the following asymptotics for the spectral counting function $N(\lambda)$ defined by \eqref{def_spectral_counting_function}.
Recall that a function $\chi:[0,+\infty)\rightarrow(0,+\infty)$ is said to be \emph{slowly varying} at $+\infty$ if for every fixed $s>0$ we have $\frac{\chi(s\lambda)}{\chi(\lambda)}\rightarrow 1$ as $\lambda\rightarrow+\infty$.
It is said to be slowly varying at $0$ if $\chi(1/\lambda)$ is slowly varying at $+\infty$.

\begin{theorem}\label{thm_weyl_measures_equiv}
Assume that there exist $\rho\in[0,+\infty)$ and a function $\chi:[0,+\infty)\rightarrow(0,+\infty)$ that is slowly varying at infinity and a linear bounded form $L$ on $C^0(M)$ (resp., on $\mathcal{S}^0(M)$) such that
\begin{equation}\label{assumption_asympt_trace}
\Tr\left(\mathcal{M}_f \, e^{t\triangle}\right) \underset{t\rightarrow 0^+}{\sim} \frac{1}{t^\rho} \chi(1/t) Lf
\qquad \left(\textrm{resp.,}\quad \Tr\left(\Op(a) e^{t\triangle}\right) \underset{t\rightarrow 0^+}{\sim} \frac{1}{t^\rho} \chi(1/t) La \right)
\end{equation}
for every nontrivial nonnegative continuous function $f$ on $M$ (resp., for every nontrivial nonnegative symbol $a$ of order $0$). 
Then the local (resp., microlocal) Weyl measure exists, and $w_\triangle = \frac{L}{L1}$ (resp., $W_\triangle = \frac{L}{L1}$) and we also have
\begin{equation}\label{def_local_cesaro-weyl_measure}
\int _M f \, dw_{\triangle} = \lim _{\lambda \rightarrow +\infty }\frac{1}{N(\lambda)}\sum _{\lambda_k \leq \lambda } \int_M f |\phi_k |^2 \, d\mu = \lim _{\lambda \rightarrow +\infty }\frac{1}{N(\lambda)}\sum _{\lambda_k \leq \lambda } \langle \mathcal{M}_f\phi_k,\phi_k\rangle_{L^2(M,\mu)} 
\end{equation}
for every continuous function $f$ on $M$; in other words, $w_{\triangle}$ is the weak limit of the sequence of probability measures $\frac{1}{N(\lambda)}\sum _{\lambda_k \leq \lambda } |\phi_k |^2 \, \mu$ (Ces\`aro mean) as $\lambda \rightarrow +\infty$.
Respectively, the microlocal Weyl measure $W_{\triangle}$ is given by
\begin{equation}\label{def_microlocal_cesaro-weyl_measure}
\int _{S^\star M} a \, dW_{\triangle} = \lim _{\lambda \rightarrow +\infty }\frac{1}{N(\lambda)}\sum _{\lambda_k \leq \lambda } \langle \Op (a)\phi_k , \phi_k \rangle_{L^2(M,\mu)} 
\end{equation}
for every classical symbol $a$ of order $0$, where $\Op$ denotes any quantization operator.
Moreover, 
$$
N(\lambda) \underset{\lambda\rightarrow +\infty}{\sim} \frac{L1}{\Gamma(\rho+1)} \lambda^\rho\chi(\lambda) .
$$
\end{theorem}

\begin{proof}
The argument is standard.
Formally, passing from the definition in terms of small-time asymptotics of the heat kernel as in \eqref{def_local_weyl_measure} and \eqref{def_microlocal_weyl_measure} to the definition in terms of asymptotic eigenvalues as in \eqref{def_local_cesaro-weyl_measure} and \eqref{def_microlocal_cesaro-weyl_measure} can be done by a Laplace transform. The precise relationship is well known in the existing literature and can be achieved thanks to the Karamata tauberian theorem (see \cite{Ka-30}). 
We recall hereafter the more general version of \cite[Chapter XIII, Section 5, Theorem 2]{Feller2}.

\begin{lemma}[\cite{Feller2}]\label{thm_karamata_feller}
Let  $F:[0,+\infty)\rightarrow\R$ be a nondecreasing function which is of locally bounded variation (i.e., the Radon measure $dF$ is nonnegative and locally finite) and whose Laplace-Stieltjes transform $\mathcal{L}(F)(t) = \int_{0}^{+\infty} e^{-t\lambda} \, dF(\lambda) $ is well defined for every $t>0$. 
Let $\chi:(0,+\infty)\rightarrow\R$ be a function that is slowly varying at $+\infty$. Let $\rho\in [0,+\infty)$. 
\begin{itemize}[leftmargin=*]
\item We have
$$
\mathcal{L}(F)(t)  \underset{t\rightarrow 0^+}{\sim} \frac{1}{t^{\rho}} \chi(1/t)
\quad\Longleftrightarrow\quad
F(\lambda) \underset{\lambda\rightarrow +\infty}{\sim}  \frac{1}{\Gamma(\rho+1)} \lambda^\rho \chi(\lambda) .
$$
\item Note that $\frac{F(\lambda)}{\mathcal{L}(F)(1/\lambda)} \rightarrow \frac{1}{\Gamma(\rho+1)}$ as $\lambda\rightarrow +\infty$ or $0$. For $\rho=+\infty$, the result still makes sense: if there exists $s>1$ such that either $\frac{\mathcal{L}(F)(st)}{\mathcal{L}(F)(t)}\rightarrow 0$ as $t\rightarrow 0$, or $\frac{F(s\lambda)}{F(\lambda)}\rightarrow +\infty$ as $\lambda\rightarrow +\infty$, then $\frac{F(\lambda)}{\mathcal{L}(F)(1/\lambda)} \rightarrow 0$ as $\lambda\rightarrow +\infty$.
\item These results are still valid when exchanging $0$ and $+\infty$ in the above limits, i.e., $\lambda\rightarrow 0$ and $t\rightarrow +\infty$.
\end{itemize}
\end{lemma}


Let us establish \eqref{def_local_cesaro-weyl_measure}.
Let $f$ be a nontrivial nonnegative continuous function on $M$. The operator $\mathcal{M}_fe^{t\triangle}$ is of trace class. We define $F_f(\lambda) = \sum_{\lambda_k\leq\lambda} \int_M f\vert\phi_k\vert^2\, d\mu$, for every $\lambda\in\R$. The function $F_f$ is nondecreasing, is of local bounded variation, and its Laplace-Stieltjes transform is
$$
\mathcal{L}(F_f)(t) = \int_0^{+\infty} e^{-t\lambda}\, dF_f(\lambda)
= \sum_{k=0}^{+\infty} e^{-\lambda_k t}\int_M f\vert\phi_k\vert^2\, d\mu  
= \int_M f(q) e(t,q,q)\, d\mu(q) 
= \Tr(\mathcal{M}_f \, e^{t\triangle}) 
$$
because $e(t,q,q')=\sum_{k=0}^{+\infty} e^{-\lambda_k t} \phi_k(q) \overline{\phi_k(q')}$.
Note that $N(\lambda)=F_1(\lambda)$.
By the assumption \eqref{assumption_asympt_trace}, we have $\mathcal{L}(N)(t)\sim \frac{1}{t^\rho} \chi(1/t) L1$ as $t\rightarrow 0^+$. The asymptotics of $N(\lambda)$ then follows from Lemma \ref{thm_karamata_feller}.

Besides, by the assumption \eqref{assumption_asympt_trace}, the limit \eqref{def_local_weyl_measure} exists and is equal to $\frac{Lf}{L1}$, hence 
$$
\int_M f\, dw_{\triangle} = \lim_{t\rightarrow 0^+} \frac{\mathcal{L}(F_f)(t)}{\mathcal{L}(F_1)(t)} = \frac{Lf}{L1} .
$$
Applying Lemma \ref{thm_karamata_feller}, the limit \eqref{def_local_cesaro-weyl_measure} exists and is equal to $\frac{Lf}{L1}$, hence
$$
\lim_{\lambda\rightarrow +\infty} \frac{F_f(\lambda)}{F_1(\lambda)} = \frac{Lf}{L1} .
$$
The conclusion follows. The proof is similar for the microlocal Weyl measure.
\end{proof}

\begin{remark}
In \cite{CHT-I}, the local and microlocal Weyl measures have been defined  by \eqref{def_local_cesaro-weyl_measure} and \eqref{def_microlocal_cesaro-weyl_measure}.
Note that these expressions do not depend on the choice of the orthonormal basis of eigenfunctions.
\end{remark}

\begin{remark}
The assumption \eqref{assumption_asympt_trace} is satisfied in all cases treated in this paper (equiregular and stratified singular cases).
We are not aware of any example where it would not be satisfied. 
Such an example would have to be a singular case where the singular set is not stratifiable.
\end{remark}

\subsection{An additional comment}\label{sec_additional}


Assuming that $\mu=h\nu$ with $h$ a positive smooth function on $M$, using \eqref{density_nilp} (definition of the nilpotentization of a measure), 
we have
\begin{equation}\label{add_densite_h}
h(q) = \frac{d\mu}{d\nu}(q) = \frac{\widehat{\mu}^{q}}{\widehat{\nu}^{q}} = \frac{\widehat{\mu}^{q}\big(\hatBsR^q(0,1)\big)}{\widehat{\nu}^{q}\big(\hatBsR^q(0,1)\big)}
\end{equation}
and $\widehat{\mu}^{q} = h(q) \widehat{\nu}^{q}$, and hence $h(q) e_{\widehat{\triangle}^{q},\widehat{\mu}^{q}} = e_{\widehat{\triangle}^{q},\widehat{\nu}^{q}}$ (see Appendix \ref{appendix_Schwartz}), yielding in particular that 
\begin{equation}\label{add_kernelnilp_mu_nu}
e_{\widehat{\triangle}^{q},\widehat{\mu}^{q}}(t,x,x')\, d\mu(q) = e_{\widehat{\triangle}^{q},\widehat{\nu}^{q}}(t,x,x')\, d\nu(q) \qquad \forall t>0\qquad \forall x,x'\in\widehat{M}^q\qquad \forall q\in M .
\end{equation}
In particular, $e_{\widehat{\triangle}^{q},\widehat{\mu}^{q}}(1,0,0)\, d\mu(q)$ does not depend on the smooth measure $\mu$.

\begin{definition}\label{def_rho}
Given an arbitrary smooth measure $\mu$ on $M$, we define the positive measure $\rho$ on $M$ as the measure whose density with respect to $\mu$ is the function $q\mapsto e_{\widehat{\triangle}^{q},\widehat{\mu}^{q}}(1,0,0)$. The measure $\rho$ does not depend on $\mu$.
\end{definition}

The function $q\mapsto e_{\widehat{\triangle}^{q},\widehat{\mu}^{q}}(1,0,0)$ may fail to be continuous or even locally integrable at singular points (for instance, it is not locally integrable near $\S$ in the Baouendi-Grushin and Martinet cases). 

We are going to see in Part \ref{part_equiregular} that, in the equiregular case, the function $q\mapsto e_{\widehat{\triangle}^{q},\widehat{\mu}^{q}}(1,0,0)$ is smooth and that the measure $\rho$ coincides, up to constant scaling, with the local Weyl measure.

In general, if the function $q\mapsto e_{\widehat{\triangle}^{q},\widehat{\mu}^{q}}(1,0,0)$ is locally integrable at $q$, then, extending the process of nilpotentization of measures to measures having a locally integrable density, we obtain that $\widehat{\rho}^q=e_{\widehat{\triangle}^{q},\widehat{\mu}^{q}}(1,0,0)\widehat{\mu}^q$ for every smooth measure $\mu$ and thus, in particular, $e_{\widehat{\triangle}^{q},\widehat{\rho}^{q}}(1,0,0)=1$.

\section{Green kernel estimates and essential selfadjointness}\label{sec_Green}
Given any $\lambda\in\C\setminus\mathrm{Spec}(\triangle)$, the resolvent kernel $G_\lambda$ is defined as the Schwartz kernel of the resolvent $(\lambda\,\mathrm{id}-\triangle)^{-1}$. 
Given any $q_0\in M$, the Green kernel $\widehat{G}^{q_0}$ of $\widehat{\triangle}^{q_0}$ (proved to exist hereafter) is defined as the Schwartz kernel of $(\widehat{\triangle}^{q_0})^{-1}$ (see \cite[Chapter 13]{Grigor'yan}).

\begin{theorem}\label{thm_green}
For every $q_0\in M$ (regular or not), the Green kernel $\widehat{G}^{q_0}$ exists and we have the asymptotic expansion at any order $N\in\N^*$ in $C^\infty(\R^n\setminus\{0\})$
$$
G_\lambda(q_0,\delta^{q_0}_\varepsilon(x)) = \varepsilon^{2-\mathcal{Q}(q_0)} \left( \widehat{G}^{q_0}(0,x)
+ \sum_{i=1}^N \varepsilon^i b_i^{q_0}(x) + \mathrm{o}(\vert\varepsilon\vert^N) \right)
$$
as $\varepsilon\rightarrow 0$, $\varepsilon\neq 0$, where the functions $b_i^{q_0}$ are smooth. 
\end{theorem}

\begin{proof}
Let us first assume that $\mathrm{Re}(\lambda)>0$.
Since $(\lambda\,\mathrm{id}-\triangle)^{-1} = \int_0^{+\infty} e^{-\lambda t} e^{t\triangle}\, dt$, we have $G_\lambda(q_0,q) = \int_0^{+\infty} e^{-\lambda t} e(t,q_0,q)\, dt$ and thus, taking some $c>0$ arbitrary,
\begin{equation}\label{green2int}
G_\lambda(q_0,\delta^{q_0}_\varepsilon(x)) = 
\int_c^{+\infty} e^{-\lambda t} e(t,q_0,\delta^{q_0}_\varepsilon(x))\, dt +
\varepsilon^{2-\mathcal{Q}(q_0)} \int_0^{c/\varepsilon^2} e^{-\lambda\varepsilon^2 t} \varepsilon^{\mathcal{Q}(q_0)}e(\varepsilon^2t,q_0,\delta^{q_0}_\varepsilon(x))\, dt
\end{equation}
By boundedness of the heat kernel for $t\geq c$, the first integral at the right-hand side of \eqref{green2int} is a smooth function of $\varepsilon$. Let us treat the second integral.
Theorem \ref{lemfondamental} in Appendix \ref{app_lemfondam} gives an expansion of $\varepsilon^{\mathcal{Q}(q_0)}e(\varepsilon^2t,q_0,\delta^{q_0}_\varepsilon(x))$ with respect to $\varepsilon$, at any order, whose first term is $\widehat{e}^{q_0}(t,0,x)$, but the difficulty here is that uniformity with respect to $t$ is ensured only on any compact interval of $(0,+\infty)$. To overcome this difficulty and apply the dominated convergence theorem, we use the exponential estimates \eqref{exp_estimates} of the heat kernel (recalled in Appendix \ref{app_sR_kernel}), which imply, since $\dsR(0,\delta_\varepsilon(x)) = \varepsilon\, \dsR(0,x)$ in the local chart, taking $c>0$ small enough and $\varepsilon^2 t\leq c$, that
$$
\varepsilon^{\mathcal{Q}(q_0)}e(\varepsilon^2t,q_0,\delta^{q_0}_\varepsilon(x)) \leq \frac{C}{t^{\mathcal{Q}(q_0)}} e^{-\dsR(0,x)^2/Ct}
$$
for some constant $C>0$, for every $t\in(0,c/\varepsilon^2]$ and for every $x\in K$, with $K$ compact subset of $\R^n$. Now, when $K$ does not contain $0$, we obtain a uniform domination term. 
As concerns the subsequent terms in the expansion, it follows from the analysis performed in \cite[Section 6.2]{CHT_AHL} that the functions $f_i^{q_0}$ in the expansion \eqref{complete_expansion} of Theorem \ref{lemfondamental} enjoy similar exponential estimates. 
The result follows, for $\mathrm{Re}(\lambda)>0$.

Now, let us take any $\lambda\in\C\setminus\mathrm{Spec}(\triangle)$. Since $e(t,q_0,q)=\sum_{j=0}^{+\infty} e^{-\lambda_j t}\phi_j(q_0)\phi_j(q)$, we split the sum in two parts, with a first finite sum $\sum_{j=0}^k$ for some $k\in\N$ large enough, and an infinite sum $\sum_{j=k+1}^{+\infty}$. The heat kernel $e_{k+1}$ defined from the latter sum is treated as above, while the first finite sum, multiplied by $e^{-\lambda t}$ and integrated with respect to $t$, gives a smooth function because all eigenfunctions $\phi_j$ have a polynomial growth (this fact follows by rough estimates of the Weyl law and by Sobolev embedding estimates as developed in \cite{CHT_AHL}).
\end{proof}

In Theorem \ref{thm_green}, not only we recover the main results of \cite{FeffermanSanchezCalle_AnnMath1986, Sac-84}, according to which the Green kernel $G(q_0,q)$ is bounded above and below, up to scaling, by $r^2/\mu(\BsR(q_0,r))$ with $r=\dsR(q_0,q)$, but we also improve them since we obtain an equivalent (and even, a full expansion) in terms of the nilpotentization. Moreover, thanks to Theorem \ref{lemfondamental}, the above argument is short. 

\paragraph{Application: essential selfadjointness of sR Laplacians.}
Given any $q_0\in M$, let us study the possibility of having different self-adjoint extensions of the Laplacian on $M\setminus\{q_0\}$. This problem has been studied in \cite{yCdV-82} in the Riemannian case and in \cite{AdamiBoscainFranceschiPrandi} in the 3D sR case. Hereafter we extend their result to the general sR case, obtaining the following new and simple result.

\begin{corollary}
Given any $q_0\in M$ (regular or not), the operator $\triangle$ on $C^\infty(M\setminus\{q_0\})$ is essentially selfadjoint if and only if $\mathcal{Q}(q_0)\geq 4$. 
\end{corollary}

Hence, the operator $\triangle$ on $C^\infty(M\setminus\{q_0\})$ is not essentially selfadjoint only in the Riemannian case in dimension $\leq 3$ and, in the non-Riemannian sR case, only in the Baouendi-Grushin case without tangency point if $q_0$ belongs to the singular set. 

\begin{proof}
We first note that $\triangle$ is not essentially selfadjoint if and only if for every $\lambda\in\C$ there exists $u\in L^2(M,\mu)$ such that $(\lambda\,\mathrm{id}-\triangle)u=0$ on $M\setminus\{q_0\}$ in the sense of distributions. In that case, $(\lambda\,\mathrm{id}-\triangle)u$ must be a finite linear combination of the Dirac $\delta_{q_0}$ and its derivatives, but there cannot be any derivatives because they are not in $H^{-2}(M,\mu)$, 
hence $(\lambda\,\mathrm{id}-\triangle)u=\alpha\delta_{q_0}$ for some $\alpha\in\R$ and thus $u=\alpha G_\lambda(q_0,\cdot)$. Now, by Theorem \ref{thm_green}, we have $G_\lambda(q_0,\cdot)\in L^2(M,\mu)$ if and only if $\mathcal{Q}(q_0)\leq 3$. Indeed, taking sR polar coordinates, this condition is equivalent to the convergence of the integral $\int_0^1 r^{3-\mathcal{Q}(q_0)}\, dr$.
The result thus follows if $\dim M\geq 3$.
\end{proof}

\begin{remark}
Another argument consists of writing
$$
\Vert G_\lambda(q_0,\cdot)\Vert_{L^2(M,\mu)}^2 = \int_{[0,+\infty)^2} {\mkern-20mu} e^{-\lambda(t+s)}e(t,q_0,q) \, e(s,q,q_0)\, dt\, ds
= \int_{[0,+\infty)^2} {\mkern-20mu}  e^{-\lambda(t+s)}e(t+s,q_0,q_0)\, dt\, ds
$$
and then of observing that the convergence of this integral only depends on the small-time asymptotic expansion of the heat kernel along the diagonal, which is $e(t,q_0,q_0)\sim \Cst/t^{\mathcal{Q}(q_0)/2}$ as $t\rightarrow 0^+$.
\end{remark}


\part{Equiregular case}\label{part_equiregular}
We assume throughout this part that the sR structure $(M,D,g)$ is equiregular.
This implies in particular that the Hausdorff dimension $\mathcal{Q}^M(q)=\mathcal{Q}^M$ and the degree of nonholonomy $r(q)=r$ are constant for $q\in M$.


Recall that $e=e_{\triangle,\mu}$ is the heat kernel on $(0,+\infty)\times M\times M$ associated with the sR Laplacian $\triangle$ defined by \eqref{deftriangle} and with the smooth measure $\mu$ on $M$. For every $q\in M$, $\widehat{e}^{q} = e_{\widehat{\triangle}^{q},\widehat{\mu}^{q}}$ is the heat kernel (nilpotentized at $q$) on $(0,+\infty)\times\widehat{M}^{q}\times\widehat{M}^{q}$ associated with $\widehat{\triangle}^{q}$ and with the measure $\widehat{\mu}^{q}$ on $\widehat{M}^{q}$ (see Appendices \ref{app_sRLaplacian}, \ref{app_nilp} and \ref{appendix_Schwartz}).

\section{Local Weyl law}

\subsection{Local Weyl law and local Weyl measure}\label{sec_local_weyl_equiregular}

\begin{theorem}\label{thm_local_weyl_equiregular}
For every $f\in C^\infty(M)$, we have
\begin{equation} \label{asympt_heat_equiregular} 
\Tr(\mathcal{M}_f \, e^{t\triangle}) 
= \int_M f(q)\, e(t,q,q) \, d\mu (q) = \frac{1}{t^{\mathcal{Q}^M/2}} F(t) \qquad\forall t>0 
\end{equation}
for some $F\in C^\infty(\R)$, with
\begin{equation} \label{asympt_heat_equiregular_F(0)} 
F(0) = \int _M f(q)\, \widehat{e}^{q}(1,0,0) \, d\mu(q) .
\end{equation}
The local Weyl measure $w_{\triangle}$ (defined by \eqref{def_local_weyl_measure}) exists, is a smooth measure on $M$, and its density with respect to 
$\mu$ is given by
\begin{equation}\label{weylequiregular}
\frac{dw_{\triangle}}{d\mu}(q) = \frac{\widehat{e}^{q}(1,0,0)}{\int_M \widehat{e}^{q'}(1,0,0)\, d\mu(q')}  \qquad \forall q\in M.
\end{equation}
In other words, we have $w_\triangle=\frac{\rho}{\rho(M)}$ (see Definition \ref{def_rho}).
Moreover,
\begin{equation}\label{asympt_N_equireg}
N(\lambda) \underset{\ \lambda\rightarrow+\infty}{\sim} \frac{\int _M \widehat{e}^{q}(1,0,0) \, d\mu(q)}{\Gamma(\mathcal{Q}^M/2+1)} \lambda^{\mathcal{Q}^M/2} .
\end{equation}
\end{theorem}

\begin{proof}
By Theorem \ref{lemfondamental} in Appendix \ref{app_lemfondam}, 
$t^{\mathcal{Q}^M/2} e(t,q,q)$ converges to $ \widehat{e}^{q}(1,0,0)$ as $t\rightarrow 0^+$, uniformly with respect to $q$
because we are in the equiregular case, hence \eqref{asympt_heat_equiregular} and \eqref{asympt_heat_equiregular_F(0)} follow by the dominated convergence theorem.
For the Taylor expansion, it suffices to note that, by \eqref{expansion_along_diagonal},
$$
t^{\mathcal{Q}^M/2} e(t,q,q) = \widehat{e}^{q}(1,0,0) + c_1(q) t + \cdots + c_N(q) t^N + \mathrm{o}(t^{N})
$$
is a smooth function of $t$.
The last part follows from Theorem \ref{thm_weyl_measures_equiv} (in Section \ref{sec_Weyl_measures}).
\end{proof}

Note that, using \eqref{add_kernelnilp_mu_nu} and \eqref{weylequiregular}, we recover the fact that $w_\triangle$ does not depend on $\mu$.

The coefficient $F(0)$ of the leading term in the asymptotic expansion given by \eqref{asympt_heat_equiregular} as $t\rightarrow 0^+$ allows us to identify the local Weyl measure. 
Note that, to obtain the expansion at the first order
$\Tr(\mathcal{M}_f \, e^{t\triangle}) = \frac{1}{t^{\mathcal{Q}^M/2}} ( F(0) + \mathrm{o}(1))$ as $t\rightarrow 0^+$, it is only required that $f$ be continuous.

The coefficients of the subsequent terms in the expansion are of the form $\int_M c_i f\, d\mu$, where the smooth functions $c_i$ can actually be expressed as iterated convolutions involving derivations and the nilpotentized heat kernel (see \cite{CHT_AHL}).

\begin{remark}
With the above theorem, we recover and generalize results of \cite{BenArous_AIF1989,Metivier1976} that are established when $M=\R^n$ and $\mu$ is the Lebesgue measure.
In this case, note that a probabilistic expression of this density is computed in \cite[Theorem 7.15]{BenArous_AIF1989}, involving the structure constants of the Lie algebra structure. In \cite{Metivier1976}, the author proves that this density is equal to the limit of $\frac{1}{\lambda^{\mathcal{Q}^M}} \mathfrak{e}(\lambda,q,q) $ as $\lambda\rightarrow+\infty$, where $\mathfrak{e}(\lambda,q,q')$ is the kernel of 
the spectral resolution of $\triangle$. 
\end{remark}

\begin{remark}
It is proved in \cite{AgrachevBarilariBoscain_CV2012} that the density of the spherical Hausdorff measure with respect to any smooth measure is of class $C^3$ but not of class $C^5$ in general. Therefore, the local Weyl measure $w_\triangle$ differs from the spherical Hausdorff measure in general. Comparing Weyl and Popp measures is done in Section \ref{sec_compare_weyl_popp} hereafter. 
\end{remark}

\begin{remark}\label{rem_othermeasure}
The constant $\rho(M)=\int_M e_{\widehat{\triangle}^{q},\widehat{\mu}^{q}}(1,0,0)\, d\mu(q)$ does not depend on the measure $\mu$.
Applying \eqref{weylequiregular} with $\mu=w_{\triangle}$, we obtain that, in the equiregular case, 
$$
e_{\widehat{\triangle}^{q},\widehat{w_\triangle}^{q}}(1,0,0) = \mathrm{Cst} = \rho(M) = \int_M e_{\widehat{\triangle}^{q'},\widehat{\mu}^{q'}}(1,0,0)\, d\mu(q')
$$
for every $q\in M$ and for every smooth measure $\mu$ on $M$.
For example, if the sR structure $(M,D,g)$ is itself a Carnot group, then $e_{\triangle,w_\triangle}(1,q,q) = \rho(M)$ for every $q\in M$. 

This is nothing else but saying that $e_{\widehat{\triangle}^{q},\widehat{\rho}^{q}}(1,0,0)=1$ for every $q\in M$ (see Section \ref{sec_additional}).
In particular, using \eqref{add_densite_h}, we have (still in the equiregular case)
$$
h(q) = \frac{d\mu}{d\nu}(q) = \frac{\widehat{\mu}^{q}}{\widehat{\nu}^{q}} = \frac{\widehat{\mu}^{q}(\hatBsR^q(0,1))}{\widehat{\nu}^{q}(\hatBsR^q(0,1))} = \frac{ e_{\widehat{\triangle}^{q},\widehat{\nu}^{q}}(1,0,0) } { e_{\widehat{\triangle}^{q},\widehat{\mu}^{q}}(1,0,0) } .
$$
Taking $\nu=\mathcal{H}_S$ (spherical Hausdorff measure associated with the sR distance on $M$), we have $\widehat{\mathcal{H}_S}^{q}(\hatBsR^{q}(0,1))=2^{\mathcal{Q}^M(q)}$, and then
$$
\frac{d\mu}{d\mathcal{H}_S}(q) = \frac{\widehat{\mu}^{q}(\hatBsR^{q}(0,1))}{2^{\mathcal{Q}^M(q)}} = \frac{ e_{\widehat{\triangle}^{q},\widehat{\mathcal{H}_S}^{q}}(1,0,0) } { e_{\widehat{\triangle}^{q},\widehat{\mu}^{q}}(1,0,0) }.
$$
\end{remark}

\subsection{The local Weyl measure as a new canonical sR measure}\label{sec_weyl_intrinsic}
In the equiregular case, the local Weyl measure $w_{\triangle}$ defined by \eqref{def_local_weyl_measure} (or, equivalently, by \eqref{def_local_cesaro-weyl_measure}) is a canonical measure in sR geometry, enjoying the same nice properties as the Popp measure $P$ (whose definition and properties are recalled in Appendix \ref{sec_popp}), which is a well known canonical smooth measure on an equiregular sR structure: like $P$, obviously, $w_\triangle$ is invariant under sR isometries of $M$; moreover, it commutes with nilpotentization, as shown in the following lemma.

\begin{lemma}
In the equiregular case, the construction of the local Weyl measure commutes with nilpotentization, i.e.,
$\widehat{w_{\triangle}}^q = w_{\widehat{\triangle}^q}$, for every $q\in M$.
\end{lemma}

\begin{proof}
By Theorem \ref{thm_local_weyl_equiregular} we have $w_{\triangle} = h \, P$ with $h(q)=e_{\widehat{\triangle}^{q},\widehat{P}^{q}}(1,0,0)$, hence $\widehat{w_\triangle}^q = h(q) \widehat{P}^q$. Besides, for the sR Laplacian $\widehat{\triangle}^q$ on $L^2(\widehat{M}^q,\widehat{P}^q)$ with $\widehat{M}^q\simeq\R^n$, Theorem \ref{thm_local_weyl_equiregular} implies that $w_{\widehat{\triangle}^q} = \gamma P_{\widehat{\triangle}^q}$ with $\gamma=e_{\widehat{\widehat{\triangle}^q}^0,\widehat{\widehat{P}^q}^0}(1,0,0)$. By equiregularity, we have $\gamma = h(q)$.
\end{proof}

\subsection{Relationship between Weyl and Popp measures}\label{sec_compare_weyl_popp}
It is natural to apply \eqref{weylequiregular} with $\mu=P$, the Popp measure on $M$ (see Appendix \ref{sec_popp}). 
Then, the local Weyl measure $w_\triangle$ coincides with the Popp measure $P$ (up to constant scaling) if and only if the function $q\mapsto e_{\widehat{\triangle}^q,\widehat{P}^q}(1,0,0)$ is constant on $M$.
We have the following result.

\begin{lemma}\label{lem_kernelconstant}
Given any $q\in M$ and any sR isometry $\phi$ of $M$, the nilpotentized sR structures $(\widehat{M}^{q},\widehat{D}^{q},\widehat{g}^{q})$ and $(\widehat{M}^{\phi(q)},\widehat{D}^{\phi(q)},\widehat{g}^{\phi(q)})$ are sR isometric.

In particular, if the group $\mathrm{Iso}_{sR}(M)$ of sR isometries of $M$ acts transitively on $M$, then the function $q\mapsto e_{\widehat{\triangle}^q,\widehat{P}^q}(1,0,0)$ is constant on $M$.
\end{lemma}

\begin{proof}
The first fact is recalled in Appendix \ref{sec_nilp_diffeo}.
Moreover, using that $\phi_*\widehat{\triangle}^{q}\phi^*=\widehat{\triangle}^{\phi(q)}$ and 
$\phi_*\widehat{P}^{q}=\widehat{P}^{\phi(q)}$, it follows from \eqref{formulas_kernel} in Appendix \ref{appendix_Schwartz} that $e_{\widehat{\triangle}^{q},\widehat{P}^{q}}(1,0,0)=e_{\widehat{\triangle}^{\phi(q)},\widehat{P}^{\phi(q)}}(1,0,0)$.
\end{proof}

\begin{remark}
As noticed in \cite{BarilariRizzi_AGMS2013}, if $\mathrm{Iso}_{sR}(M)$ acts transitively on $M$ then the Popp measure is the unique measure on $M$ that is invariant under sR isometries.
\end{remark}

\begin{remark}\label{rem_3Dcase}
In the 3D contact case, assuming that the Popp measure is normalized so that it is a probability measure on $M$, we have $e_{\widehat{\triangle}^q,\widehat{P}^q}(1,0,0)=1/16$ (see \cite{CHT-I}, see also Example \ref{ex_3D_heis} in Section \ref{sec_regdet}).
Note that this can also be obtained by using the explicit expression
$$
e(t,x,y) = \frac{1}{(2\pi t)^2} \int_\R \frac{2\tau}{\sinh(2\tau)} e^{\left( iy\tau - \vert x\vert^2\tau \mathrm{cotanh}(2\tau) \right)/t}\, d\tau
$$
established in \cite{Ga-77} (see also \cite{BGG-00}) in the Heisenberg flat case, using that $\int_\R \frac{2\tau}{\sinh(2\tau)} \, d\tau = \frac{4\pi^2}{16}$. 
\end{remark}

\begin{remark}\label{lem_nilp_cst}
The question of the equality (up to scaling) of the local Weyl measure $w_\triangle$ and the Popp measure $P$ is  related to the problem of classification of sR structures under the action of sR isometries.
They coincide if the sR structure on $M$ is free, but not in general because the classification of Carnot groups involves moduli in large dimension. For Carnot groups of dimension less than or equal to $5$, for which the classification is known (see \cite{AgrachevBarilariBoscain_CV2012}), 
they always coincide except in the bi-Heisenberg case, studied hereafter. 
%
\end{remark}

\paragraph{The bi-Heisenberg case.}
Let $G$ be the $5$-dimensional group bi-Heisenberg defined as $G=\R^5$ with the product rule
$$
(x_1,y_1,x_2,y_2,z)\star (x'_1,y'_1,x'_2,y'_2,z') = (x_1+x'_1, y_1+y'_1, x_2+x'_2, y_2+y'_2, z+z'+x_1y'_1+x_2y'_2).
$$
The contact form $\alpha=dz-y_1\, dx_1-y_2\, dx_2$ is invariant under right translations.
We define the $5$-dimensional compact manifold $M=G/\Gamma$ where $\Gamma=(\sqrt{2\pi}\Z)^4\times 2\pi\Z$ is a subgroup of $G$. The vector fields $X_i=\partial_{x_i}+y_i\, \partial z$, $Y_i=\partial_{y_i}$, $i=1,2$, form a frame of the horizontal distribution $D=\ker\alpha$. 
Let $g$ be a smooth Riemannian metric on $D$.
Denoting by $\omega_1$ and $\omega_2$ the eigenvalues of the symplectic form $d\alpha$ with respect to $g$ (note that these eigenvalues are Lipschitz functions of $q$ only, in general), the sR metric $g$ is the pullback of $\frac{1}{\omega_1}(dx_1^2+dy_1^2)+\frac{1}{\omega_2}(dx_2^2+dy_2^2)$ under the projection onto the $(x,y)$ plane.
Taking $\mu$ the Lebesgue measure, the sR Laplacian is $\triangle = \omega_1(X_1^2+Y_1^2)+\omega_2(X_2^2+Y_2^2)$.

\begin{lemma}\label{lembiHeis}
The density of $w_\triangle$ with respect to $P$ is given at any point $q\in M$ by
$$
\frac{dw_{\triangle}}{dP}(q) = \frac{e_{\widehat{\triangle}^q,\widehat{P}^q}(1,0,0)}{\int_M e_{\widehat{\triangle}^{q'},\widehat{P}^{q'}}(1,0,0)\, dP(q')}
$$ 
with
\begin{equation}\label{ehatbiH}
e_{\widehat{\triangle}^q,\widehat{P}^q}(1,0,0) = \frac{\omega_1(q)\omega_2(q)\sqrt{\omega_1(q)^2 + \omega_2(q)^2}}{2\pi^3}\sum _{\ell,\ell'\geq 0 } \frac{1}{ ((2\ell+1)\omega_1(q) + (2\ell'+1)\omega_2(q))^3} .
\end{equation}
Moreover, the function $q\mapsto e_{\widehat{\triangle}^q,\widehat{P}^q}(1,0,0)$ is smooth on $M$.
\end{lemma}

By Lemma \ref{lembiHeis}, the density of Weyl with respect to Popp is a smooth function that differs from $1$ for any normalization.

\begin{proof}
Let us first assume that $\omega_1$ and $\omega_2$ are constant.
Using the explicit expression for the Popp measure given in \cite{BarilariRizzi_AGMS2013}, we find that the density of the Popp measure with respect to the Lebesgue measure $d\mu = dx\, dy\, dz$ is $\frac{dP}{d\mu} = 1/\omega_1\omega_2\sqrt{\omega_1^2+\omega_2^2}$.
In particular, $P(M)=(2\pi)^3/\omega_1\omega_2\sqrt{\omega_1^2+\omega_2^2}$. 
Since the operator $-i\partial_z$ commutes with $\triangle$, we can write $L^2(M,d\mu)=\oplus_{m\in\Z} H_m$ and decompose functions on $M$ as $\sum_{m\in\Z} e^{imz} f_m(x_1,y_1,x_2,y_2)$. We infer that the eigenvalues are all $m((2\ell+1)\omega_1+(2\ell'+1)\omega_2)$, $m\geq 1$, $\ell,\ell'\geq 0$, of multiplicity $2m^2$, and all eigenvalues of the Riemannian Laplacian on the $4$-dimensional torus $\R^4/\sqrt{2\pi}\Z^4$. After some computations, we obtain
$$
N(\lambda)\sim \frac{2\lambda^3}{3}\sum_{\ell,\ell'\geq 0} \frac{1}{((2\ell+1)\omega_1+(2\ell'+1)\omega_2)^3}
$$
as $\lambda\rightarrow+\infty$. Since $\mathcal{Q}^M=6$ and since the function $q\mapsto e_{\widehat{\triangle}^q,\widehat{P}^q}(1,0,0)$ is constant on $M$, we infer \eqref{ehatbiH} from \eqref{asympt_N_equireg}, in the case where $\omega_1$ and $\omega_2$ are constant. 

Now, when $\omega_1$ and $\omega_2$ are not constant, \eqref{ehatbiH} is still true for every $q\in M$, by comparing at $q$ the bi-Heisenberg case with non-constant $\omega_1$ and $\omega_2$ and the bi-Heisenberg case with constant $\omega_1=\omega_1(q)$ and $\omega_2=\omega_2(q)$.

The last point is to prove that, although $\omega_1$ and $\omega_2$ are only Lipschitz functions of $q$ in general, \eqref{ehatbiH} gives anyway a smooth function of $q$. 
To prove this fact, we note that the right-hand side of \eqref{ehatbiH} is a smooth function of $(\omega_1,\omega_2)$ that is invariant under permutations. It follows from a theorem of \cite{Glaeser} (see also \cite{Schwarz}) that this function is also a smooth function of $(\omega_1+\omega_2,\omega_1\omega_2)$.
Since $\omega_1+\omega_2$ and $\omega_1\omega_2$ are smooth functions of the metric (trace and determinant of $d\alpha$ with respect to $g$), we conclude that $e_{\widehat{\triangle}^q,\widehat{P}^q}(1,0,0)$ depends smoothly on $q\in M$.
\end{proof}

\subsection{Zeta function and regularized determinant}\label{sec_regdet}
Consider the zeta function associated with the sR Laplacian $\triangle$, defined by
\begin{equation}\label{def_zeta}
\zeta_\triangle(s) = \sum_{j=1}^{+\infty} \frac{1}{\lambda_j^s} \qquad\forall s\in\C .
\end{equation}
Here and in the sequel, we remove the first mode $\lambda_0=0$, $\phi_0=1$. Recall that 
\begin{equation}\label{eqzeta}
\zeta_\triangle(s)\Gamma(s) = \int_0^{+\infty} \Tr(e^{t\triangle}) t^{s-1}\, dt
\end{equation}
where $\Tr(e^{t\triangle}) = \sum_{j=1}^{+\infty}e^{-\lambda_jt}$ and $\Gamma(s)=\int_0^{+\infty}t^{s-1}e^{-t}\, dt$. The series \eqref{def_zeta} defining $\zeta_\triangle(s)$ converges for $\mathrm{Re}(s)$ large enough. 

When the heat trace has a complete small-time expansion as $t\rightarrow 0^+$, $\zeta_\triangle(s)$ has a meromorphic extension on $\C$. Usually, $\zeta_\triangle(s)$ is holomorphic at $s=0$ and the regularized determinant is then defined by $\det' \triangle = e^{-\zeta_\triangle'(0)}$ (see, e.g., \cite[Chapter 5]{SrivastavaChoi_book}). 
As a consequence of Theorem \ref{thm_local_weyl_equiregular}, we have the following result.

\begin{lemma}
In the equiregular case, the regularized determinant exists. It is a global spectral invariant.
\end{lemma}

More generally, the regularized determinant of the sR Laplacian can be defined as soon as there is no term $t^0\vert\ln t\vert^j$ for some $j\geq 1$ in the local Weyl law (otherwise such a term would cause a pole at $s=0$).

\begin{example}\label{ex_3D_heis}
In the 3D flat Heisenberg case (see \cite[Section 3.1]{CHT-I}), we have
$$
\zeta_\triangle(s) = 2\sum_{m\geq 1,l\geq 0} \frac{1}{(2l+1)^sm^{s-1}} + \zeta_{\mathbb{T}^2}(s)
= 2\zeta(s-1)\zeta(s)(1-2^{-s}) + \zeta_{\mathbb{T}^2}(s)
$$
where $\zeta$ (resp., $\zeta_{\mathbb{T}^2}$) is the classical zeta function (resp., the zeta function on the torus $\mathbb{T}^2$). We infer that $\zeta_\triangle(s) \sim \frac{\pi^2}{4(s-2)}$ near $s=2$.
In this way, we recover the fact that $\widehat{e}^q(1,0,0)=1/16$, obtained in \cite{CHT-I} (see also Remark \ref{rem_3Dcase} in Section \ref{sec_compare_weyl_popp}). Indeed, using \eqref{eqzeta} and splitting the integral in $\int_0^1$ and $\int_1^{+\infty}$, the latter is holomorphic in $s$ and the first has a pole at $s=2$. Since we know that $\Tr(e^{t\triangle})\sim\frac{A}{t^2}$ as $t\rightarrow 0^+$, we obtain $A=\frac{1}{16}$ by identification.
\end{example}

\section{Microlocal Weyl law}\label{sec_equiregular_microlocal}
Given any vector bundle $E$ over $M$, the sphere bundle $SE$ is defined by $SE=(E\setminus\{0\})/\R^+$. Homogeneous functions of order $0$ on $E$ are identified with functions on $SE$. We also use the standard notation $S^*M = S(T^*M)$ for the co-sphere bundle.

According to Remark \ref{rem_double_covering}, if the horizontal distribution $D$ is of codimension $1$ in $TM$, then the microlocal Weyl measure $W_\triangle$ is equal to half of the pullback of $w_\triangle$ by the double covering $S\Sigma\rightarrow M$ which is the restriction to $S\Sigma$ of the canonical projection of $T^\star M$ onto $M$.
Recall that the characteristic manifold is defined by $\Sigma=D^\perp
=(g^*)^{-1}(0)$ where $g^*$ is the cometric. 

For contact closed manifolds, the microlocal Weyl law has been derived in \cite{Taylor_CPDE2020}.
In the general equiregular case, we have the following result. 

In the equiregular case, $\Sigma^{r-1}=(D^{r-1})^\perp \subset T^\star M$ (annihilator of $D^{r-1}$, see Appendix \ref{app_sRflag}), where $r$ is the degree of nonholonomy (constant on $M$), is a subbundle of $T^*M$ over $M$.
We define $\pi_{S\Sigma^{r-1}}$ as the projection of $T^*M$ onto $S\Sigma^{r-1}$. The projection $\pi_{S\Sigma^{r-1}}$ is canonical and is defined by dilations as follows: given any continuous function $a$ on $S^*M$ (identified with a function on $T^*M\setminus\{0\}$, homogeneous of degree $0$ with respect to $p$), we have $a\circ \pi_{S\Sigma^{r-1}} = \lim_{\varepsilon\rightarrow 0} a\circ \delta_{1/\varepsilon}^q$.
In a local chart of privileged coordinates around $q$, using the notations of Appendix \ref{app_sRflag}, we have, by homogeneity, $a(q,\delta_{1/\varepsilon}^q(p)) = a(q,(\varepsilon^{-w_1}p_1,\ldots,\varepsilon^{-w_n}p_n)) = a(q,(\varepsilon^{w_n-w_1}p_1,\ldots,p_n))$, whose limit as $\varepsilon\rightarrow 0$ gives $\pi_{S\Sigma^{r-1}}(q,p) = a(q,(0,\ldots,0,p_{n_{r-1}+1},\ldots,p_n))$.

\begin{theorem}\label{thm_microlocal_weyl_equiregular}
We denote by $K(q,\cdot)$ the Fourier transform of $y\mapsto e_{\widehat{\triangle}^q,m}(1,y,0)$ (where $m$ is the Lebesgue measure on $\R^n$), that is,
\begin{equation}\label{defKqp}
K(q,p) = \int_{\R^n}  e^{-iy.p} \, e_{\widehat{\triangle}^q,m}(1,y,0)  \, dy = \int_{\R^n}  e^{-iy.p} \, e_{\widehat{\triangle}^q,\widehat{\mu}^q}(1,y,0)  \, d\widehat{\mu}^q(y) .
\end{equation}
Actually, $K$ is the Kohn-Nirenberg symbol (of order $-\infty$) of the smoothing operator $e^{\widehat{\triangle}^q}$ (heat semi-group of $\widehat{\triangle}^q$ at time $1$).

We denote by $\Omega=\frac{1}{n!}\omega^n$ the canonical Liouville volume form on $T^*M$, where $\omega$ is the canonical symplectic form. 

For every pseudo-differential operator $A$ of order $0$ on $M$, of principal symbol $a$, we have
\begin{equation}\label{microlocal_weyl_law_equiregular}
\Tr(A \, e^{t\triangle}) = \frac{1}{t^{\mathcal{Q}^M/2}} F(\sqrt{t}) \qquad\forall t>0
\end{equation}
for some $F\in C^\infty(\R)$, with
\begin{equation}\label{F(0)_equiregular}
F(0) = \frac{1}{(2\pi)^n} \int_{T^*M} a\circ\pi_{S\Sigma^{r-1}}\, K \, d\Omega 
\end{equation}
that depends only on the restriction of $a$ to $S\Sigma^{r-1}$.
%
As a consequence, the microlocal Weyl measure $W_{\triangle}$ (defined by \eqref{def_microlocal_weyl_measure}) exists, is a constant times the image under $\pi_{S\Sigma^{r-1}}$ of the measure $K\,\Omega$, i.e., 
$$
W_\triangle = \frac{1}{\rho(M)} (\pi_{S\Sigma^{r-1}})_* (K\, \Omega) 
$$
where $\rho(M) = \int_M\widehat{e}^q(1,0,0)\, d\mu(q)$ (see Definition \ref{def_rho}).
Moreover, we have
$$
\mathrm{supp}(W_{\triangle}) = S\Sigma^{r-1} .
$$
\end{theorem}

\begin{remark}
In Carnot groups of small dimension (see Remark \ref{lem_nilp_cst}), the function $K$ defined by \eqref{defKqp} does not depend on $q$.
\end{remark}

\begin{remark}\label{rem_weyl_microlocal}
Compared with the asymptotic expansion given by \eqref{asympt_heat_equiregular} as $t\rightarrow 0^+$ for the local Weyl law, the function $F$ in \eqref{microlocal_weyl_law_equiregular} is a smooth function of $\sqrt{t}$. The coefficient of $t^{i-\mathcal{Q}^M/2}$ in this expansion is the sum of several coefficients (which do not have a nice expression); one of which has the same expression as the leading coefficient in \eqref{asympt_heat_equiregular}, replacing $K$ with the Fourier transform of the function $y\mapsto a_i^q(y,0)$, where $a_i^q$ is defined in Theorem \ref{lemfondamental} (see Appendix \ref{app_lemfondam}); the other coefficients involve restrictions of $a$ to $\Sigma^{j}$ for various integers $j\leq r-1$.

We underline that the leading term, which provides the microlocal Weyl measure $W_\triangle$, only involves the restriction of $a$ to $\Sigma^{r-1}$. The property $\mathrm{supp}(W_{\triangle}) = S\Sigma^{r-1}$ means that spectral concentration is exactly on $S\Sigma^{r-1}$: this reflects the fact that spectral complexity is mainly due to the highest possible Lie brackets appearing in the generating Lie algebra condition $\mathrm{Lie}(D)=TM$ (H\"ormander condition).
\end{remark}

\begin{remark}
Let us make precise the density of the microlocal Weyl measure in $S\Sigma^{r-1}$.
Consider the fibration $\pi_{S\Sigma^{r-1}}:T^*M\rightarrow S\Sigma^{r-1}$. Each fiber is given by $T^*_qM/\Sigma^{r-1}_q$, for $q\in M$.
In a local chart of coordinates $(q,p)$, we use the notation $P_1=(p_1,\ldots,p_{n_{r-1}})$ and $P_2=(p_{n_{r-1}+1},\ldots,p_n)$, so that $p=(P_1,P_2)$.
We have $\pi_{S\Sigma^{r-1}}(q,p) = (q,(0,P_2))$. In these local coordinates, we endow $\Sigma^{r-1} \simeq M\times \R^{n-n_{r-1}}$ with the measure that is the pullback under the chart of the standard Lebesgue measure $dq\, dP_2$ on $\R^n\times\R^{n-n_{r-1}}$. By disintegration of the Liouville volume form $d\Omega=dq\, dp$ of $T^*M$ with respect to $dq\, dP_2$, we write, locally, $d\Omega = dq\, dP_2\, dP_1$, where $dP_1$ is the Lebesgue measure on the vertical fiber $T^*_qM/\Sigma^{r-1}_q \simeq \R^{n_{r-1}}$. These volumes are intrinsic and correspond to Hausdorff measures.
Then, taking polar coordinates $P_2=ru$ with $r>0$ and $u\in \mathcal{S}^{n-n_{r-1}-1}$, and using that $a$ is homogeneous of degree $0$, we obtain
\begin{equation}\label{microlocal_weyl_measure_equiregular}
\frac{dW_{\triangle}(q,u)}{d(q,u)} = \frac{1}{(2\pi)^n \rho(M)}  \int_0^{+\infty} K^1(q,ru)\, r^{n-n_{r-1}-1}\, dr 
\end{equation}
where $K^1(q,P_2) = \int_{\R^{n_{r-1}}} K(q,P_1,P_2)\, dP_1$.
\end{remark}

\begin{proof}
Hereafter, we use the notations of Appendix \ref{appendix_Schwartz} for Schwartz kernels and their densities.
Let $A$ be an arbitrary pseudo-differential operator on $M$ of order $0$.
%
%
%
We have 
$$
\Tr(A \, e^{t\triangle}) = \int_M [A \, e^{t\triangle}]_\nu(t,q,q)\, d\nu(q) = \int_M G(t,q)\, d\nu(q)
$$
with
$$
G(t,q) = [A \, e^{t\triangle}]_\nu(t,q,q)= \int_M [A]_\nu(q,q')\, e_{\triangle,\nu}(t,q',q)\, d\nu(q') 
$$
for any smooth measure $\nu$ on $M$.
%
Let $q\in M$ be arbitrary and let $U$ be a neighborhood of $q$ in $M$. 
Let us compute the asymptotics of $G(t,q)$ when $t\rightarrow 0^+$. 
Since the density $[A]_\nu$ with respect to $\nu$ of the Schwartz kernel of the pseudo-differential operator $A$ is smooth outside of the diagonal, we have
$$
[A]_\nu(q,q')=\rho_q(q')[A]_\nu(q,q')+R(q,q')
$$
where $\rho_q$ is a smooth cut-off function on $M$, equal to $1$ in a neighborhood $U_1\subset U$ of $q$, to $0$ outside of $U$, and $R$ is a smooth function.
It follows from the exponential estimates \eqref{exp_estimates} for the sR heat kernel (see Appendix \ref{app_sR_kernel}) that $\int_M R(q,q') e_{\triangle,\mu}(t,q',q)\, d\mu(q') = \mathrm{O}(t^\infty)$ as $t\rightarrow 0^+$, uniformly with respect to $q$, and thus this term yields no contribution to the Taylor expansion of $G(t,q)$.
Hence, in what follows, without loss of generality we assume that $R=0$. 

Using to a partition of unity, without loss of generality we also assume that $A=\Op(a)$, where $\Op$ is the usual (left) quantization and $a$ is a classical symbol of order $0$, compactly supported with respect to its first variable in a sufficiently small neighborhood of $q$. Hereafter, we take local privileged coordinates at $q$ (in which $q=0$), and we take $\nu$ as the Lebesgue measure in this chart. The heat kernel $e_{\triangle,\nu}$ is denoted by $e$.


In these local coordinates, we have $[A]_\nu(0,x) = \frac{\rho_0(x)}{(2\pi)^n} \int_{\R^n} e^{-ix.p} a(0,p)\, dp$, 
i.e., the Schwartz kernel of the pseudo-differential operator $A=\Op(a)$ is a tempered distribution which is the Fourier transform (as an operator from $\S'(\R^n)$ to $\S'(\R^n)$) with respect to $p$ of the symbol $a$ which is a tempered distribution, and thus
$$
G(t,0) = \frac{1}{(2\pi)^n} \int_{\R^n\times\R^n} e^{-ix.p} a(0,p) \rho_0(x)\, e(t,x,0)\, dx\, dp .
$$
Making successively the change of variables $x=\delta_{\sqrt{t}}(y)$, of determinant $t^{\mathcal{Q}^M/2}$, 
using that $\delta_{\sqrt{t}}(y).p=y.\delta_{\sqrt{t}}(p)$ and then making the change of variables $p\mapsto\delta_{1/\sqrt{t}}(p)$, we infer that
\begin{equation*}
G(t,0) = \frac{1}{(2\pi)^n t^{\mathcal{Q}^M/2}} \int_{\R^n\times\R^n} {\mkern-20mu} e^{-iy.p}a\big(0,\delta_{1/\sqrt{t}}(p)\big) \rho_0(\delta_{\sqrt{t}}(y)) t^{\mathcal{Q}^M/2} \, e\left(t,\delta_{\sqrt{t}}(y),0\right)  \, dy\, dp  .
\end{equation*}
Assuming that $a$ is homogeneous of degree $0$ with respect to $p$, using the sR weights at $q$ (see Appendix \ref{app_sRflag}), we have
\begin{equation}\label{a_homog}
a\big(q,\delta_{1/\sqrt{t}}(p)\big) = a\big(q,\big(t^{-w_1/2}p_1,\ldots,t^{-w_n/2}p_n\big)\big) = a\big(0,\big(t^{(w_n-w_1)/2}p_1,\ldots,p_n\big)\big) .
\end{equation}
Similar arguments are developed for all other homogeneous components of the symbol $a$, but actually to compute the main term of the asymptotics only the above $0$-homogeneous case will be of interest.
Anyway, given any symbol of order $0$, the exponential estimates \eqref{exp_estimates} and \eqref{upper_estim_kernel}  of the heat kernel and of its derivatives (see Appendix \ref{app_sR_kernel}) imply that the function $t^{\mathcal{Q}^M/2} \, e\left(t,\delta_{\sqrt{t}}(y),0\right)$ and all its derivatives satisfy the domination property, and the Lebesgue dominated convergence theorem implies that $t^{\mathcal{Q}^M/2} \,G(t,q)$ is a smooth function of $\sqrt{t}$ and $q$, which gives \eqref{microlocal_weyl_law_equiregular}. To compute the equivalent as $t\rightarrow 0^+$, we start by noting that, by Theorem \ref{lemfondamental} in Appendix \ref{app_lemfondam}, 
\begin{equation*}
G(t,q) \sim \frac{1}{(2\pi)^n t^{\mathcal{Q}^M/2}} \int_{\R^n\times\R^n} {\mkern-20mu} e^{-iy.p}a\big(q,\delta_{1/\sqrt{t}}(p)\big)\, \widehat{e}^q(1,y,0)  \, dy\, dp  
= \frac{\int_{\R^n} a\big(q,\delta_{1/\sqrt{t}}(p)\big) K(q,p) \, dp}{(2\pi)^n t^{\mathcal{Q}^M/2}}  
\end{equation*}
as $t\rightarrow 0^+$, where $K(q,\cdot)$, defined by \eqref{defKqp}, is the Fourier transform taken at $p$ of $y\mapsto \widehat{e}^q(1,y,0)$.
Since the latter function belongs to the Schwartz class $\S(\R^n)$ (this follows, again, from the exponential estimates, as discussed above), it follows that $K(q,\cdot)\in \S(\R^n)$.
Without loss of generality, we assume that $a$ is the principal symbol of $A$ (thus, is homogeneous of degree $0$). Indeed, we have $A=\Op(a)+C$ where $C$ is a compact operator and thus $\Tr(Ce^{t\triangle})\rightarrow 0$ as $t\rightarrow 0^+$ (as in the proof of Lemma \ref{lem_prelim_measure}).
Since $a$ is homogeneous of degree $0$ with respect to $p$, letting $t$ tend to $0$ in \eqref{a_homog}, all terms $t^{(w_j-w_1)/2}$ such that $w_j>w_1$ vanish. Since $w_{n_{r-1}+1}=\cdots=w_n(=r)$, we get $a\big(q,\delta_{1/\sqrt{t}}(p)\big) = a(q,(0,\ldots,0,p_{n_{r-1}+1},\ldots,p_n)) + \mathrm{O}(\sqrt{t})$
as $t\rightarrow 0^+$, and thus
$$
G(t,q) \sim \frac{1}{(2\pi)^n t^{\mathcal{Q}^M/2}} \int_{\R^n} a(q,(0,\ldots,0,p_{n_{r-1}+1},\ldots,p_n)) K(0,p)\, dp  .
$$
By uniformity with respect to $q$ (see Theorem \ref{lemfondamental} in Appendix \ref{app_lemfondam}) and by dominated convergence, we obtain \eqref{F(0)_equiregular}.
%
%
%
It follows 
that the microlocal Weyl measure $W_{\triangle}$ exists and is the image under $\pi_{S\Sigma^{r-1}}$ of the measure $K\,d\Omega$. In particular, we have $\mathrm{supp}(W_{\triangle}) \subset S\Sigma^{r-1}$.

Let us prove that we have exactly $\mathrm{supp}(W_{\triangle}) = S\Sigma^{r-1}$.
Of course, the projection of $\mathrm{supp}(W_{\triangle})$ onto $M$ is $\mathrm{supp}(w_{\triangle})=M$. Hence, it suffices to prove that, for $q\in M$ fixed, the support of the density \eqref{microlocal_weyl_measure_equiregular} as a function of $u$ is equal to the whole space $\R^{n-n_{r-1}}$.
To prove this fact, it suffices to prove that this density is a nonnegative analytic function of $u$: indeed, then, either it is zero everywhere (which is not the case) or it is positive, which gives the result.
The fact that the density \eqref{microlocal_weyl_measure_equiregular} is analytic in $u$ is because $K(q,\cdot)$ is analytic (and this, even if the manifold $M$ is not analytic, since the argument is applied in $\R^{n-n_{r-1}}$), as the Fourier transform of a sR heat kernel taken at time $1$, which is exponentially decreasing at infinity. 
Moreover, it is a real-valued even function, because the kernel is so.

The fact that the density \eqref{microlocal_weyl_measure_equiregular} is a nonnegative function of $u$ follows from the fact that it can be written as the square of a function. Indeed, in the equiregular case $\widehat{M}^q$ is a Carnot group, isometric to $\R^n$, with the identity of the group given by $0\in\R^n$, the inverse of an element $x$ being given by $-x$, and the Haar measure being identified with the Lebesgue measure $m$ on $\R^n$. The (symmetric positive) heat kernel has then the property $\widehat{e}^q(1,x,y) = \widehat{e}^q(1,x-y,0)$ and we can write $\widehat{e}^q(1,y,0) = \int_{\R^n} \widehat{e}^q(1/2,y,z) \widehat{e}^q(1/2,z,0)\, dz = \int_{\R^n} \widehat{e}^q(1/2,y-z,0)\, \widehat{e}^q(1/2,z,0)\, dz = \big( \widehat{e}^q(1/2,\cdot,0) \star \widehat{e}^q(1/2,\cdot,0) \big) (y)$, with the usual convolution on a Lie group.
The claim follows.


Finally, let us check that $K$ is the Kohn-Nirenberg symbol of the smoothing operator $e^{\widehat{\triangle}^q}$. 
The Schwartz kernel of $e^{\widehat{\triangle}^q}$ satisfies $[e^{\widehat{\triangle}^q}](x,y)=\widehat{e}^q(1,x,y)=\widehat{e}^q(1,x-y,0)$.
Using that the Kohn-Nirenberg symbol of an operator $A$ on $\R^n$ of Schwartz kernel $[A]$ is given by $a(x,\xi)=e^{-ix.\xi}\int_{\R^n}[A](x,y)e^{iy.\xi}\, dy$, the conclusion follows.
\end{proof}


\part{Singular case}\label{part_singular}
In this part, we assume that the sR structure $(M,D,g)$ is singular, i.e., not equiregular. In contrast to Part \ref{part_equiregular}, the integer-valued functions $q\mapsto\mathcal{Q}^M(q)$ and $q\mapsto r(q)$, which are upper semi-continuous, now have discontinuities. 
The \emph{singular set} is the closed subset of $M$ consisting of all possible singular points (see Appendix \ref{app_sRflag}), i.e., the set where the sR flag is not regular, or, equivalently,
$$
\S = \{q\in M\ \mid\ \mathcal{Q}^M(q) > \inf_{q'\in M}\mathcal{Q}^M(q') \} = \{q\in M\ \mid\ \mathcal{Q}^M(q) > \Qeq \} .
$$
The latter equality is because $\mathcal{Q}^M(q)=\Qeq$ for every $q\in M\setminus\S$, where $\Qeq$ is the Hausdorff dimension of the open regular region $M\setminus\S$.

In Section \ref{sec_singular_prelim}, we first explain that, because of such discontinuities, the easy argument of dominated convergence used in the equiregular case is bound to fail in general singular cases. 
We introduce the geometric context that will be used in the subsequent sections. We explain the ``$(J+K)$-decomposition", which is instrumental in computing the local Weyl law in the singular case, by adequately splitting the integrals to be estimated. 
This preliminary analysis leads us to introduce the nilpotentizability property: we say that the horizontal distribution $D$ is $\S$-nilpotentizable if $D$ is locally diffeomorphic to its nilpotentization at any point of $\S$.
We also define the concept of double (and more generally, multiple) nilpotentization procedure.

Then, in the subsequent sections, under the assumption that the singular set $\S$ is Whitney stratified by equisingular smooth submanifolds of $M$, we establish the local Weyl law:
\begin{itemize}[parsep=0.5mm,itemsep=0.2mm,topsep=0.5mm]
\item In Section \ref{sec_equisingular_nilp}, when $\S$ has a single stratum (i.e., $\S$ is itself an equisingular smooth submanifold of $M$) and the horizontal distribution $D$ is $\S$-nilpotentizable. We give more precise results in the Baouendi-Grushin and Martinet cases, as well as consequences in terms of Quantum Ergodicity properties.
\item In Section \ref{sec_equisingular_stratified_nilp}, when $\S$ has multiple strata and $D$ is $\S$-nilpotentizable.
\item In Section \ref{sec_nonnilp}, when $D$ is not $\S$-nilpotentizable: we first provide a general result in the real analytic case and then give several examples for non-analytic sR structures.
\end{itemize}

\section{Preliminaries}\label{sec_singular_prelim}
Our objective is to compute the small-time asymptotics of the local Weyl law, i.e., given any smooth real-valued function $f$ on $M$, estimate
\begin{equation}\label{def_int_I}
I(t) = \int_M f(q)\, e(t,q,q)\, d\mu(q) 
\end{equation}
as $t\rightarrow 0^+$.
%
%
Using a partition of unity, this can be done locally. In the neighborhood of a point $q_0\in M\setminus\S$ (regular point), we have seen in Theorem \ref{thm_local_weyl_equiregular} that the argument is very easy: we use the fact that $t^{\mathcal{Q}^M(q)/2}e(t,q,q)\rightarrow\widehat{e}^q(1,0,0)$ as $t\rightarrow 0^+$, uniformly with respect to $q$ in a neighborhood of $q_0$ where all points $q$ are regular. Hence, if $f$ is supported far from $\S$ then Theorem \ref{thm_local_weyl_equiregular} can be applied and in particular the asymptotics of $I(t)$ is in $1/t^{\mathcal{Q}^M(q)/2}$ as $t\rightarrow 0^+$.
Difficulties appear when one wants to compute the asymptotics near a singular point $q_0\in\S$.

\subsection{On the domination property}\label{sec_weyl_integrable}
In the above argument, what is instrumental is to apply the dominated convergence theorem to the family of functions $h_t(q)=t^{\mathcal{Q}^M(q)/2}e(t,q,q)$, indexed by $t>0$, provided this family is dominated by an integrable function of $q$: this domination property is satisfied on every compact subset of the open regular region $M\setminus\S$, but may fail near $\S$. 

Anyway, in case the domination property is satisfied near $\S$, i.e., if the function $(t,q)\mapsto t^{\Qeq}e(t,q,q)$ is bounded above, uniformly on $(0,1]\times (M\setminus\S)$, by a locally integrable function, then the conclusion of Theorem \ref{thm_local_weyl_equiregular} remains valid.
Let us comment on this domination property.

First, it implies that the function $q'\mapsto \widehat{e}^{q'}(1,0,0)$ is locally integrable at $q$ (see Section \ref{sec_additional}). 

Second, by the exponential estimates \eqref{exp_estimates_diago} for the heat kernel (see Appendix \ref{app_sR_kernel}), given any compact subset $K$ of $M$, there exists $C_1,C_2>0$ such that 
$\mu(\BsR(q,\sqrt{t})) \, e(t,q,q)$ is bounded above by $C_2$ and below by $C_1$, uniformly with respect to $q\in K$ and to $t\in(0,1]$.
Since $\mu(\BsR(q,\sqrt{t}))\sim t^{\mathcal{Q}^M(q)/2} \widehat{\mu}^q(\hatBsR^q(0,1))$ as $t\rightarrow 0^+$ (see \eqref{equiv_mu_ball} in Appendix \ref{sec_nilp_mesures}), using \eqref{mu_ball_eps} and \eqref{equiv_muhat} in Appendix \ref{sec_uniformballbox}, 
$\mu(\BsR(q,\sqrt{t})) / \widehat{\mu}^q\big(\hatBsR^q(0,1)\big) t^{\mathcal{Q}^M(q)/2}$ is bounded above by $C_2$ and below by $C_1$, uniformly with respect to $q\in K$ and to $t\in(0,1]$,
and thus
$$
\frac{C_1}{\widehat{\mu}^q\big(\hatBsR^q(0,1)\big)} \leq t^{\mathcal{Q}^M(q)/2} e(t,q,q) \leq \frac{C_2}{\widehat{\mu}^q\big(\hatBsR^q(0,1)\big)}\qquad \forall q\in K\qquad \forall t\in(0,1] .
$$
We can replace $\widehat{\mu}^q\big(\hatBsR^q(0,1)\big)$ by $\mathtt{w}_\mu^q(X)$ (defined in Appendix \ref{sec_uniformballbox}) in the above inequality.

As noted in Appendix \ref{sec_uniformballbox}, the functions $q\mapsto  \widehat{\mu}^q\big(\hatBsR^q(0,1)\big)$ and $q\mapsto \mathtt{w}_\mu^q(X)$ are positive on $M$, smooth near regular points, but discontinuous at singular points (both of them converge to $0$ when evaluated along a sequence of regular points converging to a singular point).

We conclude that the domination property for the family of functions $h_t$ is satisfied if and only if the function $q\mapsto 1/\widehat{\mu}^q\big(\hatBsR^q(0,1)\big)$ (equivalently, the function $q\mapsto 1/\mathtt{v}_\mu^q(X)$) is locally integrable.
Such a property fails in general at singular points (see \cite{GhezziJean_NA2015, GhezziJean_TSG2015}). For instance, it fails in the Baouendi-Grushin and Martinet cases (see Sections \ref{sec_Baouendi-Grushin} and \ref{sec_Martinet}). To obtain the Weyl law, then, we will use an adequate decomposition of the integrals (which we call the ``$(J+K)$-decomposition").

\subsection{Geometric context}\label{sec_geom_context}
Recall that $\S$ is the set of all singular points of the sR flag of $D$. The region $M\setminus\S$ is the \emph{regular region}, and $\mathcal{Q}^M(q)=\mathcal{Q}^M(M\setminus\S)=\Qeq$ for every $q\in M\setminus\S$ (recall that $\Qeq$ is the Hausdorff dimension of $M$ at such a regular point $q$).

\paragraph{Equisingularity.}
We assume that $\S$ is an \emph{equisingular} (see Appendix \ref{app_sRflag}) smooth submanifold of $M$, of topological dimension $k\in\{0,\ldots,n-1\}$: this means that all integers $n_i(q) = \dim D^i_q$ and $n_i^\S(q) = \dim \left( D^i_q\cap T_q\S \right)$ remain constant as $q\in\S$. In particular, we have
$$
\mathcal{Q}^M(q)=\Cst=\mathcal{Q}^M(\S) \qquad\textrm{and}\qquad \mathcal{Q}^\S(q)=\Cst=\mathcal{Q}^\S \qquad \forall q\in\S
$$
where $\mathcal{Q}^M(q)$ is defined by \eqref{def_Q} and $\mathcal{Q}^\S(q)$ is defined by \eqref{def_QN} in Appendix \ref{app_sRflag}.
%
Note that $\mathcal{Q}^M(\S)> \max(\mathcal{Q}^\S,\Qeq)$ and that $\mathcal{Q}^\S$ is the Hausdorff dimension of $\S$.

Three dimensions are attached to $\S$: its topological dimension $k$, its Hausdorff dimension $\mathcal{Q}^\S$, and the integer $\mathcal{Q}^M(\S)$. Although the three of them are useful in the forthcoming analysis, only the Hausdorff dimension $\mathcal{Q}^\S$ will play a role in the small-time asymptotics local Weyl law (see Theorem \ref{thm_onestratum} in Section \ref{sec_equisingular_nilp} where $D$ is moreover assumed to be $\S$-nilpotentizable), with different cases depending on whether $\mathcal{Q}^\S$ is lower or greater than the Hausdorff dimension $\Qeq$ of the regular region $M\setminus\S$.

We will treat in Section \ref{sec_equisingular_stratified_nilp} the more general case where $\S$ is stratified by equisingular smooth submanifolds $\S_i$ (of topological dimension $k_i$ and of Hausdorff dimension $\mathcal{Q}^{\S_i}$). Within this more general framework, for the moment we assume that $\S=\S_1$.

\paragraph{Privileged coordinates straightening $\S$.}
Since $\S$ is equisingular, according to Appendix \ref{app_privileged}, at each point $q\in\S$, we take 
local privileged coordinates $x=(x_1,\ldots,x_n)$, 
depending smoothly on $q\in\S$, 
in which $\S = \{ x_{k+1}=\cdots=x_n = 0\}$. 
If $\S$ is a single point then $k=0$ 
and $\mathcal{Q}^\S=0$.
Recall that $\mathcal{Q}^\S=\sum_{j=1}^k w_i^\S(D)$ is the Hausdorff dimension of $\S$ 
and that $\mathcal{Q}^M(\S)=\sum_{j=1}^n w_i^\S(D)$,
where the sR weights along $\S$ are labeled according to the coordinates $x$ (see Appendix \ref{app_privileged}).
Note that $0\leq \mathcal{Q}^\S \leq \mathcal{Q}^M(\S)-1$.

\paragraph{Transverse dilations and normal bundle.}
%
The ``topological" normal bundle is the vector bundle $N\S$ over $\S$ whose fibers $N_q\S$ are defined by $N_q\S =T_q M/T_q\S$ for every $q\in\S$.
We denote by $\pi_\S:N\S\rightarrow\S$ the canonical projection.
Using privileged coordinates straightening $\S$, the normal bundle $N\S$ is identified, in a non-canonical way, to $\S\times\R^{n-k}$ with coordinates $(q,x')$ for $q\in\S$ and $x'=(x_{k+1},\ldots,x_n)\in\R^{n-k}$; the fiber $N_q\S$ is identified with $\R^{n-k}$ endowed with the dilations along these coordinates (see Appendix \ref{sec_nilp_smoothsection}). Hence, $N\S\simeq\S\times\R^{n-k}$ is endowed with the family of (so-called) transverse dilations $\delta_\varepsilon^\S (q,x') = \delta_\varepsilon^q(0,x') = (q,\delta_\varepsilon(x'))$.

Following the nilpotentization procedure (see Appendix \ref{sec_nilp_mesures}), considering the smooth measure $\mu$ on $M$, we define on $N\S$ the smooth measure $\widehat{\mu}^\S$, 
homogeneous of degree $\mathcal{Q}^M(\S)-\mathcal{Q}^\S$ with respect to transverse dilations, by
\begin{equation}\label{muepsS}
\widehat{\mu}^\S = \lim_{\varepsilon\rightarrow 0 \atop \varepsilon\neq 0} \mu_\varepsilon^\S 
\qquad\textrm{where}\qquad
\mu_\varepsilon^\S = \vert\varepsilon\vert^{-\mathcal{Q}^M(\S)+\mathcal{Q}^\S} \big(\delta_\varepsilon^\S\big)^*\mu 
\end{equation}
with convergence in the vague topology.
It is called the \emph{transverse nilpotentization of $\mu$}.
Locally near $\S$, the manifold $M$ is identified with $N\S\simeq\S\times\R^{n-k}$ and
\begin{equation}\label{muS}
\widehat{\mu}^\S=\mu_\S\otimes dx'
\end{equation}
where $\mu_\S$ is a smooth measure on $\S$ and $dx'$ is the Lebesgue measure on $\R^{n-k}$.

Finally, denoting by $\mathcal{B}^{n-k}$ the unit Euclidean ball and by $\mathcal{S}^{n-k-1}=\partial\mathcal{B}^{n-k}$ the unit Euclidean sphere 
in $\R^{n-k}$, we define \emph{transverse polar coordinates} $(q,\tau,\sigma)$ in $\S\times[0,+\infty)\times\mathcal{S}^{n-k-1}$ on $N\S$.
The set $\S\times\mathcal{S}^{n-k-1}$ is endowed with the smooth measure $(\iota_{W}\widehat{\mu}^\S)_{\vert \S\times\mathcal{S}^{n-k-1}}$, the restriction to $\S\times\mathcal{S}^{n-k-1}$ of the contraction of $\widehat{\mu}^\S$ with the infinitesimal transverse dilation vector $W = \frac{d}{d\varepsilon}\big\vert_{\varepsilon=1}\delta_\varepsilon^\S$.
In local coordinates,  we have 
$$
(\iota_{W}\widehat{\mu}^\S)_{\vert \S\times\mathcal{S}^{n-k-1}} = \mu_\S\otimes d\sigma 
\quad\textrm{with}\quad 
d\sigma = \sum_{i=k+1}^n (-1)^{i-k-1}\ord(x_i) \, dx_{k+1}\cdots dx_{i-1}\, dx_{i+1}\cdots dx_n .
$$
Given any continuous function $g$ on $M$, compactly supported near $\S$, denoting $y=(q,\sigma)\in\S\times\mathcal{S}^{n-k-1}$, we have 
\begin{equation}\label{transverse_polar}
\begin{split}
\int_M g\, d\widehat{\mu}^\S &= \int _0^{+\infty} \tau ^{\mathcal{Q}^M(\S)-\mathcal{Q}^\S-1}  \int _{\S\times\mathcal{S}^{n-k-1}}g\big(\delta^\S_\tau (y)\big) \, d(\iota_{W}\widehat{\mu}^\S)(y) \, d\tau   \\
&= \int _0^{+\infty} \tau ^{\mathcal{Q}^M(\S)-\mathcal{Q}^\S-1}  \int _{\S} \int_{\mathcal{S}^{n-k-1}} g(\delta^q_\tau (\sigma)) \, d\sigma \, d\mu_\S(q) \, d\tau 
\end{split}
\end{equation}
where, with a slight abuse of notation, we write $\delta^q_\tau(\sigma)$ instead of $\delta^q_\tau(0,\sigma)$.
Recall that the $d$ in the integrals is not the exterior derivative but a notation meaning that the integral is performed with such or such measure.

\begin{remark}\label{rem_sphere_transverse}
The formula \eqref{transverse_polar} remains valid if the unit Euclidean sphere $\mathcal{S}^{n-k_1-1}$ is replaced by any piecewise smooth sphere transverse to the fibers (possibly, depending smoothly on $\tau$). Piecewise smoothness is required to perform integrations. SR spheres cannot be used in general because they may fail to be stratifiable (see \cite{BonnardTrelat_AFST2001}).
\end{remark}

\subsection{$(J+K)$-decomposition}\label{sec_J+K}
Let us assume that $\S$ is an equisingular smooth submanifold of $M$.
The $(J+K)$-decomposition consists of splitting the integral $I(t)$ defined by \eqref{def_int_I} as the sum of two integrals: 
$$
I(t)=J(t)+K(t)
$$
with
\begin{equation}\label{I=J+K}
J(t) = \int_{\mathcal{B}(\S,\sqrt{t})} f(q')\, e(t,q',q')\, d\mu(q') \qquad\textrm{and}\qquad
K(t) = \int_{M\setminus \mathcal{B}(\S,\sqrt{t})} f(q')\, e(t,q',q')\, d\mu(q') 
\end{equation}
where, with a slight abuse of notation, $\mathcal{B}(\S,\varepsilon) = \delta^\S_\varepsilon(\S\times\mathcal{B}^{n-k})$ is an $\varepsilon$-tubular neighborhood of $\S$ in $M$. Actually, $\sqrt{t}$ is exactly the right scale to use homogeneity properties and nilpotentizations.

As we are going to see, estimating $J(t)$ does not raise any difficulty and can be done without any specific assumption. The dominating term in $J(t)$ is related to the nilpotentization of the kernel along $\S$. In contrast, computations for $K(t)$ are much more difficult. We are going to perform a kind of blow-up along $\S$ using dilations. This will lead us to the necessity to consider iterated nilpotentizations, which complicate significantly the picture.

\subsubsection{Estimating $J(t)$}\label{sec_estimating_J}
%
Making the change of variable $q'=\delta_{\sqrt{t}}^\S(y)$ and using that 
$\big(\delta_{\sqrt{t}}^\S\big)^*\mu = (\sqrt{t})^{\mathcal{Q}^M(\S)-\mathcal{Q}^\S} \mu_{\sqrt{t}}^\S$ (see \eqref{muepsS}), we have
$$
J(t) = \frac{1}{t^{\mathcal{Q}^\S/2}} 
\int_{\S\times\mathcal{B}^{n-k}} f\big(\delta_{\sqrt{t}}^\S(y)\big)\, (\sqrt{t})^{\mathcal{Q}^M(\S)}\, e\big(t, \delta_{\sqrt{t}}^\S(y), \delta_{\sqrt{t}}^\S(y) \big)\, d\mu_{\sqrt{t}}^\S(y) 
$$
and it follows from Theorem \ref{lemfondamental} in Appendix \ref{app_lemfondam} that 
\begin{equation}\label{asympt_J}
J(t) = \frac{F_J(\sqrt{t})}{t^{\mathcal{Q}^\S/2}} \qquad\forall t>0
\end{equation}
for some $F_J\in C^\infty(\R)$ such that
\begin{equation*}
F_J(0) = 
\int_{\atop\!\!\S\times\mathcal{B}^{n-k}} {\mkern-30mu} f\circ\pi_\S(y) \, \widehat{e}^\S(1,y,y)\, d\widehat{\mu}^\S(y)
= \int_\S f(q) \int_{\mathcal{B}^{n-k}} \widehat{e}^q(1,(0,x),(0,x))\, dx\, d\mu_\S(q) 
\end{equation*}
where $\widehat{e}^\S$ is the mapping on $\S$ which to any $q\in\S$ assigns the nilpotentized heat kernel $\widehat{e}^q$ generated by the nilpotentized sR Laplacian $\widehat{\triangle}^q$.
%
We have concentration on $\S$ in this integral, which depends on $f$ restricted to $\S$. 
%
Note that \eqref{asympt_J} is completely general and does not require any specific assumption.

\begin{remark}\label{rem_J_nosqrt}
When $f=1$ near $\S$, there is no odd power of $\sqrt{t}$ in \eqref{asympt_J}, i.e., $F_J(\sqrt{t})$ can be replaced with $F_J(t)$.
Indeed, using the homogeneity property \eqref{homog_property_f_i} of Theorem \ref{lemfondamental} (in Appendix \ref{app_lemfondam}) with $\varepsilon=-1$, we find that $a_i^q(x,x)=f_i^q(1,x,x)$ satisfies $a_i^q(x,x) = (-1)^i a_i^q(\delta_{-1}(x),\delta_{-1}(x))$. As noted in Theorem \ref{lemfondamental}, this property implies that $a_i^q(0,0)=0$ for $i$ odd and thus the expansion \eqref{complete_expansion_1} does not involve odd powers of $\sqrt{t}$. But, actually, we can say more: for $i$ odd, $a_i^q$ is odd with respect to all variables $x_j$ whose nonholonomic order $\ord_q(x_j)$ is odd. Hence, when integrating on $\S\times\mathcal{B}^{n-k}$, all terms in $(\sqrt{t})^{2j+1}$, for $j\in\N$, are vanishing. 

This remark remains true if $f\big(\delta_{-1}^\S(y)\big)=f(y)$ for every $y\in\delta^\S_\varepsilon(\S\times\mathcal{B}^{n-k})$ for some $\varepsilon>0$, i.e., if $f$ is even with respect to $\S$, near $\S$, in the local chart (but this evenness property depends a priori on the choice of the privileged coordinates).

Note that, when $e=\widehat{e}^q$ and $\mu=\widehat{\mu}^q$ for some $q\in M$, we have $(\sqrt{t})^{\mathcal{Q}^M(\S)}\, e\big(t, \delta_{\sqrt{t}}^\S(y), \delta_{\sqrt{t}}^\S(y) \big) = e(1,y,y) = \widehat{e}^\S(1,y,y)$ and then $F_J(t) = F_J(0)+\mathrm{O}(t^\infty)$.
\end{remark}

\subsubsection{Estimating $K(t)$}\label{sec_estimating_K}
Since $e_{\triangle,\mu}(t,q,q)\, d\mu(q) = e_{\triangle,\widehat{\mu}^\S}(t,q,q)\, d\widehat{\mu}^\S(q)$, hereafter we consider the kernel (still denoted by $e$ to keep readability) associated with the measure $\widehat{\mu}^\S$ given by \eqref{muS}.
By \eqref{transverse_polar}, we have 
\begin{equation*}
K(t) = \int_{\sqrt{t}}^{+\infty} \tau^{\mathcal{Q}^M(\S)-\mathcal{Q}^\S-1} \int_{\S} \int_{\mathcal{S}^{n-k-1}} f\left(\delta_\tau^q(\sigma)\right)\, e\left(t,\delta_\tau^q(\sigma),\delta_\tau^q(\sigma)\right) \, d\sigma\, d\mu_\S(q)\, d\tau   .
\end{equation*}
Obtaining the small-time asymptotics for $K(t)$ is much more difficult than for $J(t)$.

Having in mind the procedure of nilpotentization at $q\in\S$, \eqref{relation_eeps_e} in Appendix \ref{app_lemfondam} gives
\begin{equation}\label{kernel_change1}
e\left( t, \delta_\tau^q(\sigma), \delta_\tau^q(\sigma) \right) = \tau^{-\mathcal{Q}^M(\S)} e_\tau^q \left( \frac{t}{\tau^2}, \sigma, \sigma \right) + \mathrm{O}(\vert\tau\vert^\infty)
\end{equation}
as $\tau\rightarrow 0$, because $\mathcal{Q}^M(q)=\mathcal{Q}^M(\S)$. In what follows, we will never write the infinite-order remainder term $\mathrm{O}(\vert\tau\vert^\infty)$, because it will have no impact on the small-time asymptotics of $K(t)$.
Recall that $e_\tau^q$ is the heat kernel generated by the sR Laplacian $\triangle_\tau^q$ where, denoting $D=\mathrm{Span}(X)$ with $X=(X_1,\ldots,X_m)$, the sR Laplacian $\triangle_\tau^q$ corresponds to the $m$-tuple $X_\tau^q = ((X_1)_\tau^q,\ldots,(X_m)_\tau^q)$, and $X_\tau^q = \tau (\delta_\tau^q)^* X \rightarrow \widehat{X}^{q}$ as $\tau\rightarrow 0$, with $\widehat{X}^{q} = (\widehat{X}^{q}_1,\ldots,\widehat{X}^{q}_m)$ (see Appendix \ref{sec_nilp_smoothsection}).
We obtain
\begin{equation}\label{intK_S=S1_1} 
K(t) = \int_{\sqrt{t}}^{+\infty} \tau^{-\mathcal{Q}^\S-1} \int_{\S}\int_{\mathcal{S}^{n-k-1}} f\left(\delta_\tau^q(\sigma)\right)\, e_\tau^q \left( \frac{t}{\tau^2}, \sigma,\sigma \right) \, d\sigma\, d\mu_\S(q) \, d\tau 
\end{equation}
and since $\S$ is equisingular, by Theorem \ref{lemfondamental}, $e_\tau^q$ depends smoothly on $\tau$ and on $q\in\S$ in $C^\infty$ topology, with $e_0^q = \widehat{e}^q$ (nilpotentization at $q$ of the heat kernel).

Considering \eqref{kernel_change1} and using \eqref{relation_eeps_e} again, but this time, at some point $\sigma$, we have (still neglecting the remainder terms)
\begin{equation}\label{kernel_change2}
e\left( t, \delta_\tau^q(\sigma), \delta_\tau^q(\sigma) \right) = \tau^{-\mathcal{Q}^M(\S)} e_\tau^q \left( \frac{t}{\tau^2}, \sigma, \sigma \right) 
= \frac{\tau^{\mathcal{Q}^{\R^n}(\sigma)-\mathcal{Q}^M(\S)}}{(\sqrt{t})^{\mathcal{Q}^{\R^n}(\sigma)}} 
e_{\tau,\sqrt{t}/\tau}^{q,\sigma}(1,0,0) 
\end{equation}
where, given arbitrary (fixed) $\tau\neq 0$ and $q\in\S$, for every $\varepsilon\in\R$ and every $\sigma\in\mathcal{S}^{n-k-1}$ (depending on $(\tau,q)$), $e_{\tau,\varepsilon}^{q,\sigma} = (e_\tau^q)_\varepsilon^\sigma$ is the heat kernel generated by the sR Laplacian $\triangle_{\tau,\varepsilon}^{q,\sigma} = (\triangle_\tau^q)_\varepsilon^\sigma$ corresponding to the $m$-tuple $X_{\tau,\varepsilon}^{q,\sigma} = (X_\tau^q)_\varepsilon^\sigma = \varepsilon (\delta_\varepsilon^\sigma)^* X_\tau^q$. 
We thus deal here with a \emph{double nilpotentization} procedure as the parameters $\tau$ and $\varepsilon$ converge to $0$: the parameter $\tau$ stands for the \emph{first} nilpotentization (of $D$ at $q\in\S$), and the parameter $\varepsilon$ stands for the \emph{second} nilpotentization (of $D^q_\tau$ at $\sigma\in\mathcal{S}^{n-k-1}$). 

We infer from \eqref{intK_S=S1_1} and \eqref{kernel_change2} that
\begin{equation}\label{intK_S=S1_2} 
K(t) =  \int_{\sqrt{t}}^{+\infty} \int_{\S}\int_{\mathcal{S}^{n-k-1}} \frac{\tau^{\mathcal{Q}^{\R^n}(\sigma)-\mathcal{Q}^\S-1}}{t^{\mathcal{Q}^{\R^n}(\sigma)/2}}  f\left(\delta_\tau^q(\sigma)\right)\, 
e_{\tau,\sqrt{t}/\tau}^{q,\sigma}(1,0,0) 
\, d\sigma\, d\mu_\S(q)\, d\tau . 
\end{equation}
Several remarks are in order.

\begin{remark}
Assuming that $\S$ is an equisingular smooth submanifold, we have $\mathcal{Q}^{\R^n}(\sigma)=\Qeq$ in \eqref{intK_S=S1_2} (because $\sigma$ is outside of $\S$) and then $t^{\Qeq/2} K(t)$ is the integral of a function of $(q,\sigma,\tau,\varepsilon=\sqrt{t}/\tau)$, which is smooth with respect to $q\in\S$ and $\sigma\in M\setminus\S$ because $\S$ is equisingular (by Theorem \ref{lemfondamental}, but not necessarily smooth with respect to $(\tau,\varepsilon)$: indeed, the limit of 
$e_{\tau,\varepsilon}^{q,\sigma}(1,0,0)$ 
as $(\tau,\varepsilon)\rightarrow (0,0)$ is not necessarily well defined! 
This first remark motivates the next Section \ref{sec_nilpotentizability_doublenilp}, in which we are going to prove that, under the so-called $\S$-nilpotentizability assumption, the double limit exists, 
$e_{\tau,\varepsilon}^{q,\sigma}(1,0,0)$ 
depends smoothly on $(\tau,\varepsilon,q,\sigma)$ and 
is equal to the double nilpotentization of the heat kernel at $(\tau,\varepsilon)=(0,0)$.
Under these two assumptions, in Section \ref{sec_equisingular_nilp}, we will then infer the local Weyl law. 
\end{remark}

\begin{remark}\label{rem_prelim_JK}
Assuming that $\S$ is an equisingular smooth submanifold, but that the nilpotentizability assumption is not satisfied, 
$e_{\tau,\varepsilon}^{q,\sigma}(1,0,0)$ 
may blow up as $(\tau,\varepsilon)\rightarrow (0,0)$, and then, computing its asymptotics is required to estimate that of $K(t)$.
This issue will be investigated in Section \ref{sec_nonnilp}.

At this step, we can however make the following remark. Since we always have $1=\mathrm{O} ( e_{\tau_1,\varepsilon}^{q_1,\sigma} (1,0) )$ for $(\tau_1,\varepsilon)\in[-1,1]^2$, we claim that, when $f$ is a positive continuous function, we have $J(t)=\mathrm{O}(K(t))$ as $t\rightarrow 0^+$. More precisely, as $t\rightarrow 0^+$:
\begin{itemize}[parsep=0cm,topsep=0cm,itemsep=1mm]
\item if $\Qeq > \mathcal{Q}^\S$ then $\displaystyle \frac{1}{t^{\Qeq/2}} = \mathrm{O}(K(t))$ and $\displaystyle J(t) \sim \frac{\Cst}{t^{\mathcal{Q}^\S/2}} = \mathrm{o}(K(t))$;
\item if $\Qeq = \mathcal{Q}^\S$ then $\displaystyle \frac{\vert\ln t\vert}{t^{\Qeq/2}} = \mathrm{O}(K(t))$ and $\displaystyle J(t) \sim \frac{\Cst}{t^{\mathcal{Q}^\S/2}} = \mathrm{o}(K(t))$;
\item if $\Qeq < \mathcal{Q}^\S$ then $\displaystyle \frac{1}{t^{\mathcal{Q}^\S/2}} = \mathrm{O}(K(t))$ and $\displaystyle J(t) \sim \frac{\Cst}{t^{\mathcal{Q}^\S/2}} = \mathrm{O}(K(t))$.
\end{itemize}
Indeed, since $\mathcal{Q}^{\R^n}(\sigma) \geq \Qeq$, we have $\big( \frac{\tau}{\sqrt{t}} \big)^{\mathcal{Q}^{\R^n}(\sigma)} \geq \big( \frac{\tau}{\sqrt{t}} \big)^{\Qeq}$ because $\frac{\tau}{\sqrt{t}}\geq 1$ in the integral \eqref{intK_S=S1_2} , and the result follows, using that 
$1=\mathrm{O}\big( e_{\tau_1,\sqrt{t}/\tau_1}^{q_1,\sigma} (1,0,0) \big)$.
This implies that the dominating term in the small-time asymptotics of $I(t)$ is of the order of that of $K(t)$ (but $J(t)$ may contribute to the main term when $\Qeq < \mathcal{Q}^\S$).
\end{remark}

\begin{remark}\label{remK_strat}
When $\S$ is stratified by equisingular smooth submanifolds, $\mathcal{Q}^{\R^n}(\sigma)$ 
may take various values, depending on whether $\sigma$ belongs to $M\setminus\S$ or to some stratum of $\S$. Then, the integral $K(t)$ has to be split according to this stratification and this will lead us, in Section \ref{sec_equisingular_stratified_nilp}, in order to establish the local Weyl law in the equisingular stratified case, to apply iteratively the $(J+K)$-procedure.
\end{remark}

\subsection{Nilpotentizability and double nilpotentization}\label{sec_nilpotentizability_doublenilp}
Given any $q\in M$ and any $\tau\in\R\setminus\{0\}$, the horizontal distribution $D_\tau^q = (\delta_\tau^q)^*D = \mathrm{Span}(X_\tau^q)$ (where $X_\tau^q = \tau (\delta_\tau^q)^*X$ and $X=(X_1,\ldots,X_m)$) is diffeomorphic to $D$ but this diffeomorphism may fail to be uniform when $\tau\rightarrow 0$.
In other words, all $D_\tau^q$, $\tau>0$, are diffeomorphic, but may fail to be diffeomorphic to $\widehat{D}^q$.

\begin{example}\label{simpleexample_nonnilp}
For instance, consider a two-dimensional manifold $M$, locally identified to $\R^2$ near $q=(0,0)$, endowed with the horizontal distribution $D=\mathrm{Span}(X)$ with $X=(X_1,X_2)$ given by $X_1=\partial_1$ and $X_2=(x_1^2+x_2^2)\,\partial_2$. This is an almost-Riemannian case, of singular set $\S=\{q\}$. In $M\setminus\S$, it is Riemannian and the sR weights are $w_1^{M\setminus\S}(D)=w_2^{M\setminus\S}(D)=1$, while  $w_1^q(D)=1$ and $w_2^q(D)=3$. We have $(X_1)_\tau^q=\partial_1$ and $(X_2)_\tau^q=(x_1^2+\tau^4x_2^2)\,\partial_2$ and thus the singular set of $D_\tau^q=\mathrm{Span}(X_\tau^q)$ is $\S_\tau^q=\{(0,0)\}$ for every $\tau>0$. However, the nilpotentized distribution is $\widehat{D}^q=\mathrm{Span}(\widehat{X}^q)$ with $\widehat{X}^q=(\widehat{X}^q_1,\widehat{X}^q_2)$ given by $\widehat{X}^q_1=\partial_1$ and $\widehat{X}^q_2=x_1^2\,\partial_2$, with singular set $\widehat{\S}^q=\{x_1=0\}$. Hence, on this example, $D=\mathrm{Span}(X)$ and $\widehat{D}^q=\mathrm{Span}(\widehat{X}^q)$ are not diffeomorphic.
\end{example}

\subsubsection{Nilpotentizability}\label{sec_nilpotentizability}
The above-mentioned loss of uniformity motivates the following definition.

\begin{definition}\label{def_nilpotentizable}
%
%
%
Let $N$ be a smooth submanifold of $M$. 

The horizontal distribution $D=\mathrm{Span}(X_1,\ldots,X_m)$ is said to be $N$-\emph{nilpotentizable} 
if $D$ is locally diffeomorphic to $\widehat{D}^q$ at every point $q\in N$, smoothly with respect to $q\in N$, in the following sense: for every $q\in N$, there exist a neighborhood $U$ of $q$ in $M$, a neighborhood $V$ of $0$ in $\R^n$, and a diffeomorphism $\phi^q:U\rightarrow V$, with $\phi^q(q)=0$, such that $\phi^q_*D=\widehat{D}^q$, i.e., $d\phi^q((\phi^q)^{-1}(x)).D((\phi^q)^{-1}(x)) = \widehat{D}^q(x)$ for every $x\in V$, and such that $\phi^q$ depends smoothly on $q\in N$ in $C^\infty$ topology.

This also means that there exist smooth functions $a_{ij}^q$ on $V$, smoothly depending on $q\in N$, such that $d\phi^q((\phi^q)^{-1}(x)).X_j((\phi^q)^{-1}(x)) = \sum_{i=1}^m a_{ij}^q(x) \widehat{X}_i^q(x)$ for every $x\in V$ and for every $j\in\{1,\ldots,m\}$, and the $m$-by-$m$ matrix $A^q(x)=(a_{ij}^q(x))_{1\leq i,j\leq m}$ is invertible at $x=0$.
In terms of the m-tuples $X=(X_1,\ldots,X_m)$ and $\widehat{X}^q=(\widehat{X}^q_1,\ldots,\widehat{X}^q_m)$ viewed as $n$-by-$m$ matrices, the latter equality is written as $\phi^q_*X = \widehat{X}^q A^q$.

When $N=\cup_{k=1}^s N_k$ is a Whitney stratified submanifold of $M$, where the strata $N_k$ are smooth submanifolds of $M$, 
we say that $D$ is $N$-nilpotentizable if $D$ is $N_k$-nilpotentizable for every 
$k\in\{1,\ldots,s\}$.
\end{definition}

We have written the above definition for a general submanifold $N$, but in this article the nilpotentizability concept is always used with $N=\S$ (singular set of $D$), except in the following remark.

\begin{remark}\label{rem_nilp_equiregular_region}
Assume that $\S=\cup_{k=1}^s \S_k$ is a Whitney stratified submanifold of $M$, where the strata $\S_k$ are equisingular smooth submanifolds of $M$. 
We thus have $M=\cup_{k=1}^s \S_k \bigcup (M\setminus\S)$, i.e., $M$ is Whitney stratified by equisingular smooth submanifolds (the stratum $M\setminus\S$ is the open regular region).

Let us make the following important observation: in Definition \ref{def_nilpotentizable}, $\S$-nilpotentizability means that $D$ is $\S_k$-nilpotentizable for every $k\in\{1,\ldots,s\}$. We \emph{do not assume} that $D$ is $(M\setminus\S)$-nilpotentizable, i.e., we do not assume that $D$ is locally diffeomorphic to its nilpotentization in the regular region.
Despite the fact that $D$ may fail to be locally diffeomorphic to $\widehat{D}^q$ at every $q\in M\setminus\S$, it is however true that the sR weights of $D$ at $q$ (defined in Appendix \ref{app_sRflag}) coincide with the sR weights of $\widehat{D}^q$ at $0$ (on this concern, see also Remark \ref{rem_weakerassumptions_nilp}) further.

It is anyway interesting to note that, if $D$ is not only $\S$-nilpotentizable but also $(M\setminus\S)$-nilpotentizable, then automatically $\S$ must be Whitney stratifiable by equisingular smooth strata (because $\widehat{D}^q$ is so).
\end{remark}

The uniformity property mentioned at the beginning of Section \ref{sec_nilpotentizability_doublenilp} is recovered under the nilpotentizability assumption (it fails in Example \ref{simpleexample_nonnilp} because $D$ is not $\S$-nilpotentizable).
Indeed, when $D$ is $\S$-nilpotentizable, we have $D = (\phi^q)^* \widehat{D}^q$. Following Appendix \ref{sec_nilp_diffeo}, we set $\phi_\tau^q = \delta_{1/\tau} \circ \phi^q \circ \delta_\tau^q$ for every $\tau\in\R\setminus\{0\}$ and $\phi_0^q = \widehat{\phi^q}^q = \lim_{\tau\rightarrow 0} \phi_\tau^q$, for every $q\in N_k$. Here, $\phi_0^q = \widehat{\phi^q}^q$ is the nilpotentization of the diffeomorphism $\phi^q$ at the point $q$. Then the family of diffeomorphisms $\phi_\tau^q$ depends smoothly on $q\in \S_k$ and continuously on $\tau\in\R$.
Moreover, using the homogeneity property $\widehat{D}^q = \delta_{1/\tau}^* \widehat{D}^q$, we have $D_\tau^q = (\phi_\tau^q)^* \widehat{D}^q$ for all $q\in \S_k$ and $\tau\in\R$.
In other words, if $D$ is $\S$-nilpotentizable then all $D_\tau^q$ are diffeomorphic to $\widehat{D}^q$, and these diffeomorphisms depend smoothly on $q\in \S_k$ and continuously on $\tau$ in $C^\infty$ topology.

\paragraph{Comments on the concept of nilpotentizability.}
The nilpotentizability assumption has been much used with $N=M$ in the context of motion planning (see \cite[Sections 3.1 and 3.2]{Jean_2014}), although nilpotentizability usually means being diffeomorphic to a nilpotent distribution only. Here, our definition is stronger since we require that $D\sim\widehat{D}^q$ at any $q\in M$ (hence, a nilpotent distribution may fail to be nilpotentizable!): it coincides with the notion of being ``strongly nilpotent" considered in \cite{Mormul_2003, Mormul_2005}. Its validity is related to the theory of normal forms of distributions (see \cite{Zh-92}). When $n\leq 4$, since there are no moduli in their normal forms, all equiregular horizontal distributions are nilpotentizable (see \cite{AgrachevBarilariBoscain_CV2012, Hermes_1989}). By the Darboux theorem, every horizontal distribution of rank $m=n-1$, which is regular at $q\in M$, is nilpotentizable near $q$ (see \cite{Hermes_1989, HermesLundellSullivan_JDE1984}).
In the Baouendi-Grushin case without tangency point and in the nonsingular Martinet case, the horizontal distribution is $\S$-nilpotentizable. The nilpotentizability assumption allows however to have moduli in the horizontal distributions. Given any two integers $2\leq m<n$ and any $q\in M$, there exists a horizontal distribution of $m$ vector fields (singular at $q$) that is not nilpotentizable near $q$ (see \cite{Hermes_1989}). 
The nilpotentizability assumption is not generic when $n$ is large enough (see \cite{Jean_2014, Mormul_2003, Mormul_2005}).    


\subsubsection{Double nilpotentization}\label{sec_doublenilp}
Throughout this section, we assume that the singular set $\S$ (and thus $M$) is Whitney stratified by equisingular smooth submanifolds and that $D$ is $\S$-nilpotentizable.

Let $\S_1$ and $\S_2$ be two equisingular strata of $M$ such that $\dim\S_1<\dim\S_2$ and $\S_1\subset\overline{\S_2}$. In the case where $\S$ is an equisingular smooth submanifold of $M$ (as in Section \ref{sec_equisingular_nilp}), we have $\S_1=\S$ and $\S_2=M\setminus\S$.

In this section, we explain how to perform a double nilpotentization: the first at some point $q_1\in\S_1$ and the second (in a sense to made precise) at some point of $\S_2$ nearby $q_1$.

\medskip

Let $\psi^{q_1}$ be a chart of privileged coordinates at $q_1\in M$ (see Appendix \ref{app_privileged}) defined as the composition of a chart of privileged coordinates at $0\in\widehat{M}^{q_1}\simeq\R^n$ with the local diffeomorphism $\phi^{q_1}$ given by Definition \ref{def_nilpotentizable} that maps $D$ to $\widehat{D}^{q_1}$. 

Recalling that $X=(X_1,\ldots,X_m)$ and that $\delta^{q_1}_{\tau_1}=(\psi^{q_1})^{-1}\circ\delta_{\tau_1}$, in the chart $\psi^{q_1}$ we identify $X^{q_1}_{\tau_1}=\tau_1(\delta^{q_1}_{\tau_1})^*X$ with $X$ and with its nilpotentization $\widehat{X}^{q_1}$ at $q_1$, for any $\tau_1\in[-1,1]$ (actually, if we do not perform this identification, we have a diffeomorphism depending smoothly on $\tau_1\in[-1,1]$ and on $q_1\in\S_1$), with the agreement that $X^{q_1}_0=\widehat{X}^{q_1}$ for $\tau_1=0$.

Hence, the horizontal distribution $D^{q_1}_{\tau_1}=(\delta^{q_1}_{\tau_1})^*D=\mathrm{Span}(X^{q_1}_{\tau_1})$ is identified with $D=\mathrm{Span}(X)$ and with $\widehat{D}^{q_1}=\mathrm{Span}(\widehat{X}^{q_1})$.
The singular set $\S_{\tau_1}^{q_1} = (\delta_{\tau_1}^{q_1})^{-1}(\S)$ of $D_{\tau_1}^{q_1}$ is identified with the singular set $\S$ of $D$ and with the singular set $\widehat{\S}^{q_1}$ of $\widehat{D}^{q_1}$, and similarly for the strata: $(\S_j)^{q_1}_{\tau_1}\simeq \S_i\simeq \widehat{\S_j}^{q_1}$, for $j=1,2$.

\medskip

Now, let $q_2\in\S_2$ belong to the chart $\psi^{q_1}$, and let us nilpotentize $X^{q_1}_{\tau_1}\simeq X$ at $q_2$, for any $\tau_1\in[-1,1]$. 
To perform this second nilpotentization, we use another chart $\psi^{q_2}$ of privileged coordinates $y=\psi^{q_2}(x)$ at $q_2$.
We set 
$$
X_{\tau_1,\tau_2}^{q_1,q_2} = (X_{\tau_1}^{q_1})_{\tau_2}^{q_2} = \tau_2\big(\delta_{\tau_2}^{q_2}\big)^* X_{\tau_1}^{q_1}
\qquad\textrm{and}\qquad
D_{\tau_1,\tau_2}^{q_1,q_2} = \big(\delta_{\tau_2}^{q_2}\big)^*D^{q_1}_{\tau_1} = \mathrm{Span}(X_{\tau_1,\tau_2}^{q_1,q_2}) 
$$
(where $\delta^{q_2}_{\tau_2} = (\psi^{q_2})^{-1}\circ\delta_{\tau_2}$).
Note that the dilation $\delta_{\tau_1}$ in $\delta^{q_1}_{\tau_1}=(\psi^{q_1})^{-1}\circ\delta_{\tau_1}$ is defined with the sR weights $w_i^{\S_1}(D)$ of $D$ along the stratum $\S_1$ (see Appendix \ref{app_sRflag}), while the dilation $\delta_{\tau_2}$ in $\delta^{q_2}_{\tau_2} = (\psi^{q_2})^{-1}\circ\delta_{\tau_2}$ is defined with the sR weights $w_i^{\S_2}(D)$ of $D$ along $\S_2$. 
Since $D$ is $\S$-nilpotentizable, we have $w_i^{\S_j}(D) = w_i^{\widehat{\S_j}^{q_1}}(\widehat{D}^{q_1})$.

Therefore, $X_{\tau_1,\tau_2}^{q_1,q_2}$ has an extension at $\tau_1\tau_2=0$ that depends smoothly on $(\tau_1,\tau_2)\in[-1,1]^2$, $q_1\in\S_1$, $q_2\in\S_2$ and  
that $X_{\tau_1,\tau_2}^{q_1,q_2} = \widehat{X}^{q_1,q_2}+ \mathrm{O(\tau_1,\tau_2)}$
as $(\tau_1,\tau_2)\rightarrow (0,0)$, where we have set
$$
\widehat{X}^{q_1,q_2} = \widehat{\widehat{X}^{q_1}}^{q_2}
$$
which is the nilpotentization of $\widehat{X}^{q_1}$ at $q_2$: this \emph{double nilpotentization} is the nilpotentization at $q_2$ of the nilpotentization of $X$ at $q_1$.

Accordingly, we denote by $e_{\tau_1,\tau_2}^{q_1,q_2}$ the heat kernel generated by the sR Laplacian corresponding to $X_{\tau_1,\tau_2}^{q_1,q_2}$. When $\tau_1=\tau_2=0$, $\widehat{e}^{\, q_1,q_2} = \widehat{\,\widehat{e}^{q_1}}^{q_2}$ is the heat kernel generated by the sR Laplacian $\widehat{\triangle}^{q_1,q_2} = \widehat{\widehat{\triangle}^{q_1}}^{q_2}$ corresponding to the $m$-tuple $\widehat{X}^{q_1,q_2} = \widehat{\widehat{X}^{q_1}}^{q_2}$ (double nilpotentization).

\begin{lemma}\label{lem_uniform_double_nilp}
The heat kernel $e_{\tau_1,\tau_2}^{q_1,q_2}(t,y,y')$ is a smooth function of $(\tau_1,\tau_2)\in[-1,1]^2$, $q_1\in\S_1$, $q_2\in\S_2$, $t\in(0,+\infty)$, $(y,y')\in\R^n\times\R^n$. This function is even with respect to $\tau_2$ when $y=y'=0$, and
$e_{\tau_1,\tau_2}^{q_1,q_2}(t,y,y') = \widehat{e}^{\, q_1,q_2}(t,y,y') + \mathrm{O}(\tau_1,\tau_2)$ as $(\tau_1,\tau_2)\rightarrow (0,0)$ in $C^\infty((0,+\infty)\times\R^n\times\R^n)$.
\end{lemma}

\begin{proof}
The smoothness of the heat kernel with respect to its arguments comes from Theorem \ref{lemfondamental} in Appendix \ref{app_lemfondam}.

The fact that $e_{\tau_1,\tau_2}^{q_1,q_2}(t,0,0)$ is an even function of $\tau_2$ 
follows from \eqref{complete_expansion} and \eqref{homog_property_f_i} in Theorem \ref{lemfondamental}, applied to the kernel $e_{\tau_1}^{q_1}$ (depending smoothly on $\tau_1$ and $q_1$), which give 
$e_{\tau_1,\tau_2}^{q_1,q_2}(t,0,0) 
= \widehat{\, e_{\tau_1}^{q_1}}^{q_2}(t,0,0) + \sum_i \tau_2^{2i} f_{2i}^{q_1,q_2,\tau}(t,0,0) + \mathrm{O}(\vert\tau_2\vert^\infty)$ as $\tau_2\rightarrow 0$. 
\end{proof}

\begin{remark}\label{rem_multiple_nilp}[Multiple nilpotentization]
We have defined the double nilpotentization. By induction, it is straightforward to define the multiple nilpotentization, which will be used in Section \ref{sec_equisingular_stratified_nilp} in order to investigate the equisingular stratified nilpotentizable case.
Taking $j\geq 2$ strata of increasing dimensions such that $\S_1\subset\cdots\subset\S_j$ and taking points $q_i\in\S_j$, for $i\in\{1,\ldots,j\}$, all of them in a sufficiently small neighborhood, the $m$-tuple of vector fields
\begin{equation}\label{def_X_nilp_j}
X_{\tau_1,\ldots,\tau_j}^{q_1,\ldots,q_j} = \tau_j \big( \delta^{q_j}_{\tau_j} \big)^* X_{\tau_1,\ldots,\tau_{j-1}}^{q_1,\ldots,q_{j-1}} = \tau_1\cdots\tau_j \big( \delta^{q_1}_{\tau_1} \circ\cdots\circ \delta^{q_j}_{\tau_j} \big)^* X 
\end{equation}
has an extension at $\tau_1\cdots\tau_j$ that depends smoothly on $(\tau_1,\ldots,\tau_j)\in[-1,1]^j$ and on $(q_1,\ldots,q_j)\in\S_1\times\cdots\times\S_j$, and is equal to $\widehat{X}^{q_1,\ldots,q_j}$ at $(\tau_1,\ldots,\tau_j)=(0,\ldots,0)$. Here, the multiple nilpotentization $\widehat{X}^{q_1,\ldots,q_j}$ is the $m$-tuple $X$ that is first nilpotentized at $q_1$, then at $q_2$, etc, and finally at $q_j$; i.e., it is defined by the induction 
$$
\widehat{X}^{q_1,\ldots,q_{i+1}} = \widehat{\widehat{X}^{q_1,\ldots,q_i}}^{\!\! q_{i+1}} \qquad \forall  i\in\{1,\ldots,j\}.
$$ 
Accordingly, the singular set $(\S)_{\tau_1,\ldots,\tau_j}^{q_1,\ldots,q_j}$ of $D_{\tau_1,\ldots,\tau_j}^{q_1,\ldots,q_j} = \mathrm{Span}(X_{\tau_1,\ldots,\tau_j}^{q_1,\ldots,q_j})$ is diffeomorphic to $\S$.

Lemma \ref{lem_uniform_double_nilp} is straightforwardly generalized as follows: denoting by $e_{\tau_1,\ldots,\tau_j}^{q_1,\ldots,q_j}$ the heat kernel generated by the sR Laplacian corresponding to $X_{\tau_1,\ldots,\tau_j}^{q_1,\ldots,q_j}$, the function $e_{\tau_1,\ldots,\tau_j}^{q_1,\ldots,q_j}(t,y,y')$ depends smoothly on $(\tau_1,\ldots,\tau_j)\in[-1,1]^2$, $q_i\in\S_i$ for $i\in\{1,\ldots,j\}$, $t\in(0,+\infty)$ and $(y,y')\in\R^n\times\R^n$. 
\end{remark}

\begin{remark}\label{rem_weakerassumptions_nilp}
In Lemma \ref{lem_uniform_double_nilp}, the $\S$-nilpotentizability assumption can be slightly weakened to the following assumption: \emph{$\widehat{\S_j}^{q_1}$ is a nonempty equisingular submanifold whose sR weights coincide with the sR weights along $\S_j$, for $j=1,2$}.

More generally, if the nilpotentizability or the above more general assumption fails, although $\S_{\tau_1}^{q_1}$ is diffeomorphic to $\S$ for every $\tau_1\neq 0$, for $\tau_1=0$ the singular set $\widehat{\S}^{q_1}$ of $\widehat{D}^{q_1}$ may fail to be diffeomorphic to $\S$ in general (it may even be empty). This situation, studied in Section \ref{sec_nonnilp}, is much more challenging because $e_{\tau_1,\tau_2}^{q_1,q_2}(t,y,y')$ has no limit as $(\tau_1,\tau_2)\rightarrow 0$ and its blowing-up asymptotics must be studied.
\end{remark}

\section{Local Weyl law in the equisingular nilpotentizable case}\label{sec_equisingular_nilp}
Throughout this section, we assume that the singular set $\S$ is an equisingular smooth submanifold of $M$, of topological dimension $k\in\{0,\ldots,n-1\}$ and of Hausdorff dimension $\mathcal{Q}^\S$, and that the horizontal distribution $D$ is $\S$-nilpotentizable. 

We establish the local Weyl law in Theorem \ref{thm_onestratum} and we identify the main terms in an intrinsic way.
This case is already representative of a number of examples (see Section \ref{sec_examples_onestratum}).
The result that we obtain covers, in particular, the Baouendi-Grushin case without tangency point and the nonsingular Martinet case, which we treat in more generality in Sections \ref{sec_Baouendi-Grushin} and \ref{sec_Martinet}. In Section \ref{sec_QE}, we also derive Quantum Ergodicity properties.



\subsection{Hadamard finite part}
In order to identify terms in the local Weyl law in an intrinsic way in particular in the case where $\mathcal{Q}^\S=\Qeq$, we use the concept of \emph{Hadamard finite part}.

The Hadamard finite part can be defined in several ways, depending on the class of singular integrals under consideration. Here, we use the following definition.
Let $g$ be a function of class $C^1$ on $\R$. Let $\beta$ be the Borel measure on $\R\setminus\{0\}$ defined by 
$\frac{d\beta}{ds}=\frac{g(s)}{\vert s\vert}$. 
Given any $C^1$ function $f$ of compact support on $\R$, the integral $\int_\R f\, d\beta$ is singular at $0$. Its Hadamard finite part (in french, ``partie finie", which explains the usual short notation $\mathrm{p.f.}$) is defined by
$$
\mathrm{p.f.} \int_{\R\setminus\{0\}} f\, d\beta = \lim_{\varepsilon\rightarrow 0^+} \left( \int_{\vert s\vert\geq\varepsilon} f(s)\frac{g(s)}{\vert s\vert}\, ds - C\ln\frac{1}{\varepsilon} \right) 
$$
where $C$ is the unique real number for which this limit exists, namely, $C=2f(0)g(0)$. Note that, since the function $s\mapsto \frac{f(s)g(s)-f(0)g(0)}{\vert s\vert}$ has a continuous extension at $0$, we have
$$
\mathrm{p.f.} \int_{\R\setminus\{0\}} f\, d\beta = 2f(0)g(0)\ln(a) + \int_{-a}^a \frac{f(s)g(s)-f(0)g(0)}{\vert s\vert}\, ds 
$$
for every $a>0$ such that $\supp(f)\subset[-a,a]$. The quantity at the right-hand side does not depend on $a$.

In view of generalizing the definition to manifolds, it is useful to make the following remark. The set $\R$ is endowed with the family of dilations $\delta_\varepsilon(s)=\varepsilon s$, for $\varepsilon>0$. Defining $\beta^0=\lim_{\varepsilon\rightarrow 0^+}\delta_\varepsilon^*\beta$, we have $\frac{d\beta^0}{ds}=g(0)\frac{ds}{\vert s\vert}$ and $\beta^0$ is homogeneous of degree $0$. Denoting by $S(\R)$ the quotient of $\R$ under positive dilations and by $W=s\, \partial_s$ the infinitesimal dilation vector, we have $2g(0)=(\iota_W\beta^0)_{\vert S(\R)}$ and the above unique constant $C$ such that the limit exists is $C=\int_{S(\R)} f(0)\, (\iota_W\beta^0)_{\vert S(\R)}(dx)$.

\medskip


Let us now define the Hadamard finite part in an intrinsic way on the sR manifold $(M,D,g)$.
Let $\beta$ be a smooth measure on $M\setminus\S$, blowing up near $\S$, assumed to have a transverse nilpotentization $\widehat{\beta}^\S = \lim_{\tau\rightarrow 0} (\delta_\tau^\S)^*\beta$ that is homogeneous of degree $0$. An example is the Popp measure in the nilpotentizable case where $\mathcal{Q}^\S=\Qeq$: it has a smooth density with respect to $\mu$ in $M\setminus\S$, which blows up near $\S$.
Like in Section \ref{sec_geom_context}, $\S\times\mathcal{S}^{n-k-1}$ is then endowed with the smooth measure $(\iota_W\widehat{\beta}^\S)_{\vert \S\times\mathcal{S}^{n-k-1}}$.
Note that $\widehat{\beta}^\S = \frac{d\tau}{\tau}\otimes (\iota_W\widehat{\beta}^\S)_{\vert \S\times\mathcal{S}^{n-k-1}}$: indeed, both sides of the equality are homogeneous of degree $0$ and coincide on $\tau=1$.

Given any $C^1$ function $f$ of compact support on $M$, the integral $\int_M f\, d\beta$ 
is singular along $\S$. We define its Hadamard finite part by
$$
\mathrm{p.f.} \int_{M\setminus\S} f\, d\beta = \lim_{\varepsilon\rightarrow 0^+} \left( \int_{M\setminus \mathcal{B}(\S,\varepsilon)} f\, d\beta - C \ln\frac{1}{\varepsilon} \right)
$$
where $C = \int_{\S\times\mathcal{S}^{n-k-1}} f\circ\pi_\S(y)\, d(\iota_{W}\widehat{\beta}^\S)(y)$ is the unique real number for which the limit exists. 
The Hadamard finite part is intrinsic: it depends only on the sR structure and on $\beta$.

\medskip


Given any $q\in\S$, recalling that, near $q$, $\mu=\mu_\S\otimes dx$ where $\mu_\S$ is a smooth measure on $\S$ and $dx$ is the Lebesgue measure on $N_q\S\simeq\R^{n-k}$, we define the ``transverse trace" 
$$
\Tr_{N_q\S}(e^{\widehat{\triangle}^q}) = \int_{N_q\S} \widehat{e}^q(1,(0,x),(0,x))\, dx  
$$
of the nilpotentized heat semi-group. The mesure on $\S$ of density $\Tr_{N_q\S}(e^{\widehat{\triangle}^q})$ which respect to $\mu_\S$ does not depend on the smooth measure $\mu$.
We will see that $\Tr_{N_q\S}(e^{\widehat{\triangle}^q})<+\infty$ if and only if $\mathcal{Q}^M(q)>\Qeq$. When $\mathcal{Q}^M(q)=\Qeq$, the integral diverges in a logarithmic way at infinity and we define its Hadamard finite part by
$$
\mathrm{p.f.} \ \Tr_{N_q\S}(e^{\widehat{\triangle}^q}) = \lim_{\varepsilon\rightarrow 0^+} \left( \int_{\mathcal{B}^{n-k}(0,1/\varepsilon)} \widehat{e}^q(1,(0,x),(0,x))\, dx - C \ln\frac{1}{\varepsilon} \right)
$$
where $C$ is the unique real number for which the limit exists. 


\medskip

Finally, as a prelude to Theorem \ref{thm_onestratum} hereafter, we start by noting that, by Theorem \ref{thm_local_weyl_equiregular} in the equiregular case, given any $f\in C^\infty_c(M\setminus\S)$ (the set of smooth functions compactly supported on $M\setminus\S$), we have the expansion
\begin{equation}\label{prelude_equireg}
\Tr(\mathcal{M}_f \, e^{t\triangle}) = \int_M f(q)\, e(t,q,q)\, d\mu(q)
= \frac{1}{t^{\Qeq/2}} \sum_{j=0}^{+\infty}T_j(f) t^j  + \mathrm{O}(t^\infty)
\end{equation}
as $t\rightarrow 0^+$, where $T_j$ is the Schwartz distribution defined by $T_j(f) = \int_M a_j(q) f(q)\, d\mu(q)$, for some function $a_j\in C^\infty(M\setminus\S)$, for every $j\in\N$.

Let $C^\infty_0(M\setminus\S)$ be the closure of $C^\infty_c(M\setminus\S)$ in the Fr\'echet space $C^\infty(M)$: any function $f\in C^\infty_0(M\setminus\S)$ can be extended to a function on $M$ that is flat on $\S$, i.e., $f$ and all its derivatives vanish on $\S$.
Under the assumptions of Theorem \ref{thm_onestratum} below (in particular, $\S$-nilpotentizability), for every $j\in\N$, 
by the Hahn-Banach theorem applied in the Fr\'echet space $C^\infty(M)$ to the continuous linear form $T_j$ on the closed subspace $C^\infty_0(M\setminus\S)$, the distribution $T_j$ can be extended in a non-unique and non-canonical way to a distribution on $M$. Such an extension is done modulo a distribution supported on $\S$ and can be seen as a generalized Hadamard finite part (which we do not make explicit).

\subsection{Local Weyl law}

\begin{theorem}\label{thm_onestratum}
Given any $f\in C^\infty(M)$, the function $t\mapsto\Tr(\mathcal{M}_f \, e^{t\triangle})= \int_M f(q)\, e(t,q,q)\, d\mu(q)$ can be written in a unique way, modulo functions which are $\mathrm{O}(t^\infty)$ as well as their derivatives as $t\rightarrow 0^+$, as the sum of three terms: 
\begin{equation}\label{trace_sum3}
\frac{1}{t^{\Qeq/2}}\sum_{j=0}^{+\infty} \tilde T_j(f) t^j + \frac{1}{t^{\mathcal{Q}^\S/2}}\sum_{j=0}^{+\infty} R_j(f) t^{j/2} + \frac{\vert\ln t\vert}{t^{\min(\Qeq,\mathcal{Q}^\S)/2}} \sum_{j=0}^{+\infty} S_j(f) t^{j/2}
\end{equation}
where, for every $j\in\N$, $R_j$ and $S_j$ are distributions supported on $\S$, 
with $R_j=0$ whenever 
$t^{j/2}$ is already represented in the first sum (so that each power of $t$ appears at most one time),
and $\tilde T_j$ is a uniquely defined extension to $M$ of the distribution $T_j$ defined in \eqref{prelude_equireg}.
\end{theorem}

In addition, we have the following more precise statements, depending on the respective values of the Hausdorff dimensions $\mathcal{Q}^\S$ and $\Qeq$. We have three cases. 
%
\paragraph{Case $\mathcal{Q}^\S > \Qeq$.} In this case, we have $S_j=0$ for $j$ odd, and
$$
\Tr(\mathcal{M}_f \, e^{t\triangle}) = \frac{1}{t^{\mathcal{Q}^\S/2}} R_0(f) + \mathrm{o}\left( \frac{1}{t^{\mathcal{Q}^\S/2}} \right)
$$ 
as $t\rightarrow 0^+$ (this one-term small-time expansion only requires $f$ to be continuous), with
$$
R_0(f) = \int_{\atop\!\! N\S} {\mkern-15mu} f\circ\pi_\S(y)\, \widehat{e}^\S(1,y,y)\, d\widehat{\mu}^\S(y)
= \int_{\atop\!\!\S} {\mkern-10mu} f(q) \int_{\atop\!\! N_q\S} {\mkern-20mu} \widehat{e}^q(1,(0,x),(0,x))\, dx\, d\mu_\S(q) 
= \int_{\atop\!\!\S} {\mkern-10mu} f\, d\nu 
$$
where $\nu$ is a smooth measure on $\S$, of density $\frac{d\nu}{d\mu_\S}(q) = \Tr_{N_q\S}(e^{\widehat{\triangle}^q})$ at any $q\in\S$ with respect to the smooth measure $\mu_\S$ on $\S$ (defined by \eqref{muS}). Therefore, the local Weyl measure exists, and we have $\supp(w_\triangle)=\S$ and $w_\triangle=\nu/\nu(\S)$. 
Moreover, 
$$
N(\lambda) \underset{\ \lambda\rightarrow+\infty}{\sim} \frac{\int_{N\S} \widehat{e}^\S(1,y,y)\, d\widehat{\mu}^\S(y)}{\Gamma(\mathcal{Q}^\S/2+1)} \lambda^{\mathcal{Q}^\S/2} .
$$
In addition, if $f=1$ near $\S$ then 
$R_j(f)=0$ for every odd integer $j<\mathcal{Q}^\S-\Qeq$.
If $e=\widehat{e}^q$ and $\mu=\widehat{\mu}^q$ for some $q\in M$, 
and if $f=1$ near $\S$, then 
$R_j(f)=0$ for $0<j<\mathcal{Q}^\S-\Qeq$ and $S_j(f)=0$ for every $j\in\N$, i.e., the expansion \eqref{trace_sum3} has no term in $\vert\ln t\vert$.

\paragraph{Case $\mathcal{Q}^\S = \Qeq$.} In this case, we have $S_j=0$ for $j$ odd, and
$$
\Tr(\mathcal{M}_f \, e^{t\triangle}) = \frac{\vert\ln t\vert}{t^{\Qeq/2}} S_0(f) + \frac{1}{t^{\Qeq/2}} \tilde T_0(f)  + \mathrm{o}\left( \frac{1}{t^{\Qeq/2}} \right)
$$ 
as $t\rightarrow 0^+$ (this two-terms small-time expansion only requires $f$ to be $C^1$; when $f$ is only continuous then we get only the first term), where, defining the mapping $\widehat{e}^{\S,\mathcal{S}^{n-k-1}}$ which to any $q\in\S$ and any $\sigma\in\mathcal{S}^{n-k-1}$ assigns the heat kernel $\widehat{e}^{\, q,\sigma}=\widehat{\,\widehat{e}^q}^\sigma$ (double nilpotentization, see Section \ref{sec_doublenilp}),
\begin{equation*}
\begin{split}
S_0(f) &= \frac{1}{2} \int_{\atop\!\!\S\times\mathcal{S}^{n-k-1}} {\mkern-58mu} f\circ\pi_\S\ 
\widehat{e}^{\S,\mathcal{S}^{n-k-1}}
{\mkern-25mu}(1,0,0)\, d(\iota_{W}\widehat{\mu}^\S) 
= \frac{1}{2} \int_{\atop\!\!\S} {\mkern-10mu} f(q) \int_{\atop\!\!\mathcal{S}^{n-k-1}} {\mkern-40mu}
\widehat{e}^{\, q,\sigma}(1,0,0)
\, d\sigma \, d\mu_\S(q) 
= \frac{1}{2} \int_{\atop\!\!\S} {\mkern-10mu} f \, d\nu \\
\tilde T_0(f) &= \mathrm{p.f.}  \int_{M\setminus\S} f(q) \, \widehat{e}^q(1,0,0)\, d\mu(q) + \mathrm{p.f.} \int_\S f(q)\, \Tr_{N_q\S}(e^{\widehat{\triangle}^q}) \, d\mu_\S(q) 
\end{split}
\end{equation*}
where $\nu$ is a smooth measure on $\S$, of density
$\frac{d\nu}{d\mu_\S}(q) = \int_{\mathcal{S}^{n-k-1}} 
\widehat{e}^{q,\sigma} (1,0,0)
\, d(\iota_{W}\widehat{\mu}^\S_q)(\sigma)$
at any $q\in\S$ (moreover, this integral does not depend on the radius of the sphere on which the integration is performed). Therefore, the local Weyl measure exists, and we have $\supp(w_\triangle)=\S$ and $w_\triangle=\nu/\nu(M)$. Moreover, 
$$
N(\lambda) \underset{\ \lambda\rightarrow+\infty}{\sim} \frac{\int_{\S\times\mathcal{S}^{n-k-1}} 
\widehat{e}^{\S,\mathcal{S}^{n-k-1}} (1,0,0)\, d(\iota_{W}\widehat{\mu}^\S)}{2\,\Gamma(\mathcal{Q}^\S/2+1)} \lambda^{\mathcal{Q}^\S/2}\ln\lambda .
$$
If $e=\widehat{e}^q$ and $\mu=\widehat{\mu}^q$ for some $q\in M$, and if $f=1$ near $\S$, then 
$R_j(f)=0$ for every $j\in\N$, and $S_j(f)=0$ for every $j\geq 1$, i.e., the expansion \eqref{trace_sum3} as a unique term in $\vert\ln t\vert$ (which is also the dominating one).

\paragraph{Case $\mathcal{Q}^\S < \Qeq$.} In this case, we have $S_j=0$ for $j$ having the parity of $\Qeq-\mathcal{Q}^\S+1$, and
$$
\Tr(\mathcal{M}_f \, e^{t\triangle}) = \frac{1}{t^{\Qeq/2}} \tilde T_0(f) + \mathrm{o}\left( \frac{1}{t^{\Qeq/2}} \right)
$$ 
as $t\rightarrow 0^+$ (this one-term small-time expansion only requires $f$ to be continuous), with
$\tilde T_0(f)(0)=\int_M f \, d\nu$ where $\nu$ is a smooth measure on $M$, of density $\frac{d\nu}{d\mu}(q)=\widehat{e}^q(1,0,0)$ at any $q\in M$. Therefore, the local Weyl measure exists, is absolutely continuous with respect to any smooth measure on $M$ (in particular, $\supp(w_\triangle)=M$) and is given by $w_\triangle=\nu/\nu(M)$ as in the equiregular case. Moreover, we have as well the Weyl law \eqref{asympt_N_equireg}. This means that the equiregular part dominates in this case. 

If $e=\widehat{e}^q$ and $\mu=\widehat{\mu}^q$ for some $q\in M$, and if $f=1$ near $\S$, then 
$S_j(f)=0$ for every $j\geq 1$, and thus, the expansion \eqref{trace_sum3} as a unique term in $\vert\ln t\vert$ if $\Qeq-\mathcal{Q}^\S$ is even, and no term in $\vert\ln t\vert$ if $\Qeq-\mathcal{Q}^\S$ is odd.

\begin{remark}\label{rem_concentration_onestratum}
The local Weyl measure is such that $\supp(w_\triangle)=\mathscr{S}$ if and only if $\mathcal{Q}^\S \geq \Qeq$, in contrast to the case $\mathcal{Q}^\S < \Qeq$ in which the local Weyl law does not differ from the one in the regular region and does not detect the singular set. 

In all cases, we have identified the main term of the local Weyl law in an intrinsic way. When $\mathcal{Q}^\S=\Qeq$, we have identified intrinsically the coefficients of the two-terms small-time asymptotics; this case covers the Baouendi-Grushin case without tangency point and the nonsingular Martinet case, studied with more generality in Sections \ref{sec_Baouendi-Grushin} and \ref{sec_Martinet}.
\end{remark}

\begin{remark}
If $k=\dim\S=n-1$ and if $\mathrm{rank}\, D=n-1$ then $w_n=1$ necessarily and thus $\mathcal{Q}^\S=\mathcal{Q}^M(\S)-1$. Hence we always have $\mathcal{Q}^\S \geq \Qeq$ in this case.
\end{remark}


\begin{remark}\label{rem_thm_onestratum_sqrtt}
When $f=1$ near $\S$, we have seen in Remark \ref{rem_J_nosqrt} that the expansion of $J(t)$ does not involve any odd power of $\sqrt{t}$. This fact is however not true for $K(t)$ in general, in other words, the function $F_0$ may fail to be even (see Section \ref{sec_add_Grushin} for an explicit example). 
\end{remark}

\begin{remark}\label{one-stratum_outsideS}
In Theorem \ref{thm_onestratum} it is assumed that $D$ is $\S$-nilpotentizable, i.e., that $D$ is locally diffeomorphic to $\widehat{D}^q$ at every point $q\in\S$, but nothing is assumed at $q\in M\setminus\S$, as underlined in Remark \ref{rem_nilp_equiregular_region}.
\end{remark}

\subsection{Examples}\label{sec_examples_onestratum}
Let us give some examples of application of Theorem \ref{thm_onestratum}. In all these examples, the singular set $\S$ is an equisingular smooth submanifold and the horizontal distribution $D$ is $\S$-nilpotentizable.

\medskip

\noindent
-- Consider the $p$-Baouendi-Grushin case, that is the almost-Riemannian case around $(0,0)$ in $\R^2$ generated by
$X_1 = \partial_1$, $X_2 = x_1^p a(x_1,x_2)\, \partial_2$,
where $p\in\N^*$ and $a$ is a smooth function such that $a(0,0)=1$. 
The singular set is $\S=\{x_1=0\}$. We have $\Qeq=2$ 
and $\mathcal{Q}^\S=p+1$.
The main term in the small-time asymptotics of the local Weyl law is then of the order of 
$\frac{\vert\ln t\vert}{t}$ if $p=1$ and $t^{-(p+1)/2}$ if $p\geq 2$. The local Weyl measure is supported on $\S$.

The case $p=1$ is studied in more detail in Section \ref{sec_Baouendi-Grushin}.

\medskip

\noindent
-- Consider the sR case around $(0,0,0)$ in $\R^3$ generated by
$X_1 = \partial_1$, $X_2 = \partial_2+x_1^p a(x_1,x_2,x_3) \, \partial_3$,
where $p\in\N^*$ and $a$ is a smooth function such that $a(0,0,0)=1$.
When $p=1$, we recover the 3D contact case, which is equiregular. When $p=2$, we recover the Martinet case. The singular set is $\S=\{x_1=0\}$, and we have $\Qeq=4$ 
and $\mathcal{Q}^\S=p+2$.
The main term in the small-time asymptotics of the local Weyl law is then of the order of 
$t^{-2}$ if $p=1$, $\frac{\vert\ln t\vert}{t^2}$ if $p=2$ and $t^{-p/2-1}$ if $p\geq 2$. 
The local Weyl measure is supported on $\S$ if $p\geq 2$.

The Martinet case, for $p=2$, is studied in more detail in Section \ref{sec_Martinet}.

\medskip

\noindent
-- Consider the nilpotent tangential elliptic case (see \cite{BonnardTrelat_AFST2001}), that is the sR case in $\R^3$ generated by
$X_1=\partial_1$, $X_2=\partial_2+\big( \frac{x_1^3}{3}+x_1x_2^2\big) \partial_3$.
The singular set is $\S=\{x_1=x_2=0\}$. We have $\Qeq=4$ 
and $\mathcal{Q}^\S=4$.
The main term in the small-time asymptotics of the local Weyl law is then of the order of $\frac{\vert\ln t\vert}{t^2}$.
The local Weyl measure is supported on $\S$.

It is interesting to note that, in this example, there is a nontrivial singular curve (equivalently, there is a nontrivial abnormal), which is not minimizing (see \cite{BonnardTrelat_AFST2001}), and however, there is a logarithm in the local Weyl law. This observation invalidates the (folklore) conjecture according to which a $\log$ in the Weyl law would be due to the presence of abnormal minimizers.

\medskip

\noindent
-- Consider the (nilpotent) aR case in $\R^3$ generated by $X_1=\partial_1$, $X_2=\partial_2$, $X_3=(x_1^2+x_2^2)\, \partial_3$. We have $\S=\{x_1=x_2=0\}$ and $\mathcal{Q}^{\S}=\Qeq=3$. The main term in the small-time asymptotics of the local Weyl law is of the order of $\frac{\vert\ln t\vert}{t^{3/2}}$. The local Weyl measure is supported on $\S_1$.

\medskip

\noindent
-- Consider the (nilpotent) sR cases in $\R^4$ generated either by $X_1=\partial_1$, $X_2=\partial_2+x_1^2 \, \partial_3+x_1x_2 \, \partial_4$, or by $X_1=\partial_1$, $X_2=\partial_2+x_1 \, \partial_3+\big(\frac{x_1^3}{3}+x_1x_2^2\big) \partial_4$, or by $X_1=\partial_1$, $X_2=\partial_2+x_1 \, \partial_3+x_1^2x_2 \, \partial_4$ (these are non-isometric normal forms for horizontal distributions of rank $2$ in $\R^4$). In all cases, the singular set is $\S=\{x_1=x_2=0\}$, $\Qeq=7$ 
and $\mathcal{Q}^\S=6$. The main term in the small-time asymptotics of the local Weyl law at any point of $\S$ is then of the order of $1/t^{7/2}$. The local Weyl measure is not concentrated (the equiregular part dominates).

\medskip

\noindent
-- Consider the (nilpotent) almost-Riemannian case in $\R^n$ generated by
$X_1=\partial_1$ ,$\ldots$, $X_{n-1}=\partial_{n-1}$, $X_n=(x_1^{2\ell}+\cdots+x_{n-1}^{2\ell})\,\partial_n$
for some $\ell\in\N^*$. For $n=2$, we recover the $2\ell$-Baouendi-Grushin case. 
The singular set is $\S=\{x_1=\cdots=x_{n-1}=0\}$ and $\Qeq=n$ 
and $\mathcal{Q}^\S=2\ell+1$. 
The main term in the small-time asymptotics of the local Weyl law is then of the order of
$t^{-\ell-\frac{1}{2}}$ if $n<2\ell+1$, $t^{-n/2}\vert\ln t\vert$ if $n=2\ell+1$ and $t^{-n/2}$ if $n>2\ell+1$.
The local Weyl measure is supported on $\S$ when $n\leq 2\ell+1$ and is not concentrated if $n>2\ell+1$.

Note that, in this example, there is no nontrivial singular curve (equivalently, there is no nontrivial abnormal), and however, when $n=2\ell+1$, there is a logarithm in the Weyl law. 

\medskip

\noindent
-- As a generalization of the previous case, consider the nilpotent aR case in $\R^n$ generated by
$X_i= \partial_i$ for $i=1,\ldots,n_1$, 
$X_{n_1+1} = (x_1^{2\ell_1}+\cdots+x_{k_1}^{2\ell_1}) \, \partial_{n_1+1}$,
\ldots,
$X_{n_1+p} = (x_{k_{p-1}+1}^{2\ell_p}+\cdots+x_{k_p}^{2\ell_p}) \, \partial_{n_1+p}$,
for some integers $\ell_1\leq\ell_2\leq\cdots\leq\ell_p$, $k_p=n_1$, $n_1+p=n$.
The singular set is $\S = \{x_1=\cdots,x_{n_1}=0\}$ and $\Qeq=n$ 
and $\mathcal{Q}^\S=2\ell+p$ where $\ell=\sum_{i=1}^p\ell_i$.
The main term in the small-time asymptotics of the local Weyl law is then of the order of
$t^{-\frac{p}{2}-\ell}$ if $n_1<2\ell$, $t^{-n/2}\vert\ln t\vert$ if $n_1=2\ell$ and $t^{-n/2}$ if $n_1>2\ell$.
The local Weyl measure is supported on $\S$ if $n_1\leq 2\ell$ and is not concentrated if $n_1>2\ell$.

\paragraph{Other examples by taking products.}
We can generate plenty of other examples by taking products of cases covered by Theorem \ref{thm_onestratum}. Indeed, when taking the product of two sR structures $(M_1,D_1,g_1)$ and $(M_2,D_2,g_2)$, by the formula \eqref{tensorproduct} of Appendix \ref{appendix_Schwartz}, the resulting heat kernel on the product manifold is the tensor product of the heat kernels on $M_1$ and $M_2$.
This simple remark thus yields interesting classes of examples that are not covered by Theorem \ref{thm_onestratum}.

For instance, taking the product of $N$ Baouendi-Grushin cases without tangency point (for some $N\in\N^*$) gives a $2N$-dimensional manifold on which the main term of the small-time asymptotics of the local Weyl law is of the order of $\frac{\vert\ln t\vert^N}{t^N}$. Taking the product of a Baouendi-Grushin case without tangency point with a nonsingular Martinet case gives a $6$-dimensional manifold on which the main term of the small-time asymptotics of the local Weyl law is of the order of $\frac{\vert \ln t\vert^2}{t^3}$.

In particular, we can always find cases where the local Weyl law involves an arbitrarily large integer power of $\vert\ln t\vert$. 

\subsection{Proof of Theorem \ref{thm_onestratum} and of the subsequent remarks}\label{sec_proof_thm_onestratum}
As explained in Section \ref{sec_J+K}, our main task is to derive the small-time asymptotic expansion for the integral $K(t)$ given by \eqref{intK_S=S1_2} in Section \ref{sec_estimating_K}.
Assuming that $f$ is supported near $\S$, we have
\begin{equation}\label{formulK}
K(t) = \frac{1}{t^{\Qeq/2}} \int_{\sqrt{t}}^1 \tau^{\Qeq-\mathcal{Q}^\S-1} \, G\left(\tau,\frac{\sqrt{t}}{\tau}\right) d\tau 
\end{equation}
where the function $G:\R^2\rightarrow\R$ is defined by
\begin{equation}\label{deffunctionG}
G(\tau,\varepsilon) = \int_\S \int_{\mathcal{S}^{n-k-1}} f(\delta_\tau^q(\sigma))\, 
e_{\tau,\varepsilon}^{q,\sigma} (1,0,0) 
\, d\sigma\, d\mu_\S(q) .
\end{equation}
Since $D$ is $\S$-nilpotentizable, Lemma \ref{lem_uniform_double_nilp} in Section \ref{sec_doublenilp} implies that $G$ is smooth, even with respect to $\varepsilon$, and 
$$
G(0,0) = \int_\S f(q) \int_{\mathcal{S}^{n-k-1}} \widehat{e}^{\,q,\sigma}(1,0,0) \, d\sigma\, d\mu_\S(q) .
$$
We are going to apply Proposition \ref{prop_general_expansion} (in Appendix \ref{app_integral}) to $K(t)$, with $x=\sqrt{t}$, $k=\Qeq-\mathcal{Q}^\S-1$ and $j=0$. 
Remark \ref{rem_prop_general_expansion_even} can be applied because $G$ is even with respect to $\varepsilon$.
If $f=1$ near $\S$ and if $e=\widehat{e}^q$ and $\mu=\widehat{\mu}^q$ for some $q\in M$, then $e^q_\tau=e$ and $G$ does not depend on $\tau$, and then Remark \ref{rem_prop_general_expansion_tau} can be applied. 

\medskip

As a preliminary remark, we have, by definition (see 
\eqref{relation_eeps_e} in Appendix \ref{app_lemfondam}), 
\begin{equation}\label{epssigma_prelim}
e_{\tau,\varepsilon}^{q,\sigma}(1,0,0) 
= \varepsilon^\Qeq e_\tau^q(\varepsilon^2,\sigma,\sigma) = \varepsilon^\Qeq \tau^{\mathcal{Q}^M(\S)} \, e(\tau^2\varepsilon^2,\delta_\tau^q(\sigma),\delta_\tau^q(\sigma))
\end{equation}
and we note that $\Qeq=\mathcal{Q}^{\R^{n}}(\delta_\tau(\sigma)) = \mathcal{Q}^{\R^{n}}(\sigma)$ and $\mathcal{Q}^M(\S)=\mathcal{Q}^M(q)$.

First, taking the limit $\varepsilon\rightarrow 0$ in \eqref{epssigma_prelim} and using the homogeneity property \eqref{homog_ehat} of the nilpotentized heat kernel, we obtain
\begin{equation}\label{epssigma1}
\widehat{ \, e_\tau^q }^\sigma (1,0,0) = \tau^{\mathcal{Q}^M(\S)}\, \widehat{e}^{\delta_\tau^q(\sigma)}(\tau^2,0,0) = \tau^{\mathcal{Q}^M(\S)-\Qeq} \, \widehat{e}^{\delta_\tau^q(\sigma)}(1,0,0) .
\end{equation}
Second, taking the limit $\tau\rightarrow 0$ in \eqref{epssigma_prelim} and using the homogeneity property \eqref{homog_ehat}, we obtain
\begin{equation}\label{epssigma2}
(\widehat{e}^q)_\varepsilon^\sigma(1,0,0) = \varepsilon^\Qeq \, \widehat{e}^q(\varepsilon^2,\sigma,\sigma) 
= \varepsilon^{\Qeq-\mathcal{Q}^M(\S)}\, \widehat{e}^q (1,\delta_{1/\varepsilon}(\sigma)) .
\end{equation}
Third, we also have $(\widehat{e}^q)_\varepsilon^\sigma(1,0,0) = \varepsilon^\Qeq \, \widehat{e}^q(\varepsilon^2,\sigma,\sigma) = \varepsilon^{\Qeq} \tau^{\mathcal{Q}^M(\S)} \, \widehat{e}^q(\tau^2\varepsilon^2,\delta_\tau(\sigma),\delta_\tau(\sigma))$ and, taking the limit $\varepsilon\rightarrow 0$, we obtain
\begin{equation}\label{epssigma3}
\widehat{e}^{q,\sigma}(1,0,0) 
= \tau^{\mathcal{Q}^M(\S)} \, 
\widehat{e}^{q,\delta_\tau(\sigma)}(\tau^2,0,0)
 = \tau^{\mathcal{Q}^M(\S)-\Qeq} \, 
\widehat{e}^{q,\delta_\tau(\sigma)}(1,0,0) .
\end{equation}

\paragraph{Case $\mathcal{Q}^\S>\Qeq$.}
Let us prove that
\begin{equation}\label{thm_onestratum_cas1}
\Tr(\mathcal{M}_f \, e^{t\triangle}) 
= \frac{1}{t^{\mathcal{Q}^\S/2}} F_0(\sqrt{t}) + \frac{1}{t^{\Qeq/2}} \left( F(\sqrt{t}) + F_1(t) \vert\ln t\vert \right) \qquad\forall t\in(0,1)
\end{equation}
for some $F_0,F,F_1\in C^\infty(\R)$, with $F_0(0)=R_0(f)$.

Proposition \ref{prop_general_expansion} (in Appendix \ref{app_integral}) implies that
$$
K(t) = \frac{\tilde F_0(t)}{(\sqrt{t})^{\mathcal{Q}^\S}}  + \frac{1}{(\sqrt{t})^{\Qeq}} \left( F(\sqrt{t}) + F_1(t) \vert\ln t\vert \right)  \qquad\forall t\in(0,1)
$$
for some $\tilde F_0, F, F_1\in C^\infty(\R)$ ($F_1$ is a smooth function of $t$ by Remark \ref{rem_prop_general_expansion_even}), with
$$
\tilde F_0(0) = \int_\S f(q) \int_0^1  \int_{\mathcal{S}^{n-k-1}} \varepsilon^{\mathcal{Q}^\S-\Qeq-1} \left( \widehat{e}^q \right)_\varepsilon^\sigma (1,0,0)\, d\sigma\, d\varepsilon \, d\mu_\S(q) .
$$ 
%
Using \eqref{epssigma2},  the change of variable $\varepsilon=1/s$, and then \eqref{transverse_polar} again, we obtain
$$
\tilde F_0(0) = \int_{\S} f(q) \int_{\R^{n-k}\setminus\mathcal{B}^{n-k}(0,1)} \widehat{e}^q (1,(0,x),(0,x)) \, dx\, d\mu_\S(q) .
$$
%
%
Using the asymptotic expansion \eqref{asympt_J} of $J(t)$ obtained in Section \ref{sec_estimating_J}, we obtain \eqref{thm_onestratum_cas1} with $F_0(\sqrt{t})=F_J(\sqrt{t})+\tilde F_0(t)$ and $F_0(0) = \int_{N\S} f\circ\pi_\S(y)\, \widehat{e}^\S(1,y,y)\, d\widehat{\mu}^\S(y)$.

If $f=1$ near $\S$, then by Remark \ref{rem_J_nosqrt}, $F_J$ is even and thus $F_0$ too.
If $e=\widehat{e}^q$ and $\mu=\widehat{\mu}^q$ for some $q\in M$, and if $f=1$ near $\S$, then $F_1(t)=\mathrm{O}(t^\infty)$ (by Remark \ref{rem_prop_general_expansion_tau}) and $\tilde F_0(t)=\tilde F_0(0)$ is constant.
By Remark \ref{rem_J_nosqrt}, we have $F_J(\sqrt{t})=F_J(0)+\mathrm{O}(t^\infty)$. Hence $F_0(t)=F_0(0)+\mathrm{O}(t^\infty)$.

\paragraph{Case $\mathcal{Q}^\S=\Qeq$.} 
Let us prove that
\begin{equation}\label{thm_onestratum_cas2}
\Tr(\mathcal{M}_f \, e^{t\triangle}) 
= \frac{\vert\ln t\vert}{t^{\Qeq/2}} F_1(t) + \frac{1}{t^{\Qeq/2}} F_0(\sqrt{t})   \qquad\forall t\in(0,1)
\end{equation}
for some $F_0,F_1\in C^\infty(\R)$, with $F_1(0) = S_0(f)$ and $F_0(0)=\tilde T_0(f)$.

Proposition \ref{prop_general_expansion} (in Appendix \ref{app_integral}) implies that 
$$
K(t) = \frac{\vert\ln t\vert}{t^{\Qeq/2}} F_1(t) + \frac{1}{t^{\Qeq/2}} \tilde F_0(\sqrt{t})   \qquad \forall t\in(0,1)
$$
for some $\tilde F_0, F_1\in C^\infty(\R^n)$ ($F_1$ is a smooth function of $t$ by Remark \ref{rem_prop_general_expansion_even}), with 
\begin{equation*}
\begin{split}
F_1(0) &= \frac{1}{2} \int_\S f(q) \int_{\mathcal{S}^{n-k-1}} 
\widehat{e}^{q,\sigma} (1,0,0) 
\, d\sigma\, d\mu_\S(q) \\
\tilde F_0(0) &= \int_\S \int_{\mathcal{S}^{n-k-1}} \bigg( f(q) \int_0^1 \frac{( \widehat{e}^q )_\varepsilon^\sigma (1,0,0) - 
\widehat{e}^{q,\sigma} (1,0,0)
}{\varepsilon}\, d\varepsilon \\
&\qquad\qquad\qquad\qquad\qquad\qquad
+ \int_0^1 \frac{f(\delta_\tau^q(\sigma)) \widehat{ \, e_\tau^q }^\sigma (1,0,0) - f(q) 
\widehat{e}^{q,\sigma} (1,0,0)
}{\tau}\, d\tau \bigg) \, d\sigma\, d\mu_\S(q) .
\end{split}
\end{equation*}
%
Adding to $K(t)$ the expansion \eqref{asympt_J} of $J(t)$ obtained in Section \ref{sec_estimating_J}, we obtain \eqref{thm_onestratum_cas2} with $F_0=F_J+\tilde F_0$.
If $e=\widehat{e}^q$ and $\mu=\widehat{\mu}^q$ for some $q\in M$, and if $f=1$ near $\S$ then, by Remark \ref{rem_prop_general_expansion_tau}, $F_1(t)=F_1(0)+\mathrm{O}(t^\infty)$, $\tilde F_0$ is even, and $F_J(\sqrt{t})=F_J(0)+\mathrm{O}(t^\infty)$ by Remark \ref{rem_J_nosqrt}.

Let us identify $F_0(0)$ and $F_1(0)$ in an intrinsic way. 
For the term $F_1(0)$, let us prove that, given any $q\in\S$, the term $\int_{\mathcal{S}^{n-k-1}} 
\widehat{e}^{q,\sigma} (1,0,0)
\, d\sigma$ does not depend on the radius of the sphere on which the integration is performed. 
The image of $(\iota_{W}\widehat{\mu}^\S)_{\vert \S\times\mathcal{S}^{n-k-1}} = \mu_\S\otimes d\sigma$ under the dilation $\delta^\S_\tau$ is the measure $\tau^{-\mathcal{Q}^M(\S)+\mathcal{Q}^\S} (\iota_W\widehat{\mu}^\S)_{\vert \delta^\S_\tau(\S\times\mathcal{S}^{n-k-1})}$ and the result follows by using \eqref{epssigma3} and the fact that $\mathcal{Q}^\S = \Qeq$.

The intrinsic identification of $F_0(0)$ is inferred from the following lemma.

\begin{lemma}\label{lem_intrinsic_identification}
We have
\begin{equation}\label{lem_intrinsic_identification_1}
\mathrm{p.f.}  \int_{M\setminus\S} f(q) \, \widehat{e}^q(1,0,0)\, d\mu(q)
= \int_{\S\times\mathcal{S}^{n-k-1}\times[0,1]} {\mkern-80mu}  \frac{f(\delta_\tau^q(\sigma)) \widehat{ \, e_\tau^q }^\sigma (1,0,0) - f(q) 
\widehat{e}^{q,\sigma} (1,0,0)
}{\tau}\, d\tau \, d\sigma\, d\mu_\S(q) 
\end{equation}
\begin{multline}\label{lem_intrinsic_identification_2}
\int_\S f(q)\, \mathrm{p.f.} \ \Tr_{N_q\S}(e^{\widehat{\triangle}^q}) \, d\nu(q) 
= \int_\S f(q) \int_{\mathcal{B}^{n-k}(0,1)} 
\widehat{e}^q(1,(0,x),(0,x))\, dx\, d\mu_\S(q) \\
+ \int_\S \int_{\mathcal{S}^{n-k-1}} f(q) \int_0^1 \frac{( \widehat{e}^q )_\varepsilon^\sigma (1,0,0) - 
\widehat{e}^{q,\sigma} (1,0,0)
}{\varepsilon}\, d\varepsilon \, d\sigma\, d\mu_\S(q)
\end{multline}
\end{lemma}

\begin{proof}
We establish \eqref{lem_intrinsic_identification_1} by using \eqref{epssigma1}, the Fubini theorem and \eqref{transverse_polar}, with $\mathcal{Q}^\S=\Qeq$.
To establish \eqref{lem_intrinsic_identification_2}, using successively \eqref{epssigma2}, the change of variable $\tau=1/\varepsilon$, and \eqref{transverse_polar} again, we obtain
\begin{equation*}
\int_{\atop\!\!\mathcal{S}^{n-k-1}} {\mkern-20mu} f(q) \int_0^1 \frac{( \widehat{e}^q )_\varepsilon^\sigma (1,0,0) - 
\widehat{e}^{q,\sigma} (1,0,0)
}{\varepsilon}\, d\varepsilon \, d\sigma\, d\mu_\S(q)  
= \mathrm{p.f.}  \int_{\atop\R^{n-k}\setminus\mathcal{B}^{n-k}(0,1)} {\mkern-60mu}  f(q)\, \widehat{e}^q(1,(0,x),(0,x))\, dx
\end{equation*}
where the latter Hadamard finite part is considered at infinity, because the integral diverges logarithmically at infinity.
Adding the contribution $F(0)$ due to $J(t)$ (see Section \ref{sec_estimating_J}) gives the integral over the whole $\R^{n-k}$, which is exactly $\Tr_{N_q\S}(e^{\widehat{\triangle}^q})$.
\end{proof}

\paragraph{Case $\mathcal{Q}^\S<\Qeq$.}
Let us prove that
\begin{equation}\label{thm_onestratum_cas3}
\Tr(\mathcal{M}_f \, e^{t\triangle}) 
= \frac{1}{t^{\Qeq/2}}F(t) + \frac{1}{t^{\mathcal{Q}^\S/2}} \left( F_0(\sqrt{t}) + F_1(\sqrt{t}) \vert\ln t\vert \right) \qquad\forall t\in(0,1)
\end{equation}
for some $F, F_0,F_1\in C^\infty(\R)$, 
with $F(0)=\tilde T_0(f)$.

Proposition \ref{prop_general_expansion} (in Appendix \ref{app_integral}) implies that
$$
K(t) = \frac{F(t)}{(\sqrt{t})^{\Qeq}} + \frac{1}{(\sqrt{t})^{\mathcal{Q}^\S}} \left( \tilde F_0(\sqrt{t}) + F_1(\sqrt{t}) \vert\ln t\vert \right)  \qquad\forall t\in(0,1)
$$
for some $F, \tilde F_0, F_1\in C^\infty(\R^n)$ such that 
$F_1$ has the parity of $\Qeq-\mathcal{Q}^\S$ (by Remark \ref{rem_prop_general_expansion_even}) and
$$
F(0) = \int_0^1 \tau^{\Qeq-\mathcal{Q}^\S-1} \int_\S \int_{\mathcal{S}^{n-k-1}} f(\delta_\tau^\S(\sigma)) \widehat{ \, e_\tau^q }^\sigma (1,0,0)\, d\sigma\, d\nu(q)\, d\tau  .
$$
We infer from \eqref{epssigma1} and from \eqref{transverse_polar} that $F(0) = \int_M f(q) \, \widehat{e}^{q}(1,0,0) \, d\mu(q)$.
Adding the expansion \eqref{asympt_J} of $J(t)$ obtained in Section \ref{sec_estimating_J} gives \eqref{thm_onestratum_cas3} with $F_0=F_J+\tilde F_0$. 
If $e=\widehat{e}^q$ and $\mu=\widehat{\mu}^q$ for some $q\in M$, and if $f=1$ near $\S$, then by Remark \ref{rem_prop_general_expansion_tau}, $F_1(\sqrt{t})=F_1(0)+\mathrm{O}(t^\infty)$.

\paragraph{Proof of \eqref{trace_sum3}.}
%
Having in mind the prelude to the theorem and in particular \eqref{prelude_equireg}, now, to arrive at the decomposition \eqref{trace_sum3}, we organize all terms in a different way, as follows.
%
%
Splitting $f$ into two parts, one supported near $\S$ and the other supported in $M\setminus\S$ to which \eqref{prelude_equireg} is applied, to this expansion coming from the regular region, we add all other possible terms in $t^{i-\Qeq/2}$ (for the three cases), for $i\in\N$, and this gives the first term in \eqref{trace_sum3}. 
The resulting decomposition is canonical.

\subsection{Baouendi-Grushin case}\label{sec_Baouendi-Grushin}
The Baouendi-Grushin case is often named ``Grushin case". It has been mentioned to us by N. Garofalo that, actually, S. Baouendi designed this famous model in 1967, while V. Grushin studied its hypoellipticity later, in 1970. This interesting story is reported in \cite[Section 11]{Garofalo_thoughts}.

\subsubsection{Definition of the model and local Weyl law}
We assume that $n=2$ and that $\mathrm{rank}(D)=2$ except along the singular set $\S$. 
This means that the $2D$ closed surface $M$ is Riemannian outside of $\S$.
Under generic assumptions (see \cite{ABS08, BoscainCharlotGhezzi_DGA2013}), the set $\S$ (along which $\mathrm{rank}(D)=1$) is a smooth closed curve, i.e., a finite union of embedded circles. 
Assume that $D=\mathrm{Span}(X_1,X_2)$ locally.
Outside isolated points of $\S$ (called \emph{tangency points}), $TM $ is spanned by $X_1$, $X_2$ and the Lie bracket $[X_1,X_2]$: such points are called \emph{Baouendi-Grushin points}.

Near Baouendi-Grushin points, we have the normal form for the metric
$$
X_1=\partial_1, \qquad X_2 = x_1 a(x_1,x_2) \, \partial_2
$$ 
($g$-orthonormal basis) where $a$ is a smooth function such that $a(0,x_2)=1$ for every $x_2$. Locally, the singular set is $\S=\{x_1=0\}$: it is equisingular and we have $\mathcal{Q}^\S=\Qeq=2$ and $\mathcal{Q}^M(\S)=3$. Moreover, $D$ is $\S$-nilpotentizable. We are thus in the framework of Theorem \ref{thm_onestratum} there.

Near tangency points, we need one more bracket, and we have the normal form for the distribution
$$
X_1=\partial_1, \qquad X_2=(x_1^2-x_2)\,\partial_2
$$
(there exists a normal form for the metric but we will not need it). Locally, the singular set is $\S=\{x_1^2-x_2=0\}$: it is not equisingular but is however stratified by two equisingular manifolds: $\S_1=\{(0,0)\}$ and $\S_2=\S\setminus\S_1$. The stratum $\S_2$ consists of Baouendi-Grushin points and thus Theorem \ref{thm_onestratum} can be applied there. For the stratum $\S_1$, we have $\mathcal{Q}^{\S_1}=0$ and $\mathcal{Q}(\S_1)=4$. Theorem \ref{thm_onestratum} does not cover this case.
Moreover, $D$ is not $\S$-nilpotentizable near tangency points: indeed the nilpotentization at a tangency point $q$ is $\widehat{X}^q_1=\partial_1$, $\widehat{X}^q_2=x_1^2\,\partial_2$, whose singular set is $\widehat{\!\S}^q=\{x_1=0\}$ which is not diffeomorphic to $\S$.

Let $P$ be the Popp measure on $M\setminus \S$, that is, here, the Riemannian measure associated with the metric $g$ (which is Riemannian outside of $\S$). Near $\S$, we have $dP=d\nu \otimes \big\vert\frac{d\phi}{\phi}\big\vert$ where $\nu$ is a smooth measure on $\S$ 
and $\phi=0$ is a local equation of $\S$, with $\phi$ a local submersion. 
The measure $\nu$ does not depend on the choice of $\phi$.\footnote{In the case without tangency point, the measure $\nu$ coincides with the Popp measure $P_\S$ along the equisingular submanifold $\S$, as defined in \cite{GhezziJean_TSG2015} (see also Appendix \ref{sec_popp}).}

\begin{theorem}\label{thm_Baouendi-Grushin}
Given any $f\in C^\infty(M)$ vanishing near the tangency points, we have
\begin{equation}\label{expansion_Baouendi-Grushin}
\Tr(\mathcal{M}_f \, e^{t\triangle}) = \int_M f(q)\, e(t,q,q)\, d\mu(q) 
= \frac{\vert\ln t\vert}{t} F_1(t) + \frac{1}{t} F_0(\sqrt{t})    \qquad\forall t\in(0,1)
\end{equation}
for some $F_0, F_1\in C^\infty(\R)$, with $F_1(0) = \frac{1}{4\pi} \int_\S f \, d\nu$. 
When $f$ is only of class $C^1$ and vanishes near the tangency points, we have a two-terms small-time asymptotics, with
$$
F_0(0) = \frac{1}{4\pi}\left( \mathrm{p.f.}\int _{M\setminus\S} f\, dP + (\gamma + 4 \ln 2)\int_\S f \, d\nu \right)
$$
where $\gamma$ is the Euler constant.
Moreover, if $(M,D,g)$ is a nilpotent Lie group (and $\mu$ is the Haar measure) and if $f=1$ near $\S$ then $F_1(t)=F_1(0)+\mathrm{O}(t^\infty)$ and $F_0$ is even.

Given any $f\in C^0(M)$ (which may be nonzero near the tangency points), we have 
$$
\Tr(\mathcal{M}_f \, e^{t\triangle}) 
= \frac{\vert\ln t\vert}{t} F_1(0) + \mathrm{o}\left( \frac{\vert\ln t\vert}{t}\right)
$$
as $t\rightarrow 0^+$.
The local Weyl measure is supported on $\S$ and is given by $w_\triangle=\nu /\nu(\S)$.
Moreover,
$$
N(\lambda) \underset{\lambda\rightarrow+\infty}{\sim} \frac{\nu(\S)}{4\pi}\lambda \ln \lambda .
$$

\end{theorem}

When there are no tangency points, the result follows from Theorem \ref{thm_onestratum}. In this case, the consequence on the asymptotics of the spectral counting function is already known: it follows from results of \cite{Me-Sj-78}. But it is new when there are tangency points.
What is also new here is the expression of the local Weyl measure and the intrinsic two-terms small time-asymptotics of the local Weyl law (far from the tangency points).


Moreover, we establish in Section \ref{sec_QE} a Quantum Ergodicity (QE) result in the Baouendi-Grushin case when $\S$ is connected with at most one tangency point: there exists a density-one subsequence of probability measures $|\phi_{j_k}|^2 \mu$ converging weakly to $w_\triangle$. This is the first QE result in sR geometry where the limit measure is singular. 

The more general $k$-Baouendi-Grushin case $X_1=\partial_1$, $X_2=(x_1^k-x_2)$, for $k\geq 2$ (for which $\supp(w_\triangle)=\S$), as well as other variants like $X_2=(x_1^2-x_2^3)$ (for which $\supp(w_\triangle)=\{(0,0)\}$), much more difficult to treat, are considered in Section \ref{sec_nonnilp} in an analytic framework.

\subsubsection{Proof of Theorem \ref{thm_Baouendi-Grushin}}\label{sec_proof_thm_Baouendi-Grushin}
When there is no tangency point (or, when $f$ vanishes near the tangency points), Theorem \ref{thm_onestratum} can be applied.
Outside of $\S$, the Popp measure is given by $dP = \frac{1}{\vert x_1\vert a(x_1,x_2)}\, dx_1\, dx_2$.

Since the Baouendi-Grushin case is Riemannian outside of $\S$, noting that $a(0,0)=1$, we have 
$$
\widehat{e}^{\,q,\sigma}(1,0,0) = e_{\mathrm{Euclid}}(1,0,0) = \frac{1}{4\pi}
$$
for all $q\in\S$ and $\sigma\in \mathcal{S}^0=\{\pm1\}$, where $e_{\mathrm{Euclid}}(t,x,y)=\frac{1}{4\pi t}\exp(-\Vert x-y\Vert^2/4t)$ is the Euclidean heat kernel in $\R^2$. We obtain $F_1(0) = \frac{1}{4\pi} \int_\S f \, d\nu$.

The term $F_2(0)$ given by Theorem \ref{thm_onestratum} is the sum of two terms. The first is 
$$
\mathrm{p.f.}  \int_{M\setminus\S} f(q) \, \widehat{e}^q(1,0,0)\, d\mu(q) 
= \mathrm{p.f.}  \int_{M\setminus\S} f(q) \, e_{\widehat{\triangle}^q,\widehat{P}^q(1,0,0)}\, dP(q)
$$
where we have used \eqref{add_kernelnilp_mu_nu}. But since $e_{\widehat{\triangle}^q,\widehat{P}^q}(1,0,0) = e_{\mathrm{Euclid}}(1,0,0) = \frac{1}{4\pi}$, this first term is $\frac{1}{4\pi}\, \mathrm{p.f.}  \int_{M\setminus\S} f(q)\, dP(q)$.
The second is $\int_\S f(q)\, \mathrm{p.f.} \, \Tr_{N_q\S}(e^{\widehat{\triangle}^q}) \, d\nu(q)$, and we claim that, for every $q\in\S$,  
$\mathrm{p.f.} \ \Tr_{N_q\S}(e^{\widehat{\triangle}^q}) = \frac{\gamma+4\ln 2}{4\pi}$.
To prove this fact, we consider an explicit example in the equivalence class, and then to identify the unknown coefficient from this example: this is done in Section \ref{sec_Baouendi-Grushin_sphere}.\footnote{Alternatively, one may use the explicit formula for the flat Baouendi-Grushin case (see, e.g., \cite{ChangLi_2015})
$$
e(t,x,x') = \frac{1}{(2\pi t)^{3/2}} \int_\R \sqrt{\frac{\tau}{\sinh\tau}} e^{\left( i\tau(x_2'-x_2) - \frac{1}{2}(x_1^2+{x_1'}^2)\tau\,\mathrm{cotanh}\,\tau - 2x_1x_1' \right)/t} \, d\tau .
$$
}

It remains to prove that, when there are tangency points, we still have an asymptotic expansion \eqref{expansion_Baouendi-Grushin} (but in which we cannot identify $F_0(0)$.
Since we are in the generic case, the metric normal form near a tangency point $q$ is $X_1=a(x_1,x_2)\,\partial_1$, $X_2=(x_1^2-x_2)b(x_1,x_2)\,\partial_2$ ($g$-orthonormal basis), with $a$ and $b$ germs of smooth positive functions such that $a(0,0)=b(0,0)=1$. By the uniform ball-box theorem (see \eqref{mu_ball_eps} in Appendix \eqref{sec_uniformballbox}), we have $\mu(\BsR(q,r)) \geq C r^2 \max(\vert x_1^2-x_2\vert, r\vert x_1\vert, r^2)$ for some constant $C>0$. Now, it suffices to split the integral $I(t)$ as an integral over $\BsR(q,r)$ (we do that, actually, around each tangency point) and an integral over the rest. The latter is estimated by Theorem \ref{thm_onestratum}. For the integral over $\BsR(q,r)$, we use the heat kernel estimate \eqref{exp_estimates_diago} (see Appendix \ref{app_sR_kernel}), which yields the upper bound $\frac{\vert\ln t\vert}{t}\mathrm{O}(r)$ with a uniform $\mathrm{O}(r)$. The result follows.

\subsubsection{The Baouendi-Grushin case on the sphere $\mathcal{S}^2$}\label{sec_Baouendi-Grushin_sphere}
In this section, we present an explicit example of the Baouendi-Grushin case without tangency point on the $2$-sphere where we can compute explicitly the eigenelements of the sR Laplacian and the two-terms expansion of the heat asymptotics.

We set $M=\mathcal{S}^2$, the Euclidean $2$-sphere of $\R^3$ endowed with the canonical Riemannian metric $g_R$. We denote by $\mu_R$ the associated (smooth) Riemannian measure on $M$. 
The vector fields on $M$ defined by 
$$
X_1= -z\,\partial_x +x\,\partial_z, \qquad
X_2 = z\,\partial_y-y\,\partial_z, \qquad X_3=y\,\partial_x-x\,\partial_y
$$
span the tangent space to $M$ at every point. The Laplace-Beltrami Laplacian on $M$ (Riemannian Laplacian, associated with the canonical Riemannian metric and with the canonical Riemannian measure) is $\triangle_R = \triangle_{g_R,\mu_R} = X_1^2 + X_2^2 + X_3^2$. 

Note that $[X_1,X_2]=-X_3$, $[X_1,X_3]=X_2$ and $[X_2,X_3]=-X_1$. Note also that, in spherical coordinates $x=\cos\varphi\cos\theta$, $y=\sin\varphi\cos\theta$, $z=\sin\theta$, with $0\leq\varphi\leq2\pi$ and $-\pi/2\leq\theta\leq\pi/2$, we have $X_3= -\partial_\varphi$. For $\vert\theta \vert\neq\pi/2$, we have $X_1=\sin\varphi\tan\theta\,\partial_\varphi+\cos\varphi\,\partial_\theta$ and $X_2=\cos\varphi\tan\theta\,\partial_\varphi-\sin\varphi\,\partial_\theta$, and the canonical Riemannian measure is $d\mu_R = \cos\theta\, d\theta\, d\varphi$.

We consider on $M$ the almost Riemannian structure $(M,D,g)$ with $D=\mathrm{Span}(X_1,X_2)$ and $g$ the metric on $D$ for which $(X_1,X_2)$ is an orthonormal frame.
We denote by $\triangle=\triangle_{g,\mu_R}=X_1^2+X_2^2$ the sR Laplacian associated with the sR metric $g$ and with the (smooth) Riemannian measure $\mu_R$ on $M$. The singular set $\S$ is the equator $\{(x,y,z)\in\mathcal{S}^2\ \mid\ z=0\}$: $\triangle$ is elliptic outside of $\S$ but is only subelliptic on $\S$ (it is \emph{almost-Riemannian}).
This operator (which is selfadjoint with respect to the smooth measure $\mu_R$) has been considered in \cite{Casarino}. Note that the almost-Riemannian Laplacian studied in \cite{BPS-14} is different: they consider the sR Laplacian $\triangle_{g,P}$ associated with the sR metric $g$ and with the Popp measure $P$, which is singular along $\S$. 
Here, we have $dP = \frac{1}{\tan\theta}\,d\theta\,d\varphi$ outside of $\S$, and thus, with the notations of Theorem \ref{thm_Baouendi-Grushin}, we have $d\nu=d\varphi$ along $\S$.


Since $[\triangle, X_3]=0$, considering the standard basis $(Y_{l,m})_{|m|\leq l} $ of spherical harmonics, satisfying $-\triangle_R Y_{l,m}=l(l+1) Y_{l,m}$, $X_3 Y_{l,m}=-im  Y_{l,m}$ for $l\in\N$ and $|m|\leq l$, we find 
$-\triangle Y_{l,m}= \left( l(l+1) -m^2\right)Y_{l,m}$.
The eigenvalues are $\lambda_{l,m}= l(l+1) -m^2$ with multiplicity given by the number of such decompositions of $\lambda_{l,m}$. The eigenfunctions are $Y_{l,m}$ but the eigenvalues are ordered in different ways for $\triangle_R $ and for $\triangle$.

The zeta function of $\triangle + 1/4$ (whose eigenvalues are $(l+1/2)^2-m^2$ with $|m|\leq l$), as defined in Section \ref{sec_regdet}, is then
$$
\zeta(s) = \sum_{|m|\leq l} \frac{1}{(l+m+\frac{1}{2})^s }\frac{1}{(l-m+\frac{1}{2})^s} .
$$
Given any $p,q\geq 0$, we can find integers $l$ and $m$ such that $p=l+m$ and $q=l-m$ if and only $p$ and $q$ have the same parity. Therefore
$$
\zeta(s) = 4^{-s} (\zeta(s,1/4))^2 + 4^{-s} (\zeta(s,3/4))^2
$$
where the functions $\zeta(\cdot,\cdot)$ are the corresponding Hurwitz zeta functions%
\footnote{Given any $a>0$, the classical Hurwitz zeta function $\zeta(\cdot,a)$ is defined by $\zeta(s,a)=\sum_{j=0}^{+\infty}\frac{1}{(j+a)^s}$. Its meromorphic extension has a unique pole, at $s=1$, and $\zeta(s,a)=\frac{1}{s-1}-\psi(a)+\mathrm{O}(1)$ as $s\rightarrow 1$, where $\psi$ is the digamma function. We have $\psi_0 (1/4)=-\pi/2 - 3\ln 2 -\gamma $ and $\psi_0 (3/4)=\pi/2 - 3\ln 2 -\gamma $.}
(see \cite[Section 1.3]{SrivastavaChoi_book}). 
The function $\zeta\Gamma$ is holomorphic for $\mathrm{Re}(s)>1$. Its meromorphic extension has a pole of order $2$ at $s=1$.
When $s\rightarrow 1$, we have the expansions 
$$
4^{-s}= \frac{1}4 -\frac{\ln2}{2} (s-1) + \mathrm{o}(s-1), \qquad
(\zeta(s, 1/4))^2= \left( \frac{1}{s-1} +\pi/2 + 3\ln 2 +\gamma +\mathrm{o}(1) \right)^2,
$$
and a similar one for $(\zeta(s, 3/4))^2$ by replacing $+\pi/2$ by $-\pi/2$. Hence
$$
\zeta(s)= \frac{1}{2(s-1)^2}+ \frac{\gamma+2\ln 2}{s-1} + \mathrm{O}(1) ,
\qquad
\zeta(s)\Gamma (s)  = \frac{1}{2 (s-1)^2} + \frac{\gamma + 4 \ln 2}{2(s-1)} +\mathrm{O}(1)
$$
as $s\rightarrow 1$.
Setting $Z(t)=\Tr(e^{t(\triangle+1/4)})$, we have $\zeta(s)\Gamma(s) = \int_0^{+\infty} Z(t) t^{s-1}\, dt$. Splitting this integral in $\int_0^1$ and $\int_1^{+\infty}$, the latter gives a holomorphic function, while the first is meromorphic with a pole at $s=1$.
Since we already know that $Z(t)=A\frac{\vert\ln t\vert}{t} + \frac{B}{t} + \mathrm{O}(1)$ as $t\rightarrow 0^+$, we find, by identification, $A=\frac{1}{2}$ and $B= \frac{\gamma + 4 \ln 2}{2}$. 
Noting that $\nu(\S)=2\pi$ and that $\mathrm{p.f.}\int_{M\setminus\S}dP=0$, we can thus identify the unknown coefficient in Section \ref{sec_proof_thm_Baouendi-Grushin}.

\subsubsection{An additional remark}\label{sec_add_Grushin}
In Remark \ref{rem_thm_onestratum_sqrtt}, we have claimed that, when $f=1$ near $\S$, the function $F_0$ may fail to be even. In this section, we give an explicit example. 
Take the local Baouendi-Grushin model near $(0,0)$ given by the $g$-orthonormal frame $X_1=(1+x_1)\,\partial_1$, $X_2 = x_1 \, \partial_2$. Let us compute the small-time expansion of $I(t)$ near $(0,0)$. According to the procedure described in Sections \ref{sec_singular_prelim} and \ref{sec_equisingular_nilp}, we perform a first nilpotentization at $q=(0,0)$ and then at $\sigma=(1,0)$. Here, we obtain 
$$
(X_1)_{\tau,\varepsilon}^{q,\sigma} (y_1,y_2) 
= (1+\tau+\tau\varepsilon y_1)\, \partial_1 ,
\qquad
(X_2)_{\tau,\varepsilon}^{q,\sigma} (y_1,y_2) 
= ( 1+\varepsilon y_1 ) \,\partial_2  
$$
and thus 
$\triangle_{\tau,\varepsilon}^{q,\sigma}
= \triangle_\tau + R$ with 
\begin{multline*}
\triangle_\tau = (1+\tau)^2\,\partial_1^2 + \partial_2^2, \qquad
R = R_1+R_2, \\
R_1 = ((1+\tau)\tau\varepsilon + \tau^2\varepsilon^2 y_1)\, \partial_1  + \tau\varepsilon y_1(2(1+\tau)+\tau\varepsilon y_1)\,\partial_1^2, \qquad
R_2 = \varepsilon y_1(1+\varepsilon y_1)\, \partial_2^2 .
\end{multline*}
Using \eqref{formulas_kernel} in Appendix \ref{appendix_Schwartz}, the heat kernel $e_\tau$ generated by $\triangle_\tau$ satisfies
$e_\tau(1,0,0) = \frac{1}{4\pi(1+\tau)^2}$ (because $e_{\mathrm{Euclid}}(1,0,0) = \frac{1}{4\pi}$). In order to estimate the heat kernel 
$e_{\tau,\varepsilon}^{q,\sigma}$ generated by $\triangle_{\tau,\varepsilon}^{q,\sigma}$, 
we follow \cite{CHT_AHL}, noting that, by the Duhamel formula,
$e^{t \triangle_{\tau,\varepsilon}^{q,\sigma}}
= e^{t\triangle_\tau} + 
e^{t \triangle_{\tau,\varepsilon}^{q,\sigma}} 
\star Re^{t\triangle_\tau}$, where the convolution is defined by $A(t)\star B(t) = \int_0^t A(t-s)B(s)\, ds$. Hence we have the expansion 
$$
e^{t \triangle_{\tau,\varepsilon}^{q,\sigma}} 
= e^{t\triangle_\tau} + e^{t\triangle_\tau}\star Re^{t\triangle_\tau} + e^{t\triangle_\tau}\star Re^{t\triangle_\tau}\star Re^{t\triangle_\tau} + \cdots
$$
at any order in $(\tau,\varepsilon)$, as a sum of locally smoothing operators, as proved in \cite[Section 6]{CHT_AHL} (this is highly nontrivial).
Since $R=R_1+R_2$, it follows that 
$$
e_{\tau,\varepsilon}^{q,\sigma}(1,0,0) 
= e_\tau(1,0,0) + h_1(\tau,\varepsilon) + h_2(\tau,\varepsilon)
$$
where $h_i(\tau,\varepsilon)$ is the Schwartz kernel of $e^{t\triangle_\tau}\star R_ie^{t\triangle_\tau} + e^{t\triangle_\tau}\star R_ie^{t\triangle_\tau}\star R_ie^{t\triangle_\tau} + \cdots$, for $i=1,2$. Since $R_1= \mathrm{O}(\tau\varepsilon)$ and $R_2 = \mathrm{O}(\varepsilon)$ and since we already know that $(e^q_\tau)^\sigma_\varepsilon(1,0,0)$ is even with respect to $\varepsilon$, we obtain that 
$$
e_{\tau,\varepsilon}^{q,\sigma}(1,0,0) 
= \frac{1}{4\pi(1+\tau)^2} + \tau^2\varepsilon^2 h_1(\tau,\varepsilon) + \varepsilon^2 h_2(\tau,\varepsilon)
$$
for some $h_1,h_2\in C^\infty(\R^2)$. Therefore
we have \eqref{formulK} with $\Qeq=2$, $\mathcal{Q}^\S=2$, $f=1$ near $\S$ and 
$$
G(\tau,\varepsilon) = \int_\S \int_{\mathcal{S}^0} 
e_{\tau,\varepsilon}^{q,\sigma}(1,0,0) 
\, d\sigma\, d\nu(q) = \frac{\nu(\S)}{2\pi}\frac{1}{(1+\tau)^2} +\tau^2\varepsilon^2 H_1(\tau,\varepsilon) + \varepsilon^2 H_2(\tau,\varepsilon)
$$
for some $H_1,H_2\in C^\infty(\R^2)$.
In particular, $G(0,0)=\frac{\nu(\S)}{2\pi}$, $\partial_2G(\tau,0)=\partial_1\partial_2G(0,0)=0$ and $\partial_1G(0,\varepsilon)=-\frac{\nu(\S)}{\pi}$. According to Remark \ref{rem_expansion_4} in Appendix \ref{app_integral}, the term in $\sqrt{t}$ in the expansion of $I(t)$ is equal to $-\partial_1 G(0,0)=\frac{\nu(\S)}{\pi}$.
We finally obtain
$$
I(t) = \frac{\nu(\S)}{4\pi}\frac{\vert\ln t\vert}{t} + F_0(0)\frac{1}{t} + \frac{\nu(\S)}{\pi}\frac{1}{\sqrt{t}} + \mathrm{O}\left(\vert\ln t\vert\right)
$$
as $t\rightarrow 0^+$ (where $F_0(0)$ is given by Theorem \ref{thm_Baouendi-Grushin}). There is a nontrivial term in $1/\sqrt{t}$.

\subsection{Martinet case}\label{sec_Martinet}
We assume that $n=3$ and that $D$ is locally defined by $D=\ker\alpha$ where $\alpha$ is a real-valued $1$-form on $M$ such that $\alpha\wedge d\alpha$ vanishes transversally on a $2D$ smooth surface  $\S$ (called Martinet surface). Let $L=(L_z)_{z\in \S}$ be defined by $L=D\cap T\S$.
We assume that $L$ is line bundle over $\S$. Generically, $L$ admits singularities (see \cite{Ma-70,Zh-92}). We speak of the \emph{nonsingular Martinet case} when $L$ has no singularities, and in this case, according to \cite{Ma-70}, the distribution $D$ can be defined locally near $\S$ by $D=\ker\alpha$ with $\alpha = dx - x^2\, dy $.
The distribution $D$ is of contact type outside of $\S$.

Let $P$ be the Popp measure on $M\setminus \S$ (canonical contact measure outside of $\S$). Near $\S$, we have $dP=d\nu \otimes \big\vert\frac{d\phi}{\phi}\big\vert$ where $\nu$ is a smooth measure on $\S$ 
and $\phi =0$ is a local equation of $\S$, with $\phi$ a local submersion. 
The measure $\nu$ does not depend on the choice of $\phi$.

\begin{theorem}\label{thm_Martinet}
Given any $f\in C^\infty(M)$ vanishing near the singularities of $\S$, we have
$$
\Tr(\mathcal{M}_f \, e^{t\triangle}) = \int_M f(q)\, e(t,q,q) \, d\mu(q) 
= \frac{\vert\ln t\vert}{t^2} F_1(t) + \frac{1}{t^2} F_0(\sqrt{t})  \qquad \forall t\in(0,1)
$$
for some $F_0,F_1\in C^\infty(\R)$, with $F_1(0) = \frac{1}{16} \int_\S  f \, d\nu$.
When $f$ is only of class $C^1$ and vanishes near the singularities of $\S$, we have a two-terms small-time asymptotics, with
$$
F_0(0) = \mathrm{p.f.} \int_{M\setminus\S} fh \, d\mu + A \int_\S  f \, d\nu
$$
for some universal constant $A\in\R$ and for some smooth function $h$ on $M$, where $\mathrm{p.f.}\int _M fh\, d\mu$ is the Hadamard finite part of the integral with respect to the surface $\S$.
Moreover, if $(M,D,g)$ is a nilpotent Lie group (and $\mu$ is the Haar measure) and if $f=1$ near $\S$ then $F_1(t)=F_1(0)+\mathrm{O}(t^\infty)$ as $t\rightarrow 0^+$ and $F_0$ is even.

Given any $f\in C^0(M)$ (which may be nonzero at the singularities of $\S$), we have
$$
\Tr(\mathcal{M}_f \, e^{t\triangle}) 
= \frac{\vert\ln t\vert}{t^2} F_1(0) + \mathrm{o}\left( \frac{\vert\ln t\vert}{t^2}\right)
$$
as $t\rightarrow 0^+$.
The local Weyl measure is the probability measure supported on $\S$ defined by $w_\triangle=\nu /\nu(\S)$. Moreover,
$$
N(\lambda)\underset{\lambda\rightarrow +\infty}{\sim} \frac{\nu(\S)}{32}\lambda^2 \ln \lambda .
$$
\end{theorem}


The proof is similar to the one done in the Baouendi-Grushin case and is thus skipped.
The coefficient $\frac{1}{16}$ in $F_1(0)$ is computed thanks to Remark \ref{rem_3Dcase} (because, outside of the singular set $\S$, the Martinet case is a 3D contact case).
The fact that the constant $A$ is universal is due to the knowledge of normal forms (see \cite{ABCK97}).
In contrast to the Baouendi-Grushin case, we do not know any sufficiently tractable model for the Martinet case that would allow us to identify the constant $A$ and the function $h$ in $F_0(0)$.

\subsection{Quantum Ergodicity properties}\label{sec_QE}
In this section, we derive QE (Quantum Ergodicity) properties, as a consequence of the local Weyl law.
We denote by $(\phi_j)_{j\in\N}$ an orthonormal eigenbasis, as considered in Section \ref{sec_karamata}.

\subsubsection{Concentration of Quantum Limits on the singular set} 
Under the assumptions done in Theorem \ref{thm_onestratum}, if $\mathcal{Q}^\S \geq \Qeq$ then the Weyl measure is supported on the singular set $\S$.
It follows from Theorem \ref{thm_onestratum} and from Theorem \ref{thm_weyl_measures_equiv} that
$$
\lim_{\lambda\rightarrow+\infty} \frac{1}{N(\lambda)}\sum_{\lambda_k\leq\lambda} \int_M f\vert\phi_k\vert^2\, d\mu = \int_\S f\, dw_\triangle \qquad\forall f\in C^0(M).
$$
By using a well known lemma\footnote{Given a bounded sequence $(u_n)_{n\in\N}$ of nonnegative real numbers, the Ces\'aro mean $\frac{1}{n}\sum_{k=0}^{n-1}u_k$ converges to $0$ if and only if there exists a subset $S\subset\N$ of density
 one such that $(u_k)_{k\in S}$ converges to $0$. We recall that $S$ is of density one if $\frac{1}{n}\#\{k\in S\mid k\leq n-1\}$ converges to $1$
 as $n$ tends to $+\infty$.} due to Koopman and Von Neumann (see, e.g., \cite[Chapter 2.6, Lemma 6.2]{Petersen}), we infer the following corollary.

\begin{corollary}\label{cor_concentration_onestratum}
In the framework of Theorem \ref{thm_onestratum}, if $\mathcal{Q}^\S \geq \Qeq$ then there exists a density-one sequence $(k_j)_{j\in\N}$ of positive integers such that, for every real-valued continuous function $f$ on $M$ whose support does not intersect the singular set $\S$, 
$$
\lim_{j\rightarrow+\infty} \int_M f\vert\phi_{k_j}\vert^2\, d\mu = 0 .
$$
This means that all Quantum Limits (weak limits) on the base of $(\phi_{k_j})_{j\in\N}$ are supported on $\S$.
\end{corollary}

\begin{remark}
For the Baouendi-Grushin case on the sphere $\mathcal{S}^2$, this result may seem surprising because the eigenfunctions of the aR Laplacian on $\mathcal{S}^2$ coincide with those of the usual Laplace-Beltrami Laplacian, which are the spherical harmonics $Y_{l,m}$. As mentioned in the previous section, the explanation is that the eigenvalues are ordered in different ways for the two Laplacians. Being a sequence of density one depends on the order of the sequence: any infinite sequence of density zero can be reordered as a sequence of density one! More precisely, any weak limit of a sequence $\vert Y_{l,m}\vert^2\mu$ is supported on the equator if and only if $\vert m/l\vert$ converges to $1$.
\end{remark}

\subsubsection{Quantum Limits (QLs) and Quantum Ergodicity (QE)}
Let us first briefly recall the definition of what is a QL and of what is QE. 
Let $E$ be a smooth compact manifold, endowed with a probability measure $\Theta$, and let $T$ be a self-adjoint operator on $L^2(E,\Theta)$, bounded below and having a compact resolvent (and hence a discrete spectrum). Let $(\Phi_k)_{k\in\N}$ be a (complex-valued) Hilbert basis of $L^2(E,\Theta)$, consisting of eigenfunctions of $T$, associated with the ordered sequence of eigenvalues $\lambda_0\leq\cdots\leq\lambda_j\leq\cdots$. 

A \emph{local Quantum Limit} (local QL, or QL on the base) is a probability measure on $E$ that is the weak limit of a subsequence of probability measures $\vert\Phi_{k_j}\vert^2 \Theta$. A \emph{microlocal QL} (or QL in the phase space) is a probability measure on the co-sphere bundle $S^*E$ that is the weak limit of a subsequence of Radon measures $\mu_{k_j}(a)=\langle\Op(a)\Phi_{k_j},\Phi_{k_j}\rangle_{L^2(E,\Theta)}$ (the measures $\mu_j$ are asymptotically positive), where $\Op$ is an arbitrary quantization and $a$ is a classical symbol of order $0$. Microlocal QLs do not depend on the choice of the quantization. Any local QL is the image of a (not necessarily unique) microlocal QL under the canonical projection $S^*E\rightarrow E$.

We say that Quantum Ergodicity (QE) on the base (resp., in the phase space) is satisfied for the eigenbasis $(\Phi_j)_{j\in\N}$ of  $T$ if there exist a local QL $\beta$ on $E$ (resp., a microlocal QL $\beta$ on $S^*E$) and a density-one sequence $(k_j)_{j\in\N}$ of positive integers such that the sequence of probability measures $\vert\Phi_{k_j}\vert^2\Theta$ (resp., $\mu_{k_j}$) converges weakly to $\beta$.

QE theorems have a glorious history, starting with the well known Shnirelman theorem (see \cite{Shn-74} and see \cite{yCdV-85,Zel-87} for a proof), establishing QE in the Riemannian case provided that the geodesic flow be ergodic for the canonical Riemannian measure. This theorem has been extended in a number of ways in elliptic cases and it is not our objective here to report on such extensions. In \cite{CHT-I}, we have established the first QE result in a sR case, namely, in the 3D contact case: QE is satisfied under the assumption that the Reeb flow be ergodic for the Popp measure, with the Popp measure (canonical contact measure) as a limit measure. 

In all above-mentioned results, the limit QL $\beta$ is absolutely continuous. Hereafter, we provide the first example in sR geometry of a QE property with a limit measure that is singular.

\subsubsection{QE in the Baouendi-Grushin case}

\begin{theorem}\label{thm_QE_Grushin}
QE is satisfied in the Baouendi-Grushin case when $\S$ is connected with at most one tangency point, with $w_\triangle = \nu/\nu(\S)$ (Weyl measure) as a limit measure on the base, and $W_\triangle$ in the phase space, which is half of the pullback of $w_\triangle$ by the double covering $S\Sigma\rightarrow M$ which is the restriction to $S\Sigma$ (with $\Sigma=D^\perp$) of the canonical projection of $T^\star M$ onto $M$.
\end{theorem}

\begin{proof}
By Corollary \ref{cor_concentration_onestratum}, there exists a local QL $\beta$, supported on $\S$, which is the weak limit of a density-one subsequence of probability measures $|\phi_{k_j}|^2 \mu$ on $M$.

We first treat the Baouendi-Grushin case without tangency point. In order to apply results on QLs established in \cite{CHT-I} in the 3D contact case, let us lift this 2D case in dimension three. A local normal form is given by the $g$-orthonormal frame $X_1=\partial_1$, $X_2=x_1a(x_1,x_2)\,\partial_2$ with $a(0,x_2)=1$. As a particular case of the general desingularization procedure (see \cite{Jean_2014, RS1976}), 
the (2D) Baouendi-Grushin case without tangency point is the projection onto $M$ of the 3D contact structure given on $M\times\mathcal{S}^1$ by the contact form $\alpha=dx_2 - x_1a(x_1,x_2)\, dx_3$, associated with the orthonormal frame $\tilde X_1=\partial_1$, $\tilde X_2=X_2+\partial_3$. We denote by $p:M\times\mathcal{S}^1\rightarrow M$ the canonical projection. Endowing $M\times\mathcal{S}^1$ with the measure $\tilde\mu=\mu\otimes dx_3$, we have $p_*\tilde\mu=\mu$. Let $\tilde\triangle = \tilde X_1^2+\tilde X_2^2$ be the sR Laplacian on $M\times\mathcal{S}^1$. The Reeb vector field is $(a+x_1\partial_1 a)\,\partial_2$. 

Noting that $p^*\phi_{k_j}$ is an eigenfunction of $\tilde\triangle$, the sequence of probability measures $|p^*\phi_{k_j}|^2 \tilde\mu$ on $M\times\mathcal{S}^1$ converges weakly to a QL $\tilde\beta$ of the 3D contact case, such that $p_*\tilde\beta=\beta$, which is, in the local coordinates, supported in $\{x_1=0\}$ (by the choice of $\beta$) and in $\{\xi_3=0\}$ (because $p^*\phi_{k_j}$ does not depend on $x_3$). By \cite[Theorem B]{CHT-I}, we have the Radon-Nikodym decomposition $\tilde\beta = \tilde\beta_0+\tilde\beta_\infty$ with $\tilde\beta_0$ that is invariant under the 3D contact geodesic flow and $\tilde\beta_\infty$ that is invariant under the Reeb flow. 
We claim that there is no 3D contact geodesic contained in $\{x_1=\xi_3=0\}$. Indeed, if such a geodesic were to exist, since the Hamiltonian of the lifted case is $\tilde H=\frac{1}{2}\xi_1^2+\frac{1}{2}(x_1\xi_2+\xi_3)^2$, this geodesic should be contained in $\vert\xi_1\vert=1$ and satisfy $\dot x_1=\xi_1\neq 0$ on $\S$, thus its projection should be transverse to $\S$, which is a contradiction. Therefore $\tilde\beta_0=0$ and $\tilde\beta = \tilde\beta_\infty$. Moreover, the projection of the Reeb vector field is tangent to $\S$ and lets $w_\triangle$ invariant, as expected. The quantum ergodicity property on the base follows.
In the phase space, since $\supp(W_\triangle)\subset\{x_1=\xi_1=0\}$, the only possibility is that $W_\triangle$ is supported on the $\mathcal{S}^1$-fiber bundle (in $\xi_2$) over $\S$, whence the result.

Let us now treat the Baouendi-Grushin case when $\S$ is a circle with exactly one tangency point $q_0$. Since $\S\setminus\{q_0\}$ is connected and since the Weyl measure $w_\triangle$ is the unique measure that is invariant under the projection of the Reeb flow, it follows from the above result obtained for the Baouendi-Grushin case without tangency point that the restriction of $\beta$ to $\S\setminus\{q_0\}$ coincides with $\nu$ up to scaling. Hence, there exists $\alpha\in[0,1]$ such that $\beta=(1-\alpha)\nu/\nu(\S)+\alpha\delta_{q_0}$, where $\delta_{q_0}$ is the Dirac mass at $q_0$. Let $\alpha_0\in(0,1)$ be fixed. Since the Weyl measure is $\nu/\nu(\S)$, any subsequence of the sequence of probability measures $|\phi_{k_j}|^2 d\mu$, for which $\alpha>0$, must have density zero.
The result follows by a classical diagonal argument.
\end{proof}

\subsubsection{QE in the Martinet case}
In the Martinet case, QE properties are open. 
The classical dynamics to consider might be a suitable normalization of the abnormal geodesics (more precisely, the dynamics of a characteristic vector field $Y$ on $\S$ leaving invariant the measure $\nu$). There does not seem to exist such a vector field in the general case. Assuming the existence of such a characteristic vector field and the ergodicity of the corresponding dynamics, we conjecture that the QE property is satisfied for any eigenbasis of the sR Laplacian $\triangle$.

\subsection{The Baouendi-Grushin and Martinet Laplacians associated with the Popp measure}
We have previously considered the Baouendi-Grushin and Martinet Laplacians $\triangle_\mu$ associated with the \emph{smooth} measure $\mu$ on $M$.
To complete our study, in this section we consider the Baouendi-Grushin and Martinet Laplacians $\triangle_P$ associated with the Popp measure $P$, which is singular on $\S$.
It is proved in \cite{BL-13} that both Baouendi-Grushin and Martinet Laplacians $\triangle_P$ are essentially selfadjoint on $M\setminus\S$.

\begin{theorem}\label{thm_BS_M_Popp}
Given any $f\in C^0(M)$, we have
$$
\int_M f(q)\, e_{\triangle,P}(t,q,q)\, dP(q) \sim 
\left\{ \begin{array}{ll}
\frac{1}{4\pi} \int_\S f \, d\nu\ \frac{\vert\ln t\vert}{t} & \textrm{in the Baouendi-Grushin case} \\[1mm]
\frac{1}{16} \int_\S f \, d\nu\ \frac{\vert\ln t\vert}{t^2} & \textrm{in the Martinet case} 
\end{array}\right.
$$
as $t\rightarrow 0^+$, and thus, as in Theorems \ref{thm_Baouendi-Grushin} and \ref{thm_Martinet}, $w_\triangle=\nu /\nu(\S)$ and 
$$
N(\lambda) \underset{\lambda\rightarrow+\infty}{\sim}
\left\{ \begin{array}{ll}
\frac{\nu(\S)}{4\pi}\lambda \ln \lambda & \textrm{in the Baouendi-Grushin case,} \\[1mm]
\frac{\nu(\S)}{32}\lambda^2 \ln \lambda & \textrm{in the Martinet case.}
\end{array}\right.
$$
Moreover, as in Corollary \ref{cor_concentration_onestratum}, there exists density-one sequence $(k_j)_{j\in\N}$ of positive integers such that, if $\supp(f)\cap\S=\emptyset$ then
$$
\lim_{j\rightarrow+\infty} \int_M f\vert\phi_{k_j}\vert^2\, dP = 0 .
$$
\end{theorem}

In the Baouendi-Grushin case we can even obtain as in Theorem \ref{thm_Baouendi-Grushin} an intrinsic two-terms expansion. But, although the first term is the same with a smooth measure and with the Popp measure, the second terms differ: the intrinsic constants are different. We refer to \cite{BPS-14} for the computation of these constants.

To obtain Theorem \ref{thm_BS_M_Popp} in the Baouendi-Grushin case, we proceed as follows. In a local model where $X_1=\partial_1$, $X_2=x_1\partial_2$, we have $dP=\frac{1}{\vert x_1\vert} \, dx_1\, dx_2$ and, following \cite{BL-13} the aR Laplacian $\triangle_P = \partial_1^2+x_1^2\partial_2-\frac{1}{x_1}\partial_1$ is unitarily equivalent to the second-order operator $\partial_1^2+x_1^2\partial_2-\frac{3}{4x_1^2}$ considered on $L^2(\R^2)$ with the Euclidean measure, which is proved to be essentially selfadjoint on $M\setminus\S$.
Then, we use the same general method as before, i.e., a $(J+K)$-decomposition and the use of Theorem \ref{lemfondamental} in Appendix \ref{app_lemfondam}.
The only difference is that the nilpotentized Laplacian is not the same as the one obtained with a smooth measure: it has a potential that is singular along $\S$. It remains anyway Riemannian far from $\S$, which explains why the first term is the same as with a smooth measure. 

The method is the same in the Martinet case: the local model $\triangle_P = \partial_1^2+(\partial_2+\frac{x_1^2}{2}\partial_3)^2-\frac{1}{x_1}\partial_1$ is unitarily equivalent to the second-order operator $\partial_1^2+(\partial_2+\frac{x_1^2}{2}\partial_3)^2-\frac{3}{4x_1^2}$ considered on $L^2(\R^2)$ with the Euclidean measure. 

In both cases however, with respect to Theorem \ref{lemfondamental}, which is established for more general operators in \cite{CHT_AHL}, we need an additional ingredient that is not written in that article: in \cite{CHT_AHL}, we have taken smooth bounded potentials, in particular in order to ensure dissipativity of the operator and thus existence of the semigroup. But actually the main result of \cite{CHT_AHL} can be extended to potentials $V$ blowing up only mildly, in the following sense: we assume that $\varepsilon^2 (\delta_\varepsilon^q)^* V  \rightarrow  \widehat{V}^q$ as $\varepsilon\rightarrow 0$ in $C^\infty$ topology (indeed, in the framework of \cite{CHT_AHL}, the operator $\triangle^\varepsilon$ involves the term $\varepsilon^2 (\delta_\varepsilon^q)^* V$ and this assumption is exactly devised to give a sense to its limit). 
For the above Baouendi-Grushin and Martinet cases with the Popp measure where the potential $V(x)=1/x_1^2$ is added to a smooth sR Laplacian, we obtain exactly $\widehat{V}^q = 1/x_1^2$.

\section{Local Weyl law in the equisingular stratified nilpotentizable case}\label{sec_equisingular_stratified_nilp}

Throughout this section, we assume that the singular set $\S$ is \emph{stratified by equisingular smooth submanifolds},
i.e., $\S$ is a Whitney stratified submanifold of $M$, disjoint union of strata, and
each stratum 
is an equisingular smooth connected submanifold of $M$ (see Appendix \ref{app_sRflag}). 
Hence $M$ is stratified as well by equisingular smooth submanifolds: its strata are the open set $M\setminus\S$ (regular region), of Hausdorff dimension $\Qeq$, and the strata of $\S$. 
%
We consider the set $\{\mathcal{Q}^1,\ldots,\mathcal{Q}^s\}$ (where $s\in\N\setminus\{0,1\}$) of all possible Hausdorff dimensions of strata of $M$, ordered in the increasing way, i.e., $\mathcal{Q}^1 < \cdots < \mathcal{Q}^s$. This means that, for each stratum of $M$, there exists $i\in\{1,\ldots,s\}$ such that $\mathcal{Q}^i$ is the Hausdorff dimension of this stratum. In particular, $\Qeq$ is equal to one of the $\mathcal{Q}^i$'s. The integer $\mathcal{Q}^s$ is the maximal possible Hausdorff dimension of a stratum.

The equisingular stratification assumption, which is a refinement of the stratification by topological dimension, has been introduced in \cite{Gromov} and considered as well in \cite{GhezziJean_NA2015, GhezziJean_TSG2015}. It is satisfied for generic smooth sR structures and for analytic sR structures (see \cite[Section 1.3.A]{Gromov}), 
in particular for nilpotent Lie groups and their quotients (because they have an analytic structure), and thus for nilpotentizable sR structures. 

As alluded in Remark \ref{rem_nilp_equiregular_region}, if $D$ is locally diffeomorphic to $\widehat{D}^q$ for every $q\in M$ (this is stronger than $\S$-nilpotentizability where this property is required for every $q\in\S$ only) then the equisingular stratification property is automatically satisfied.

\subsection{Main result}\label{sec_thm_multistrates_nilp}
Let $q\in\S$ be arbitrary. By definition of the Whitney stratification, $q$ belongs to a stratum $\S_1$ and to the closure of $p$ other nested strata of increasing topological dimensions: 
$$
q\in\S_1\subset\overline\S_2\subset\cdots\subset\overline\S_p\subset\overline\S_{p+1}=\overline{M\setminus\S}=M
$$
(where $p\in\N^*$ and the sequence of strata depend on $q$), with $\dim\S_i<\dim\S_{i+1}\leq n$ for $i=1,\ldots,p+1$ (hence, $p\leq n$). We call the sequence $\mathscr{C}=(\S_1,\ldots,\S_p,\S_{p+1})$ a \emph{chain of strata} at $q$, of length $p+1$. By convention, we always have $\S_{p+1}=M\setminus\S$. Of course, there may exist several chains of strata at $q$, of possibly different lengths. 
The chains are however invariant along $\S_1$, meaning that if $\mathscr{C}$ is a chain at $q\in\S_1$ then $\mathscr{C}$ is a chain at any other $q'\in\S_1$.

%


Let $q\in\S$ and let $\mathscr{C}=(\S_1,\ldots,\S_p,\S_{p+1})$ be a 
chain of strata at $q$, of length $p+1$. 
For every $j\in\{1,\ldots,p+1\}$, the Hausdorff dimension $\mathcal{Q}^{\S_j}$ is equal to $\mathcal{Q}^i$ for some $i\in\{1,\ldots,s\}$. We denote by $m_i(q,\mathscr{C}) \in\{0,\ldots, n+1\}$ the ``multiplicity" of $\mathcal{Q}^i$ in the chain $\mathscr{C}$ at $q$, that is, the number of integers $j\in\{1,\ldots,p+1\}$ such that $\mathcal{Q}^{\S_j} = \mathcal{Q}^i$.
The multiplicity $m_i(q,\mathscr{C})$ does not depend on $q\in\S_1$.

Finally, we define the ``maximal multiplicity" of $\mathcal{Q}^i$ by
$$
m_i = \displaystyle\max\{m_i(q,\mathscr{C})\ \mid\ q\in\S,\ \mathscr{C}\textrm{ chain of strata at } q \}\qquad \forall i\in\{1,\ldots,s\} .
$$
Note that $1\leq m_i\leq n+1$.
The $2s$ integers $\mathcal{Q}^i,m_i$, standing for all possible Hausdorff dimensions of strata of $M$ together with their maximal multiplicity, are the characteristic integers appearing in the local Weyl law.

We say that a chain $\mathscr{C}=(\S_1,\ldots,\S_p,\S_{p+1})$ of strata at $q$, of length $p+1$, is \emph{maximal} if it contains $m_s$ strata of maximal Hausdorff dimension $\mathcal{Q}^s$. Let $M_s$ be the equisingular stratified submanifold of $\S$ defined as the set of all $q\in M$ at which there exists a maximal chain.


\begin{theorem}\label{thm_multistrates_nilp}
Assume that the horizontal distribution $D$ is $\S$-nilpotentizable. 
Given any $f\in C^\infty(M)$, the function $t\mapsto\Tr(\mathcal{M}_f \, e^{t\triangle})= \int_M f(q)\, e(t,q,q)\, d\mu(q)$ has an asymptotic expansion as $t\rightarrow 0^+$ with respect to the scale of functions $t^{(k-\mathcal{Q}^i)/2}\vert\ln t\vert^j$, with $i\in\{1,\ldots,s\}$, $j\in\{0,\ldots,m_i-1\}$ and $k\in\N$.
%
Moreover, there exists a nontrivial Borel measure $\nu$ on $M$ such that, for every $f\in C^0(M)$,
$$
\Tr(\mathcal{M}_f \, e^{t\triangle}) = \left( \int_{\mathcal{N}_s} f \, d\nu\right) \frac{\vert\ln t\vert^{m_s-1}}{t^{\mathcal{Q}^s/2}}  + \mathrm{o}\left( \frac{\vert\ln t\vert^{m_s-1}}{t^{\mathcal{Q}^s/2}} \right)
$$
as $t\rightarrow 0^+$. 
The support of $\nu$ is the equisingular stratified submanifold $\mathcal{N}_s$ of $M_s$ defined as follows:
take any $q\in M_s$ and any maximal chain $\mathscr{C}$ of strata at $q$; among the strata of $\mathscr{C}$ of maximal Hausdorff dimension $\mathcal{Q}^s$, consider the stratum that is of minimal topological dimension; then $\mathcal{N}_s$ is the closure of the union of all such strata, over all $q\in M_s$ and all maximal chains at $q$.

As a consequence, the local Weyl measure exists, is smooth on $\mathcal{N}_s$ and is given by $w_\triangle=\nu/\nu(\mathcal{N}_s)$, and we have
$$
N(\lambda) \underset{\ \lambda\rightarrow+\infty}{\sim}  \frac{\nu(\mathcal{N}_s)}{\Gamma(\mathcal{Q}^s/2+1)} \lambda^{\mathcal{Q}^s/2} (\ln\lambda)^{m_s-1} .
$$
\end{theorem}

\begin{remark}\label{rem_density_equisingular_stratified_nilp}
Let us describe the density of $\nu$ on $\mathcal{N}_s$.
For every $q\in\mathcal{N}_s$, there exists a chain $\mathscr{C}=(\S_1,\ldots,\S_{p},\S_{p+1})$ of strata at $q$, of length $p+1$, containing $m_s$ strata $\S_{i_j}$ ($j=1,\ldots,m_s$, with $i_1<\cdots<i_s$) of maximal Hausdorff dimension $\mathcal{Q}^{\S_{i_j}} = \mathcal{Q}^s$ and of maximal multiplicity $m_s$. Among these $m_s$ strata $\S_{i_j}$, the stratum $\S_{i_1}$ is the one that is of minimal topological dimension.
Then the density of $\nu$ with respect to $\mu_{\S_{i_1}}$ (the smooth measure on 
$\S_{i_1}$ inferred from $\mu$ as in \eqref{muS} in Section \ref{sec_geom_context}) is smooth and is given at any point $q_1\in \S_{i_1}$ by
\begin{multline}\label{densite_multistrates}
\frac{d\nu}{d\mu_{\S_{i_1}}}(q_1) = \frac{1}{2^{m_s-1}\, (m_s-1)!} \int_{M_2(q_1)} \cdots \int_{M_{j+1}(q_1,\ldots,q_j)} \cdots \int_{M_{i_{m_s}}(q_1,\ldots,q_{m_s-1})} \\
\int_{\R^{n-\dim\S_{i_{m_s}}}} \widehat{e}^{\, q_1,q_2,\ldots,q_{m_s}}(1,(0,x),(0,x))\, dx\, d\mu_{m_s}(q_{m_s})\cdots d\mu_2(q_2) 
\end{multline}
where 
$M_{j+1}(q_1,\ldots,q_j) = \mathcal{S}^{n-\dim\S_{i_j}-1}\cap \widehat{\S_{i_{j+1}}}^{q_1,\ldots,q_j}$.

Here, for a stratum or for the heat kernel, the notation $\widehat{\star}^{\, q_1,q_2,\ldots,q_j}$ stands for the $j$-th nilpotentization of $\star$, successively at $q_1\in\S_{i_1}$, then at $q_2\in M_1(q_1)$, etc, and finally at $q_j \in M_j(q_1,\ldots,q_{j-1})$ (it is defined by induction, see Remark \ref{rem_multiple_nilp} in Section \ref{sec_doublenilp}).
The measure $\mu_i$ is the smooth measure on $M_j(q_1,\ldots,q_{j-1})$ inferred from $\mu$ as in \eqref{muS}.
In \eqref{densite_multistrates}, by convention, we remove the nested integrals $\int_{M_{j+1}(q_1,\ldots,q_j)}$ ($j=1,\ldots,i_{m_s}-1$) whenever $m_s=1$.

Note that, if $\mathcal{Q}^{M\setminus\S}=\mathcal{Q}^s$ then $\S_{i_{m_s}}=M\setminus\S$ and thus $\dim\S_{i_{m_s}}=n$ and, in \eqref{densite_multistrates}, we have $\int_{\R^{n-\dim\S_{i_{m_s}}}} \widehat{e}^{\, q_1,q_2,\ldots,q_{m_s}}(1,(0,x),(0,x))\, dx = \widehat{e}^{\, q_1,q_2,\ldots,q_{m_s}}(1,0,0)$.
\end{remark}


\begin{remark}
Theorem \ref{thm_multistrates_nilp} generalizes Theorem \ref{thm_onestratum}: 
\begin{itemize}[parsep=0cm,topsep=0cm]
\item if $m_s=1$ and $\mathcal{Q}^s>\Qeq$ then \eqref{densite_multistrates} gives $\frac{d\nu}{d\mu_{\S_{i_1}}}(q_1) = \Tr_{N_{q_1}\S_{i_1}}(e^{\widehat{\triangle}^{q_1}})$, as in the first case of Theorem \ref{thm_onestratum}: the dominating term and the density are given by the stratum $\S_{i_1}$;
\item if $m_s=2$ and $\mathcal{Q}^s=\Qeq$ then $M_2(q_1)=\mathcal{S}^{n-\dim\S_{i_1}-1}$, $\S_{i_2}=M\setminus\S$ and \eqref{densite_multistrates} gives $\frac{d\nu}{d\mu_{\S_{i_1}}}(q_1) = \frac{1}{2} \int_{M_2(q_1)} \widehat{e}^{\, q_1,\sigma}(1,0,0)\, d\sigma$, as in the second case of Theorem \ref{thm_onestratum};
\item if $m_s=1$ and $\mathcal{Q}^s=\Qeq$ then $\mathcal{N}_s=M$, $\dim\S_{i_{m_s}}=n$ and \eqref{densite_multistrates} gives $\frac{d\nu}{d\mu}(q)=\widehat{e}^q(1,0,0)$ at any $q\in M$, as in the third case of Theorem \ref{thm_onestratum}: the equiregular part dominates. 
\end{itemize}
\end{remark}

\begin{remark}\label{rem_concentration_multistrates_nilp}
In Theorem \ref{thm_multistrates_nilp}, the only situation where one has $\mathcal{N}_s=M$ is when all strata of the singular set $\S$ have a Hausdorff dimension (stricly) less than $\Qeq$ (third case $m_s=1$ and $\mathcal{Q}^s=\Qeq$ of the latter remark). As soon as the Hausdorff dimension of one of the strata of $\S$ is greater than or equal to $\Qeq$, the support of the Weyl measure is contained in $\S$.
\end{remark}

\begin{remark}
As in Remark \ref{one-stratum_outsideS}, note that, in Theorem \ref{thm_multistrates_nilp}, it is assumed that $D$ is locally diffeomorphic to $\widehat{D}^q$ at every point $q$ of every stratum of $\S$, but nothing is assumed at $q\in M\setminus\S$.
\end{remark}

\subsection{Examples}\label{sec_examples_stratified_nilp}
In the examples hereafter, the singular set $\S$ is stratified by equisingular smooth submanifolds and the horizontal distribution $D$ is $\S$-nilpotentizable.

%

\medskip

\noindent
-- Consider the sR case in $\R^3$ generated by $X_1=\partial_1$, $X_2 = \partial_2 + x_1^k x_2 \, \partial_3$, for $k\in\N^*$. When $k=2$, this is the so-called nilpotent tangential hyperbolic case (see \cite{BonnardTrelat_AFST2001}). The singular set $\S=\{x_1=0\}\cup\{x_2=0\}$ consists of the three strata $\S_1 = \{x_1=x_2=0\}$, $\S_2 = \{x_1=0, x_2\neq 0\}$ and $\S_2' = \{x_1\neq 0, x_2=0\}$.
We have 
$\mathcal{Q}^{\S_1}=\mathcal{Q}^{\S_2}=k+2$ and $\mathcal{Q}^{\S_2'}=\Qeq=4$. 
Note that the possible chains of strata at points of $\S_1$ are either $(\S_1,\S_2,M\setminus\S)$ or $(\S_1,\S_2',M\setminus\S)$. 
If $k=1$ then $m_s=2$ and thus the main term in the small-time asymptotics of the local Weyl law is of the order of $\frac{\vert\ln t\vert}{t^2}$, and the local Weyl measure is supported on $\overline{\S_2'}=\{x_2=0\}$.
For $k\geq 2$, we have $m_s=3$ if $k=2$ and $m_s=2$ if $k\geq 3$, and thus
the main term in the small-time asymptotics of the local Weyl law is of the order of $\frac{\vert\ln t\vert^2}{t^2}$ if $k=2$ and $\frac{\vert\ln t\vert}{t^{1+\frac{k}{2}}}$ if $k\geq 3$, and the local Weyl measure is supported on $\S_1$.

\medskip

\noindent
-- Consider the (nilpotent) aR case in $\R^5$ generated by $X_i=\partial_i$ for $i=1,2,3,4$, and $X_5=(x_1^2+x_2^2)(x_3^2+x_4^2)\, \partial_5$. The singular set $\S=\{x_1=x_2=0\}\cup\{x_3=x_4=0\}$ consists of the three strata $\S_1 = \{x_1=x_2=x_3=x_4=0\}$, $\S_2 = \{x_1=x_2=0, x_3^2+x_4^2\neq 0\}$ and $\S_2' = \{x_1^2+x_2^2\neq 0, x_3=x_4=0\}$. We have 
$\mathcal{Q}^{\S_1}=\mathcal{Q}^{\S_2}=\mathcal{Q}^{\S_2'}=\Qeq=5$. The main term in the small-time asymptotics of the local Weyl law is of the order of $\frac{\vert\ln t\vert^2}{t^{5/2}}$. The local Weyl measure is supported on $\S_1$.

\medskip

\noindent
-- As noted as the end of in Section \ref{sec_examples_onestratum}, we can generate other examples by taking products.

\subsection{Proof of Theorem \ref{thm_multistrates_nilp}}\label{sec_proof_thm_multistrates_nilp}


Let $q\in\S$ and let $\mathscr{C}=(\S_1,\ldots,\S_{p},\S_{p+1})$ be a chain of strata at $q$ (with $\S_{p+1}=M\setminus\S$).
We are going to estimate the local Weyl law around $q$ by iterating the $(J+K)$-decomposition procedure along the sequence of strata $\S_i$, $i=1,\ldots,p$. We assume that $p\geq 2$.
Without loss of generality, we assume that $\mathscr{C}$ is \emph{exhaustive} in the following sense: for every $i\in\{1,\ldots,p\}$, if $P$ is a stratum such that $\S_i\subset\overline P\subset\overline\S_{i+1}$ then $P=\S_i$ or $P=\S_{i+1}$.
Setting $k_i=\dim(\S_i)$ (topological dimension), we have $k_1< \cdots < k_p<n$.

Following Section \ref{sec_geom_context}, we take privileged coordinates at $q$ straightening the stratum $\S_1$. 
Following Section \ref{sec_J+K}, we write $I(t)=J(t)+K(t)$ with $J(t)$ and $K(t)$ defined by \eqref{I=J+K} (with $\S=\S_1$ in that formula).
The analysis done in Section \ref{sec_estimating_J}, which yields the expansion \eqref{asympt_J} for $J(t)$ remains valid here (with $\S=\S_1$).
%
Concerning $K(t)$, the formula \eqref{intK_S=S1_1} in Section \ref{sec_estimating_K} remains valid (with $\S=\S_1$), but, as explained in Remark \ref{remK_strat}, \eqref{intK_S=S1_2} raises problems because $\mathcal{Q}^{\R^n}(\sigma)$ is not constant on $\mathcal{S}^{n-k_1-1}$: the set $\S_1\times\mathcal{S}^{n-k_1-1}$, viewed as a cylinder around the stratum $\S_1$, intersects the other strata $\S_i$ of the singular set $\S$.
We are thus led to consider the stratification of $(\S_1\times\mathcal{S}^{n-k_1-1})\cap\S$ and to iterate the $(J+K)$-decomposition along this stratified submanifold. 

By \eqref{intK_S=S1_1}, we have
\begin{equation}\label{K_2strat}
K(t) 
 = \int_{\sqrt{t}}^1 \tau_1^{-\mathcal{Q}^{\S_1}-1}  I_{\tau_1}\left( \frac{t}{\tau_1^2}\right)  d\tau_1 
\end{equation}
where
$$
I_{\tau_1}(t_1) = \int_{\S_1} I_{\tau_1}^{q_1}(t_1) \, d\mu_{\S_1}(q_1), \qquad
I_{\tau_1}^{q_1}(t_1) = \int_{\mathcal{S}^{n-k_1-1}} f(\delta_{\tau_1}^{q_1}(0,\sigma)) \, e_{\tau_1}^{q_1} (t_1,(0,\sigma),(0,\sigma)) \, d\sigma
$$
for every $\tau_1\in\R$ and every $t_1>0$ (we will take $t_1=\frac{t}{\tau_1^2}$). Given any $\tau_1\in\R$ and any $q_1\in\S_1$, the integral $I_{\tau_1}^{q_1}(t_1)$ is of the same kind as the integral $I(t)$ defined by \eqref{def_int_I}, and we are thus going to apply to $I_{\tau_1}^{q_1}(t_1)$ the $(J+K)$-decomposition procedure developed in Section \ref{sec_equisingular_nilp}; except that now the integration is performed over the 
submanifold $\mathcal{S}^{n-k_1-1}$ (instead of $M$), which is viewed as a sphere centered at $q_1$, and we have to consider the heat kernel $e_{\tau_1}^{q_1}$ (instead of $e$) associated with the sR Laplacian $\triangle_{\tau_1}^{q_1}$, depending on the parameters $\tau_1\in\R$ and $q_1\in\S_1$. 

For any $\tau_1\neq0$, 
the singular set $\S_{\tau_1}^{q_1}=(\delta^{q_1}_{\tau_1})^{-1}(\S)$ of $D_{\tau_1}^{q_1}=(\delta^{q_1}_{\tau_1})^*D$ is stratified by the equisingular smooth submanifolds $(\S_i)_{\tau_1}^{q_1}=(\delta^{q_1}_{\tau_1})^{-1}(\S_i)$, for $i=2,\ldots,p$. 
Hence $\mathcal{S}^{n-k_1-1}$ is Whitney stratified by the smooth (but not necessarily equisingular) submanifolds $\mathcal{S}^{n-k_1-1} \cap (\S_i)_{\tau_1}^{q_1}$, for $i=2,\ldots,p$. 
This Whitney stratification of $\mathcal{S}^{n-k_1-1}$ leads to write $I_{\tau_1}^{q_1}(t_1)$, and thus $K(t)$, as a sum of integrals over various regions. 
%
Note that $\mathcal{Q}^{\R^n}(\mathcal{S}^{n-k_1-1}\setminus\S_{\tau_1}^{q_1})=\Qeq$, 
and the results of Section \ref{sec_equisingular_nilp} can be applied in the region $\mathcal{S}^{n-k_1-1}\setminus\S_{\tau_1}^{q_1}$: far from all strata $\S_i$, $i\geq 2$, we only see the influence of the single stratum $\S_1$.


Let us now investigate the influence of the second stratum $\S_2$.
Note that $\mathcal{S}^{n-k_1-1}$ and $(\S_2)_{\tau_1}^{q_1}$ 
intersect transversally, and $\dim (\mathcal{S}^{n-k_1-1}\cap (\S_2)_{\tau_1}^{q_1}) = k_2-k_1-1$. This is due to the Whitney stratification property and to the fact that $n-k_1-1+k_2 \geq n$. 

Thanks to the nilpotentizability assumption, $(\S_i)_{\tau_1}^{q_1}$ depends smoothly on $\tau_1\in\R$, including $\tau_1=0$, and for $\tau_1=0$ we have $(\S_i)_0^{q_1} = \widehat{\S_i}^{q_1}$, which is the $i^\textrm{th}$ stratum of the singular stratification of the nilpotentized sR structure at $q_1$.
The manifold $\mathcal{S}^{n-k_1-1}\cap \widehat{\S_2}^{q_1}$ is stratified by equisingular smooth submanifolds, hence, by nilpotentizability, the manifold $\mathcal{S}^{n-k_1-1}\cap (\S_2)_{\tau_1}^{q_1}$ is stratified by equisingular smooth submanifolds depending smoothly on $q_1\in\S_1$ and on $\tau_1\in[-1,1]$.
Since the reasoning hereafter can be applied to each equisingular stratum of $\mathcal{S}^{n-k_1-1}\cap (\S_2)_{\tau_1}^{q_1}$, without loss of generality we assume that $\mathcal{S}^{n-k_1-1}\cap (\S_2)_{\tau_1}^{q_1}$ is equisingular.


Let us choose an arbitrary point $q_2\in \mathcal{S}^{n-k_1-1}\cap (\S_2)_{\tau_1}^{q_1}$ (depending on $q_1$ and $\tau_1$).
%
Following Section \ref{sec_geom_context}, we take privileged coordinates around $q_2$ straightening $(\S_2)_{\tau_1}^{q_1}$, depending smoothly on $q_1\in\S_1$ and, thanks to the nilpotentizability assumption, on $\tau_1\in[-1,1]$ (including $\tau_1=0$).
We are led to consider the following characteristic integers, which do not depend on $\tau_1\in\R$ for $\vert\tau_1\vert>0$ small enough 
nor on $q_1\in\S_1$:
\begin{itemize}[parsep=0.1cm,itemsep=0.25mm,topsep=0.1cm]
\item $\mathcal{Q}^M(\S_2) = \mathcal{Q}^{\R^n}(  (\S_2)_{\tau_1}^{q_1} )$ 
(defined by \eqref{def_Q} in Appendix \ref{app_sRflag});
\item $\mathcal{Q}^{\S_2} = \mathcal{Q}^{ (\S_2)_{\tau_1}^{q_1}}$ 
(Hausdorff dimension of $\S_2$, defined by \eqref{def_QN} in Appendix \ref{app_sRflag}); 
\item $\mathcal{Q}^{\mathcal{S}^{n-k_1-1}\cap (\S_2)_{\tau_1}^{q_1}} = \mathcal{Q}^{\delta^{q_1}_{\tau_1}(\mathcal{S}^{n-k_1-1})\cap \S_2}$ 
(Hausdorff dimension of $\mathcal{S}^{n-k_1-1}\cap (\S_2)^{q_1}_{\tau_1}$);  
\item $\mathcal{Q}^{\mathcal{S}^{n-k_1-1}} \left( \mathcal{S}^{n-k_1-1}\cap (\S_2)_{\tau_1}^{q_1} \right)$. 
\end{itemize}
Actually, since $D$ is $\S$-nilpotentizable, the above four integers are respectively equal to $\mathcal{Q}^{\R^n}(\widehat{\S_2}^q)$, $\mathcal{Q}^{\widehat{\S_2}^{q_1}}$, $\mathcal{Q}^{\mathcal{S}^{n-k_1-1}\cap \widehat{\S_2}^{q_1}}$ and $\mathcal{Q}^{\mathcal{S}^{n-k_1-1}} ( \mathcal{S}^{n-k_1-1}\cap \widehat{\S_2}^{q_1} )$, but this fact is not useful.

By Lemma \ref{lemQinter} and Remark \ref{rem_lemQinter_extension} (in Appendix \ref{app_privileged}) applied to $P_1(q_1,\tau_1)=\mathcal{S}^{n-k_1-1}\cap (\S_2)_{\tau_1}^{q_1}$ and $N_2(q_1,\tau_1)=(\S_2)_{\tau_1}^{q_1}$, for every $q_1\in\S_1$ and for every $\tau_1$ with $\vert\tau_1\vert>0$ small enough, including $\tau_1=0$ by using the nilpotentizability assumption, there exists a smooth submanifold $P(q_1,\tau_1)$ of $M$ of topological dimension $n-k_1-1$, depending smoothly on $q_1\in\S_1$ and on $\tau_1$, containing $P_1(q_1,\tau_1)$ and intersecting $(\S_2)_{\tau_1}^{q_1}$ transversally, such that 
$$
\mathcal{Q}^{P(q_1,\tau_1)}(P(q_1,\tau_1)\cap (\S_2)_{\tau_1}^{q_1}) - \mathcal{Q}^{P(q_1,\tau_1)\cap (\S_2)_{\tau_1}^{q_1}} = \mathcal{Q}^M(\S_2) - \mathcal{Q}^{\S_2} .
$$
Note that the $\S$-nilpotentizability assumption ensures that $P(q_1,\tau_1)$ depends smoothly on $\tau_1$ at $\tau_1=0$. 

Now, we replace $\mathcal{S}^{n-k_1-1}$ by $P(q_1,\tau_1)$. Although it now depends on $\tau_1$, we choose to keep the same notation $\mathcal{S}^{n-k_1-1}$ in what follows. The important formula \eqref{transverse_polar} is still valid although $\mathcal{S}^{n-k_1-1}$ is not exactly the unit Euclidean sphere (see Remark \ref{rem_sphere_transverse}). 
Hence, we have 
\begin{equation}\label{diff_Q}
\mathcal{Q}^{\mathcal{S}^{n-k_1-1}} \left( \mathcal{S}^{n-k_1-1}\cap (\S_2)_{\tau_1}^{q_1} \right) - \mathcal{Q}^{\mathcal{S}^{n-k_1-1}\cap (\S_2)_{\tau_1}^{q_1}} = \mathcal{Q}^M(\S_2)-\mathcal{Q}^{\S_2}
\end{equation}
for every $q_1\in\S_1$ and every $\tau_1$ with $\vert\tau_1\vert$ small enough. Note that 
\eqref{diff_Q} may fail if $\mathcal{S}^{n-k_1-1}$ differs from $P(q_1,\tau_1)$ (see Remark \ref{rem_lemQinter} in Appendix \ref{app_privileged}). Hence, adapting the definition of $\mathcal{S}^{n-k_1-1}$ is important.
Note also that, when $\mathcal{S}^{n-k_1-1}\cap (\S_2)_{\tau_1}^{q_1}$ is stratified by equisingular smooth submanifolds, it is understood that \eqref{diff_Q} holds for any equisingular stratum of $\mathcal{S}^{n-k_1-1}\cap (\S_2)_{\tau_1}^{q_1}$.

\medskip

Applying the $(J+K)$-decomposition to the integral $I_{\tau_1}^{q_1}(t_1)$, 
we have
\begin{equation}\label{J+K_tau1}
I_{\tau_1}^{q_1}(t_1)=J_{\tau_1}^{q_1}(t_1)+K_{\tau_1}^{q_1}(t_1)
\end{equation}
for every $t_1>0$ with $J_{\tau_1}^{q_1}(t_1)$ and $K_{\tau_1}^{q_1}(t_1)$ defined as in \eqref{I=J+K}, with $M$ replaced by $\mathcal{S}^{n-k_1-1}$ and $\S$ by $\mathcal{S}^{n-k_1-1}\cap (\S_2)_{\tau_1}^{q_1}$.
In what follows, we denote 
$$
f_{\tau_1}^{q_1}(\sigma) = f(\delta_{\tau_1}^{q_1}(0,\sigma)),\qquad
f_{\tau_1,\tau_2}^{q_1,q_2}(\sigma) = f_{\tau_1}^{q_1}(\delta_{\tau_2}^{q_2}(0,\sigma)),
\qquad
e_{\tau_1,\tau_2}^{q_1,q_2} = ( e_{\tau_1}^{q_1} )_{\tau_2}^{q_2} .
$$
Following Section \ref{sec_estimating_J}, 
since $(\sqrt{t_1})^{\mathcal{Q}^M(\S_2)}\, e_{\tau_1}^{q_1} \big( t_1,\delta_{\sqrt{t_1}}^{q_2}(x),\delta_{\sqrt{t_1}}^{q_2}(x) \big) = e_{\tau_1,\sqrt{t_1}}^{q_1,q_2}(1,x,x)$, we obtain  
\begin{equation*}
J_{\tau_1}^{q_1}(t_1) = \frac{1}{t_1^{\mathcal{Q}^{\S_2}/2}} \int_{\mathcal{S}^{n-k_1-1}\cap (\S_2)_{\tau_1}^{q_1}} \int_{\mathcal{B}^{n-k_2}}  f_{\tau_1,\sqrt{t_1}}^{q_1,q_2}(x) \, e_{\tau_1,\sqrt{t_1}}^{q_1,q_2}(1,x,x)  \, dx \, d\mu_2(q_2)  .
\end{equation*}
In order to express $K_{\tau_1}^{q_1}(t_1)$, we first need to adapt the change of variable \eqref{transverse_polar} to the new context: in Section \ref{sec_geom_context}, we considered transverse polar coordinates along $\S$ on the manifold $M \simeq [0,+\infty)\times\S\times\mathcal{S}^{n-k-1}$ endowed with the measure $\tau^{\mathcal{Q}^M(\S)-\mathcal{Q}^\S-1}\, d\sigma\otimes\mu_\S\otimes d\tau$. Now, we consider transverse polar coordinates along $\mathcal{S}^{n-k_1-1}\cap (\S_2)_{\tau_1}^{q_1}$ on the manifold $\mathcal{S}^{n-k_1-1} \simeq [0,+\infty)\times \left( \mathcal{S}^{n-k_1-1}\cap (\S_2)_{\tau_1}^{q_1} \right) \times \mathcal{S}^{n-k_2-1}$ endowed, thanks to \eqref{diff_Q}, with the measure
$$
\tau_2^{\mathcal{Q}^M(\S_2)-\mathcal{Q}^{\S_2}-1} \, d\tau \otimes\mu_2 \otimes d\sigma
$$
where $\mu_2$ is the smooth measure defined as in Section \ref{sec_geom_context} on the $(k_2-k_1-1)$-dimensional submanifold $\mathcal{S}^{n-k_1-1}\cap (\S_2)_{\tau_1}^{q_1}$ 
(in the local chart, $\mathcal{S}^{n-k_2-1}$ is considered as a submanifold of $\mathcal{S}^{n-k_1-1}$).
Recall that, since $D$ is $\S$-nilpotentizable, $(\S_2)_{\tau_1}^{q_1}$ is diffeomorphic to $\widehat{\S_2}^{q_1}$ with a diffeomorphism depending smoothly on $\tau_1$.
Actually, this construction is performed on every equisingular stratum of $\mathcal{S}^{n-k_1-1}\cap (\S_2)_{\tau_1}^{q_1}$.
Then, following Section \ref{sec_estimating_K} and in particular \eqref{intK_S=S1_1}, we obtain
\begin{equation*}
K_{\tau_1}^{q_1}(t_1) = \int_{\sqrt{t_1}}^1 \tau_2^{-\mathcal{Q}^{\S_2}-1} \int_{\mathcal{S}^{n-k_1-1}\cap (\S_2)_{\tau_1}^{q_1}}  I_{\tau_1,\tau_2}^{q_1,q_2}\left( \frac{t_1}{\tau_2^2}\right) \, d\mu_2(q_2) \, d\tau_2 
\end{equation*} 
where 
\begin{equation}\label{Itau1tau2}
I_{\tau_1,\tau_2}^{q_1,q_2}(t_2) = \int_{\mathcal{S}^{n-k_2-1}} f_{\tau_1,\tau_2}^{q_1,q_2}(\sigma) \, e_{\tau_1,\tau_2}^{q_1,q_2} (t_2, \sigma, \sigma) \, d\sigma   \qquad \forall t_2>0 .
\end{equation}
In \eqref{Itau1tau2}, $\mathcal{S}^{n-k_2-1}$ is identified with a submanifold of codimension one of $\mathcal{S}^{n-k_1-1}$, that is an hypersphere of $\mathcal{S}^{n-k_2-1}$ centered around $\mathcal{S}^{n-k_1-1}\cap (\S_2)^{q_1}_{\tau_1}$.
%
Taking $t_1=\frac{t}{\tau_1^2}$, it follows from \eqref{K_2strat} and \eqref{J+K_tau1} that $K(t)$ is the sum of two terms:
\begin{multline}\label{KS1S2}
K(t) 
= \frac{1}{t^{\mathcal{Q}^{\S_2}/2}} \int_{\sqrt{t}}^1 \tau_1^{\mathcal{Q}^{\S_2}-\mathcal{Q}^{\S_1}-1}  G_1\left( \tau_1 , \frac{\sqrt{t}}{\tau_1} \right) \, d\tau_1  \\
+ \int_{\sqrt{t}}^1 \tau_1^{-\mathcal{Q}^{\S_1}-1} \int_{\frac{\sqrt{t}}{\tau_1}}^1 \tau_2^{-\mathcal{Q}^{\S_2}-1}  I_{\tau_1,\tau_2}\left( \frac{t}{\tau_1^2\tau_2^2}\right) \, d\tau_2  \, d\tau_1
\end{multline}
with
\begin{equation}\label{defG1}
G_1(\tau_1,\tau_2) = \int_{\S_1} \int_{\mathcal{S}^{n-k_1-1}\cap (\S_2)_{\tau_1}^{q_1}} \int_{\mathcal{B}^{n-k_2}} f_{\tau_1,\tau_2}^{q_1,q_2}(x) \, e_{\tau_1,\tau_2}^{q_1,q_2}(1,x,x) \, dx \, d\mu_2(q_2) \, d\mu_{\S_1}(q_1)
\end{equation}
for all $(\tau_1,\tau_2)\in\R^2$, and
$$
I_{\tau_1,\tau_2}(t_2) = \int_{\S_1} \int_{\mathcal{S}^{n-k_1-1}\cap (\S_2)_{\tau_1}^{q_1}}  I_{\tau_1,\tau_2}^{q_1,q_2} (t_2) \, d\mu_2(q_2) \, d\mu_{\S_1}(q_1) 
$$
for all $(\tau_1,\tau_2)\in\R^2$ and for every $t_2>0$ (we will take $t_2=\frac{t_1}{\tau_2^2}=\frac{t}{\tau_1^2\tau_2^2}$).
Note that the two terms at the right-hand side of \eqref{KS1S2} are coming from a $(J+K)$-decomposition. By Remark \ref{rem_prelim_JK}, we know in advance that the dominating term in the small-time asymptotics of $K(t)$ will be of the order of that of the second term at the right-hand side of \eqref{KS1S2}, although the first term may contribute to the main term.


\paragraph{Case $p=2$.}
When $p=2$, in the integral \eqref{Itau1tau2}, the variable $\sigma$, which ranges over the sphere $\mathcal{S}^{n-k_2-1}$, belongs to the regular region $\S_3=M\setminus\S$, hence $\mathcal{Q}^{\R^n}(\sigma)=\Qeq=\mathcal{Q}^{\S_3}$ and thus, using \eqref{relation_eeps_e} (and neglecting the remainder terms),
$$
e_{\tau_1,\tau_2}^{q_1,q_2}\left( \frac{t}{\tau_1^2\tau_2^2}, \sigma, \sigma \right) = \frac{\tau_1^{\mathcal{Q}^{\S_3}} \tau_2^{\mathcal{Q}^{\S_3}}}{t^{\mathcal{Q}^{\S_3}/2}} \, e_{\tau_1,\tau_2,\frac{\sqrt{t}}{\tau_1\tau_2}}^{q_1,q_2,\sigma}(1,0,0)
$$
and therefore, in this case, the second term at the right-hand side of \eqref{KS1S2} is
$$
\frac{1}{t^{\Qeq/2}} \int_{\sqrt{t}}^1 \tau_1^{\mathcal{Q}^{\S_3}-\mathcal{Q}^{\S_1}-1} \int_{\frac{\sqrt{t}}{\tau_1}}^1 \tau_2^{\mathcal{Q}^{\S_3}-\mathcal{Q}^{\S_2}-1} G_2\left(\tau_1,\tau_2,\frac{\sqrt{t}}{\tau_1\tau_2} \right) d\tau_2\, d\tau_1
$$
with
\begin{equation}\label{defG2}
G_2\left(\tau_1,\tau_2,\varepsilon \right) = \int_{\S_1} \int_{\mathcal{S}^{n-k_1-1}\cap (\S_2)_{\tau_1}^{q_1}}  \int_{\mathcal{S}^{n-k_2-1}} f_{\tau_1,\tau_2}^{q_1,q_2}(\sigma) \, e_{\tau_1,\tau_2,\varepsilon}^{q_1,q_2,\sigma}(1,0,0) \, d\sigma \, d\mu_2(q_2) \, d\mu_{\S_1}(q_1)
\end{equation}
for all $(\tau_1,\tau_2,\varepsilon)\in\R^3$. 
Hence, with the notations of Appendix \ref{sec_nested}, we have $I(t)=J(t)+K(t)$ with
\begin{equation}\label{K_sum2}
K(t) = \frac{1}{t^{\mathcal{Q}^{\S_2}/2}} I_{\mathcal{Q}^{\S_2}-\mathcal{Q}^{\S_1}-1}[G_1](\sqrt{t}) + \frac{1}{t^{\mathcal{Q}^{\S_3}/2}} I_{\mathcal{Q}^{\S_3}-\mathcal{Q}^{\S_1}-1,\mathcal{Q}^{\S_3}-\mathcal{Q}^{\S_2}-1}[G_2](\sqrt{t}) .
\end{equation}
Since $D$ is $\S$-nilpotentizable, by Lemma \ref{lem_uniform_double_nilp}, 
the function $G_1$ defined by \eqref{defG1} is smooth. 
Using the extension of Lemma \ref{lem_uniform_double_nilp} stated in Remark \ref{rem_multiple_nilp}, $e_{\tau_1,\tau_2,\varepsilon}^{q_1,q_2,\sigma}(1,0,0)$ depends smoothly on $\tau_1,\tau_2,\varepsilon\in\R$, $q_1\in\S_1$, $q_2\in \mathcal{S}^{n-k_1-1}\cap (\S_2)_{\tau_1}^{q_1}$, $\sigma\in\mathcal{S}^{n-k_2-1}$, is even with respect to $\varepsilon$, and hence the function $G_2$ defined by \eqref{defG2} is smooth and is even with respect to $\varepsilon$.

\paragraph{Case $p>2$.}
When $p>2$, the procedure is continued: we apply the $(J+K)$-decomposition to $I_{\tau_1,\tau_2}(t_2)$. This iterative procedure is performed until we reach the stratum $\S_{p+1}=M\setminus\S$. We do not give any details.
We obtain finally, with the notations of Appendix \ref{app_integral}, 
\begin{equation}\label{I_multistrates}
I(t) = \frac{F_J(\sqrt{t})}{t^{\mathcal{Q}^{\S_1}/2}} + \sum_{i=2}^{p+1} \frac{1}{t^{\mathcal{Q}^{\S_i}/2}} I_{\mathcal{Q}^{\S_i}-\mathcal{Q}^{\S_1}-1,\ldots,\mathcal{Q}^{\S_i}-\mathcal{Q}^{\S_{i-1}}-1}[G_{i-1}](\sqrt{t})
\end{equation}
where all functions $G_i$ are smooth thanks to the $\S$-nilpotentizability assumption. 

The results stated in the theorem now readily follow from Proposition \ref{prop_equiv_iter} in Appendix \ref{app_integral}. 


\section{Local Weyl law in the non-nilpotentizable case}\label{sec_nonnilp}
In this section, we investigate the more general case where the horizontal distribution $D$ may fail to be $\S$-nilpotentizable, as in Example \ref{simpleexample_nonnilp} given in Section \ref{sec_nilpotentizability_doublenilp} ($X_1=\partial_1$, $X_2=(x_1^2+x_2^2)\,\partial_2$).
We first give in Section \ref{sec_nonnilp_analytic} a general result for analytic sR structures, proved in Sections \ref{sec_proof_thm_analyticsingularities} and \ref{sec_proof_lem_nonnilp_general}. In Section \ref{sec_nonnilp_flat} we give examples of local Weyl laws for non-analytic sR structures, exhibiting non-standard Weyl asymptotics.

\subsection{Analytic sR structures}\label{sec_nonnilp_analytic}

In this section, we assume that the sR structure $(M,D,g)$ is real analytic, meaning that the manifold $M$ and the vector fields $X_1,\ldots,X_m$ defining the sR structure are real analytic. 
We also assume that the measure $\mu$ is real analytic.
Actually, the following slightly weaker assumption is sufficient: given any $q\in M$, there exists a chart at $q$ in which the vector fields $X_1,\ldots,X_m$ and the density of $\mu$ are real analytic. 


Since $(M,D,g)$ is analytic, the singular set 
$\S$ and thus also the manifold $M$ are Whitney stratified by strata that are 
equisingular analytic submanifolds of $M$ (this is a consequence of subanalytic theory, see Appendix \ref{app_subanalytic}).

With the notations given at the beginning of Section \ref{sec_equisingular_stratified_nilp}, 
we consider the integers $\mathcal{Q}^s$ (maximal Hausdorff dimension of equisingular strata of $M$) and $m_s\in\{1,\ldots,n+1\}$ (maximal multiplicity of $\mathcal{Q}^s$). 
Let 
$p_{\max}+1$ be the maximal length of all chains of strata at $q$, over all $q\in M$. Note that $p_{\max}\in\{1,\ldots,n\}$.

%

\begin{theorem}\label{thm_analyticsingularities}
There exist $k\in\{0,\ldots,p_{\max}\}$ and a rational number $\gamma\in\Q$, 
depending only on the horizontal distribution $D=\mathrm{Span}(X)$ (and not on the metric $g$), satisfying 
$\gamma\geq\frac{1}{2}\mathcal{Q}^s$ and if $\gamma=\frac{1}{2}\mathcal{Q}^s$ then $k\geq m_s-1$,
and there exists $\ell\in\N^*$ 
such that, for any $f\in C^\infty(M)$, 
the function $t\mapsto\Tr(\mathcal{M}_f \, e^{t\triangle})= \int_M f(q)\, e(t,q,q)\, d\mu(q)$ 
has an asymptotic expansion as $t\rightarrow 0^+$ with respect to the scale of functions $t^{j/\ell-\gamma}\vert\ln t\vert^i$, with $i\in\{0,\ldots,k\}$ and $j\in\N$.
%
Moreover, there exists a nontrivial Borel measure $\nu$ on $M$ such that, for every $f\in C^0(M)$,
$$
\Tr(\mathcal{M}_f \, e^{t\triangle}) = \left( \int_{\mathcal{N}} f \, d\nu\right) \frac{\vert\ln t\vert^k}{t^\gamma} + \mathrm{o}\left( \frac{\vert\ln t\vert^k}{t^\gamma} \right)
$$
as $t\rightarrow 0^+$. The support of $\nu$ is an equisingular stratified submanifold $\mathcal{N}$ of $M$.  
As a consequence, the local Weyl measure exists and is $w_\triangle=\nu/\nu(\mathcal{N})$, and 
$$
N(\lambda) \underset{\ \lambda\rightarrow+\infty}{\sim}  \frac{\nu(\mathcal{N})}{\Gamma(\gamma+1)} \lambda^\gamma \ln^k\lambda .
$$
\end{theorem}

Theorem \ref{thm_analyticsingularities} shows that, in the real analytic case, the maximal complexity of the asymptotics of $N(\lambda)$ is a positive \emph{rational} power of $\lambda$ times an integer power of $\ln\lambda$. The asymptotic expansion of the local Weyl law cannot involve an irrational power of $t$ nor a term in $\ln\vert\ln t\vert$ for instance, as it may happen for non-analytic sR structures (see Section \ref{sec_nonnilp_flat}). 

The statement on $\gamma$ and $k$ implies that the dominating term $\vert\ln t\vert^k / t^\gamma$ of the small-time asymptotics of the local Weyl law is always greater than or equal to the dominating term $\vert\ln t\vert^{m_s-1} / t^{\mathcal{Q}^s/2}$ obtained in Theorem \ref{thm_multistrates_nilp} in the equisingular stratified nilpotentizable case.


In contrast to Theorem \ref{thm_multistrates_nilp}, as confirmed by the examples hereafter, the Hausdorff dimensions of equisingular strata do not seem to play a role and 
we do not have any information on the submanifold $\mathcal{N}$ on which the Weyl measure is concentrated.
The next examples and the proof of the theorem (done in Section \ref{sec_proof_thm_analyticsingularities}) 
show anyway that, in the absence of nilpotentizability, the situation may become extremely complicated. Actually, $\mathcal{N}$, $\gamma$ and $k$ depend not only on the nilpotentizations at points of $M$, but also on terms of higher order which do not seem to have any evident geometric interpretation.


\paragraph{Examples.}
In the examples given in Table \ref{table_examples}, we give the dominating term (up to a multiplying scalar) of the trace $\mathrm{Tr}(\mathcal{M}_f \, e^{t\triangle})$ as $t\rightarrow 0^+$ and of the Weyl counting function $N(\lambda)$ as $\lambda\rightarrow +\infty$. 
For each example, the computations are quite lengthy and are not reported here. They consist in following the algorithmic procedure described throughout the steps of the proof of the theorem given in Sections \ref{sec_proof_thm_analyticsingularities} and \ref{sec_proof_lem_nonnilp_general}.
\begin{table}[h]
\caption{Examples of trace asymptotics and Weyl laws}\label{table_examples}
$$
\arraycolsep=3pt 
\begin{array}{|c|c|cccc|}
\hline \rule{0pt}{3.5mm}
& \textrm{Case} & \mathrm{Tr}(\mathcal{M}_f \, e^{t\triangle}) 
& N(\lambda)
& \mathcal{N}  & \\ 
\hline
1 & \begin{array}{c} X_1=\partial_1,\quad X_2=(x_1^k-x_2) \, \partial_2 \\ k\geq 2 \end{array}
& \frac{\vert\ln t\vert}{t} 
& \lambda\ln\lambda 
& \{x_2=x_1^k\} 
& \\
\hline \rule{0pt}{5mm}
2 & X_1=\partial_1,\quad X_2=(x_1^{2p}+x_1x_2^k) \, \partial_2 
&  \frac{\vert\ln t\vert^2}{t}
& \lambda\ln^2\lambda 
& \{(0,0)\}
& \textrm{if}\ k=1 \\[1mm]
& p,k\in\N^* 
&  \frac{1}{t^{p+\frac{1}{2}-\frac{2p-1}{2k}}}
& \lambda^{p+\frac{1}{2}-\frac{2p-1}{2k}}
& \{(0,0)\}
& \textrm{if}\ k\geq 2 \\[2mm]
\hline \rule{0pt}{5mm}
3 & X_1=\partial_1,\quad X_2=(x_1^2+x_2^2)^k \, \partial_2
&  \frac{\vert\ln t\vert}{t}
& \lambda\ln\lambda
& \{(0,0)\}
& \textrm{if}\ k=1 \\[1mm]
& k\in\N^* &  \frac{1}{t^k}
& \lambda^k
& \{(0,0)\}
& \textrm{if}\ k\geq 2 \\[1mm]
\hline \rule{0pt}{5mm}
4 & X_1=\partial_1,\quad X_2=x_1^m(x_1^{2p}+x_2^{2k}) \, \partial_2 
&  \frac{\vert\ln t\vert}{t^{1+\frac{m}{2}}}
& \lambda^{1+\frac{m}{2}}\ln\lambda
& \{(0,0)\} 
& \textrm{if}\ p=k=1 \\[1mm]
& m\in\N, \quad p,k\in\N^* 
&  \frac{1}{t^{\frac{m+1}{2}+p-\frac{p}{2k}}}
& \lambda^{\frac{m+1}{2}+p-\frac{p}{2k}} 
& \{(0,0)\} 
& \{(0,0)\} \\[2mm]
\hline \rule{0pt}{4mm}
5 & X_1=\partial_1,\quad X_2=(x_1^2-x_2^3) \, \partial_2
& \frac{1}{t^{7/6}}
& \lambda^{7/6}
& \{(0,0)\} & \\[1mm]
\hline \rule{0pt}{4mm}
6 & X_1=\partial_1,\quad  X_2=(x_1^4+x_1^2x_2^2+x_2^{2k}) \, \partial_2 
& \frac{1}{t^2}
& \lambda^2
& \{(0,0)\} 
& \textrm{if}\ k\geq 3 \\[1mm]
\hline 
7 & X_1=\partial_1,\quad X_2=\partial_2+x_1 \, \partial_3+x_1^2 \, \partial_5 ,
& \frac{1}{t^{7/2}}
& \lambda^{7/2} 
& \R^5
& \textrm{if}\ k=2 \\[1mm]
& X_3=\partial_4+(x_1^k+x_2^k) \, \partial_5
&  \frac{\vert\ln t\vert}{t^{\frac{7}{2}}}
& \lambda^{7/2}\ln\lambda
& \{ x_1=x_2=0 \}
& \textrm{if}\ k=3 \\[2mm]
& k\geq 2 &  \frac{1}{t^{4-\frac{1}{k-1}}}
& \lambda^{4-\frac{1}{k-1}}
& \{ x_1=x_2=0 \}
& \textrm{if}\ k\geq 4 \\[2mm]
\hline
\end{array}
$$
\end{table}

Case 1 of Table \ref{table_examples} 
is the $k$-Baouendi-Grushin case with tangency point. For $k=2$, it has been studied in Section \ref{sec_Baouendi-Grushin} where it was already observed that the tangency point does not add any complexity to the trace asymptotics. 

Cases 2, 4 and 5 illustrate the genuine rationality of $\gamma$ in Theorem \ref{thm_analyticsingularities}, in contrast to Theorem \ref{thm_multistrates_nilp}. Note that, in Cases 2 and 4, the Hausdorff dimensions of the singular strata do not depend on $p$ and $k$. 
For instance in Case 2 we have $\S=\S_1\cup\S_2\cup\S_2'$ with $\S_1=\{x_1=x_2=0\}$, $\S_2=\{x_1=0, x_2\neq 0\}$ and $\S_2'=\{x_1\neq 0, x_1^{2p-1}+x_2^k=0\}$ (equisingular strata), and we have $\mathcal{Q}^{\S_1}=0$ and $\mathcal{Q}^{\S_2}=\mathcal{Q}^{\S_2'}=\Qeq=2$ for any $p,k\in\N^*$, while the dominating term in the small-time asymptotics of the Weyl law depends on $p$ and $k$.
Since the sR weights do not depend on $k$, this shows that, here, terms of higher order, not related to the Lie structure, have an impact on the Weyl law (contrarily to Case 6).

Case 3 is a generalization of Example \ref{simpleexample_nonnilp}, which was given in Section \ref{sec_nilpotentizability_doublenilp} as a simple example where nilpotentizability fails. We have $\S=\{x_1=x_2=0\}$ and $\mathcal{Q}^{\S}=0$ and $\Qeq=2$ for any $k\in\N^*$. Note that, although $M$ consists of only two strata and $\Qeq>\mathcal{Q}^{\S}$, the Weyl measure is concentrated on $\S$, in contrast to Theorem \ref{thm_onestratum} (see in particular Remark \ref{rem_concentration_onestratum}).

More generally, in contrast to Theorem \ref{thm_multistrates_nilp} where the support $\mathcal{N}_s$ of the Weyl measure is the closure of a union of strata of maximal Hausdorff dimension, in Cases 2, 3, 4 and 5 we have $\mathcal{N}=\{(0,0)\}$ which is of Hausdorff dimension $0$ and thus not maximal. We have $\mathcal{N}\subsetneq\mathcal{N}_s$ on these examples.


Case 7 corresponds to \cite[Example 6.5]{GhezziJean_NA2015}.
We have $\S=\{x_1=x_2=0\}$ and $\mathcal{Q}^{\S}=6$ and $\Qeq=7$ for any $k\geq 2$ (see also Example \ref{ex_GhezziJean} further). For $k=3$ a logarithm appears in the asymptotics of the Weyl law, in contrast to Theorem \ref{thm_multistrates_nilp} where the power of the logarithm was $m_s-1$ (here, $m_s=1$).

In view of those examples, one may wonder whether, given any rational number $\gamma\geq\frac{1}{2}\mathcal{Q}^s$ and any $k\in\{0,\ldots,n\}$, there exists a sR structure whose Weyl counting function's asymptotics is $\lambda^\gamma\ln^k\lambda$ up to a multiplying scalar. We leave this issue open.

%

\begin{remark}
The fact that $N(\lambda)/\lambda^\gamma\ln^k\lambda$ is bounded above and below by some positive constants follows from the Fefferman-Phong estimate \eqref{FeffermanPhong} recalled in the introduction. Indeed, it can be proved, by following (in a much simpler way) the developments done in Section \ref{sec_proof_lem_nonnilp_general}, that $\int_M \frac{1}{ \mu( \BsR(q,1/\sqrt{\lambda}) ) } \, d\mu(q)$ is of the order of $\lambda^\gamma\ln^k\lambda$ as $\lambda\rightarrow+\infty$. This is what is done in the recent preprint \cite{ChenChenLi_2022} for homogeneous sR structures (also dealing with Dirichlet boundary conditions), which we have discovered while finishing the present article. As already said in the introduction, here, 
our analysis not only leads to an equivalent, but also to a small-time expansion of the local Weyl law at any order.
\end{remark}

\subsection{Proof of Theorem \ref{thm_analyticsingularities}}\label{sec_proof_thm_analyticsingularities}

\subsubsection{A first motivating example} 
As a prelude, let us consider Example \ref{simpleexample_nonnilp}, which is a very simple example where the nilpotentizability property fails. Near $q_1=(0,0)\in M\simeq\R^2$, we have $X_1(x)=\partial_1$ and $X_2(x)=(x_1^2+x_2^2)\,\partial_2$ and thus $\S=\{(0,0)\}$ (one singular stratum). For any $\tau_1>0$, we have $(X_1)^{q_1}_{\tau_1}(x)=\partial_1$ and $(X_2)^{q_1}_{\tau_1}(x)=(x_1^2+\tau_1^4 x_2^2)\,\partial_2$, so that $\S^{q_1}_{\tau_1}=\{(0,0)\}$ (one singular stratum), $\mathcal{Q}^{\S}=0$ and $\mathcal{Q}^{M\setminus\S}=2$. 
As in \eqref{def_int_I} in Section \ref{sec_singular_prelim}, we set $I(t)=\int_{\R^2} f(q)\, e(t,q,q)\, dq$ with $f$ smooth, compactly supported in $B(0,1)$, such that $f(q_1)\neq 0$. We perform the $(J+K)$-decomposition (see Section \ref{sec_J+K}) and we obtain $I(t)=J(t)+K(t)$ with $J(t)$ that is a smooth function of $\sqrt{t}$, and $J(t)\sim f(q_1)\int_{\R^2}\widehat{e}^{q_1}(t,y,y)\, dy$ as $t\rightarrow 0^+$ (see Section \ref{sec_estimating_J}), while $K(t)$ is much more difficult to estimate. Here, since there is only one singular stratum (which is a singleton), we can follow the beginning of Section \ref{sec_proof_thm_onestratum}: according to \eqref{formulK} and \eqref{deffunctionG}, we have
$$
K(t) = \frac{1}{t} \int_{\sqrt{t}}^1 \tau_1 \int_{\mathcal{S}^1} f(\delta^{q_1}_{\tau_1}(\sigma))\, e^{q_1,\sigma}_{\tau_1,\sqrt{t}/\tau_1}(1,0,0)\, d\sigma\, d\tau_1 
$$
but $e^{q_1,\sigma}_{\tau_1,\tau_2}(1,0,0)$ does not depend smoothly on $(\tau_1,\tau_2)$ due to the lack of nilpotentizability at $q_1$. 

Any $\sigma\in\mathcal{S}^1$ can be written as $\sigma=(\sin\theta,\cos\theta)$ for $\theta\in[-\pi,\pi]$, and making the change of variable $x_1=\sin\theta+y_1$, $x_2=\cos\theta+y_2$ we find $(X_1)^{q_1,\sigma}_{\tau_1,\tau_2}(y)=\partial_1$ and $(X_2)^{q_1,\sigma}_{\tau_1,\tau_2}(y)=(\sin^2\theta+\tau_1^4\cos^2\theta+2\tau_2\sin\theta\, y_1+\tau_2^2y_1^2+2\tau_1^4\tau_2\cos\theta\, y_2+\tau_1^3\tau_2^2 y_2^2)\,\partial_2$. 

Far from $\theta\in\{0,\pi\}$, the sR structure generated by $((X_1)^{q_1,\sigma}_{\tau_1,\tau_2},(X_2)^{q_1,\sigma}_{\tau_1,\tau_2})$ is ``uniformly" Riemannian with respect to $(\tau_1,\tau_2)$ 
and thus in this region $e^{q_1,\sigma}_{\tau_1,\tau_2}(1,0,0)$ depends smoothly on $(\tau_1,\tau_2)$ and we can apply the arguments developed in Section \ref{sec_proof_thm_onestratum} (see in particular \eqref{thm_onestratum_cas3}): the corresponding contribution in $K(t)$ is equivalent to $\frac{\Cst}{t}$ as $t\rightarrow 0$.
Hence, the difficulty is concentrated near $\theta=0$ and $\theta=\pi$. Let us focus on $\theta=0$.

We have $\sigma_\theta\simeq(\theta,1)$ and thus $(X_2)^{q_1,\sigma_\theta}_{\tau_1,\sqrt{t}/\tau_1}(y) \simeq (\tau_1^4+\theta^2+2\frac{\theta\sqrt{t}}{\tau_1} y_1+\frac{t}{\tau_1^2}y_1^2+\cdots)\,\partial_2$: the above ``uniform" Riemannian property is lost. Here, there is a kind of ``competition" between the parameters $t$, $\tau_1$ and $\theta$, leading to an algebraic decomposition of $[0,1]^3$:
\begin{itemize}[leftmargin=*,parsep=0cm,itemsep=0cm,topsep=1mm]
\item The monomial $\tau_1^4$ dominates in the region of $[0,1]^3$ where $\tau_1^2>\vert\theta\vert$ and $\tau_1>t^{1/6}$. In this region, we have $(X_2)^{q_1,\sigma_\theta}_{\tau_1,\tau_2}(y) = \tau_1^4 \, U\big(t,\tau_1,\theta,\frac{\theta}{\tau_1^2}, \frac{t^{1/6}}{\tau_1},y_1\big)\,\partial_2  = \tau_1^4 Y_2(y)$ where $U$ is a non-vanishing smooth function and, setting $Y_1=(X_1)^{q_1,\sigma_\theta}_{\tau_1,\tau_2}$, we see that the sR structure generated by $(Y_1,Y_2)$ is Riemannian in this region and therefore the corresponding heat kernel is a smooth function of the variables inside the function $U$. The corresponding contribution in the integral $K(t)$ is then equivalent, as $t\rightarrow 0$, to
$$
\frac{\Cst}{t} \int_{t^{1/6}}^1 \tau_1 \int_{-\tau_1^2}^{\tau_1^2} \frac{1}{\tau_1^4}\, d\theta\, d\tau_1 \sim \Cst\frac{\vert\ln t\vert}{t} .
$$
\item The monomial $\frac{t}{\tau_1^2}$ dominates in the region of $[0,1]^3$ where $\tau_1<t^{1/6}$ and $\vert\theta\vert<\sqrt{t}/\tau_1$. In this region, we have $(X_2)^{q_1,\sigma_\theta}_{\tau_1,\tau_2}(y) = \frac{t}{\tau_1^2} Y_2(y)$ where $Y_2(y) = (y_1^2 + \cdots)\, \partial_{y_2}$. The sR structure generated by $(Y_1,Y_2)$ is a $2$-Baouendi-Grushin structure in this region. Then, similarly, the corresponding contribution in the integral $K(t)$ is equivalent, as $t\rightarrow 0$, to
$$
\frac{\Cst}{t} \int_{\sqrt{t}}^{t^{1/6}} \tau_1 \int_{-\sqrt{t}/\tau_1}^{\sqrt{t}/\tau_1} \frac{\tau_1^2}{t}\, d\theta\, d\tau_1 \sim \frac{\Cst}{t} .
$$
\item The monomial $\theta^2$ dominates in the region of $[0,1]^3$ where $\tau_1<\sqrt{\theta}$ and $\sqrt{t}<\tau_1\theta$. In this region, we have $(X_2)^{q_1,\sigma_\theta}_{\tau_1,\tau_2}(y) = \theta^2 Y_2(y)$, where $(Y_1,Y_2)$ is Riemannian. The corresponding contribution in the integral $K(t)$ is equivalent, as $t\rightarrow 0$, to
$$
\frac{\Cst}{t} \int_t^{t^{1/6}} \tau_1 \int_{\sqrt{t}/\tau_1}^1 \frac{d\theta}{\theta^2}\, d\tau_1
+ \frac{\Cst}{t} \int_{t^{1/6}}^1 \tau_1 \int_{\tau_1^2}^1 \frac{d\theta}{\theta^2}\, d\tau_1
\sim \Cst\frac{\vert\ln t\vert}{t} .
$$
\item The monomial $\frac{\sqrt{t}}{\tau_1}$ can never dominate.
\end{itemize}
We thus conclude that $I(t) \sim K(t)\sim \mathrm{Cst}\frac{\vert\ln t\vert}{t}$ as $t\rightarrow 0^+$.

\subsubsection{Preliminaries: difficulties due to the absence of nilpotentizability}
We follow the strategy developed in Section \ref{sec_equisingular_nilp} (case of one singular stratum) and in Section \ref{sec_equisingular_stratified_nilp} (case of multiple singular strata), but we have to adapt the arguments at all steps where the nilpotentizability assumption was used: the analysis performed in these sections remains valid except when we consider limits as $\tau_i\rightarrow 0$. 
Indeed, in the absence of nilpotentizability, these limits do not exist in general. This complicates significantly the analysis.

Here and in the sequel, we use the following convenient notation:
given any $k$-tuple $z=(z_1,\ldots,z_k)$ of elements of some set, for any $i\in\{1,\ldots,k\}$ we denote $z_{\leq i}=(z_1,\ldots,z_i)$ and $z_{<i}=(z_1,\ldots,z_{i-1})$, with the agreement that $z_{<1}=\emptyset$.

\paragraph{What has to be done.}
Like in Section \ref{sec_proof_thm_multistrates_nilp}, we consider $q\in M$ and an exhaustive (in terms of topological dimensions) chain $\mathscr{C}=(\S_1,\ldots,\S_{p},\S_{p+1})$ of strata at $q$, with $\S_{p+1}=M\setminus\S$. 
When $p=1$, we have $\S_2=M\setminus\S$ like in the case of only one singular stratum treated in Section \ref{sec_equisingular_nilp}.

When $p=1$, in the $(J+K)$-decomposition \eqref{I=J+K} written in Section \ref{sec_J+K}, the integral $J(t)$ is still expanded as \eqref{asympt_J} and the integral $K(t)$, which is performed on $\sqrt{t}\leq\tau_1\leq 1$, is still written as \eqref{formulK} in Section \ref{sec_proof_thm_onestratum} 
and involves the term $e^{q_1,\sigma}_{\tau_1,\frac{\sqrt{t}}{\tau_1}}(1,0,0)$, with $q_1\in\S_1$ and $\sigma\in\mathcal{S}^{n-k_1-1}$.

When $p\geq 2$, we still have \eqref{K_sum2} or the general formula \eqref{I_multistrates}, 
and we have to estimate the small-time asymptotics of nested integrals over $q_1\in\S_1$, over $q_j\in\mathcal{S}^{n-k_{j-1}-1}\cap(\S_j)_{\tau_{<j}}^{q_{<j}}$ 
for $j\in\{2,\ldots,p+1\}$, 
and possibly over $y\in\mathcal{B}^{n-k_p}$, of nested integrals with respect to $\tau_1,\ldots,\tau_j$ that are of the form
\begin{equation}\label{gen_form_int}
\frac{1}{t^{\ell/2}} \int_{\sqrt{t}}^1 \tau_1^{\ell-\ell_1} \int_{\frac{\sqrt{t}}{\tau_1}}^1 \tau_2^{\ell-\ell_2} \cdots \int_{\frac{\sqrt{t}}{\tau_1\cdots\tau_{j-1}}}^1 \tau_j^{\ell-\ell_j} \, f_{\tau_{\leq j},\frac{\sqrt{t}}{\tau_1\cdots\tau_j}}^{q_{\leq j},q_{j+1}}(y)\, e_{\tau_{\leq j},\frac{\sqrt{t}}{\tau_1\cdots\tau_j}}^{q_{\leq j},q_{j+1}}(1,y,y)\, d\tau_j\cdots d\tau_2\, d\tau_1
\end{equation}
with $\ell=\mathcal{Q}^{\S_{j+1}}$ and $\ell_i=\mathcal{Q}^{\S_i}+1$ for $i=1,\ldots,j$.


In the absence of nilpotentizability, such integrals \eqref{gen_form_int} are challenging to estimate in small time $t$, because not only we have to estimate, for any $j\in\N^*$, the (blowing-up) asymptotics of $e_{\tau_{\leq j}}^{q_{\leq j}}(1,y,y)$ as $\tau_1\cdots\tau_j\rightarrow 0$, but we also have to parametrize the manifold $\mathcal{S}^{n-k_{j-1}-1}\cap(\S_j)_{\tau_{<j}}^{q_{<j}}$ in an adequate way in order to estimate its asymptotics as $\tau_1\cdots\tau_j\rightarrow 0$, as discussed hereafter.

Before coming to that point, similarly as in Remark \ref{rem_prelim_JK} in Section \ref{sec_estimating_K}, we can already observe that, since 
the heat kernel function inside the nested integral \eqref{gen_form_int} is always greater than a positive constant (on the domain of integration), it follows that $I(t)$ is always greater than \eqref{I_multistrates}, i.e., than the trace asymptotics given by Theorem \ref{thm_multistrates_nilp}. Taking $f=1$, this already shows that $\gamma\geq\frac{1}{2}\mathcal{Q}^s$ and that if $\gamma=\frac{1}{2}\mathcal{Q}^s$ then $k\geq m_s-1$, in the statement of Theorem \ref{thm_analyticsingularities}.

\paragraph{Loss of smoothness of the heat kernel.}
The conclusion of Lemma \ref{lem_uniform_double_nilp} (and Remark \ref{rem_multiple_nilp}) in Section \ref{sec_doublenilp} fails if $D$ is not $\S$-nilpotentizable: given any $j\in\N^*$, it is not true anymore that $e_{\tau_{\leq j}}^{q_{\leq j}}(1,y,y)$ can be extended to a smooth function of $\tau_{\leq j}\in[-1,1]^j$, equal to $\widehat{e}^{\, q_{\leq j}}(1,y,y)$ (nilpotentized heat kernel) at $\tau_{\leq j}=(0,\ldots,0)$. 
Therefore, the functions $G$ defined by \eqref{deffunctionG}, $G_1$ defined by \eqref{defG1}, $G_2$ defined by \eqref{defG2}, and $G_j$ in \eqref{I_multistrates}, may fail to have a smooth extension, and thus Propositions \ref{prop_general_expansion} or \ref{prop_equiv_iter} in Appendix \ref{app_integral} cannot be applied directly to estimate the integrals \eqref{gen_form_int}.

We recall that, by definition, $e_{\tau_{\leq j}}^{q_{\leq j}}$ is the heat kernel generated by the sR Laplacian corresponding to the $m$-tuple of vector fields $X_{\tau_{\leq j}}^{q_{\leq j}} 
= \tau_1\cdots\tau_j ( \delta^{q_1}_{\tau_1} \circ\cdots\circ \delta^{q_j}_{\tau_j} )^* X$ (see \eqref{def_X_nilp_j} in Section \ref{sec_doublenilp}). In the absence of nilpotentizability, $X_{\tau_{\leq j}}^{q_{\leq j}}$ does not depend smoothly on its arguments.

Actually, as alluded above, we are going to show that, in the non-nilpotentizable case, $e_{\tau_{\leq j}}^{q_{\leq j}}(1,y,y)$ may blow up as $\tau_1\cdots\tau_j\rightarrow 0$, with an asymptotics that is given ``piecewise" as a fractional monomial in the $\tau_i$ (see Lemma \ref{lem_nonnilp_general} in the next section).
Knowing this asymptotics will enable us to estimate the integral \eqref{gen_form_int}.


\paragraph{Problem of smooth parametrization.}
The manifold $\mathcal{S}^{n-k_{j-1}-1}\cap(\S_j)_{\tau_{<j}}^{q_{<j}}$ is constructed by induction on $j\in\{2,\ldots,p+1\}$. Recall that, by definition,
$$
(\S_j)_{\tau_{<j}}^{q_{<j}} = (\S_j)_{\tau_{<j}}^{q_{<j}} = \big(\delta^{q_{j-1}}_{\tau_j}\big)^{-1} \big( (\S_j)_{\tau_1,\ldots,\tau_{j-2}}^{q_1,\ldots,q_{j-2}} \big) = \big( \delta^{q_1}_{\tau_1} \circ\cdots\circ \delta^{q_{j-1}}_{\tau_{j-1}} \big)^{-1} (\S_j)
$$
with $\delta^{q_i}_{\tau_i} = (\psi^{q_i})^{-1}\circ\delta_{\tau_i})$, where $\psi^{q_i}$ is a chart of privileged coordinates at $q_i$.

For $j=2$, we consider $X_{\tau_1,\tau_2}^{q_1,q_2}=(\delta^{q_2}_{\tau_2})^*X_{\tau_1}^{q_1}=(\delta^{q_1}_{\tau_1}\circ\delta^{q_2}_{\tau_2})^*X$ for $\tau_1,\tau_2\in(0,1]$, $q_1\in\S_1$ and $q_2\in\mathcal{S}^{n-k_1-1}\cap(\S_1)^{q_1}_{\tau_1}$. Extending the $m$-tuple $X_{\tau_1,\tau_2}^{q_1,q_2}$ 
at $\tau_1=0$ 
requires to find a parametrization of the manifold $\mathcal{S}^{n-k_1-1}\cap(\S_1)^{q_1}_{\tau_1}$ that can be extended at $\tau_1=0$ and that depends at least piecewise smoothly on $\tau_1\in[0,1]$ and $q_1\in\S_1$.
But, when nilpotentizability fails, it is not true anymore that 
$(\S_1)^{q_1}_{\tau_1}$ is diffeomorphic to $\widehat{\S_1}^{q_1}$, uniformly with respect to $\tau_1\in(0,1]$: in Example \ref{simpleexample_nonnilp} (given at the beginning of Section \ref{sec_nilpotentizability_doublenilp}), we have $\S_\tau^q=\{(0,0)\}$ for every $\tau>0$ but $\widehat{\S}^{\, q}=\{x_1=0\}$.
Let us give another example where, in contrast, we have $\widehat{\S}^{q_1}=\emptyset$ while $(\S)^{q_1}_\tau\neq\emptyset$ for $\tau>0$. 

\begin{example}\label{ex_GhezziJean}
Following \cite[Example 6.5]{GhezziJean_NA2015}, consider the sR case in $\R^5$ generated by 
$$
X_1=\partial_{x_1} , \qquad
X_2=\partial_{x_2}+x_1\,\partial_{x_3}+x_1^2\,\partial_{x_5}, \qquad
X_3=\partial_{x_4}+(x_1^k+x_2^k)\,\partial_{x_5} ,
$$
for $k\geq 2$. We have $\S=\{x_1=x_2=0\}$ (equisingular), 
$\mathcal{Q}^{\S}=6$ and $\Qeq=7$.
Taking $q_1=(0,0,0,0,0)$, for every $\tau\neq 0$ we have 
$(X_1)^{q_1}_\tau = X_1$, $(X_2)^{q_1}_\tau = X_2$ and $(X_3)^{q_1}_\tau = \partial_{x_4}+\tau^{k-2}(x_1^k+x_2^k)\,\partial_{x_5}$, 
hence $(\S)^{q_1}_\tau=\S=\{x_1=x_2=0\}$. However, for $\tau=0$ we find $\widehat{X_1}^{q_1} = X_1$, $\widehat{X_2}^{q_1} = X_2$ and $\widehat{X_3}^{q_1} = \partial_{x_4}$ if $k\geq 3$ (while $\widehat{X_3}^{q_1} = \partial_{x_4}+(x_1^2+x_2^2)\,\partial_{x_5}$ if $k=2$), hence $\widehat{\S}^{q_1}=\emptyset$ if $k\geq 3$, meaning that the nilpotentization at $q_1$ is equiregular.
\end{example}


Moreover, once an adequate parametrization of $\mathcal{S}^{n-k_1-1}\cap(\S_1)^{q_1}_{\tau_1}$ in function of $\tau_1\in[0,1]$ (giving a sense to the limit at $\tau_1=0$) and $q_1\in\S_1$ has been found, in order to define $X^{q_1,q_2}_{\tau_1,\tau_2}$ and compute an expansion of it for small $\tau_1$ and $\tau_2$,  we have to define $\delta^{q_2}_{\tau_2} = (\psi^{q_2})^{-1}\circ\delta_{\tau_2}$, which requires to construct a chart of privileged coordinates $\psi^{q_2}$ at any $q_2\in\mathcal{S}^{n-k_1-1}\cap(\S_1)^{q_1}_{\tau_1}$ that does not degenerate as $\tau_1\rightarrow 0$, i.e., such that its Jacobian remains positive and bounded. 

Finally, once all these problems have been solved, one has to estimate, by induction, the asymptotics of the heat kernel $e_{\tau_{\leq j}}^{q_{\leq j}}(1,y,y)$ as $\tau_1\cdots\tau_j\rightarrow 0$ and infer the small-time asymptotics of the integrals \eqref{gen_form_int}.

\subsubsection{The key lemma}
The lemma hereafter is the key result to establish Theorem \ref{thm_analyticsingularities}.
A mapping $F$ defined on a smooth finite-dimensional manifold $P$ is said to be piecewise smooth if $P$ is Whitney stratified and $F$ is smooth on every stratum of maximal dimension (up to the boundary), i.e., $P$ is the closure of the finite union of disjoint open smooth submanifolds of $P$ and $F$ is has a smooth extension on an open set containing the closure of every such submanifold.

\begin{lemma}\label{lem_nonnilp_general}
We set $N_1=\S_1$, $\theta_1=q_1$, $\tau=(\tau_1,\ldots,\tau_p)$ and $\theta=(\theta_1,\ldots,\theta_p)$. For every $j\in\{2,\ldots,p+1\}$, there exists a piecewise smooth parametrization of $\mathcal{S}^{n-k_{j-1}-1}\cap(\S_j)_{\tau_{<j}}^{q_{<j}}$ given as follows:
there exists a smooth compact manifold $N_j$ of dimension $k_j-k_{j-1}-1$ and a piecewise smooth mapping $(\tau_{<j},\theta_{<j},\theta_j)\mapsto q_j(\tau_{<j},\theta_{<j},\theta_j)$ on $[0,1]^{j-1}\times N_{<j}\times N_j$ 
such that 
$$
q_j(\tau_{<j},\theta_{<j},N_j) = \mathcal{S}^{n-k_{j-1}-1}\cap(\S_j)_{\tau_{<j}}^{q_{<j}}
\qquad \forall \tau_{<j} \in(0,1]^{j-1} \qquad\forall \theta_{<j} \in N_{<j} .
$$
For every $j\in\{1,\ldots,p\}$, we define
\begin{equation}\label{def_Dj}
\mathcal{D}_{j+1} = \{ (t,\tau_{\leq j})\in(0,1]^{j+1} \ \mid\ \sqrt{t}\leq \tau_1\cdots\tau_j\leq 1 \} .
\end{equation}
Denoting for short $q_j = q_j(\tau_{<j},\theta_{\leq j})$, we consider the function 
\begin{equation}\label{function_e}
e_{\tau_{\leq j},\frac{\sqrt{t}}{\tau_1\cdots\tau_j}}^{q_{\leq j},q_{j+1}}(1,y,y)
\end{equation}
of $(t,\tau_{\leq j})\in\mathcal{D}_{j+1}$, of $\theta_{\leq j+1}\in N_{\leq j+1}$ and of $y\in W$ where $W$ is a relatively compact open neighborhood of $0$ in $\R^n$. 
The set $\mathcal{D}_{j+1}\times N_{\leq j+1} \times W$ is the closure of the finite union of disjoint open smooth submanifolds, the projection onto $(0,1]^{j+1}$ of each of them having the (cylindrical) form
\begin{eqnarray*}
a_0(\theta_{\leq j+1},y) &\!\!\! <\ t\ < &\!\!\! b_0(\theta_{\leq j+1},y) , \\
a_i(t,\tau_{<i},\theta_{\leq j+1},y) &\!\! <\, \tau_i\, < &\!\! b_i(t,\tau_{<i},\theta_{\leq j+1},y) , \qquad i=1,\ldots,j,
\end{eqnarray*}
for some smooth functions $a_i$ and $b_i$ such that $0\leq a_i(\cdot)<b_i(\cdot)\leq 1$ for every $i\in\{0,\ldots,j\}$, with either $a_i(\cdot)>0$ or $a_i(\cdot)\equiv 0$, having a Puiseux expansion with respect to $t$ near $t=0$ at any order with coefficients that are smooth functions of $(\tau_{<i},\theta_{\leq j+1},y)$, and in each such submanifold the function \eqref{function_e} can be written as
\begin{equation}\label{function_e_cell}
t^{-\ell'/2} \prod_{i=1}^j \tau_i^{\alpha_i} \, 
F \left( t^{1/\ell}, \left( \frac{a_i(t,\tau_{<i},\theta_{\leq j+1},y)}{\tau_i} \right)^{1/\ell_i}_{1\leq i\leq j}, \left( \frac{\tau_i}{b_i(t,\tau_{<i},\theta_{\leq j+1},y)} \right)^{1/\ell_i}_{1\leq i\leq j} , \theta_{\leq j+1}, y \right) 
\end{equation}
for some $\ell'\in\N$, $\alpha_1,\ldots,\alpha_j\in\Q$, $\ell,\ell_1,\ldots,\ell_j\in\N^*$, and some positive $C^\infty$ function $F$ defined on an open set containing $[0,1]^{1+2j}\times N_{\leq j+1}\times W$. 
\end{lemma}

What is important in this lemma is that, in each submanifold of the partition, the function $F$ is smooth up to the boundary of the submanifold and is bounded below and above by a positive constant. 
Note that, when $t\rightarrow 0$ and $\tau_{\leq j}\rightarrow 0$, the limit value of $F(\star)$ in \eqref{function_e_cell} depends on those of the ratios ${a_i(t,\tau_{<i},\theta_{\leq j+1},y)}/{\tau_i}$ and ${\tau_i}/{b_i(t,\tau_{<i},\theta_{\leq j+1},y)}$. Hence, the term $F(\star)$ in \eqref{function_e_cell} is \emph{not} a piecewise smooth function of $(t,\tau_{\leq j},\theta_{\leq j+1},y)$ in the closure of the submanifold.
In some sense, $F(\star)$ plays the role of an angle in polar coordinates, and for each fixed ``angle" the asymptotics of $e_{\tau_{\leq j},\frac{\sqrt{t}}{\tau_1\cdots\tau_j}}^{q_{\leq j},q_{j+1}}(1,y,y)$ as $t\rightarrow 0$ and $\tau_{\leq j}\rightarrow 0$ is the fractional monomial $t^{-\ell'/2} \prod_{i=1}^j \tau_i^{\alpha_i}$ multiplied by a positive constant. 
Lemma \ref{lem_nonnilp_general} thus provides a 
kind of
desingularization of the heat kernel. 

Lemma \ref{lem_nonnilp_general} is proved in Section \ref{sec_proof_lem_nonnilp_general}, where we establish more precise results, based on subanalytic geometry and in particular on subanalytic cell preparation theorems (see Appendix \ref{app_subanalytic}), showing that any subanalytic function can be written as a locally fractional normal crossings form in adequate cells partitioning the manifold, i.e., in each cell, as the product of a fractional monomial with a unit analytic function.
The prepared form \eqref{function_e_cell} is a bit technical (anyway, this is the nature of things) but is adequately devised to infer the small-time asymptotics of the integrals \eqref{gen_form_int}, as shown hereafter.

\subsubsection{End of the proof of Theorem \ref{thm_analyticsingularities}}\label{sec_end_proof_thm_analyticsingularities}
Thanks to Lemma \ref{lem_nonnilp_general}, we can establish Theorem \ref{thm_analyticsingularities}.

For $p=1$, it follows from Lemma \ref{lem_nonnilp_general} that $K(t)$ (given by \eqref{formulK}) is a finite sum of integrals over submanifolds of $\S_1\times\mathcal{S}^{n-k_1-1}$ of integrals of the form
$$
\frac{1}{t^{(\Qeq+\ell')/2}} \int_{a_1(t,q_1,\sigma)}^{b_1(t,q_1,\sigma)} \tau_1^{\beta_1} F\left( t^{1/\ell}, \left( \frac{a_1(t,q_1,\sigma)}{\tau_1} \right)^{1/\ell_1}, \left( \frac{\tau_1}{b_1(t,q_1,\sigma)} \right)^{1/\ell_1}, q_1, \sigma \right) d\tau_1
$$
for some $\ell'\in\N$, $\beta_1\in\Q$ and some smooth function $F$. 
Using that $a_1$ and $b_1$ have Puiseux expansions with respect to $t$ at $t=0$, with coefficients that are smooth functions of $q_1$ and $\sigma$, by an easy adaptation of the proof of Proposition \ref{prop_general_expansion} in Appendix \ref{app_integral}, such integrals have an infinite-order asymptotic expansion as $t\rightarrow 0$ of the form
$$
\frac{1}{t^\gamma} \sum_{j=1}^{+\infty} \left( c_j(q_1,\sigma) t^{j/\ell}\vert\ln t\vert + d_j(q_1,\sigma)t^{j/\ell} \right) + \mathrm{O}(t^\infty)
$$
for some $\gamma\in\Q$ and some smooth functions $c_j$ and $d_j$. 
The result follows.

For $p\geq 2$, we have to estimate integrals \eqref{gen_form_int}, which are finite sums of integrals of the form
$$
\frac{1}{t^{(\Qeq+\ell')/2}} \int_{\sqrt{t}}^1 \tau_1^{\beta_1} \int_{\frac{\sqrt{t}}{\tau_1}}^1 \tau_2^{\beta_2} \cdots \int_{\frac{\sqrt{t}}{\tau_1\cdots\tau_{j-1}}}^1 \tau_j^{\beta_j} F(\star)\, d\tau_j\cdots d\tau_2\, d\tau_1
$$
for some $\ell'\in\N$, $\beta_i\in\Q$, where $F(\star)$ is like in \eqref{function_e_cell}.
Integrating iteratively as in the proof of Proposition \ref{prop_equiv_iter} in Appendix \ref{sec_nested}, we obtain the result. A new occurence of $\vert\ln t\vert$ may (only) occur each time we integrate with respect to a parameter $\tau_i$, whence $k\leq p_{\max}$ in the theorem. 

\begin{remark}
The above argument of proof is close to that done in \cite[Th\'eor\`eme 1]{LionRolin_AIF1998}, in \cite[Theorem 1.3 and Section 8]{CluckersMiller_DMJ2011} or in \cite[Corollary 6.4]{Parusinski_2001}, where it is shown that parametric integrals of globally subanalytic functions are log-analytic functions of order at most $1$ (see \cite{LionRolin_AIF1997}), meaning that they can always be written, piecewise, as sums of products of globally subanalytic functions and their logarithms and thus can be expanded exactly like in the theorem. 
%
\end{remark}

\subsection{Proof of Lemma \ref{lem_nonnilp_general}}\label{sec_proof_lem_nonnilp_general}
As alluded above, the proof of Lemma \ref{lem_nonnilp_general} uses some results of subanalytic geometry that are recalled in Appendix \ref{app_subanalytic}. The reader is thus invited to read this appendix before going through the details of the present section.
%
%
Since the proof of Lemma \ref{lem_nonnilp_general} also provides an algorithmic way to compute the trace asymptotics, we also give, along the steps of the proof, a number of examples. 
As a remark, we will also see in Section \ref{sec_further_comments_exp_estimates} that Lemma \ref{lem_nonnilp_general} allows us to recover the exponential estimates \eqref{exp_estimates} (see Appendix \ref{app_sR_kernel}) for the sR heat kernel.

The strategy is in three steps. 
We first show how to parametrize the manifold $\mathcal{S}^{n-k_{j-1}-1}\cap(\S_j)_{\tau_{<j}}^{q_{<j}}$, for $j\in\{2,\ldots,p+1\}$, in a subanalytic way.
Second, we prove that there exists a ``uniform" chart of privileged coordinates at any point 
$q_j\in\mathcal{S}^{n-k_{j-1}-1}\cap(\S_j)_{\tau_{<j}}^{q_{<j}}$, 
``uniform" in the sense that its Jacobian does not tend to $0$ nor blow up as 
the $\tau_i$ converge to $0$, 
in which the $m$-tuple of vector fields $X_{\tau_{\leq j}}^{q_{\leq j}}$ 
depends subanalytically on its parameters: in other words, we perform a desingularization of $X_{\tau_{\leq j}}^{q_{\leq j}}$, in each cell of a subanalytic cell decomposition. 
Finally, in each cell, we infer the asymptotics of the heat kernel $e_{\tau_{\leq j}}^{q_{\leq j}}(1,y,y)$ as $\tau_1\cdots\tau_j\rightarrow 0$.



\subsubsection{Subanalytic parametrization of $\mathcal{S}^{n-k_{j-1}-1}\cap(\S_j)_{\tau_{<j}}^{q_{<j}}$} 
Since $\S_j$ is a globally subanalytic and analytic submanifold of $M$ (with $\dim\S_j=k_j$) whose closure contains the lower dimensional strata $\S_1,\ldots,\S_{j-1}$, for all $(\tau_1,\ldots,\tau_{j-1})\in(0,1]^{j-1}$ the set 
$\mathcal{S}^{n-k_{j-1}-1}\cap (\S_j)_{\tau_{<j}}^{q_{<j}}$
is a globally subanalytic and analytic submanifold of $\mathcal{S}^{n-k_{j-1}-1}$ of topological dimension $k_j-k_{j-1}-1$.

\begin{lemma}\label{lem_subanalytic_param}
We set $N_1=\S_1$, $\theta_1=q_1$, $\tau=(\tau_1,\ldots,\tau_p)$ and $\theta=(\theta_1,\ldots,\theta_p)$. 
For every $j\in\{1,\ldots,p+1\}$, there exist a real analytic compact manifold $N_j$ of dimension $k_j-k_{j-1}-1$ and a bounded globally subanalytic mapping $(\tau_{<j},\theta_{<j},\theta_j)\mapsto q_j(\tau_{<j},\theta_{<j},\theta_j)$ on $[0,1]^{j-1}\times N_{<j}\times N_j$ 
such that
$$
q_j(\tau_{<j},\theta_{<j},N_j) = \mathcal{S}^{n-k_{j-1}-1}\cap(\S_j)_{\tau_{<j}}^{q_{<j}}
\qquad \forall \tau_{<j}\in(0,1]^{j-1} \qquad \forall \theta_{<j}\in N_{<j} .
$$
\end{lemma}

\begin{proof}
For $j=1$ there is nothing to prove.
For $j=2$, let us prove that
there exist a real analytic compact manifold $N_2$ of dimension $k_2-k_1-1$ and a bounded globally subanalytic mapping $(\tau_1,q_1,\theta_2)\mapsto q_2(\tau_1,q_1,\theta_2)$ on $[0,1]\times\S_1\times N_2$ such that $q_2(\tau_1,q_1,N_2) = \mathcal{S}^{n-k_1-1}\cap(\S_2)_{\tau_1}^{q_1}$ for all $\tau_1\in(0,1]$ and $q_1\in\S_1$.
We consider the globally subanalytic compact subset $X$ of $[0,1]\times\S_1\times\R^n$ defined as the closure of the set of all $(\tau_1,q_1,x)\in(0,1]\times\S_1\times\R^n$ such that $x\in\mathcal{S}^{n-k_1-1}\cap(\S_2)_{\tau_1}^{q_1}$.
Given any $\tau\in(0,1]$ and $q_1\in\S_1$, the fiber $X_{\tau_1,q_1} = \{ x\in\R^n\ \mid\ (\tau_1,q_1,x)\in X\}$ is exactly $\mathcal{S}^{n-k_1-1}\cap(\S_2)_{\tau_1}^{q_1}$. 
Applying Lemma \ref{lem_parametric_Hironaka} in Appendix \ref{app_subanalytic_useful} with $P=[0,1]\times\S_1$, $p=1+k_1$ and $k=k_2-k_1-1$, the claim follows.
A straightforward induction, using the same argument, gives the lemma.
\end{proof}

Taking $j=2$ to simplify, it is interesting to note that the mapping $(\tau_1,q_1,\theta_2)\mapsto q_2(\tau_1,q_1,\theta_2)$ 
may fail to be analytic in $\tau_1$ at $\tau_1=0$. Anyway, it follows from the subanalytic preparation theorem recalled in Appendix \ref{app_subanalytic} that there exists a subanalytic cell decomposition of $[0,1]\times \S_1\times N_2$ 
such that, in each cell of maximal dimension, we can write $q_2(\tau_1,q_1,\theta_2) = \sum_{i\in\N} a_i(q_1,\theta_2) \tau_1^{i/\ell_1}$ for some $\ell_1\in\N^*$, i.e., $q_2(\cdot,q_1,\theta_2)$ can be expanded as a convergent Puiseux series in $\tau_1$ with analytic coefficients (see also \cite{Pawlucki_1984}).
%
Here, a nontrivial integer $\ell_1$ comes into the picture and cannot in general be avoided, as shown in Example \ref{ex_BG} below.
This is in contrast with the nilpotentizable case where the dependence in $\tau_1$ was smooth at $\tau_1=0$.
Note also that, due to analyticity, this is a rational power of $\tau_1$, and not, for instance, a flat term in $\tau_1$ as this could be in non-analytic non-nilpotentizable cases (see Section \ref{sec_nonnilp_flat}).

\begin{example}\label{ex_BG}
Consider the Baouendi-Grushin case with a tangency point (see Section \ref{sec_Baouendi-Grushin}), whose local model in privileged coordinates $(x_1,x_2)$ at $q_1=(0,0)\in\R^2$ is given by $X_1=\partial_{x_1}$ and $X_2=(x_1^2-x_2)\, \partial_{x_2}$. We have $\S=\S_1\cup \S_2$ with $\S_1=\{q_1\}$ and $\S_2=\{x_2=x_1^2, x_1\neq 0\}$ (equisingular smooth strata), and the sR weights along these strata are $w_1^{q_1}(D)=1$ and $w_2^{q_1}(D)=3$, $w_1^{\S_2}(D)=1$ and $w_2^{\S_2}(D)=2$, while outside of $\S$ the sR structure is Riemannian. 
For every $\tau_1\neq 0$, we have 
$(X_1)^{q_1}_{\tau_1}=\partial_{x_1}$ and $(X_2)^{q_1}_{\tau_1}=(x_1^2-\tau_1 x_2)\, \partial_{x_2}$,
hence $\S^{q_1}_{\tau_1}=(\S_1)^{q_1}_{\tau_1}\cup (\S_2)^{q_1}_{\tau_1}$ with $(\S_1)^{q_1}_{\tau_1}=\{(0,0)\}$ and $(\S_2)^{q_1}_{\tau_1}=\{x_1^2=\tau_1 x_2, \, x_1\neq 0\}$. For $\tau_1=0$, we have $\widehat{X_1}^{q_1}=\partial_{x_1}$ and $\widehat{X_2}^{q_1}=x_1^2\, \partial_{x_2}$ and thus $\widehat{\S}^{q_1} = \{ x_1=0 \}$ (equisingular submanifold). 

Here, $\mathcal{S}^{n-k_1-1}\cap(\S_2)_{\tau_1}^{q_1}$ is reduced to the single point $q_2(\tau_1) = (x_1(\tau_1),x_2(\tau_1))$, and an obvious argument shows that $x_1(\cdot)$ and $x_2(\cdot)$ are analytic functions of $\sqrt{\tau_1}$, near $\tau_1=0$ and for $\tau_1\geq 0$, and we have $x_1(\tau_1)\sim\sqrt{\tau_1}$ and $x_2(\tau_1)\sim 1$ as $\tau_1\rightarrow 0^+$. Hence, here, $\ell_1=2$.
\end{example}

More generally, according to Lemma \ref{lem_subanalytic_param}, the picture is the following: 
given any fixed $\theta_{\leq j}\in N_{\leq j}$, the curve $\tau_{<j}\mapsto q_j(\tau_{<j},\theta_{\leq j})$ on $\mathcal{S}^{n-k_{j-1}-1}\cap(\S_j)_{\tau_{<j}}^{q_{<j}}$ (for $\tau_1\cdots\tau_{j-1}>0$) has an extension at $\tau_1\cdots\tau_{j-1}=0$, which is globally subanalytic.
In particular, there exists a subanalytic cell decomposition of $[0,1]^{j-1}\times N_{\leq j}$ such that, in each cell of maximal dimension, we can write $q_j(\tau_{<j},\theta_{\leq j})$ as the product of $\prod_{i=1}^{j-1} \tau_i^{\alpha_i}$ with a unit function (i.e., not vanishing), for some $\alpha_i\in\Q$.

\subsubsection{Desingularization of $X_{\tau_{\leq j}}^{q_{\leq j}}$ in subanalytic cells}

In this section, we are going to define $X_{\tau_{\leq j}}^{q_{\leq j}}$ by induction on $j\in\{2,\ldots,p+1\}$ and show how it can be desingularized in each cell of a subanalytic cell decomposition.

\paragraph{Preliminary remark.}
In particular, we are going to construct, by induction on $j$, a system of privileged coordinates at 
$q_j=q_j(\tau_{<j},\theta_{\leq j})$, depending subanalytically on $(\tau_{<j},\theta_{\leq j})$ and not degenerating as $\tau_1\cdots\tau_j\rightarrow 0$.
Starting with $j=2$, the first (naive) guess is to take exponential coordinates of the first or second kind (see Appendix \ref{app_privileged}) with an adapted frame consisting of iterated Lie brackets of $X^{q_1}_{\tau_1}$, but this may fail because the resulting system of coordinates may degenerate as $\tau_1\rightarrow 0$, as shown in the following example.

\begin{example}\label{ex_BG2}
Consider again Example \ref{ex_BG} (Baouendi-Grushin case with a tangency point), with $q_2=q_2(\tau_1)=(\sqrt{\tau_1},1)$ to simplify. We have $[(X_1)^{q_1}_{\tau_1},(X_2)^{q_1}_{\tau_1}](x)=2x_1\,\partial_{x_2}$, whose value at $q_2$ is $2\sqrt{\tau_1}\,\partial_{x_2}$. For every $\tau_1>0$, the frame $((X_1)^{q_1}_{\tau_1}),[(X_1)^{q_1}_{\tau_1},(X_2)^{q_1}_{\tau_1}])$ is adapted to the sR flag of $D^{q_1}_{\tau_1}$ at $q_2$ but the corresponding change of privileged exponential coordinates 
$$
x = \exp(y_1 (X_1)^{q_1}_{\tau_1}) \circ \exp(y_2 [(X_1)^{q_1}_{\tau_1},(X_2)^{q_1}_{\tau_1}]) (q_2)
$$
degenerates when $\tau_1\rightarrow 0$ in the sense that its Jacobian tends to $0$. 

Here, instead, we can choose another adapted frame yielding a system of privileged coordinates at $q_2$ that depends smoothly on $\sqrt{\tau_1}$ as $\tau_1\rightarrow 0$: we simply take the adapted frame $(\partial_{x_1},\partial_{x_2})$, noticing that $\partial_{x_2}=\frac{1}{2\sqrt{\tau_1}} [(X_1)^{q_1}_{\tau_1},(X_2)^{q_1}_{\tau_1}](q_2(\tau))$ for $\tau_1>0$. The resulting privileged coordinates are then $y_1=x_1-\sqrt{\tau_1}$, $y_2=x_2-1$, and we have $(X_1)^{q_1}_{\tau_1}(y)=\partial_{y_1}$ and $(X_2)^{q_1}_{\tau_1}(y) = (2\sqrt{\tau_1} y_1+y_1^2-\tau_1 y_2)\, \partial_{y_2}$.
For $\tau_1>0$, the sR structure generated by $X^{q_1}_{\tau_1}$ is of Baouendi-Grushin type, hence the sR weights at $(0,0)$ are $w_1=1$, $w_2=2$ and thus 
$(X_1)^{q_1,q_2}_{\tau_1,\tau_2}(y) = \partial_{y_1}$ and $(X_2)^{q_1,q_2}_{\tau_1,\tau_2}(y) = (2\sqrt{\tau_1} y_1+\tau_2 y_1^2-\tau_1\tau_2 y_2) \, \partial_{y_2}$.
We observe that, in the privileged coordinates $y$, $X^{q_1,q_2}_{\tau_1,\tau_2}$ depends analytically on $\sqrt{\tau_1}$ and on $\tau_2$.
\end{example}

\paragraph{Volume of $X_{\tau_{\leq j}}^{q_{\leq j}}$.}
As the $m$-tuple $X_{\tau_{\leq j}}^{q_{\leq j}}$ will be constructed by induction, we will consider its ``volume" $\mathtt{V}_{\tau_{\leq j}}^{q_{\leq j}}$, defined as follows.

In Appendix \ref{sec_uniformballbox}, it is recalled how to associate to any $m$-tuple $Y=(Y_1,\ldots,Y_m)$ of vector fields on $\R^n$ endowed with its Lebesgue measure $m$, the function $\mathtt{v}_m^{x,\rho}(Y_1,\ldots,Y_m)$ of $x\in\R^n$ and of $\rho>0$, defined by \eqref{def_v}, that is bounded above and below, up to scaling, by the volume $m(\BsR(x,\rho))$ of the sR ball of center $x$ and radius $\rho$, uniformly with respect to $x$ in a compact subset of $\R^n$ and to $\rho\in(0,1]$ (see \eqref{mu_ball_eps}). Considering the sR heat kernel $e$ associated with the sR Laplacian corresponding to $Y$, the latter fact, combined with the exponential estimate \eqref{exp_estimates} recalled in Appendix \ref{app_sR_kernel}, implies that $\mathtt{v}_m^{x,\sqrt{\rho}}(X_1,\ldots,X_m)\, e(\rho,x,x)$ is bounded above and below by positive constants on any compact, uniformly with respect to $\rho\in(0,1]$. Here, $\rho$ plays the role of a small parameter. 

Let us apply this fact to $Y=X_{\tau_{\leq j}}^{q_{\leq j}}$ (constructed by induction in what follows) with $x=0$ and $\rho=1$.
Following Appendix \ref{sec_uniformballbox}, given any ordered set $I=(i_1,\ldots,i_p)$ of $p$ indices taken in $\{1,\ldots,m\}$, we consider the vector field $(X_{\tau_{\leq j}}^{q_{\leq j}})_I$ defined as the Lie bracket of vector fields $(X_{i_j})_{\tau_{\leq j}}^{q_{\leq j}}$ of length $p=\vert I\vert$ according to $I$. 
Finally, we define
\begin{equation}\label{def_V_iter}
\mathtt{V}_{\tau_{\leq j}}^{q_{\leq j}} = \mathtt{v}^{0,1}_m(X_{\tau_{\leq j}}^{q_{\leq j}})
\end{equation}
for all $\tau_1,\ldots,\tau_j\in(0,1]$, $q_1\in\S_1$ and $q_i\in\mathcal{S}^{n-k_{i-1}-1}\cap(\S_i)_{\tau_{<i}}^{q_{<i}}$, for $i\in\{2,\ldots,j\}$.
Now, in \eqref{def_V_iter}, each $\tau_i$ plays the role of a small parameter.


\begin{lemma}\label{lem_expansion_epstau}
For every $j\in\{1,\ldots,p+1\}$, there exist relatively compact open neighborhoods $V_j$ and $W_j$ of $0$ in $\R^n$ and a chart $\psi^{q_j}:V_j\rightarrow W_j$ of privileged coordinates at $q_j=q_j(\tau_{<j},\theta_{\leq j})$, depending subanalytically on $(\tau_{<j},\theta_{\leq j})\in[0,1]^{j-1}\times N_{\leq j}$, whose Jacobian is uniformly bounded above and below by positive constants on $[0,1]^{j-1}\times N_{\leq j}\times V_j$.
In the chart $\psi^{q_j}$, the function $\mathtt{V}_{\tau_{\leq j}}^{q_{\leq j}}$ and all coefficients of the $m$-tuple of vector fields $X_{\tau_{\leq j}}^{q_{\leq j}}$ 
have an extension at $\tau_1\cdots\tau_j=0$, and are bounded globally subanalytic functions of $(\tau_{\leq j},\theta_{\leq j},y)\in[0,1]^j\times N_{\leq j}\times W_j$.
\end{lemma}

By subanalyticity, there exists a subanalytic cell decomposition of $[0,1]^j\times N_{\leq j}\times W_j$ such that, in each cell, the $m$-tuple $X_{\tau_{\leq j}}^{q_{\leq j}}$ depends analytically on $(\tau_{\leq j},\theta_{\leq j},y)$. 

\begin{proof}
The proof is done by induction on $j$. There is nothing to prove for $j=1$.
Let us assume that the result is true until the step $j-1$ and let us establish it at the step $j$.

Since $M$ is compact, let $\mathcal{Q}_\mathrm{max}$ be the maximum of $\mathcal{Q}^M(q)$ over all $q\in M$.
Using the notations of Appendix \ref{sec_uniformballbox}, we consider the finite set $\mathscr{X}_{\tau_{<j}}^{q_{<j}}$ of all $n$-tuples 
$$
\mathtt{X}_{\tau_{<j}}^{q_{<j}}=((X_{\tau_{<j}}^{q_{<j}})_{I_1}, \ldots, (X_{\tau_{<j}}^{q_{<j}})_{I_n})
$$
such that $\vert\mathtt{X}_{\tau_{<j}}^{q_{<j}}\vert = \sum_{i=1}^n \vert I_i\vert \leq \mathcal{Q}_\mathrm{max}$. 

Using the induction assumption,
let $\mathcal{F}$ be the finite set of globally subanalytic functions of $(\tau_{<j},\theta_{<j},y) \in [0,1]^{j-1}\times N_{<j}\times W_{j-1}$, consisting of the function $\mathtt{V}_{\tau_{<j}}^{q_{<j}}$ (which does not depend on $y$) and of all coefficients of all vector fields of the frames $\mathtt{X}_{\tau_{<j}}^{q_{<j}}\in\mathscr{X}_{\tau_{<j}}^{q_{<j}}$. 
We consider all these functions in a neighborhood $W_j$ of the point $q_j=q_j(\tau_{<j},\theta_{<j},\theta_j)$, with $\theta_j\in N_j$, given by Lemma \ref{lem_subanalytic_param}. The elements of $\mathcal{F}$ are then globally subanalytic functions of $(\tau_{<j},\theta_{\leq j},y) \in [0,1]^{j-1}\times N_{\leq j}\times W_j$.
By the subanalytic cell preparation theorem (see Appendix \ref{app_subanalytic}), there exists a subanalytic cell decomposition of $[0,1]^{j-1}\times N_{\leq j}\times W_j$ such that, in each cell, each $f\in\mathcal{F}$ can be prepared with respect to $\tau_{<j}$ (we can choose $\zeta_i=0$ as a center, for every $i\in\{1,\ldots,j\}$) and written as
the product of a fractional monomial $\prod_{i=1}^{j-1} \tau_i^{\alpha_i}$, for some $\alpha_1,\ldots,\alpha_j\in\Q$, with a unit (i.e., not vanishing in the cell) globally analytic and analytic function.


We remove from $\mathscr{X}_{\tau_{<j}}^{q_{<j}}$ all elements $\mathtt{X}_{\tau_{<j}}^{q_{<j}}$ such that $\det(\mathtt{X}_{\tau_{<j}}^{q_{<j}})=0$ (they are the same for all $\tau_1,\ldots,\tau_{j-1}>0$) so that, now, every $\mathtt{X}_{\tau_{<j}}^{q_{<j}}\in\mathscr{X}_{\tau_{<j}}^{q_{<j}}$ is a frame at $q_j$, anyway not necessarily adapted to the sR flag. 

In each cell, by definition of the volume, there exists a $n$-tuple $(Y_1,\ldots,Y_n)\in\mathscr{X}_{\tau_{<j}}^{q_{<j}}$ that is an adapted frame at $q_j$, whose determinant at $0$ is equivalent to $\mathtt{V}_{\tau_{<j}}^{q_{<j}}$ as $\tau_{<j}\rightarrow 0$. 
This means that
\begin{equation}\label{desingXj}
Y_k = \prod_{i=1}^{j-1} \tau_i^{\beta_{k,i}} Z_k \qquad \textrm{with}\qquad Z_k = \tilde Z_k + \mathrm{o}(1)\quad\textrm{as}\quad \tau_{<j}\rightarrow 0 ,
\end{equation}
and the product of all $\tau_i^{\beta_{k,i}}$, for $1\leq i\leq j-1$ and $1\leq k\leq n$, is equivalent to $\mathtt{V}_{\tau_{<j}}^{q_{<j}}$ as $\tau_{<j}\rightarrow 0$. 
Moreover, any coefficient of any vector field $Z_k$ can be written as the product of a monomial $\prod_{i=1}^{j-1} \tau_i^{\alpha_i}$, for some $\alpha_1,\ldots,\alpha_j\in\Q\cap[0,+\infty)$, with a unit (i.e., not vanishing in the cell) globally analytic and analytic function.
Note that $Y_k = (X_k)_{\tau_{<j}}^{q_{<j}}$ if $1\leq k\leq m$.
Since the $m$-tuple of vector fields $X_{\tau_{<j}}^{q_{<j}}$ satisfies the H\"ormander condition for all $\tau_1,\ldots,\tau_{j-1}>0$, with a uniform degree of nonholonomy, it follows that the $m$-tuple of vector fields $Z=(Z_1,\ldots,Z_m)$ satisfies the H\"ormander condition with a uniform degree of nonholonomy.
In this way, \eqref{desingXj} provides a desingularization of $X_{\tau_{<j}}^{q_{<j}}$ in each cell of a subanalytic cell decomposition. 

Now, in the cell, we define the privileged change of coordinates $y=\psi^{q_j}(x) = \exp(x_1 Z_1)\circ\cdots\circ\exp(x_n Z_n ) (0)$. By construction, the Jacobian of $\psi^{q_j}$ at $0$ is positive.
Setting $(X)_{\tau_{\leq j}}^{q_{\leq j}} = \tau_j (\delta^{q_j}_{\tau_j})^* (X)_{\tau_{<j}}^{q_{<j}}$, the lemma follows. 
\end{proof}


Let us give several examples.

\begin{example}\label{ex_BG2_2}
In Example \ref{ex_BG2} (Baouendi-Grushin case with a tangency point), we have $\mathtt{V}^{q_1,q_2}_{\tau_1,\tau_2} = 2\sqrt{\tau_1}+2\tau_2$,
$(X_2)_{\tau_1,\tau_2}^{q_1,q_2}$ converges to $0$, but we see here that the set of all $(\tau_1,\tau_2)\in(0,1]^2$ has a partition in two cells, delimited by the (globally subanalytic) curve $\tau_2=\sqrt{\tau_1}$:
\begin{itemize}[parsep=0.5mm,itemsep=0.2mm,topsep=0.5mm]
\item If $\sqrt{\tau_1}>\tau_2$ then 
$(X_2)^{q_1,q_2}_{\tau_1,\tau_2}(y) = \sqrt{\tau_1}\big( 2y_1 + \frac{\tau_2}{\sqrt{\tau_1}} y_1^2 - \sqrt{\tau_1}\tau_2 y_2 \big) \, \partial_{y_2}$: 
the asymptotics is given by a smooth perturbation of a sR structure of Baouendi-Grushin type where the second vector field is multiplied by $\sqrt{\tau_1}$. 
\item If $\tau_2>\sqrt{\tau_1}$ then 
$(X_2)^{q_1,q_2}_{\tau_1,\tau_2}(y) = \tau_2\big( y_1^2 + 2\frac{\sqrt{\tau_1}}{\tau_2} y_1 - \frac{\tau_1}{\tau_2} y_2 \big) \, \partial_{y_2}$: 
the asymptotics is given by a smooth perturbation of a sR structure of $2$-Baouendi-Grushin type where the second vector field is multiplied by $\tau_2$. 
\end{itemize}
In each cell, the sR structure ``resembles" either to a Baouendi-Grushin or to a $2$-Baouendi-Grushin sR structure, multiplied by a rational power of $\tau_i$.
\end{example}

\begin{example}\label{simpleexample_nonnilp_1}
Consider the sR structure of Example \ref{simpleexample_nonnilp}, and set $q_1=(0,0)$ and $q_2=(0,1)$. 
We have $\S_1=\{q_1\}$ and $\S_2=\R^2\setminus\S_1$.
We obtain 
$(X_1)^{q_1,q_2}_{\tau_1,\tau_2}(y)=\partial_{y_1}$ and $(X_2)^{q_1,q_2}_{\tau_1,\tau_2}(y)=(\tau_2^2 y_1^2+\tau_1^4 + 2\tau_1^4\tau_2 y_2 + \tau_1^4\tau_2^2y_2^2) \, \partial_{y_2}$.
We have $\mathtt{V}^{q_1,q_2}_{\tau_1,\tau_2} \sim \tau_1^4+\tau_2^2$ as $(\tau_1,\tau_2)\rightarrow (0,0)$.
The set of parameters $(\tau_1,\tau_2)\in(0,1]^2$ is decomposed in two cells, delimited by the (analytic) curve $\tau_2=\tau_1^2$:
\begin{itemize}[parsep=0.5mm,itemsep=0.2mm,topsep=0.5mm]
\item If $\tau_1^2>\tau_2$ then $(X_2)^{q_1,q_2}_{\tau_1,\tau_2}(y)$ is a smooth perturbation of $\tau_1^4\,\partial_{y_2}$ (more precisely, the coefficient is written as the product of $\tau_1^4$ with a unit analytic function of $(\tau_1,\tau_2,\frac{\tau_2}{\tau_1^2},y)$), i.e., the sR structure is a smooth perturbation of a Riemannian structure where the second vector field is multiplied by $\tau_1^4$. 
\item If $\tau_2>\tau_1^2$ then $(X_2)^{q_1,q_2}_{\tau_1,\tau_2}(y)$ is a smooth perturbation of $\tau_2^2 y_1^2\,\partial_{y_2}$, i.e., the sR structure is a smooth perturbation of a $2$-Baouendi-Grushin type where the second vector field is multiplied by $\tau_2^2$; similarly, we can factorize by $\tau_2^2$.
\end{itemize}
\end{example}

\begin{example}\label{example_tower_2BG}
Let us consider a ``tower" of intricated Baouendi-Grushin cases with tangency points, with two ``floors": we set 
$X_1=\partial_{x_1}$, $X_2=(x_1^2-x_2)\,\partial_{x_2}$ and $X_3=(x_2^2-x_3)\,\partial_{x_3}$.
The singular set $\S = \{ x_1^2=x_2 \} \cup \{ x_2^2=x_3 \}$ consists of $7$ equisingular strata. Setting $q_1=(0,0,0)$ and $q_2 = (\tau_1^{3/4}, \tau_1^{1/2}, 1)$, we find
\begin{multline*}
(X_1)^{q_1,q_2}_{\tau_1,\tau_2}(y) = \partial_{y_1}, \qquad
(X_2)^{q_1,q_2}_{\tau_1,\tau_2}(y) = (2\tau_1^{3/4}y_1+\tau_2 y_1^2-\tau_1\tau_2 y_2)\,\partial_{y_2}, \\
(X_3)^{q_1,q_2}_{\tau_1,\tau_2}(y) = (2\tau_1^{1/2}y_2+\tau_2^2y_2^2-\tau_1\tau_2 y_3)\,\partial_{y_3} .
\end{multline*}
We have $\mathtt{V}^{q_1,q_2}_{\tau_1,\tau_2} \sim \tau_1^2+\tau_1^{1/2}\tau_2^2+\tau_2^4$ as $(\tau_1,\tau_2)\rightarrow(0,0)$. 
The set of parameters $(\tau_1,\tau_2)\in[0,1]^2$ is decomposed in three cells, delimited by the globally subanalytic curves $\tau_2=\tau_1^{1/4}$ and $\tau_2=\tau_1^{3/4}$:
\begin{itemize}[leftmargin=1cm,parsep=0.5mm,itemsep=0.2mm,topsep=0.5mm]
\item If $\tau_2<\tau_1^{3/4}$ then $(X_2)^{q_1,q_2}_{\tau_1,\tau_2}(y)$ (resp., $(X_3)^{q_1,q_2}_{\tau_1,\tau_2}(y)$) is a smooth perturbation of $2\tau_1^{3/4} y_1\, \partial_{y_2}$ (resp., of $2\tau_1^{1/2} y_2\, \partial_{y_3}$).
\item If $\tau_1^{3/4}<\tau_2\ll\tau_1^{1/4}$ then $(X_2)^{q_1,q_2}_{\tau_1,\tau_2}(y)$ (resp. $(X_3)^{q_1,q_2}_{\tau_1,\tau_2}(y)$) is a smooth perturbation of $\tau_2 y_1^2\, \partial_{y_2}$ (resp., of $2\tau_1^{1/2} y_2\, \partial_{y_3}$).
\item 
If $\tau_1^{1/4}<\tau_2$ then $(X_2)^{q_1,q_2}_{\tau_1,\tau_2}(y)$ (resp., $(X_3)^{q_1,q_2}_{\tau_1,\tau_2}(y)$) is a smooth perturbation of $\tau_2 y_1^2\, \partial_{y_2}$ (resp., of $\tau_2^2 y_2^2\, \partial_{y_3}$).
\end{itemize}
Note that, here, we have $\ell_1=4$. 
%
Taking a tower with more floors, the integer $\ell_1$ could be arbitrarily large.
\end{example}

\begin{example}\label{example_plarge}
Consider the sR structure generated in $\R^2$ by $X_1=\partial_{x_1}$ and $X_2=(x_1^4+x_1^2x_2^2+x_2^{2k})\, \partial_{x_2}$ where $k\geq 3$ (this is the case 6 of the table of examples given in Section \ref{sec_nonnilp_analytic}).
The singular set is $\S=\{(0,0)\}$ and we have $w_1=1$ and $w_2=5$ at $q_1=(0,0)$. We have 
$$
(X_1)^{q_1}_{\tau_1}=\partial_{x_1}, \qquad
(X_2)^{q_1}_{\tau_1}=(x_1^4+\tau_1^8 x_1^2x_2^2+\tau_1^{10k-4}x_2^{2k})\, \partial_{x_2} .
$$
Hence $\widehat{X_1}^{q_1}=\partial_{x_1}$ and $\widehat{X_2}^{q_1}=x_1^4\,\partial_{x_2}$ and thus $\widehat{\S}^{q_1}=\{x_1=0\}$. Wet set $q_2=(0,1)$, $x_1=y_1$ and $x_2=1+y_2$. We find
$$
(X_1)^{q_1,q_2}_{\tau_1,\tau_2}(y) = \partial_{y_1}, \qquad
(X_2)^{q_1,q_2}_{\tau_1,\tau_2}(y) = ( \tau_2^4 y_1^4 + \tau_1^8 \tau_2^2 y_1^2 (1+\tau_2 y_2)^2 + \tau_1^{10k-4} (1+\tau_2 y_2)^{2k} ) \, \partial_{y_2} .
$$
We have $\mathtt{V}^{q_1,q_2}_{\tau_1,\tau_2} \sim \tau_1^{10k-4}+\tau_1^8\tau_2^2+\tau_2^4$ as $(\tau_1,\tau_2)\rightarrow(0,0)$.
Here, we have $\ell_1=1$, 
and the ``competition" is between the three monomials $\tau_1^{10k-4}$, $\tau_1^8\tau_2^2$ and $\tau_2^4$, yielding as well an appropriate cell decomposition.
\end{example}

\subsubsection{Desingularization of the heat kernel}\label{sec_cell_decomp_VF}

The volume function $\mathtt{V}_{\tau_{\leq j}}^{q_{\leq j}}$ defined by \eqref{def_V_iter} is the good quantity in order to desingularize the heat kernel $e_{\tau_{\leq j}}^{q_{\leq j}}(1,y,y)$.
This is not surprising, because we already know that the product $\mathtt{V}_{\tau_{\leq j}}^{q_{\leq j}}\, e_{\tau_{\leq j}}^{q_{\leq j}}(1,y,y)$ is bounded above and below by positive constants on any compact set of $y$. Actually, by extending the uniform ball-box theorem and in particular the estimate \eqref{mu_ball_eps} of Appendix \ref{sec_uniformballbox} and the exponential estimates \eqref{exp_estimates} of Appendix \ref{app_sR_kernel} to parameter-dependent vector fields (framework of Appendix \ref{app_parameter_kernels}), it can even be proved that these bounds are uniform with respect to the $\tau_i$.
But the following lemma, which can be seen as a counterpart to Lemma \ref{lem_uniform_double_nilp}, is much more precise.

\begin{lemma}\label{lem_nonnilp_uniform_multiple_nilp}
Let $V$ be a relatively compact open neighborhood of $0$ in $\R^n$. 
For every $j\in\{1,\ldots,p\}$, denoting for short $q_j = q_j(\tau_{<j},\theta_{\leq j})$ as before, the function $\mathtt{V}_{\tau_{\leq j}}^{q_{\leq j}}\, e_{\tau_{\leq j}}^{q_{\leq j}}(1,y,y)$ has an extension at $\tau_1\cdots\tau_j=0$ that is a bounded positive function of $(\tau_{\leq j},\theta_{\leq j},y)\in [0,1]^j\times N_{\leq j}\times V$. More precisely, there exists a subanalytic cell decomposition of $[0,1]^j\times N_{\leq j}\times V$ such that the projection onto $[0,1]^j$ of each cell of maximal dimension has the (cylindrical) form
$$
a_i(\tau_{<i},\theta_{\leq j},y) < \tau_i <  b_i(\tau_{<i},\theta_{\leq j},y) , \qquad i=1,\ldots,j,
$$
for some analytic and globally subanalytic functions $a_i$ and $b_i$ such that $0\leq a_i(\cdot)<b_i(\cdot)\leq 1$ for every $i\in\{0,\ldots,j\}$, with either $a_i(\cdot)>0$ or $a_i(\cdot)\equiv 0$, 
and such that in each such cell, we have
$$
\mathtt{V}_{\tau_{\leq j}}^{q_{\leq j}} = t^{\ell'/2} \prod_{i=1}^j \tau_i^{\alpha_i} \, U(h(\tau_{\leq j},\theta_{\leq j},y))
\qquad\textrm{and}\qquad
\mathtt{V}_{\tau_{\leq j}}^{q_{\leq j}}\, e_{\tau_{\leq j}}^{q_{\leq j}}(1,y,y) = F(h(\tau_{\leq j},\theta_{\leq j},y))
$$
where
$$
h(\tau_{\leq j},\theta_{\leq j},y) =  \left(  \left( c_i(\theta_{\leq j},y) \right)_{1\leq i\leq N}, \left( \frac{a_i(\tau_{<i},\theta_{\leq j},y)}{\tau_i} \right)^{1/\ell_i}_{1\leq i\leq j}, \left( \frac{\tau_i}{b_i(\tau_{<i},\theta_{\leq j},y)} \right)^{1/\ell_i}_{1\leq i\leq j}  \right) 
$$
for some $\ell'\in\N$, $\alpha_1,\ldots,\alpha_j\in\Q$, $N, \ell_1,\ldots,\ell_j\in\N^*$, 
for some analytic and globally subanalytic functions $c_1,\ldots,c_N$ taking their values in $[0,1]$, 
and for some analytic unit (i.e., not vanishing) function $U$ and some positive $C^\infty$ function $F$ defined on an open set containing $[0,1]^{N+2j}$.
\end{lemma}

\begin{remark}
Lemma \ref{lem_nonnilp_uniform_multiple_nilp} shows that the function $\mathtt{V}_{\tau_{\leq j}}^{q_{\leq j}}\, e_{\tau_{\leq j}}^{q_{\leq j}}(1,y,y)$ is smooth on each cell of maximal dimension of a subanalytic cell decomposition of $[0,1]^j\times N_{\leq j}\times V$. Such cells are analytic manifolds and are also globally subanalytic sets. Anyway, $\mathtt{V}_{\tau_{\leq j}}^{q_{\leq j}}\, e_{\tau_{\leq j}}^{q_{\leq j}}(1,y,y)$ is not a piecewise smooth function of $(\tau_{\leq j},\theta_{\leq j},y)\in [0,1]^j\times N_{\leq j}\times V$ in the sense used in Lemma \ref{lem_nonnilp_general}, i.e., it is not smooth up to the boundary of each cell: in particular, its limits as $\tau_{\leq j}\rightarrow 0$ depend of those of the ratios $a_i(\tau_{<i},\theta_{\leq j},y)/\tau_i$ and $\tau_i/b_i(\tau_{<i},\theta_{\leq j},y)$.
\end{remark}

\begin{remark}
Note that, in Lemma \ref{lem_nonnilp_uniform_multiple_nilp}, the function $U$ is analytic, while the function $F$ is only  smooth and actually may fail to be analytic. Indeed, the heat kernel $e_{\tau_{\leq j}}^{q_{\leq j}}(1,y,y)$ is \emph{not}, in general, a subanalytic function of its arguments. 
\end{remark}

\begin{proof}
We use the same notations as in the proof of Lemma \ref{lem_expansion_epstau}: in particular, we consider the finite set $\mathscr{X}_{\tau_{\leq j}}^{q_{\leq j}}$, from which we remove all elements $\mathtt{X}_{\tau_{\leq j}}^{q_{\leq j}}$ such that $\det(\mathtt{X}_{\tau_{\leq j}}^{q_{\leq j}})=0$, so that every $\mathtt{X}_{\tau_{\leq j}}^{q_{\leq j}}\in\mathscr{X}_{\tau_{\leq j}}^{q_{\leq j}}$ is a frame at $q_j$, anyway not necessarily adapted to the sR flag. 

By Lemma \ref{lem_expansion_epstau}, there exists a subanalytic cell decomposition of $[0,1]^j\times N_{\leq j}\times W_j$ such that, in each cell, the $m$-tuple $X_{\tau_{\leq j}}^{q_{\leq j}}$ depends analytically on $(\tau_{\leq j},\theta_{\leq j},y)$. Each cell can be written in a cylindrical form as in the statement of the lemma.

Like in the proof of Lemma \ref{lem_expansion_epstau}, in each cell, by definition, there exists a $n$-tuple $(Y_1,\ldots,Y_n)\in\mathscr{X}_{\tau_{\leq j}}^{q_{\leq j}}$ that is an adapted frame at $q_j$, whose determinant at $0$ is equivalent to $\mathtt{V}_{\tau_{\leq j}}^{q_{\leq j}}$ as $\tau_{\leq j}\rightarrow 0$. Then, in the cell, we define the change of coordinates $y=\varphi_j(z) = \exp(z_1 Y_1)\circ\cdots\circ\exp(z_n Y_n ) (0)$. By construction, the Jacobian of $\varphi_j$ at $0$ is equal to $\mathtt{V}_{\tau_{\leq j}}^{q_{\leq j}}$ multiplied by a unit globally subanalytic function of $(\tau_{\leq j},\theta_{\leq j})$.

Hence, the pullback $m$-tuple $\tilde X_{\tau_{\leq j}}^{q_{\leq j}} = \varphi_j^* X_{\tau_{\leq j}}^{q_{\leq j}}$ and its corresponding volume $\mathtt{\tilde V}_{\tau_{\leq j}}^{q_{\leq j}}$ (as defined in Appendix \ref{sec_uniformballbox}) depend globally subanalytically on $(\tau_{\leq j},\theta_{\leq j},z)$. 
Moreover, defining $\tilde{\mathcal{F}} =  \varphi_j^* \mathcal{F}$, in the cell, each function of $\tilde{\mathcal{F}}$ can be written as 
$$
F(\theta_{\leq j},z) \, 
U \left( \left( c_i(\theta_{\leq j},z) \right)_{1\leq i\leq N}, \, \tau_1^{1/\ell_1}, \left( \frac{a_i(\tau_{<i},\theta_{\leq j},z)}{\tau_i} \right)^{1/\ell_i}_{2\leq i\leq j} , \left( \frac{\tau_i}{b_i(\tau_{<i},\theta_{\leq j},z)} \right)^{1/\ell_i}_{2\leq i\leq j} \right)
$$
for some $N,\ell_1,\ldots,\ell_j\in\N^*$, some analytic and globally subanalytic functions $F$, $c_1,\ldots,c_N$, with $c_1,\ldots,c_N$ taking their values in $[0,1]$, and some unit (i.e., not vanishing) analytic function $U$ defined on an open subset of $\R^{N+2j-1}$ containing $[0,1]^{N+2j-1}$.

In particular, all possible limits of $\tilde X_{\tau_{\leq j}}^{q_{\leq j}}$ as $\tau_{\leq j}\rightarrow 0$ in the cell satisfy the H\"ormander condition with a uniform degree of nonholonomy and depend smoothly in the cell on the variables that are inside the above function $U$,
up to the boundary of the cell.
Therefore, denoting by $\tilde e_{\tau_{\leq j}}^{q_{\leq j}}$ the heat kernel generated by $\tilde X_{\tau_{\leq j}}^{q_{\leq j}}$, it follows from Theorem \ref{thm_parameter_heat_smooth} in Appendix \ref{app_parameter_kernels} and Theorem \ref{lemfondamental} in Appendix \ref{app_lemfondam} (in particular, the end of Theorem \ref{lemfondamental}) that $\tilde e_{\tau_{\leq j}}^{q_{\leq j}}(1,z,z)$ depends smoothly on the same variables in the cell. 

To end the proof, it suffices to note that, applying the first formula of \eqref{formulas_kernel} in Appendix \ref{appendix_Schwartz} (change of variable in a heat kernel), we have 
$\tilde e_{\tau_{\leq j}}^{q_{\leq j}}(1,z,z) = \mathtt{V}_{\tau_{\leq j}}^{q_{\leq j}} \, e_{\tau_{\leq j}}^{q_{\leq j}}(1,y,y)$.
The lemma follows.
\end{proof}

\subsubsection{End of the proof of Lemma \ref{lem_nonnilp_general}}
In order to end the proof Lemma \ref{lem_nonnilp_general}, this is not exactly the functions $\mathtt{V}_{\tau_{\leq j}}^{q_{\leq j}}$ and $e_{\tau_{\leq j}}^{q_{\leq j}}(1,y,y)$ studied in Lemma \ref{lem_nonnilp_uniform_multiple_nilp} that we need to consider, but rather the functions
$$
f_1(t,\tau_{\leq j},\theta_{\leq j+1},y) = \mathtt{V}^{q_{\leq j},q_{j+1}}_{\tau_{\leq j},\frac{\sqrt{t}}{\tau_1\cdots\tau_j}} 
\qquad\textrm{and}\qquad
f_2(t,\tau_{\leq j},\theta_{\leq j+1},y) = e_{\tau_{\leq j},\frac{\sqrt{t}}{\tau_1\cdots\tau_j}}^{q_{\leq j},q_{j+1}}(1,y,y)  
$$
($f_1$ does not depend on $y$), which are globally subanalytic functions on $\mathcal{D}_{j+1}\times N_{\leq j+1}\times W$ (recall that $\mathcal{D}_{j+1}$ is defined by \eqref{def_Dj}), where $W$ is a relatively compact open neighborhood of $0$ in $\R^n$. Moreover, $f_1$ and $f_1f_2$ are bounded.
We have anyway chosen to present Lemma \ref{lem_nonnilp_uniform_multiple_nilp} and its proof in order to keep a better readability.

Now, to obtain Lemma \ref{lem_nonnilp_general}, we proceed similarly as in the proof of Lemma \ref{lem_nonnilp_uniform_multiple_nilp}, but instead of preparing the functions $f\in\mathcal{F}$ with respect to $\tau_{\leq j}\in[0,1]^j$, we now prepare them with respect to $(t,\tau_{\leq j})\in\mathcal{D}_{j+1}$: there exists a subanalytic cell decomposition of $\mathcal{D}_{j+1}\times N_{\leq j+1} \times W$, the projection onto $(0,1]^{j+1}$ of each cell of maximal dimension having the (cylindrical) form
\begin{eqnarray*}
a_0(\theta_{\leq j+1},y) &\!\!\! <\ t\ < &\!\!\! b_0(\theta_{\leq j+1},y) , \\
a_i(t,\tau_{<i},\theta_{\leq j+1},y) &\!\! <\, \tau_i\, < &\!\! b_i(t,\tau_{<i},\theta_{\leq j+1},y) , \qquad i=1,\ldots,j,
\end{eqnarray*}
such that in each such cell the function $f_1$ can be written as $\displaystyle t^{\ell'/2} \prod_{i=1}^j \tau_i^{\beta_i} \, U(h(t,\tau_{\leq j},\theta_{\leq j+1},y))$
and the product function $f_1f_2$ as $F(h(t,\tau_{\leq j},\theta_{\leq j+1},y))$
with $h(t,\tau_{\leq j},\theta_{\leq j+1},y)$ equal to
$$
\left( \left( c_i(\theta_{\leq j+1},y) \right)_{1\leq i\leq N}, \, t^{1/\ell}, \left( \frac{a_i(t,\tau_{<i},\theta_{\leq j+1},y)}{\tau_i} \right)^{1/\ell_i}_{1\leq i\leq j}, \left( \frac{\tau_i}{b_i(t,\tau_{<i},\theta_{\leq j+1},y)} \right)^{1/\ell_i}_{1\leq i\leq j} \right) 
$$
for some $\ell'\in\N$, $\beta_1,\ldots,\beta_j\in\Q$, $N,\ell,\ell_1,\ldots,\ell_j\in\N^*$, for some analytic and globally subanalytic functions $a_1,\ldots,a_j$, $b_1,\ldots,b_j$, $c_1,\ldots,c_N$, with $0\leq a_i(\cdot)<b_i(\cdot)\leq 1$ for every $i\in\{1,\ldots,j\}$, with either $a_i(\cdot)>0$ or $a_i(\cdot)\equiv 0$, and with $c_1,\ldots,c_N$ taking their values in $[0,1]$, and for some unit analytic function $U$ and some positive $C^\infty$ function $F$ defined on an open subset of $\R^{N+2j+1}$ containing $[0,1]^{N+2j+1}$.
Lemma \ref{lem_nonnilp_general} follows.

\subsubsection{Further comments}\label{sec_further_comments_exp_estimates}
\paragraph{$\gamma$ and $k$ depend only on $D$, and not on the metric $g$.}
Indeed, if two horizontal distributions $D_1$ and $D_2$ are diffeomorphic, i.e., there exists a diffeomorphism $\phi$ such that $\phi_*D_1=D_2$ (this is weaker than the concept of sR isometry defined in Appendix \ref{app_isom}), then the corresponding sR distances $d_1$ and $d_2$ are equivalent in the sense that $d_1$ is bounded above and below, up to scaling, by $d_2$ on any compact; therefore the same is true for the volumes of the corresponding sR balls, and the claim follows from the Fefferman-Phong estimate \eqref{FeffermanPhong}.

\paragraph{Recovering the exponential estimates for the heat kernel.}
In this section, we show how to recover the exponential estimates \eqref{exp_estimates} for sR heat kernels, or equivalently \eqref{exp_estimates_diago} (see Appendix \ref{app_sR_kernel}), for real analytic sR structures, thanks to Lemma \ref{lem_nonnilp_uniform_multiple_nilp} and to the uniform ball-box theorem. 

As a preliminary remark, by Lemma \ref{lem_nonnilp_uniform_multiple_nilp}, there exist $C_1>0$ and $C_2>0$ such that 
\begin{equation}\label{prelimtau1tau2}
C_1 \leq \mathtt{V}^{q_1,q_2}_{\tau_1,\tau_2}\, e_{\tau_1,\tau_2}^{q_1,q_2}(1,0,0) \leq C_2
\qquad \forall(\tau_1,\tau_2)\in(0,1]^2\qquad \forall q_1\in\S_1\qquad \forall \theta_2\in N_2
\end{equation}
(recall that $q_2=q_2(\tau_1,q_1,\theta_2)$). 
Denoting by $(\BsR)^{q_1,q_2}_{\tau_1,\tau_2}(0,1)$ the unit sR ball for the sR structure generated by $X^{q_1,q_2}_{\tau_1,\tau_2}$, we have 
$$
\delta^{q_1}_{\tau_1}\circ\delta^{q_2}_{\tau_2}\big( (\BsR)^{q_1,q_2}_{\tau_1,\tau_2}(0,1) \big) = \delta^{q_1}_{\tau_1} \big( (\BsR)^{q_1}_{\tau_1}(q_2,\tau_2) \big) = \BsR( \delta^{q_1}_{\tau_1}(q_2), \tau_1\tau_2) ,
$$
hence
$$
\mu( \BsR( \delta^{q_1}_{\tau_1} (q_2), \tau_1\tau_2) ) = \tau_1^{\mathcal{Q}^M(q)} \tau_2^{\mathcal{Q}^M(q_2)} m ( (\BsR)^{q_1,q_2}_{\tau_1,\tau_2}(0,1) )
$$
for all $\tau_1,\tau_2>0$, where $m$ is the Lebesgue measure.
By the uniform ball-box theorem (see Appendix \ref{sec_uniformballbox}), the ratio $m ( (\BsR)^{q_1,q_2}_{\tau_1,\tau_2}(0,1) ) / \mathtt{V}^{q_1,q_2}_{\tau_1,\tau_2}$ is bounded above and below by positive constants, uniformly with respect to $(\tau_1,\tau_2)\in(0,1]^2$. Therefore, using \eqref{prelimtau1tau2}, the quantity
$$
\mu( \BsR( \delta^{q_1}_{\tau_1} (q_2), \tau_1\tau_2) \, e(\tau_1^2\tau_2^2, \delta^{q_1}_{\tau_1}(q_2), \delta^{q_1}_{\tau_1}(q_2) )
= \frac{ \mu( \BsR( \delta^{q_1}_{\tau_1} (q_2), \tau_1\tau_2)) } { \tau_1^{\mathcal{Q}^M(q)} \tau_2^{\mathcal{Q}^M(q_2)} } \, e^{q_1,q_2}_{\tau_1,\tau_2}(1,0,0)
$$
is bounded above and below by positive constants, uniformly with respect to $(\tau_1,\tau_2)\in(0,1]^2$.
Taking $\tau_2=\sqrt{t}/\tau_1$ and $q=\delta^{q_1}_{\tau_1}(q_2)$, we infer that the product $\mu(\BsR(q,\sqrt{t}))\, e(t,q,q)$ is bounded above and below by positive constants, uniformly with respect to $q$ in the compact manifold $M$ and with respect to $t\in(0,1]$.
This gives \eqref{exp_estimates_diago}. 

\subsection{Examples of non-analytic sR structures}\label{sec_nonnilp_flat}
In this section, we give classes of examples where the vector fields defining the sR structure are not analytic, i.e., may involve flat terms (however, $\S$ is stratified by equisingular smooth submanifolds, as in Section \ref{sec_equisingular_stratified_nilp}). Such examples have more exotic Weyl laws.

\begin{proposition}\label{prop_nonanalytic}
Consider in $\R^2$ the sR structure generated by the two vector fields $X_1=\partial_1$ and $X_2=(x_1^2+g(x_2))\, \partial_2$, where $g$ is a continous function such that $g(0)=0$ and $g(s)>0$ if $s\neq 0$, and satisfying $g(s)=\mathrm{o}(s^2)$ as $s\rightarrow 0$.
Then, for any $f\in C^0(\R^2)$ such that $f(0,0)>0$, we have
$$
\Tr(\mathcal{M}_f \, e^{t\triangle})\ \sim\  \mathrm{Cst}\ \frac{\vert \{g<t\}\vert}{t^{3/2}} +  \mathrm{Cst}\ \frac{1}{t} \int_{t < g(s)<1} \frac{1}{\sqrt{g(s)}}\, ds
$$
as $t\rightarrow 0^+$, where $\vert \{g<t\}\vert$ is the Lebesgue measure of the set of all $s\in(0,1)$ such that $g(s)<t$.
Moreover, the Weyl measure is the Dirac at $(0,0)$.
\end{proposition}

If $g(s)=s^{2k}$ with $k\geq 1$ then $\Tr(\mathcal{M}_f \, e^{t\triangle})\ \sim\  \mathrm{Cst} / t^{\frac{3}{2}-\frac{1}{2k}}$.

Denoting by $g^*$ the nondecreasing rearrangement of $g$ on $[0,1]$, we have, by equimeasurability, $\vert \{g<t\}\vert=\vert \{g^*<t\}\vert$ and $\int_{t < g(s)<1} g(s)^{-1/2}\, ds=\int_{t < g^*(s)<1} g^*(s)^{-1/2}\, ds$. Hence, without loss of generality, we can take $g$ nondecreasing on $[0,1]$.

Note that, if $g$ is increasing on $[0,1]$, then 
$$
\Tr(\mathcal{M}_f \, e^{t\triangle})\ \sim\ \mathrm{Cst}\ \frac{g^{-1}(t)}{t^{3/2}} +  \mathrm{Cst}\ \frac{1}{t} \int_t^1 \frac{1}{\sqrt{s}g'(g^{-1}(s))}\, ds .
$$
It is easy to see that $\int_t^1 \frac{1}{\sqrt{s}g'(g^{-1}(s)}\, ds = \mathrm{o}\big( \frac{g^{-1}(t)}{\sqrt{t}} \big)$ if $\frac{1}{g'(s)} = \mathrm{o}\big( \frac{s}{g(s)} \big)$ and if the integral diverges (this latter property is satisfied if $g'(s)=\mathrm{o}(\sqrt{g(s)})$); in this case, $\Tr(\mathcal{M}_f \, e^{t\triangle})\ \sim\ \mathrm{Cst}\ \frac{g^{-1}(t)}{t^{3/2}}$.

We obtain interesting examples by taking $g$ flat at $0$ (see Table \ref{table_examples_flat}).

\begin{table}[h]
\caption{Examples with $g$ flat at $0$.}\label{table_examples_flat}
$$
\begin{array}{|c|cc|}
\hline \rule{0pt}{3.5mm}
g(s) & \mathrm{Tr}(\mathcal{M}_f \, e^{t\triangle}) 
& N(\lambda) 
\\ 
\hline 
\displaystyle \frac{1}{e^{1/\vert s\vert^\alpha}},\ \ \alpha>0 &\displaystyle \frac{1}{t^{3/2}\vert\ln t\vert^{1/\alpha}} &\displaystyle \frac{\lambda^{3/2}}{\left(\ln\lambda\right)^{1/\alpha}} \\[3.5mm]
\hline
\displaystyle \frac{1}{e^{\beta e^{1/\vert s\vert^\alpha}}},\ \ \alpha,\beta>0  &\displaystyle \frac{1}{t^{3/2}\left(\ln\vert\ln t^{1/\beta}\vert\right)^{1/\alpha}} &\displaystyle \frac{\lambda^{3/2}}{\left(\ln\ln\lambda^{1/\beta}\right)^{1/\alpha}} \\[4.5mm]
\hline
\displaystyle \frac{1}{\exp^{[k]}\vert s\vert} = \frac{1}{e^{e^{\dots^{e^{1/\vert s\vert}}}}} &\displaystyle \frac{1}{t^{3/2}\ln^{[k]}\frac{1}{t}} &\displaystyle \frac{\lambda^{3/2}}{\ln^{[k]}\lambda} = \frac{\lambda^{3/2}}{\ln\cdots\ln\lambda} \\[4mm]
\hline 
\displaystyle e^{-\frac{\ln^2s}{s}} &\displaystyle \frac{\ln^2\vert\ln t\vert}{t^{3/2}\vert\ln t\vert}  &\displaystyle \frac{\lambda^{3/2}\ln^2\ln\lambda}{\ln\lambda}  \\[2.5mm]
\hline
\end{array}
$$
\end{table}

\begin{proof}
We have $\S=\{q\}$ with $q=(0,0)$ and $w_1=1$, $w_2=3$ at $q$. We compute $(X_1)^q_\tau=\partial_{x_1}$ and $(X_2)^q_\tau = (x_1^2+\tau^{-2}g(\tau^3 x_2))\, \partial_{x_2}$. The nilpotentization is $\widehat{X}^q_1=\partial_{x_1}$ and $\widehat{X}^q_2=x_1^2\,\partial_{x_2}$. Outside of $x_1=0$, the sR case is nilpotentizable. We thus focus on the point $q_2(\tau)=q_2=(0,1)$. Setting $x_1=y_1$ and $y_1=1+y_2$, we obtain $(X_1)^{q_1,q_2}_{\tau,\varepsilon}=\partial_{y_1}$ and $(X_2)^{q_1,q_2}_{\tau,\varepsilon}=(\varepsilon^2x_1^2+\tau^{-2}g(\tau^3+\tau^3\varepsilon y_2))\,\partial_{y_2}$. 

Therefore, if $\varepsilon < \sqrt{g(\tau^3)}/\tau$ (resp., $\varepsilon > \sqrt{g(\tau^3)}/\tau$) then the asymptotics of $e^{q_1,q_2}_{\tau,\varepsilon}(1,(u,0),(u,0))$ is $\frac{\tau^2}{g(\tau^3)}$ (resp., $\frac{1}{\varepsilon^2}$) times a smooth function of $u$.
This leads to split the integral defining $K(t)$, which is an integral over $\tau\in[\sqrt{t},1]$, in two integrals performed on the intersection of $[\sqrt{t},1]$ with either $g(\tau^3)<t$ or $g(\tau^3)>t$.
The formula now follows from computations.
\end{proof}

Another example, not coming from Proposition \ref{prop_nonanalytic}, is the following.

\begin{example}
Consider in $\R^5$ the sR structure generated by $X_1=\partial_1$, $ X_2=\partial_2+x_1\,\partial_3+x_1^2\,\partial_5$, $X_3=\partial_4+e^{-1/(x_1^2+x_2^2)}\,\partial_5$. Then, for any $f\in C^\infty(\R^2)$ such that $f>0$ along $x_1=x_2=0$, we have
$$
\mathrm{Tr}(\mathcal{M}_f \, e^{t\triangle}) \underset{t\rightarrow 0^+}{\sim} \frac{\Cst}{t^4\vert\ln t\vert} \qquad\qquad N(\lambda)\underset{\lambda\rightarrow+\infty}{\sim}\Cst\, \frac{\lambda^4}{\ln\lambda}.
$$
The computations, which are quite lengthy, are not reported.
\end{example}

As a final comment, we may wonder whether there exists a class $\mathfrak{C}$ of functions such that, if $M$ is stratified by equisingular smooth submanifolds and all singularities of the sR flag restricted to each stratum are in $\mathfrak{C}$, then the function $t\mapsto\Tr(\mathcal{M}_f \, e^{t\triangle})$ has a small-time asymptotic expansion in a well-identified asymptotic scale. Candidates for $\mathfrak{C}$ might be provided by some specific o-minimal classes which are stable under integration, as the class of log-analytic functions (see \cite{CluckersMiller_DMJ2011}).

\appendix
\part{Appendix}
In this part, we gather all reminders and results that are useful for our study.
Appendix \ref{sec_sR_geom} gathers some (well known and less known) definitions and facts in sub-Riemannian geometry.
In Appendix \ref{app_parameter_kernels}, we recall two useful statements : a hypoelliptic version of Kac's principle, stating the local nature of the small-time asymptotics of heat kernels; a result on continuity of heat kernels with respect to parameters. 
In Appendix \ref{app_lemfondam}, we give the complete small-time asymptotic expansion of heat kernels. 
Finally, in Appendix \ref{app_integral}, we establish a lemma on asymptotic expansion of some integrals.

\section{Sub-Riemannian geometry}\label{sec_sR_geom}
In this section (that one can also find in \cite{CHT_AHL}), we recall well known definitions and facts in sub-Riemannian geometry (see the textbooks \cite{AgrachevBarilariBoscain_book2019, Bellaiche, Gromov, Jean_2014, LeDonne_book2021, Mo-02, Rifford_2014}).
Throughout, ``\emph{sR}" means ``\emph{sub-Riemannian}".

\subsection{Definition}
Let $n\in\N^*$ and let $M$ be a smooth connected manifold of dimension $n$. 
Let $m\in\N^*$ and let $X=(X_1,\ldots,X_m)$ be a $m$-tuple of smooth vector fields on $M$. We set $D=\mathrm{Span}(X)$ (called \emph{horizontal distribution}).
The \emph{sR metric} $g$ associated with the $m$-tuple $X$ is defined as follows: given any $q\in M$ and any $v\in D(q)=\mathrm{Span}(X_1(q),\ldots,X_m(q))$, we define the positive definite quadratic form $g_q$ on $D(q)$ by
\begin{equation}\label{def_metric}
g_q(v) = \inf \left\{ \sum_{i=1}^m u_i^2 \ \ \Big\vert \ \ v=\sum_{i=1}^mu_iX_i(q) \right\} .
\end{equation}
The pair $(M,X)$ (or the triple $(M,D,g)$) is called the \emph{sub-Riemannian structure} on $M$ generated by $X$.
When $D$ has constant rank $m$ on $M$ with $m\leq n$, $D$ is a subbundle of $TM$, $g$ is a Riemannian metric on $D$ and the frame $X=(X_1,\ldots,X_m)$ of $D$ is $g$-orthonormal. But the rank of $D$ may vary (i.e., $D$ is a subsheaf of $TM$) and the above definition encompasses the so-called \emph{almost-Riemannian case}, for which $m\geq n$ and $\mathrm{rank}(D)<n$ at some singular points.

More formally, a sR structure on $M$ can be defined by giving an Euclidean vector bundle $E$ over $M$ and a smooth vector bundle morphism $\sigma:E\rightarrow TM$ such that $D(q)=\sigma(E(q))$ for every $q\in M$, and then
$$
g_q(V)=\inf\{\Vert u\Vert_{E(q)}^2\ \mid\ u\in E(q),\ \sigma(u)=V\}
$$
When $E=M \times\R^m$ and $\sigma(x,u)=\sum_{i=1}^m u_i X_i(x)$, we recover the definition of a sR structure attached with the $m$ vector fields $X_1,\ldots,X_m$.

A \emph{horizontal path} is, by definition, an absolutely continuous path $q(\cdot):[0,1]\rightarrow M$ for which there exist $m$ functions $u_i\in L^1(0,1)$ such that $\dot q(t) = \sum_{i=1}^m u_i(t)X_i(q(t))$ for almost every $t\in[0,1]$.
The metric $g$ induces a length on the set of horizontal paths, and thus a distance $\dsR$ on $M$ that is called the sR distance.
Given any $q\in M$ and any $R>0$, the \emph{sR ball} $\BsR(q,R)$ centered at $q$ of radius $R$ is the set of all $q'\in M$ such that $\dsR(q,q')<R$. 

The \emph{cometric} $g^*$ associated with $X$ is the nonnegative quadratic form on $T^*M$ defined as follows: given any $q\in M$, $g^*_q$ is the nonnegative quadratic form defined on $T_q^*M$ by $g^*_q(\xi)=\sum_{i=1}^m\langle\xi,X_i(q)\rangle^2$.
Note that $\frac{1}{2} g_q(v) = \sup_{\xi\in T^*_qM} \left( \langle\xi,v\rangle - \frac{1}{2}g^*_q(v) \right)$ (Legendre transform).
The cometric $g^*$ completely determines the horizontal distribution $D$ and the sR metric $g$. 

Given any smooth function $f$ on $M$, the \emph{horizontal gradient} $\nabla_gf$ of $f$ is the smooth section of $D$ defined by $g(\nabla_gf,Y)=df.Y$ for every smooth section $Y$ of $D$. We have $\nabla_gf = \sum_{i=1}^m (X_if)X_i$.

Let $\mu$ be an arbitrary smooth (Borel) measure on $M$.

\subsection{Sub-Riemannian Laplacian}\label{app_sRLaplacian}
Let $L^2(M,\mu)$ be the set of complex-valued functions $u$ such that $|u|^2$ is $\mu$-integrable over $M$. We define $-\triangle$ as the nonnegative selfadjoint operator on $L^2(M,\mu)$ that is the Friedrichs extension of the Dirichlet integral
$$
Q(\phi)=\int _M \Vert d\phi \Vert_{g^*}^2 \, d\mu  \qquad \forall \phi\in C^\infty_c(M)
$$
where the norm of $d\phi $ is calculated with respect to the (degenerate) dual metric $g^\star $ (also called co-metric) on $T^\star M $ associated with $g$.
The sR Laplacian $\triangle$ depends on $g$ and $\mu$.

We denote by $\mathrm{div}_\mu$ the divergence operator associated with the measure $\mu$, defined by $\mathcal{L}_Y \mu = \mathrm{div}_\mu(Y)\, \mu$ for any vector field $Y$ on $M$.
Hence
$\triangle\phi = \mathrm{div}_\mu(\nabla_g\phi)$ for every $\phi\in C^\infty_c(M)$
and, since $\nabla_g\phi=\sum_{i=1}^m(X_i\phi)X_i$ and $Q(\phi) = \int_M  \sum_{i=1}^m (X_i\phi)^2 \, d\mu$, it follows that
\begin{equation}\label{app_def_triangle}
\triangle= - \sum_{i=1}^m X_i^\star X_i = \sum_{i=1}^m \left( X_i^2+\mathrm{div}_\mu(X_i) X_i \right) 
\end{equation}
where the transpose is taken in $L^2(M,\mu)$.
Under the H\"ormander condition 
\begin{equation}\label{Hormander_assumption}
\mathrm{Lie}(D)=\mathrm{Lie}(X_1,\ldots,X_m)=TM  ,
\end{equation}
$\triangle$ is subelliptic (see \cite{Ho-67}), and thus, if $M$ is compact, has a compact resolvent and a discrete spectrum. 
Note that if $M$ is compact then $\triangle$ is essentially selfadjoint.

\subsection{Sub-Riemannian flag}\label{app_sRflag}
We define the sequence of subsheafs $D^k$ of $TM$ by $D^0=\{0\}$, $D^1=D=\mathrm{Span}(X_1,\ldots,X_m)$ and $D^{k+1}=D^k+[D,D^k]$ for $k\geq 1$. Under the H\"ormander condition \eqref{Hormander_assumption}, given any point $q\in M$, we consider the \emph{sR flag} of $D$
$$
\{0\}=D^0(q)\subset D(q)=D^1(q)\subset D^2(q)\subset \cdots\subset D^{r(q)-1}(q)\subsetneq D^{r(q)}(q)=T_qM 
$$
where $r(q)$ is called the \emph{degree of nonholonomy} at $q$. We set $n_i(q) = \dim D^i(q)$. The $r(q)$-tuple of integers $(n_1(q),\ldots,n_{r(q)}(q))$ is called the \emph{growth vector} at $q$, and we have $n_{r(q)}(q)=n=\dim M$. By convention, we set $n_0(q)=0$.

We define the nondecreasing sequence of \emph{sR weights} $w_i(q)$ (also denoted $w_i^q(D)$ when one wants to underline that they refer to the horizontal distribution $D$) 
as follows:
given any $i\in\{1,\ldots,n\}$, there exists a unique $j\in\{1,\ldots,r(q)\}$ such that $n_{j-1}(q)+1\leq i\leq n_j(q)$, and we set $w_i(q)=j$.
By definition, we have $w_1(q)=\cdots=w_{n_1}(q)=1$, and $w_{n_{j-1}+1}(q)=\cdots=w_{n_j}(q)=j$ when $n_j(q)>n_{j-1}(q)$. We also have $w_{n_{r-1}+1}(q)=\cdots=w_{n_r}(q)=r(q)$.

Given any $q\in M$, we set
\begin{equation}\label{def_Q}
\mathcal{Q}^M(q) = \sum_{i=1}^r i ( n_i(q)-n_{i-1}(q) ) =\sum_{i=1}^n w_i(q).
\end{equation}
If $q$ is regular then $\mathcal{Q}^M(q)$ is the Hausdorff dimension of a small ball in $M$ containing $q$ for the induced corresponding sR distance (see \cite{Gromov}).
Note that the maps $q\mapsto \mathcal{Q}^M(q)$ and $q\mapsto w_i(q)$, for $i=1,\ldots,n$, are upper semi-continuous. 

A point $q\in M$ is said to be \emph{regular} if the growth vector is constant in a neighborhood of $q$; otherwise it is said to be \emph{singular}. Throughout the paper, the singular set is denoted by $\S$; it depends on the horizontal distribution $D$ but not on the metric. The open set $M\setminus\S$ is called the \emph{regular region}; we have $\mathcal{Q}^M(q)=\mathcal{Q}^M(M\setminus\S)=\mathcal{Q}^{M\setminus\S}$ for every $q\in M\setminus\S$, and $\mathcal{Q}^{M\setminus\S}$ is the Hausdorff dimension of $M\setminus\S$.
The sR structure is said to be \emph{equiregular} if all points of $M$ are regular; in this case, the weights and the Hausdorff dimension are constant as well on $M$.

At a regular point $q$, 
$\triangle$ is locally subelliptic with a gain of regularity $2/r(q)$, meaning that if $\triangle u=v$ with $v$ of Sobolev class $H^s$ locally at $q$ then $u$ is (at least) of Sobolev class $H^{s+2/r(q)}$ locally at $q$ (see \cite{Ho-67}). 

We also define $\Sigma^i=(D^i)^\perp \subset T^\star M$ (annihilator of $D^i$) for $i=1,\ldots,r$.
For $i=1$, $\Sigma=\Sigma^1=D^\perp$ is called the characteristic manifold of the sR structure, and we also have $\Sigma=(g^*)^{-1}(0)$. 

\paragraph{Sub-Riemannian flag restricted to a submanifold.}
Let $N$ be a smooth submanifold of $M$. Given any $q\in N$, we consider the sR flag of $D$ at $q$ restricted to $N$ (also called the sR flag of $D\cap TN$)
\begin{equation}\label{sRflag_N}
\{0\} \subset \left( D^1(q)\cap T_qN \right) \subset 
\cdots\subset  \left( D^{r(q)-1}(q)\cap T_qN \right) \subset \left( D^{r(q)}(q)\cap T_qN \right) = T_qN 
\end{equation}
and we set $n_i^N(q) = \dim \left( D^i(q)\cap T_qN \right)$ and
\begin{equation}\label{def_QN}
\mathcal{Q}^N(q) = \sum_{i=1}^{r(q)} i ( n_i^N(q)-n_{i-1}^N(q) ) .
\end{equation}
The $r(q)$-tuple of integers $(n_1^N(q),\ldots,n_{r(q)}^N(q))$ is called the growth vector at $q$ restricted to $N$.
Following \cite{GhezziJean_NA2015,Gromov}, we say that $N$ is \emph{equisingular} if all integers $n_i(q)$ and $n_i^N(q)$ remain constant as $q\in N$, i.e., if the growth vector and the growth vector restricted to $N$ are constant on $N$. In this case, $\mathcal{Q}^N$ is the Hausdorff dimension of $N$ (see \cite[Theorem 5.3]{GhezziJean_NA2015}).
Note that a smooth submanifold of an equisingular smooth submanifold may fail to be equisingular.

\subsection{Sub-Riemannian isometries}\label{app_isom}
Given two sR structures $(M_1,D_1,g_1)$ and $(M_2,D_2,g_2)$, of respective cometrics $g^*_1$ and $g^*_2$, a (local) sR isometry $\phi:M_1\rightarrow M_2$ is a (local) smooth diffeomorphism mapping $g^*_1$ to $g^*_2$. 
This is stronger than requiring that $D_1=\mathrm{Span}(X)$ and $D_2=\mathrm{Span}(Y)$ are diffeomorphic which means that we have only $\phi_*D_1=D_2$ (see Appendix \ref{sec_nilp_diffeo}).

\subsection{Nilpotentization}\label{app_nilp}
\subsubsection{Definition}
Let $q\in M$ be arbitrary. The nilpotentization of the sR structure $(M,D,g)$ at $q$ is the sR structure $(\widehat{M}^{q},\widehat{D}^{q},\widehat{g}^{q})$ defined as the metric tangent space of $M$ (endowed with its sR distance) in the sense of Gromov-Hausdorff (see \cite{Bellaiche,Gromov}).\footnote{This means that $(\delta^q_\varepsilon)^{-1}(\BsR(q,\varepsilon))\rightarrow\hatBsR^q(0,1)$ for the Gromov-Hausdorff topology in the privileged coordinates introduced hereafter, where $\hatBsR^q(0,1)$ is the sR unit ball for the sR structure $(\widehat{M}^{q},\widehat{D}^{q},\widehat{g}^{q})$.}
In this triple, $\widehat{M}^{q}$ is a smooth connected manifold of dimension $n$ 
(the vector space $T_0\widehat{M}^q$ is canonically identified with $T_qM$)
which is identified to $\R^n$ thanks to the privileged coordinates defined hereafter,
the horizontal distribution is $\widehat{D}^{q}=\mathrm{Span}(\widehat{X}^{q}_1,\ldots,\widehat{X}^{q}_m)$ with smooth vector fields $\widehat{X}^{q}_1,\ldots,\widehat{X}^{q}_m$ on $\widehat{M}^q$ (given hereafter) called nilpotentizations at $q$ of the vector fields $X_1,\ldots,X_m$ at $q$, and the sR metric $\widehat{g}^{q}$ is defined accordingly as in \eqref{def_metric}.
The metric $\widehat{g}^{q}$ induces 
a distance $\hatdsR^{q}$ on $\widehat{M}^{q}$.


\subsubsection{Privileged coordinates}\label{app_privileged}
We first recall the notion of \emph{nonholonomic order} (see \cite{Bellaiche,Jean_2014,Mo-02} for details).
Given a germ $f$ of a real-valued smooth function at $q$, given $k\in\N$ and integers $j_1,\ldots,j_k$ in $\{1,\ldots,m\}$, the Lie derivative $(X_{j_1}\cdots X_{j_k}f)(q)$ is called a nonholonomic derivative of order $k$. By definition, the \emph{nonholonomic order} of $f$ at $q$, denoted by $\ord_{q}(f)$, is the smallest integer $k$ for which at least one nonholonomic derivative of $f$ of order $k$ at $q$ is not equal to zero.
Given a germ $Y$ of a smooth vector field at $q$, the nonholonomic order of $Y$ at $q$ is the largest integer $k$ such that $\ord_{q}(Yf)\geq k+\ord_{q}(f)$, for every germ $f$ at $q$.

The \emph{nonholonomic length} at $q$ of a vector field $Z$ on $M$ is defined by $\ell_q(Z) = \min\{j\in\N\ \vert\ Z(q)\in D^j(q)\}$. A family $(Z_1,\ldots,Z_n)$ of $n$ vector fields is said to be \emph{adapted to the sR flag} of $D$ at $q$ if $D^j(q) = \mathrm{Span}\, \{Z_i(q)\ \vert\ 1\leq i\leq n,\ \ell_q(Z_i)\leq j\}$ for every $j\in\{1,\ldots,r(q)\}$.%
\footnote{A usual way to construct an adapted local frame $(Z_1,\ldots,Z_n)$ of $T_{q}M$ at $q$ is the following: choose vector fields $Z_1,\ldots,Z_{n_1(q)}\in D$ whose values at $q$ form a basis of $D(q)$; complete them to vector fields $Z_1,\ldots,Z_{n_2(q)}\in D^2$ whose values at $q$ form a basis of $D^2(q)$; etc. With such a choice, we have $Z_i(q)\in D^{w_i(q)}(q)$ for every $i\in\{1,\ldots,n\}$; but in the definition of adapted frame that we adopt, we can consider as well a permutation of $(Z_1,\ldots,Z_n)$.}

Given any chart at $q$, i.e., given any smooth diffeomorphism $\psi^q:U\rightarrow V$, where $U$ is a neighborhood of $q$ in $M$ and $V$ is a neighborhood of $0$ in $\R^n$, with $\psi^q(q)=0$, inducing local coordinates $x=(x_1,\ldots,x_n)$, we have 
$$
\ord_q(x)=\sum_{j=1}^n\ord_q(x_j)\leq\mathcal{Q}(q)=\sum_{j=1}^nw_j(q) .
$$
We say that $\psi^q$ is a chart of \emph{privileged coordinates} at $q$ if it is ``maximal", in the sense that $\ord_q(x)=\mathcal{Q}(q)$, i.e., $\{\ord_q(x_1),\ldots,\ord_q(x_n)\} = \{ w_1(q),\ldots,w_n(q) \}$.

The sR weights $w_j(q)$, defined in Appendix \ref{app_sRflag}, are a nondecreasing sequence. Anyway, it may be convenient to relabel the weights so that $w_i(q) = \ord_q(x_i)$: in this case, following \cite{GhezziJean_NA2015}, we say that the weights are \emph{labeled according to the coordinates $x$}.
Note that, then, $dx_i(D^{w_i(q)}(q))\neq 0$ and $dx_i(D^{w_i(q)-1}(q))= 0$, meaning that $\partial_{x_i}\in D^{w_i(q)}(q)\setminus D^{w_i(q)-1}(q)$ at $q$, i.e., privileged coordinates are always adapted to the sR flag.

Classical examples of charts of privileged coordinates at $q$ are given by the coordinates of the first kind
$$
(\psi^q)^{-1}(x_1,\ldots,x_n) = \exp \left( x_1 Z_1 + \cdots + x_n Z_n \right) (q)
$$
and by the coordinates of the second kind $\exp \left( x_1 Z_1 \right) \circ \cdots \circ \exp \left( x_n Z_n \right) (q)$
where $(Z_i)_{1\leq i\leq n}$ is a frame of vector fields that is adapted to the sR flag at $q$.
Privileged coordinates 
can be obtained from adapted coordinates by a triangular change of variables (see \cite{Jean_2014}).

\paragraph{Privileged coordinates straightening an equisingular smooth submanifold.}
Let $N$ be an equisingular smooth submanifold of $M$ (see Appendix \ref{app_sRflag} for the definition) of topological dimension $k$. Let $q\in N$ and let $U$ be a neighborhood of $q$ in $M$. At each point $q'\in N\cap U$, assuming that $U$ is sufficiently small, there exist local privileged coordinates $x=(x_1,\ldots,x_n)$, 
depending smoothly on $q'\in N\cap U$, 
in which $N = \{ x_{k+1}=\cdots=x_n = 0\}$. 

The existence of such coordinates is proved in \cite[Lemma 4.3]{GhezziJean_NA2015}, thanks to the following argument: since all integers $n_i=\dim(D^i)$ and $n_i^N=\dim(D^i\cap TN)$ are constant along $N$, assuming $U$ small enough, we can choose a local frame $(Z_1,\ldots,Z_n)$ of $n$ vector fields that is adapted to the sR flag of $D$ on $U$ such that, moreover, the family $(Z_1,\ldots,Z_k)$ is adapted to the sR flag of $D\cap TN$ (defined by \eqref{sRflag_N}) on $N\cap U$.%
\footnote{Indeed, take $n_1^N$ vector fields $Z_1,\ldots,Z_{n_1^N}\in D\cap TN$ whose values at any point $q'\in N\cap U$ form a basis of $D(q')\cap T_{q'}N$; complete them to vector fields $Z_1,\ldots,Z_{n_1}\in D$ whose values at any point $q'\in N\cap U$ form a basis of $D(q')$; 
then iterate this construction along the flag. Finally, re-index the vector fields by a permutation.}
Then, it suffices to take exponential privileged coordinates (of the first or second kind) associated with the vector fields $Z_1,\ldots,Z_n$.

Assuming that the sR weights are labeled according to the coordinates $x$, for every $j\in\{1,\ldots,n\}$, $w_j(q)=w_j^q(D)=\ord_q(x_j)$ does not depend on $q\in N$ and we denote it by $w_j^N(D)$.
Hence, to the equisingular submanifold $N$ are attached the two integers $\mathcal{Q}^M(N) = \ord_q(x) = \sum_{j=1}^n\ w_j^N(D)$ (see \eqref{def_Q}) and $\mathcal{Q}^N = \ord_q(x_1,\ldots,x_k) = \sum_{j=1}^k w_j^N(D)$ (see \eqref{def_QN}), the latter being the Hausdorff dimension of $N$.
%
By convention, when $N$ is a single point, we set $k=0$ and $\mathcal{Q}^N=0$.

\begin{remark}\label{rem_transversally_smooth_system_stratified}
The above straightening procedure cannot be iterated on strata when $N=\bigcup_{i=1}^s N_i$ is (Whitney) stratified by equisingular smooth submanifolds, for some $s\in\N^*$, where $N_1,\ldots,N_s$ are (disjoint) equisingular smooth submanifolds of $M$ such that $N_i\subset \overline{N_{i+1}}$ for $i=1,\ldots,s-1$. More precisely, let $q\in N_1$. We set $k_i=\dim(N_i)$. There exist privileged coordinates in a (small enough) neighborhood $U$ in $N$ of $q\in N_1$, straightening $N_1$, depending smoothly on $q\in N_1$, in which $N_1=\{x_{k_1+1}=\cdots=x_n=0\}$. But, in general, it is not possible to construct privileged coordinates such that $N_i=\{x_{k_i+1}=\cdots=x_n=0\}$ for every $i\in\{1,\ldots,s\}$ when $s>1$: a counterexample is given by the Baouendi-Grushin case with a tangency point; the same counterexample shows that it is not possible in general to construct privileged coordinates at $q\in N_1$ that would depend smoothly on $q_i\in\overline{N_i}$ for $i>1$ with $q_i$ converging to $q$.
\end{remark}

\begin{lemma}\label{lemQinter}
Let $N_2$ be an equisingular smooth submanifold of $M$ of topological dimension $k_2\in\N^*$. Let $P_1$ be 
an equisingular submanifold of $N_2$ of topological dimension $k_2-k_1-1$ for some integer $0\leq k_1<k_2$. 
For any $q\in P_1$, locally around $q$ there exists a smooth submanifold $P$ of $M$ (not equisingular) of topological dimension $n-k_1-1$, satisfying $P\cap N_2=P_1$ and intersecting $N_2$ transversally at $q$ (i.e., $T_qM=T_qP+T_qN_2$), such that 
$$
\mathcal{Q}^{P}(P\cap N_2) - \mathcal{Q}^{P\cap N_2} = \mathcal{Q}^M(N_2) - \mathcal{Q}^{N_2} .
$$
\end{lemma}

\begin{proof} 
For any $q\in P_1$, we claim that there exist privileged coordinates $y=(y_1,y_2,y_3)$ at $q$, depending smoothly on $q\in P_1$, straightening $P_1$ and $N_2$ so that $P_1=\{y_2=y_3=0\}$ and $N_2=\{y_3=0\}$: indeed, since $P_1$ and $N_2$ are equisingular, in a neighborhood $U$ of $q$, we can choose a local frame $(Z_1,\ldots,Z_n)$ of $n$ vector fields that is adapted to the sR flag of $D$ on $U$ such that, moreover, 
the family $(Z_1,\ldots,Z_{k_2-k_1-1})$ is adapted to the sR flag of $D\cap TP_1$ on $P_1\cap U$,
and the family $(Z_1,\ldots,Z_{k_2})$ is adapted to the sR flag of $D\cap TN_2$ on $N_2\cap U$.
We have then $\mathcal{Q}^{P_1}=\ord_q(y_1)$, $\mathcal{Q}^{N_2}(P_1)=\mathcal{Q}^{N_2}=\ord_q(y_1,y_2)$ and $\mathcal{Q}^M(N_2)=\ord_q(y)$.
We define $P$ as the bundle over $P_1$ whose fiber over each $q\in P_1$ is the submanifold $\{y_1=y_2=0\}$ (of topological dimension $n-k_2$). 
Then $\mathcal{Q}^{P}(P\cap N_2) = \mathcal{Q}^{P}(P_1)=\ord_q(y_1,y_3)$ and the lemma follows. 
\end{proof}

\begin{remark}\label{rem_lemQinter}
The conclusion of Lemma \ref{lemQinter} is not valid in general for any smooth local submanifold $P$ of $M$ containing $P_1$ and transverse to $N_2$ at $q$.
\end{remark}

\begin{remark}\label{rem_lemQinter_extension}
Lemma \ref{lemQinter} is obviously extended to the case where $N_2(\kappa)$ and $P_1(\kappa)$ depend smoothly on a parameter $\kappa$ (belonging to some smooth manifold), and yields $P(\kappa)$ depending as well smoothly on $\kappa$.

Another useful extension of the lemma is when $P_1$ is stratified by equisingular submanifolds of $N_2$: in this case, $P$ is a stratified bundle over $P_1$ and the conclusion of the lemma is satisfied for any stratum of $P$.
\end{remark}

\subsubsection{Dilations and nilpotentization of smooth sections of $D$}\label{sec_nilp_smoothsection}
We consider a chart $\psi^q$ of privileged coordinates at $q$. 
Given any $\varepsilon\in\R$, the dilation $\delta_\varepsilon^q$ at $q$, according to the flag at $q$, and the dilation $\delta_\varepsilon$ in $\R^n$, are defined by 
\begin{equation}\label{def_deltaepsilon}
\delta_\varepsilon^q = (\psi^q)^{-1}\circ\delta_\varepsilon , \qquad
\delta_\varepsilon(x) = \left( \varepsilon^{w_1(q)} x_1,\ldots, \varepsilon^{w_n(q)} x_n \right) \quad\forall x=(x_1,\ldots,x_n)\in\R^n 
\end{equation}
where the sR weights are labeled according to the coordinates $x$.
Note that, denoting by $m$ the Lebesgue measure on $\R^n$ (given by $dm=dx_1\cdots dx_n$), we have $\delta_\varepsilon^*m=\vert\varepsilon\vert^{\mathcal{Q}^M(q)}m$ for every $\varepsilon\neq 0$.

Given any vector field $Y$ on $M$ that is a smooth section%
\footnote{Note that we consider a smooth section of the subsheaf $D$, otherwise there are some difficulties: take $M=\R^2$, $D$ spanned by $X_1=\partial_x$ and $X_2=x^6\,\partial_y$, and the vector field $Y=x^2\,\partial_y$.}
of $D$ (i.e., $Y(q) = \sum_{i=1}^m a_i(q) X_i(q)$ at any $q\in M$, with smooth functions $a_i$), the nilpotentization $\widehat{Y}^{q}$ at $q$ of $X$ is the (nilpotent and complete) vector field on $\R^n$ defined by 
$$
\widehat{Y}^{q} = \lim_{\varepsilon\rightarrow 0\atop \varepsilon\neq 0} Y_\varepsilon^q 
\qquad\textrm{where}\qquad
Y_\varepsilon^q = \varepsilon (\delta_\varepsilon^q)^* Y = \varepsilon \delta_\varepsilon^* \psi^q_* Y .
$$
Actually this convergence is valid in $C^\infty$ topology (uniform convergence of all derivatives on compact subsets of $\R^n$); we also refer to \cite[Section 6.1.1]{CHT_AHL} for stronger results.
Note that $\widehat{Y}^{q}$ is homogeneous of degree $-1$ with respect to dilations, i.e., $\lambda\delta_\lambda^*\widehat{Y}^{q}=\widehat{Y}^{q}$ for every $\lambda\neq 0$,
and that the nonholonomic order of $Y-\widehat{Y}^{q}$ at $q$ is nonnegative.
Actually, writing in $C^\infty$ topology the Taylor expansion $Y=Y^{(-1)}+Y^{(0)}+Y^{(1)}+\cdots$ around $0$, where $Y^{(k)}$ is polynomial and homogeneous of degree $k$ (with respect to dilations), we get that $Y_\varepsilon^q$ has a Taylor expansion at any order $N$ with respect to $\varepsilon$, in $C^\infty$ topology:
\begin{equation*}
Y_\varepsilon^q = \varepsilon (\delta_\varepsilon^q)^* Y = \widehat{Y}^{q} + \varepsilon Y^{(0)} + \varepsilon^2 Y^{(1)} + \cdots + \varepsilon^N Y^{(N-1)} + \mathrm{o}\big(\vert\varepsilon\vert^N\big)
\end{equation*}
with $\widehat{Y}^{q} = Y^{(-1)}$ (see also \cite[Lemma 1]{Ba-13}), i.e., setting $Y_0^q=\widehat{Y}^{q}$ for $\varepsilon=0$, $Y_\varepsilon^q$ depends smoothly on $\varepsilon$ in $C^\infty$ topology. We also have
$Y_\varepsilon^q = \widehat{Y}^{q} + \varepsilon Z_\varepsilon^q$
for every $\varepsilon\in\R$ with $\vert\varepsilon\vert$ small enough so that we are in the chart, where $Z_\varepsilon^q$ is a smooth vector field depending smoothly on $\varepsilon$ in $C^\infty$ topology.

\subsubsection{Nilpotentization of the sR structure}\label{sec_def_nilpotentization}
In the above chart, we have $\widehat{M}^{q}\simeq\R^n$, endowed with the sR structure (denoted by $(\widehat{M}^{q},\widehat{D}^{q},\widehat{g}^{q})$) induced by the vector fields $\widehat{X}^{q}_i$, $i=1,\ldots,m$.
This definition does not depend on the choice of privileged coordinates at $q$ because two sets of such coordinates produce two sR-isometric sR structures. 
This is due to the fact that, since transition maps of charts of privileged coordinates are triangular with respect to the flag, the nilpotentization of any transition map is a sR isometry (see \cite[Proposition 5.20]{Bellaiche}).
Note that the nilpotent sR structure $(\widehat{M}^{q},\widehat{D}^{q},\widehat{g}^{q})$ is homogeneous with respect to the above dilations and that the corresponding sR distance is homogeneous of degree $1$. Moreover, the growth vector of $\widehat{D}^{q}$ coincides with that of $D$ at $q$, and $\mathrm{Lie}(\widehat{X}^{q}_1,\ldots,\widehat{X}^{q}_m)$ is a nilpotent Lie algebra of step $r(q)$.
%
Setting
$$
\widehat{g}^{q} = \lim_{\varepsilon\rightarrow 0\atop \varepsilon\neq 0} g_\varepsilon^q
\qquad \textrm{where}\qquad
g_\varepsilon^q = \varepsilon^{-2} (\delta_\varepsilon^q)^* g ,
$$
we have $\widehat{g}^{q}_x(\widehat{X}^{q}(x),\widehat{Y}^{q}(x)) = g_{q}(X(q),Y(q))$
for every $x\in\R^n$, for all vector fields $X$ and $Y$ on $M$ that are smooth sections of $D$. 

Another geometric identification of $(\widehat{M}^{q},\widehat{D}^{q},\widehat{g}^{q})$ is the following.
Let $\mathcal{G}_{q}$ be the (nilpotent) Lie group of diffeomorphims of $\R^n$ generated by $\exp(t\widehat{X}^{q}_i)$, for $t\in\R$ and $i=1,\ldots,m$. Its Lie algebra is
$$
\mathfrak{g}_{q} = \mathrm{Lie}(\widehat{X}^{q}_1,\ldots,\widehat{X}^{q}_m)
= \bigoplus_{i=1}^{r(q)} \big( \widehat{D}^{q} \big)^i / \big( \widehat{D}^{q} \big)^{i-1}  ,
$$
it is nilpotent, graded, and generated by its first component $\widehat{D}^{q}$. In other words, $\mathcal{G}_{q}$ is a Carnot group (see \cite{Mo-02}).
Under the H\"ormander condition $\mathrm{Lie}(D)=TM$, $\mathcal{G}_{q}$ acts transitively on $\R^n$.
Defining the isotropy group $H_{q}=\{\varphi\in \mathcal{G}_{q}\ \mid\ \varphi(0)=0\}$, of Lie algebra $\mathfrak{h}_{q}=\{Y\in \mathfrak{g}_{q}\ \mid\ Y(0)=0\}$, we identify $\widehat{M}^{q}$ to the homogeneous (coset) space $\mathcal{G}_{q}/H_{q}$. If $q$ is regular then $H_{q}=\{0\}$ and thus $\widehat{M}^{q}\simeq \mathcal{G}_{q}$ is a Carnot group endowed with a left-invariant sR structure.

\begin{remark}
Carnot groups are to sub-Riemannian geometry as Euclidean spaces are to Riemannian geometry.
However, there is a major difference, which is of particular importance here.
In Riemannian geometry, all tangent spaces are isometric, but this is not the case in sub-Riemannian geometry: given two points $q_1$ and $q_2$ of $M$, the nilpotentizations $(\widehat{M}^{q_1},\widehat{D}^{q_1},\widehat{g}^{q_1})$ and $(\widehat{M}^{q_2},\widehat{D}^{q_2},\widehat{g}^{q_2})$ of the sR structure respectively at $q_1$ and $q_2$ may not be sR-isometric, even though the growth vectors at $q_1$ and $q_2$ coincide.\footnote{Actually, the flags of two sR structures coincide at any point if and only if the sR structures are locally Lipschitz equivalent, meaning that the corresponding sR distances satisfy $c_1 d_2(q,q')\leq d_1(q,q')\leq c_2 d_2(q,q')$ for some uniform constants $c_1>0$ and $c_2>0$.
}
There are many algebraically non-isomorphic (and thus non-isometric) $n$-dimensional Carnot groups, and even uncountably many for $n\geq 5$ (due to moduli in their classification). 
We refer to \cite{AgrachevMarigo_2005,Marigo_2007} for a complete classification of rigid and semi-rigid Carnot algebras.

Note that, in dimension three, if the growth vector is $(2,3)$ 
then we have a unique model that is the Heisenberg flat case 
in the equivalence class of sR-isometric Carnot groups.
\end{remark}

\subsubsection{Nilpotentized sR Laplacian}
Let $q\in M$ be arbitrary. Associated with the sR structure $(\widehat{M}^{q},\widehat{D}^{q},\widehat{g}^{q})$, we define on $C^\infty(\widehat{M}^{q})$ the differential operator
\begin{equation}\label{def_triangle_nilp}
\widehat{\triangle}^{q} = \sum_{i=1}^m (\widehat{X}^{q}_i)^2 .
\end{equation}
Under the H\"ormander condition \eqref{Hormander_assumption}, we have as well $\mathrm{Lie}(\widehat{D}^{q})=T\widehat{M}^{q}$ and thus $\widehat{\triangle}^{q}$ is subelliptic.

\subsubsection{Nilpotentization of measures}\label{sec_nilp_mesures}
The nilpotentization of measures is defined by duality of the nilpotentization of functions. 
Let $\mu$ be a smooth measure on $M$ and let $q\in M$. Using a chart $\psi^q$ of privileged coordinates at $q$, the measure $\widehat{\mu}^{q}$ on $\widehat{M}^{q}\simeq\R^n$ (nilpotentization of the measure $\mu$ at $q$) is given by 
$$
\widehat{\mu}^{q} = \lim_{\varepsilon\rightarrow 0\atop \varepsilon\neq 0} \mu_\varepsilon^q
\qquad\textrm{where}\qquad \mu_\varepsilon^q = \vert\varepsilon\vert^{-\mathcal{Q}^M(q)} (\delta_\varepsilon^q)^* \mu
$$
with convergence in the vague topology (i.e., the weak star topology of $C_c(M)'$, where $C_c(M)$ is the set of continuous functions on $M$ of compact support).
Note that, since $(\delta_\varepsilon^q)^{-1}(\BsR(q,\varepsilon))\rightarrow\hatBsR^q(0,1)$ for the Gromov-Hausdorff topology, 
we have
\begin{equation}\label{equiv_mu_ball}
\mu(\BsR(q,\varepsilon))\sim\varepsilon^{\mathcal{Q}^M(q)}\widehat{\mu}^q\big(\hatBsR^q(0,1)\big)
\end{equation}
as $\varepsilon\rightarrow 0^+$ (see also \cite[Remark 3.6]{GhezziJean_TSG2015}).
According to the above definition of the nilpotentization of a measure, if $\mu$ and $\nu$ are two smooth measures on $M$, with $\mu=h\nu$, where $h$ is a positive smooth function on $M$, then $\widehat{\mu}^{q}=h(q)\widehat{\nu}^{q}$. Equivalently, this means that
\begin{equation}\label{density_nilp}
h(q) = \frac{d\mu}{d\nu}(q) = \frac{\widehat{\mu}^{q}}{\widehat{\nu}^{q}}.
\end{equation}
In particular, the nilpotentizations at $q$ of all smooth measures are proportional to the Lebesgue measure $m$ on $\widehat{M}^{q}\simeq\R^n$. If $q$ is regular, then $\widehat{\mu}^{q}$ is a left-invariant measure on the Carnot group $\widehat{M}^{q}$; in this case, $\widehat{M}^{q}$ is a nilpotent Lie group and thus is unimodular, and hence $\widehat{\mu}^{q}$ coincides with the Haar measure, up to scaling.
If $q$ is singular, $\widehat{M}^{q}$ is a homogeneous (quotient) space and $\widehat{\mu}^{q}$ is a left-invariant measure on it.

In passing, note that, applying \eqref{density_nilp} to the measure $\nu=\mathcal{H}_S$ that is the spherical Hausdorff measure and using the fact (proved in \cite{AgrachevBarilariBoscain_CV2012}) that $\widehat{\mathcal{H}_S}^{q}\big(\hatBsR^{q}(0,1)\big)=2^{\mathcal{Q}^M(q)}$, we obtain that the density at $q$ of $\mu$ with respect to the spherical Hausdorff measure is 
$$
h(q) = \frac{d\mu}{d\mathcal{H}_S}(q) = \frac{\widehat{\mu}^{q}\big(\hatBsR^{q}(0,1)\big)}{2^{\mathcal{Q}^M(q)}} .
$$

\begin{remark}\label{rem_nilpLap}
Let $q\in M$ be arbitrary, and let $\mu$ be an arbitrary smooth measure on $M$.
Endowing $\widehat{M}^{q}$ with the lest-invariant measure $\widehat{\mu}^{q}$, we have 
\begin{equation}\label{div_zero}
\mathrm{div}_{\widehat{\mu}^{q}}(\widehat{X}^{q}_i)=0 \qquad \forall i\in\{1,\ldots,m\}.
\end{equation}
Indeed, $\widehat{\mu}^{q}$ is invariant and the vector fields $\widehat{X}^{q}_i$ are the generators of the group action and thus must have a zero divergence.
As a consequence of \eqref{div_zero}, 
we have $(\widehat{X}^{q}_i)^*=-\widehat{X}^{q}_i$, where the transpose is considered in $L^2(\widehat{M}^{q},\widehat{\mu}^{q})$.  It follows that
$$
\widehat{\triangle}^{q} = \sum_{i=1}^m (\widehat{X}^{q}_i)^2 = \sum_{i=1}^m - (\widehat{X}^{q}_i)^*\widehat{X}^{q}_i  .
$$
Due to the cancellation of the divergence term, there are no terms of order one (compare with the general formula for a sR Laplacian, given, e.g., in \cite{CHT-SEDP}).
%
\end{remark}

\subsubsection{Nilpotentization of diffeomorphisms}\label{sec_nilp_diffeo}
Let $M_1$ and $M_2$ be two manifolds of same dimension, and let $D_1=\mathrm{Span}(X)$ (resp., $D_2=\mathrm{Span}(Y)$) be a horizontal distribution on $M_1$ (resp., on $M_2$), with $X=(X_1,\ldots,X_m)$ (resp., $Y=(Y_1,\ldots,Y_m)$). Let $q\in M_1$. 

We assume that the horizontal distributions $D_1$ and $D_2$ are locally diffeomorphic around $q$, i.e., there exists a germ of smooth diffeomorphism $\phi:M_1\rightarrow M_2$ around $q$ such that $\phi_*D_1=D_2$. This means that, for every $j\in\{1,\ldots,m\}$, $\phi_*X_j = \sum_{j=1}^m a_{ij} Y_i$ for some germs at $\phi(q)$ of smooth functions $a_{ij}$ on $M_2$. Then, $\widehat{D}^q_1$ and $\widehat{D}^{\phi(q)}_2$ are diffeomorphic, i.e., $\widehat{\phi}^q_*\widehat{D}^q_1=\widehat{D}^{\phi(q)}_2$, with a diffeomorphism $\widehat{\phi}^q:\widehat{M}^q_1\rightarrow\widehat{M}^q_2$ satisfying $\widehat{\phi}^q_* \widehat{X}^q_j = \sum_{j=1}^m a_{ij}(\phi(q)) \widehat{Y}^{\phi(q)}_i$ for every $i\in\{1,\ldots,m\}$.

Actually, the diffeomorphism $\widehat{\phi}^q$ is the limit as $\varepsilon\rightarrow 0$ of $\phi_\varepsilon^q$, where
$$
\phi_\varepsilon^q = (\delta_\varepsilon^{\phi(q)})^{-1}\circ\phi\circ\delta_\varepsilon^q
= \delta_{1/\varepsilon} \circ \psi_2^{\phi(q)} \circ \phi \circ (\psi_1^q)^{-1} \circ \delta_\varepsilon
$$
for every $\varepsilon>0$, where $\psi_1^q$ (resp., $\psi_2^{\phi(q)}$) is a local chart of privileged coordinates at $q$ (resp., at $\phi(q)$). To see that $\widehat{\phi}^q$, defined as this limit, is indeed a diffeomorphism, it suffices to choose the local charts of privileged coordinates so that $\phi$ is the identity in those coordinates.

Moreover, if $q$ is regular, i.e., if $q\in M\setminus\S$, then the diffeomorphism $\widehat{\phi}^q$ depends smoothly (in $C^\infty$ topology) on $q$ in $M\setminus\S$. More generally, if $M$ is Whitney stratified by equisingular smooth strata then the latter property is satisfied along strata.

\subsection{Uniform ball-box theorem}\label{sec_uniformballbox}
We follow \cite[Chap. 2, Sec. 2.2.2]{Jean_2014} (see also \cite{NSW-85}). 
Considering the $m$-tuple $X=(X_1,\ldots,X_m)$ of vector fields, given an ordered set $I=(i_1,\ldots,i_p)$ of $p$ indices taken in $\{1,\ldots,m\}$, we define the vector field $X_I$ as the Lie bracket of length $p=\vert I\vert$ given by 
$$
X_I = [\cdots[[X_{i_1},X_{i_2}],X_{i_3}],\ldots ,X_{i_p}] .
$$
Let $K$ be a compact subset of $M$ and let $\mathcal{Q}^K_\mathrm{max}$ be the maximum of $\mathcal{Q}^M(q)$ over all $q\in K$.
Let $\mathscr{X}$ be the (finite) set of all $n$-tuples $\mathtt{X}=(X_{I_1},\ldots,X_{I_n})$ such that $\vert\mathtt{X}\vert = \sum_{i=1}^n \vert I_i\vert \leq \mathcal{Q}^K_\mathrm{max}$.
Note that a $n$-tuple $\mathtt{X}\in\mathscr{X}$ of rank $n$ at $q$ is adapted to the sR flag of $D$ at $q$ if and only if $\vert\mathtt{X}\vert=\mathcal{Q}^M(q)$.
Given any $q\in K$ and any $\rho>0$, we define
\begin{equation}\label{def_v}
\mathtt{v}_\mu^{q,\rho}(X) = \sum_{\mathtt{X}\in\mathscr{X}} \rho^{\vert\mathtt{X}\vert} \left\vert {\det}_\mu \left( \mathtt{X}(q) \right) \right\vert .
\end{equation}
The function $q\mapsto \mathtt{v}_\mu^{q,\rho}(X)$ is continuous, and for every $q$ fixed, $\mathtt{v}_\mu^{q,\rho}(X)$ is polynomial in $\rho$, of valuation $\mathcal{Q}^M(q)$ and of degree not greater than $\mathcal{Q}^K_\mathrm{max}$. Actually, following \cite{GhezziJean_NA2015,GhezziJean_TSG2015} and defining
$$
\mathtt{w}_\mu^q(X) = \sum_{\mathtt{X}\in\mathscr{X}, \vert\mathtt{X}\vert = \mathcal{Q}^M(q)} \left\vert {\det}_\mu \left( \mathtt{X}(q) \right) \right\vert  ,
$$
we have $\mathtt{v}_\mu^{q,\rho}(X) \sim \rho^{\mathcal{Q}^M(q)} \, \mathtt{w}_\mu^q(X)$ as $\rho\rightarrow 0$, but this limit is not uniform with respect to $q$ near singular points.
The function $q\mapsto\mathtt{w}^q_\mu(X)$ is positive on $K$ (because there always exists an adapted frame at any point $q$, consisting of Lie brackets), is smooth in the regular region $K\setminus\S$, and is not continuous at singular points (if a singular point $q$ is the limit of regular points $q_k$ then $\mathtt{w}_\mu^{q_k}(X)\rightarrow 0$ while $\mathtt{w}_\mu^q(X)>0$).

Given any $q\in K$, any $\mathtt{X}\in\mathscr{X}$ and any $\rho>0$, we define the $\mathtt{X}$-box 
$$
\mathrm{Box}_{\mathtt{X}}(q,\rho) = \{ \exp(x_1 X_{I_1}) \circ\cdots\circ \exp(x_n X_{I_n})(q)\ \mid\ \vert x_i\vert\leq \rho^{\vert I_i\vert},\ i=1,\ldots,n\} .
$$
By the \emph{uniform ball-box theorem} (see \cite{Jean_2001} or \cite[Theorem 2.4]{Jean_2014}, see also \cite[Proposition A.1]{GhezziJean_NA2015}), there exist $C>0$ and $\rho_0>0$ (depending on $K$) such that, for every $q\in K$ and every $\rho\in(0,\rho_0]$, for every $\mathtt{X}\in\mathscr{X}$ (depending on $q$ and $\rho$) achieving the maximum of $\rho^{\vert\mathtt{X}\vert} \left\vert {\det}_\mu \left( \mathtt{X}(q) \right) \right\vert$ over all $\mathtt{X}\in\mathscr{X}$,
we have 
$$
\mathrm{Box}_{\mathtt{X}}(q,\rho/C) \subset \BsR(q,\rho) \subset \mathrm{Box}_{\mathtt{X}}(q,\rho C) .
$$
As a consequence, there exists $C>0$ (depending on $K$) such that
\begin{equation}\label{mu_ball_eps}
\frac{1}{C}\mathtt{v}_\mu^{q,\rho}(X_1,\ldots,X_m) \leq \mu(\BsR(q,\rho)) \leq C \mathtt{v}_\mu^{q,\rho}(X_1,\ldots,X_m)  \qquad \forall q\in K \qquad\forall \rho\in(0,\rho_0] .
\end{equation}
The above double inequality is the main result of \cite{NSW-85}, from which the authors infer the volume doubling property.
Note that, using \eqref{equiv_mu_ball}, we easily infer (see also \cite[Remark 5.8]{GhezziJean_NA2015}) that there exists $C>0$ (depending on $K$) such that
\begin{equation}\label{equiv_muhat}
\frac{1}{C}\mathtt{w}_\mu^q(X)\leq \widehat{\mu}^q\big(\hatBsR^q(0,1)\big) \leq C \mathtt{w}_\mu^q(X) \qquad \forall q\in K .
\end{equation}
%
In passing, one can note that, applying \eqref{mu_ball_eps} to the $m$-tuple $\widehat{X}^q = (\widehat{X}^q_1,\ldots,\widehat{X}^q_m)$ of vector fields, we obtain that $\mathtt{w}^q_\mu(X)$ is bounded above and below on $K$ by constant multiples of $\mathtt{v}^{0,1}_{\widehat{\mu}^q}(\widehat{X}^q)$.


\subsection{The Popp measure}\label{sec_popp}
Similarly to the fact that any Riemannian manifold has a canonical smooth measure associated with the metric, any sR structure has canonical (intrinsic) measures, i.e., measures that depend only on the sR structure and not on the choice of a system of coordinates or of a local frame.
This question has been addressed in \cite{AgrachevBarilariBoscain_CV2012,Mo-02}.
The $n$-dimensional Hausdorff measure and the $n$-dimensional spherical Hausdorff measure are of course intrinsic measures on a sR structure, because a sR manifold is a metric space for the corresponding sR distance.
The Popp measure, introduced in \cite{Mo-02}, is another canonical measure, associated with the sR metric and with the flag structure, in the equiregular case. It is even ``doubly intrinsic", in the sense that it commutes with nilpotentization, as recalled below.

In $M\setminus\S$, the Popp volume is defined as the inverse image of $\vert\nu_1\wedge\cdots\wedge\nu_r\vert$ under the canonical isomorphism%
\footnote{Indeed, following \cite{ABGR}, considering a basis $(e_1,\ldots,e_n)$ of $T_{q}M$ that is adapted to the flag, that is, such that $e_i\in D^{w_i(q)}(q)$, the wedge product $e_1\wedge\cdots\wedge e_n$ depends only on $e_i\mod D_{q}^{w_i(q)-1}$. This induces the canonical isomorphism.}
$$
\Lambda^n(T^\star_{q} M) \simeq \Lambda^n\bigg(  \bigoplus_{k=1}^{r(q)} D^k(q)/D^{k-1}(q) \bigg)^* ,
$$
where $\nu_k$ is the canonical volume form on $D(q)^k/D(q)^{k-1}$ induced by the Euclidean structure coming from the surjection $D(q)^{\otimes k} \rightarrow D(q)^k/D(q)^{k-1}$ defined with Lie brackets modulo $D(q)$.
The corresponding measure $P$, smooth $M\setminus\S$, is called the \emph{Popp measure}.

By construction, the Popp measure is invariant under local sR isometries. 
Actually, if the group of sR isometries acts transitively on $M$, then the Popp measure is the unique invariant measure, up to scaling (see \cite{BarilariRizzi_AGMS2013}).
If $M$ is a Lie group equipped with a left-invariant sR structure (and thus in particular if $M$ is a Carnot group), then, since the left action is an isometry, the Popp measure is left-invariant.
In this case, by uniqueness (up to scaling) of the Haar measure on a locally compact topological group, the Popp measure is therefore a constant multiple of the Haar measure. In particular, this is the case on the nilpotentization of the sR structure at some given point.

Moreover, 
the construction of the Popp measure commutes with nilpotentization, in the following sense. Let $q\in M\setminus\S$ be arbitrary. We denote temporarily by $P_M$ the Popp measure on $M$ associated with the sR structure $(M,D,g)$, and by $P_{\widehat{M}^{q}}$ the Popp measure on $\widehat{M}^{q}$ associated with the sR structure $(\widehat{M}^{q},\widehat{D}^{q},\widehat{g}^{q})$. Note that, since the sR structure $(\widehat{M}^{q},\widehat{D}^{q},\widehat{g}^{q})$ is a class of equivalence under sR isometries and that the Popp measure is invariant under sR isometries, it follows that $P_{\widehat{M}^{q}}$ is an intrinsic measure on $\widehat{M}^{q}$.
Let $\psi^q$ be a chart of local privileged coordinates at $q$ (note that this is as well a local sR isometry), and let $\widehat{P_M}^{q}$ be the nilpotentization of $P_M$ at $q$ in this chart. Then $\widehat{P_M}^{q} = P_{\widehat{M}^{q}}$. 
Since it is an intrinsic object, we simply denote it by $\widehat{P}^{q}$.


When the singular set $\S$ is stratified by equisingular smooth submanifolds, the Popp measure can also be defined along each stratum $\S_j$ of $\S$ (see \cite{GhezziJean_TSG2015}): the construction is the same as above, considering the sR flags of $D\cap T\S_j$, and gives a smooth measure on each stratum.

As noticed in \cite[Corollary 4.5]{GhezziJean_TSG2015}, it is remarkable that, given any compact subset $K$ of $M$, there exists $C>0$ (depending on $K$) such that $\frac{1}{C}\leq\widehat{P}^{q}(\widehat{B}^q(0,1))\leq C$ for every $q\in M\setminus\S$ (even near $\S$). Indeed, it follows from the explicit expression of the Popp measure given in \cite{BarilariRizzi_AGMS2013} that, in $M\setminus\S$, the Radon-Nikodym derivative $\frac{dP}{d\mu}$ is bounded above and below, up to scaling, by $\frac{1}{\mathtt{w}_\mu}$. The claim then follows from \eqref{density_nilp} and \eqref{equiv_muhat}.


In Section \ref{sec_Weyl_measures}, we introduce a new canonical sR measure, that we call the \emph{Weyl measure}, which is of a spectral nature, in contrast to the Popp measure that is of algebraic nature.

\subsection{Schwartz kernels, heat kernels}\label{appendix_Schwartz}
We set $\mathscr{D}(M)=C_c^\infty(M)$ and we denote by $\mathscr{D}'(M)$ the  space of distributions on $M$, i.e., the topological dual of $\mathscr{D}(M)$ endowed with the weak topology.
Let $\mu$ be a smooth measure on $M$.

\subsubsection{Schwartz kernels}
According to the Schwartz kernel theorem, there is a linear bijection between $\mathscr{D}'(M\times M)$ and the set of bilinear continuous functionals on $\mathscr{D}(M)\times \mathscr{D}(M)$.
Given a linear continuous mapping $A: \mathscr{D}(M)\rightarrow \mathscr{D}'(M)$, the Schwartz kernel of $A$ is the unique distribution $[A]\in \mathscr{D}'(M\times M)$ defined by 
$$
\langle Af,g\rangle_{\mathscr{D}'(M),\mathscr{D}(M)} = \langle [A],g\otimes f\rangle 
\qquad \forall f,g\in \mathscr{D}(M)
$$
where $\langle\cdot,\cdot\rangle$ is the duality bracket.

When $[A]\in C^0(M\times M)$, identifying the distribution bracket by an integral with respect to the measure $\mu\otimes\mu$ and denoting by $[A]_\mu$ the density function, we have the familiar formula 
$$
Af(q) = \int_M [A]_\mu(q,q') f(q')\, d\mu(q') \qquad \forall q\in M\qquad \forall f\in\mathscr{D}(M).
$$
We stress that, although the density function $[A]_\mu$ depends on $\mu$, given any $q\in M$, the absolutely continuous measure $[A]_\mu(q,\cdot)\, d\mu(\cdot)$ depends only on $A$: it does not depend on the smooth measure $\mu$, in the sense that 
$$
[A]_\mu(q,\cdot)\, d\mu(\cdot) = [A]_\nu(q,\cdot)\, d\nu(\cdot)
$$
for any other smooth measure $\nu$ on $M$.

Actually, in geometric terms, $[A]$ is a continuous section of the bundle $\pi_2^*(\Omega_M)$ on $M\times M$, where $\Omega_M$ is the line bundle of smooth measures (densities) on $M$ and $\pi_2:M\times M\rightarrow M$ is the projection defined by $\pi_2(q,q')=q'$.

Similarly, the diagonal part $[A]_\mu(q,q)\, d\mu(q)$ is an absolutely continuous measure, which does not depend on $\mu$. Denoting by $\mathcal{M}_f$ the operator of multiplication by $f$, we have
$$
\Tr(A \mathcal{M}_f) = \int_M [A]_\mu(q,q) f(q)\, d\mu(q) \qquad\forall f\in \mathscr{D}(M).
$$

\subsubsection{Heat kernels}
Let $A:D(A)\rightarrow L^2(M,\mu)$ be a densely defined operator on $L^2(M,\mu)$, generating a strongly continuous semigroup $(e^{tA})_{t\geq 0}$. 
For every $t>0$, the heat kernel $e_A(t)$ associated with $A$ is the measure on $M$ defined as the Schwartz kernel of $e^{tA}$, i.e., $e_A(t)=[e^{tA}]$. 
Of course, it does not depend on $\mu$.


When this measure has a density $[e^{tA}]_\mu$ with respect to $\mu$ which is locally integrable, we define the heat kernel $e_{A,\mu}(t,\cdot,\cdot)$ associated with $A$ and with the measure $\mu$ by $e_{A,\mu}(t,q,q') = [e^{tA}]_\mu(q,q')$. This means that 
$$
u(t,q) = (e^{tA}f)(q) = \int_M e_{A,\mu}(t,q,q')f(q')\, d\mu(q')
$$
is the unique solution to $\partial_t u-Au=0$ for $t>0$, $u(0,\cdot)=f(\cdot)$, for every $f\in\mathscr{D}(M)$.
In other words, we have
$$
e_A(t)(q,q') = [e^{tA}](q,q') = e_{A,\mu}(t,q,q')\, d\mu(q') \qquad\forall t>0\qquad \forall q,q'\in M.
$$
As said above, this expression depends only on $A$, not on the smooth measure $\mu$.

\medskip

Extending $e_{A,\mu}$ by $0$ for $t<0$, for any fixed $q'\in M$ the mapping $(t,q)\mapsto e_{A,\mu}(t,q,q')$ is also solution of $(\partial_t-A)e_{A,\mu}(\cdot,\cdot,q')=\delta_{(0,q')}$ in the sense of distributions, where the distribution pairing is considered with respect to the measure $dt\times d\mu(q)$ on $\R\times M$. 

We gather hereafter some useful facts.

\begin{itemize}[leftmargin=*,parsep=0cm,itemsep=0cm,topsep=0cm]
\item
Let $\varphi:M\rightarrow M$ be a diffeomorphism, representing a change of variable in the manifold $M$. We have $\varphi^*\mu = \vert J_\mu(\varphi)\vert\mu$, where $J_\mu(\varphi)$ is the Jacobian of $\varphi$ with respect to $\mu$, and where $\varphi^*\mu$ is the pullback of $\mu$ under $\varphi$.
Then
\begin{equation}\label{formulas_kernel}
\begin{split}
e_{\varphi^*A\varphi_*,\mu}(t,q,q') &= \vert J_\mu(\varphi)(q')\vert\, e_{A,\mu}(t,\varphi(q),\varphi(q'))  \\
e_{A,\varphi^*\mu}(t,q,q') &= \frac{1}{\vert J_\mu(\varphi)(q')\vert}\, e_{A,\mu}(t,q,q')  \\
e_{\varphi^*A\varphi_*,\varphi^*\mu}(t,q,q') &= e_{A,\mu}(t,\varphi(q),\varphi(q'))
\end{split}
\end{equation}
for every $t>0$ and all $q,q'\in M$.
Note that the last one follows from the two first ones, in which we have replaced $A$ with $\varphi^*A\varphi_*$ in the second one.
The two first formulas in \eqref{formulas_kernel} are not symmetric, but there is no contradiction there: indeed if $A$ is selfadjoint in $L^2(M,\mu)$ then $e_{A,\mu}$ is symmetric, but $A$ need not be selfadjoint in $L^2(M,\varphi^*\mu)$ and thus $e_{A,\varphi^*\mu}$ need not be symmetric.
Actually, given any other smooth measure $\nu$ on $M$, we have
$$
e_{\varphi^*A\varphi_*,\mu}(t,q,q')\, d(\varphi^*\mu)(q') = e_{\varphi^*A\varphi_*,\nu}(t,q,q')\, d\nu(q') .
$$

\item As a particular case, given any $\lambda>0$, we have
$\lambda\, e_{A,\lambda\mu} = e_{A,\mu} $.

\item
Given any $\varepsilon>0$, the kernel associated with $\varepsilon^2 A$ and with the measure $\nu$ is 
\begin{equation*}
e_{\varepsilon^2 A, \nu}(t,q,q') = e_{A,\nu}(\varepsilon^2 t,q,q')
\qquad \forall t>0\qquad \forall q,q'\in M .
\end{equation*}

\item
We assume that $\mu=h\nu$ with $h$ a positive smooth function on $M$ (density of $\mu$ with respect to $\nu$). Then
$h(q')e_{A,\mu}(t,q,q')=e_{A,\nu}(t,q,q')$
for all $(t,q,q')\in(0,+\infty)\times M\times M$, or equivalently,
\begin{equation*} 
e_{A,\mu}(t,q,q')\, d\mu(q') = e_{A,\nu}(t,q,q')\, d\nu(q').
\end{equation*}



\item 
Let $(M_1,\mu_1)$ and $(M_2,\mu_2)$ be smooth manifolds with smooth measures. For $i=1,2$, let $A_i:D(A_i)\rightarrow L^2(M_i,\mu_i)$ be a densely defined operator on $L^2(M_i,\mu_i)$, assumed to generate a heat kernel that has a density $e_{\triangle_i,\mu_i}=e_{\triangle_i,\mu_i}(t,q_i,q_i')$ with respect to $\mu_i$ which is locally integrable. We define $M=M_1\times M_2$ and $\mu=\mu_1\otimes\mu_2$ and we consider on $M$ the operator 
$$
A=(A_1)_{q_1}+(A_2)_{q_2} = A_1\otimes\mathrm{id}_{M_2} + \mathrm{id}_{M_1} \otimes A_2 .
$$
We have
\begin{equation}\label{tensorproduct}
e_{A,\mu}(t,(q_1,q_2),(q_1',q_2')) = e_{A_1,\mu_1}(t,q_1,q_1') \, e_{A_2,\mu_2}(t,q_2,q_2')
\end{equation}
for every $t>0$, all $q_1,q_1'\in M_1$ and all $q_2,q_2'\in M_2$.
\end{itemize}

\subsubsection{Sub-Riemannian heat kernels}\label{app_sR_kernel}
In this paper, the above facts facts are applied to the nonpositive selfadjoint operator $\triangle:D(\triangle)\rightarrow L^2(M,\mu)$ (defined by \eqref{app_def_triangle}), or to the operator $\widehat{\triangle}^{q}$ (nilpotentization of $\triangle$, defined by \eqref{def_triangle_nilp}) defined on $D(\widehat{\triangle}^{q}) = \{ f\in L^2(\widehat M^{q},\widehat\mu^{q})\ \mid\ \widehat{\triangle}^{q} f\in L^2(\widehat M^{q},\widehat\mu^{q}) \}$. Note that, according to Remark \ref{rem_nilpLap}, the operator $\widehat{\triangle}^{q}:D(\widehat{\triangle}^{q})\rightarrow L^2(\widehat M^{q},\widehat\mu^{q})$ is essentially selfadjoint, and since $\widehat M^{q}$ is complete (indeed, sR balls of small radius are compact, and $\widehat M^{q}$ is invariant by dilatations) $\widehat{\triangle}^{q}$ is selfadjoint (see \cite{Strichartz_JDG1986}). Therefore, both operators generate strongly continuous contraction semigroups.

Under the H\"ormander condition $\mathrm{Lie}(D)=TM$, the operators $\partial_t-\triangle$ and $\partial_t-\widehat{\triangle}^{q}$ are hypoelliptic and therefore the corresponding heat kernels have smooth densities: we denote by $e=e_{\triangle,\mu}$ the density of the heat kernel of $\triangle$ with respect to $\mu$, defined on $(0,+\infty)\times M\times M$, and by $\widehat{e}^q = e_{\widehat{\triangle}^{q},\widehat{\mu}^{q}}$ the density of the heat kernel of $\widehat{\triangle}^q$ with respect to $\widehat{\mu}^q$, defined on $(0,+\infty)\times\widehat{M}^{q}\times\widehat{M}^{q}$.
By the maximum principle for hypoelliptic operators (see \cite{Bony}), the smooth functions $e$ and $\widehat{e}^q$ are positive symmetric 
(see \cite{Strichartz_JDG1986}).

Note that the nilpotentized heat kernel $\widehat{e}^q$ satisfies the homogeneity property
\begin{equation}\label{homog_ehat}
\vert\varepsilon\vert^{\mathcal{Q}^M(q)} \, \widehat{e}^q(\varepsilon^2t,\delta_\varepsilon(x),\delta_\varepsilon(x')) = \widehat{e}^q(t,x,x')
\end{equation}
for every $\varepsilon\in\R\setminus\{0\}$ and for all $(t,x,x')\in(0,+\infty)\times\R^n\times\R^n$.

\paragraph{Exponential estimates for sR heat kernels.}
It is well known that, for every compact subset $K$ of $M$, for every $\varepsilon\in(0,1]$ and for every $T>0$, there exist $C_1,C_2>0$ such that
\begin{equation}\label{exp_estimates}
\frac{C_1}{\mu(\BsR(q,\sqrt{t}))} \exp \left(-\frac{\dsR(q,q')^2}{(4-\varepsilon)t}\right) \leq e(t,q,q') \leq \frac{C_2}{\mu(\BsR(q,\sqrt{t}))} \exp \left( -\frac{\dsR(q,q')^2}{(4+\varepsilon)t} \right)
\end{equation}
and
\begin{equation}\label{upper_estim_kernel}
\left\vert \partial_t^m X_q^I e(t,q,q') \right\vert \leq \frac{1}{t^{m+\vert I\vert/2}} \frac{C_2}{\mu(\BsR(q,\sqrt{t}))} \exp \left( -\frac{\dsR(q,q')^2}{(4+\varepsilon)t} \right)
\end{equation}
for all $q,q'\in K$, for every $t\in(0,T)$, for all $i_1,\ldots,i_s\in\{1,\ldots,m\}$ and for every $s\in\N^*$, 
where $X^I = X_{i_1}\cdots X_{i_s}$, $I=(i_1,\ldots,i_s)$ and $\vert I\vert=s$. Here, $X^I_q$ means that the derivation $X^I$ is applied with respect to the variable $q$.
In particular, 
\begin{equation}\label{exp_estimates_diago}
\frac{C_1}{\mu(\BsR(q,\sqrt{t}))} \leq e(t,q,q) \leq \frac{C_2}{\mu(\BsR(q,\sqrt{t}))} 
\qquad \forall q\in K\qquad \forall t\in(0,T].
\end{equation}
These exponential estimates have been established, e.g., in \cite{CoulhonSikora_PLMS2008, Je-Sa-86, KusuokaStroock, Saloff-Coste_IMRN1992, Sac-84, Varopoulos} (see also \cite[Appendix C]{CHT_AHL} for a survey and more results on this issue). 
Actually, the estimates \eqref{exp_estimates_diago} along the diagonal imply the general estimates \eqref{exp_estimates} by standard considerations.

\section{Parameter-dependent sR heat kernels}\label{app_parameter_kernels}
We recall here results that have been established in \cite{CHT_AHL} in the more general framework of H\"ormander operators. Hereafter, to avoid technicalities, we specify the statements to sR Laplacians.

Let $M$ be a smooth connected manifold and let $\Omega$ be an open subset of $M$.
Let $m\in\N^*$ and let $\mathcal{K}$ be a compact set. For every $\tau\in\mathcal{K}$, 
let $\mu^\tau$ be a smooth density on $M$,
let $X_1^\tau,\ldots,X_m^\tau$ be smooth vector fields on $M$, all of them depending continuously on $\tau$ in $C^\infty$ topology. We denote by $g^\tau$ the corresponding sR metric. We consider the sR Laplacian
\begin{equation*}
\triangle^\tau = \triangle_{g^\tau,\mu^\tau} = -\sum_{i=1}^m (X_i^\tau)^*X_i^\tau = \sum_{i=1}^m \big( (X_i^\tau)^2 + \mathrm{div}_{\mu^\tau}(X_i^\tau) X_i^\tau \big)
\end{equation*}
where the star is the transpose in $L^2(M,\mu^\tau)$.
We assume that the Lie algebra $\mathrm{Lie}(X_1^\tau,\ldots,X_m^\tau)$ generated by the vector fields is equal to $T_qM$ at any point $q\in M$, with a degree of nonholonomy that is uniform with respect to $\tau\in\mathcal{K}$ (\emph{uniform H\"ormander condition}).

We still denote by $\mu^\tau$ the volume induced on $\Omega$.
Let $D(\triangle^\tau)$ be a subset of $\{ f\in L^2(\Omega,\mu^\tau) \mid (\triangle^\tau f)_{\vert \Omega} \in L^2(\Omega,\mu^\tau) \}$, standing for a domain of $\triangle^\tau$ for which $(\triangle^\tau,D(\triangle^\tau))$ is selfadjoint and generates a strongly continuous contraction semigroup on $L^2(\Omega,\mu^\tau)$ and thus a smooth positive symmetric heat kernel $e^\tau(t,q,q')=e_{\triangle^\tau,\mu^\tau}(t,q,q')$ on $(0,+\infty) \times\Omega\times\Omega$.

\subsection{Hypoelliptic Kac's principle}\label{app_kac}
Let $\Omega_1$ be another arbitrary open subset of $M$. We define the operator $\triangle^\tau_1$ on $L^2(\Omega_1,\mu^\tau)$ exactly as we did above on $\Omega$, so that the selfadjoint operator $(\triangle^\tau_1,D(\triangle^\tau_1))$ generates another smooth positive symmetric heat kernel $e^\tau_1(t,q,q')=e_{\triangle^\tau_1,\mu^\tau}(t,q,q')$ on $(0,+\infty) \times\Omega_1\times\Omega_1$.

\begin{theorem}[{\cite[Theorem 3.1]{CHT_AHL}}]\label{thm_Kac}
For all $(k,\alpha,\beta)\in \N\times\N^d\times\N^d$, we have
$$
(\partial_t^k \partial_q^\alpha \partial_{q'}^\beta e^\tau) (t,q,q') = (\partial_t^k \partial_q^\alpha \partial_{q'}^\beta e^\tau_1) (t,q,q') + \mathrm{O}(t^\infty) 
$$
as $t\rightarrow 0^+$, uniformly with respect to $\tau\in\mathcal{K}$ and to $q,q'$ varying in any compact subset of $\Omega\cap \Omega_1$.
\end{theorem}

Theorem \ref{thm_Kac} reflects Kac's principle of ``not feeling the boundary", showing that the small-time asymptotic behavior of the heat kernel is purely local. The above version is moreover uniform with respect to parameters.  

This result, which follows from uniform local subellipticity estimates, is particularly useful to develop local arguments using the heat kernel in small time.

\subsection{Continuity with respect to parameters}\label{app_smooth}

\begin{theorem}[{\cite[Theorem 3.2]{CHT_AHL}}]\label{thm_parameter_heat_smooth}
The heat kernel $e^\tau$ is smooth on $(0,+\infty)\times\Omega\times\Omega$, for every $\tau\in\mathcal{K}$, and depends continuously on $\tau\in\mathcal{K}$ in $C^\infty((0,+\infty)\times\Omega\times\Omega)$ topology.
\end{theorem}

Theorem \ref{thm_parameter_heat_smooth} is obtained, first, by applying the Trotter-Kato theorem in general semigroup theory (see, e.g., \cite[Chapter III]{EngelNagel}), which gives dependence in $C^{-\infty}$ for the weak-star topology; second, dependence in $C^\infty$ topology is obtained by using the Heine-Borel property, because, by uniform subellipticity, the family $(e^\tau)_{\tau\in\mathcal{K}}$ is bounded.

This result can be applied to singular perturbations $\triangle^\tau$ of $\triangle^0$. It generalizes to hypoelliptic operators well known results established for elliptic operators (see, e.g., \cite{Lions_perturbations}).

\section{Small-time asymptotic expansion of sR heat kernels near the diagonal}\label{app_lemfondam}
Using the notations and assumptions made in Appendix \ref{sec_sR_geom}, let $q\in M$ be arbitrary (regular or not) and let $U$ be a relatively compact open connected neighborhood of $q$ in $M$. 
Recall that $e=e_{\triangle,\mu}$ is the sR heat kernel associated with the sR Laplacian $\triangle$ defined by \eqref{app_def_triangle} (see Appendix \ref{app_sR_kernel}).

Hereafter, we identify $\widehat{M}^{q}\simeq\R^n$ (with a sR isometry).

Let $\psi^q:U\rightarrow V\subset\R^n$ be a chart of privileged coordinates at $q$, such that $\psi^q(q)=0$, 
where $U$ is a neighborhood of $q$ in $M$ and $V$ is a neighborhood of $0$ in $\R^n$.
%
Let $\varepsilon_0>0$ be small enough such that $\delta_\varepsilon^q(V)\subset U$ (and $\delta_\varepsilon(V)\subset V$) for every $\varepsilon\in(-\varepsilon_0,\varepsilon_0)$, where the dilations are defined by \eqref{def_deltaepsilon}. 
By definition of the nilpotentization, for every $i\in\{1,\ldots,m\}$, the vector field $(X_i)_\varepsilon^q = \varepsilon(\delta_\varepsilon^q)^*X_i$ converges to $\widehat{X}^q_i$ in $C^\infty$ topology as $\varepsilon\rightarrow 0$ (see Section \ref{sec_nilp_smoothsection}). 
Also, the metric $g_\varepsilon^q = \varepsilon^{-2}(\delta_\varepsilon^q)^*g$ converges to the nilpotentized metric $\widehat{g}^q$ (see Section \ref{sec_def_nilpotentization}), and the smooth measure $\mu_\varepsilon^q = \vert\varepsilon\vert^{-\mathcal{Q}^M(q)}(\delta_\varepsilon^q)^*\mu$ converges to the nilpotentized measure $\widehat{\mu}^q$ (see Section \ref{sec_nilp_mesures}).
Hence the operator 
$$
\triangle_\varepsilon^q 
= \varepsilon^2(\delta_\varepsilon^q)^*\triangle(\delta_\varepsilon^q)_*  
= \sum_{i=1}^m \Big( \big( (X_i)_\varepsilon^q \big)^2 + \mathrm{div}_{\mu_\varepsilon^q} \big( (X_i)_\varepsilon^q \big) \, (X_i)_\varepsilon^q  \Big)
$$
converges to $\widehat{\triangle}^{q} = \sum_{i=1}^m \big( \widehat{X}^q_i\big)^2$ in $C^\infty$ topology (use Remark \ref{rem_nilpLap}). 

Extending the vector fields $(X_i)_\varepsilon^q$ (and thus the differential operator $\triangle_\varepsilon^q$) and the measure $\mu_\varepsilon^q$ by $0$ outside of the neighborhood $V$, we obtain a selfadjoint operator $(\triangle_\varepsilon^q,D(\triangle_\varepsilon^q))$ on $L^2(\R^n,\mu_\varepsilon^q)$, which generates a strongly continuous contraction semigroup. Its Schwartz kernel restricted to $(0,+\infty)\times V\times V$ has a smooth density, which is the smooth positive symmetric heat kernel denoted by $e_\varepsilon^q = e_{\triangle_\varepsilon^q, \mu_\varepsilon^q}$. The subscript $q$ underlines that the nilpotentization is performed at the point $q$.
It follows from Theorem \ref{thm_Kac} (in Appendix \ref{app_kac}) that the way we extend has no impact on the small-time asymptotics of the heat kernel. Therefore we have
\begin{equation}\label{relation_eeps_e}
e_\varepsilon^q(s,x,x') = \vert\varepsilon\vert^{\mathcal{Q}^M(q)}\, e(\varepsilon^2s, \delta_\varepsilon^q(x), \delta_\varepsilon^q(x')) + \mathrm{O}(\vert\varepsilon\vert^\infty)
\end{equation}
as $\varepsilon\rightarrow 0$, in $C^\infty$ topology.
Besides, by Theorem \ref{thm_parameter_heat_smooth} (in Appendix \ref{app_smooth}), $e_\varepsilon^q$ converges to $\widehat{e}^q$ in $C^\infty$ topology as $\varepsilon\rightarrow 0$.
Hence, at this step, we have obtained that
\begin{equation}\label{lim_eeps}
\lim_{\varepsilon\rightarrow 0} e_\varepsilon^q(s,x,x') = \lim_{\varepsilon\rightarrow 0} \vert\varepsilon\vert^{\mathcal{Q}^M(q)}\, e(\varepsilon^2s, \delta_\varepsilon^q(x), \delta_\varepsilon^q(x')) = \widehat{e}(s,x,x')
\end{equation}
uniformly with respect to $(s,x,x')$ on every compact subset of $(0,+\infty)\times\R^n\times\R^n$. 
Moreover, when $q$ is regular, taking smaller neighborhoods $U$ and $V$ if necessary so that every point of $U$ is regular, all operators and functions above depend smoothly on $q$ and the convergence \eqref{lim_eeps} is uniform with respect to $q$. When $q$ is singular and when the singular set is Whitney stratified, the same conclusion is true along each stratum along which the sR weights remain constant.
In other words, setting $X_i^0 = \widehat{X}^q_i$, $g^0=\widehat{g}^q$, $\mu^0=\widehat{\mu}^q$, $\triangle^0=\widehat{\triangle}^q$ and $e^0_q=\widehat{e}^q$, we have the following result.

\begin{lemma}\label{lem_eeps}
The family $(e_\varepsilon^q)_{\varepsilon\in[-\varepsilon_0,\varepsilon_0]}$ depends continuously on $\varepsilon$ in $C^\infty((0,+\infty)\times V\times V)$ topology.
If the singular set (set of singular points) is Whitney stratified by equisingular smooth strata, then the family depends continuously on $(\varepsilon,q)$ along each stratum.
\end{lemma}

\begin{remark}\label{rem_eeps}
In particular, the function $\varepsilon\mapsto e_\varepsilon^q(1,0,0)$ (resp., the function $(\varepsilon,q)\mapsto e_\varepsilon^q(1,0,0)$ along strata of the singular set)
is continuous, and its value at $\varepsilon=0$ is $\widehat{e}^q(1,0,0)$. 
\end{remark}

These facts are developed in detail in \cite[Part I]{CHT_AHL}.
In particular, \eqref{lim_eeps} gives the first term of the small-time asymptotic expansion of the heat kernel. Smoothness with respect to $\varepsilon$ and the complete expansion, which are much more difficult to obtain, are given in the next theorem, which is the main result of \cite{CHT_AHL}.

\begin{theorem}[\cite{CHT_AHL}]\label{lemfondamental}
%
The family $(e_\varepsilon^q)_{\varepsilon\in[-\varepsilon_0,\varepsilon_0]}$ depends smoothly on $\varepsilon$ in $C^\infty((0,+\infty)\times V\times V)$ topology.
Given any $N\in\N^*$, 
we have the asymptotic expansion in $C^\infty((0,+\infty)\times V\times V)$
\begin{equation}\label{complete_expansion}
\begin{split}
e_\varepsilon^q(s,x,x') &= \vert\varepsilon\vert^{\mathcal{Q}^M(q)}\, e(\varepsilon^2s, \delta_\varepsilon^q(x), \delta_\varepsilon^q(x')) + \mathrm{O}(\vert\varepsilon\vert^\infty) \\
&= \widehat{e}^{q}(s,x,x') + \sum_{i=1}^N \varepsilon^i f^{q}_i(s,x,x') + \mathrm{o}\big(\vert\varepsilon\vert^N\big) 
\end{split}
\end{equation}
as $\varepsilon\rightarrow 0$, 
where the functions $f^{q}_i$ are smooth and satisfy the homogeneity property
\begin{equation}\label{homog_property_f_i}
f^q_i(s,x,x') = \varepsilon^{-i} \vert\varepsilon\vert^{\mathcal{Q}^M(q)} f^q_i(\varepsilon^2 s,\delta_\varepsilon(x),\delta_\varepsilon(x')) 
\end{equation}
for all $(s,x,x')\in(0,+\infty)\times\R^n\times\R^n$ and for every $\varepsilon\neq 0$. In particular, $f_i^q(s,0,0)=0$ if $i$ is odd.

Taking $s=1$, $\varepsilon=\sqrt{t}$ and setting $a_i^{q}(x,x')=f^{q}_i(1,x,x')$, it follows that, given any $N\in\N$, 
we have the asymptotic expansion in $C^\infty(V\times V)$
\begin{equation}\label{complete_expansion_1}
t^{\mathcal{Q}^M(q)/2}\, e\left(t, \delta_{\sqrt{t}}^q(x), \delta_{\sqrt{t}}^q(x')\right) 
= \widehat{e}^{q}(1,x,x') + \sum_{i=1}^N t^{i/2} a_i^{q}(x,x')  + \mathrm{o}(t^{N/2})  
\end{equation}
as $t\rightarrow 0^+$, where the functions $a_i^{q}$ are smooth and satisfy $a_{2j-1}^{q}(0,0)=0$ for every $j\in\N^*$.

Moreover, if $q$ is regular, then the above convergence and asymptotic expansion are locally uniform with respect to $q$, and the functions $\widehat{e}^{q}$, $f^{q}_i$ and $a_i^{q}$ depend smoothly (in $C^\infty$ topology) on $q$ in any open neighborhood of $q$ consisting of regular points. If the manifold $M$ is Whitney stratified by equisingular smooth strata (i.e., the sR weights $w_1(q),\ldots, w_n(q)$ are constant along each stratum) then the latter property is satisfied along strata.
\end{theorem}

\begin{remark}
As a particular case, take $x=x'=0$ in \eqref{complete_expansion_1} and set $c_j(q) = a_{2j}^{q}(0,0)$.
Since $a_{2j-1}^{q}(0,0)=0$, it follows that, given any $N\in\N$, for every $q\in M$,
\begin{equation}\label{expansion_along_diagonal}
t^{\mathcal{Q}^M(q)/2}\, e(t,q,q) \\
= \widehat{e}^{q}(1,0,0) + c_1(q) t + \cdots + c_N(q) t^N + \mathrm{o}(t^{N})  
\end{equation}
as $t\rightarrow 0^+$. Moreover if $q$ is regular then the functions $c_j$ are smooth locally around $q$.
This small-time expansion of the heat kernel along the diagonal was already known (see \cite{BenArous_AIF1989}, see also \cite{Metivier1976}). The equivalent $t^{\mathcal{Q}^M(q)/2} e(t,q,q)\sim \widehat{e}^q(1,0,0)$ gives the main term in the local Weyl law in the equiregular case. 

The expansion \eqref{complete_expansion} is more general because, in addition to have identified the main coefficients in terms of the nilpotentization, the expansion is valid in an \emph{asymptotic neighborhood of the diagonal}, which is instrumental to derive the microlocal Weyl law and to treat the case of singular sR structures, as done in the present paper. 
\end{remark}

\section{Asymptotic expansions of some integrals}\label{app_integral}
\subsection{Integrals with a single layer}
In the following, given any $G\in C^\infty(\R^2)$, any $k\in\Z$ and any $j\in\N$, we define
$$
I_k^j[G](x) = \int_x^1 \tau^k\, G\left( \tau, \frac{x}{\tau} \right) \ln^j\frac{\tau}{x}\, d\tau  \qquad\forall x\in (0,1).
$$

\begin{proposition}\label{prop_general_expansion}
We have
\begin{equation}\label{expansionIkj}
I_k^j[G](x) = x^{\min(0,k+1)} F_0(x) + x^{\max(0,k+1)} \sum_{i=1}^{j+1} F_i(x) \ln^i\frac{1}{x}  \qquad\forall x\in (0,1)
\end{equation}
for some $F_0, \ldots, F_{j+1}\in\C^\infty(\R)$, and more precisely,
$$
I_k^j[G](x) =  \left\{ \begin{array}{lll} 
\displaystyle \sum_{i=0}^k \frac{x^i}{i!} \int_0^1 \tau^{k-i} \, \partial_2^iG(\tau,0) \ln^j\frac{\tau}{x}\, d\tau + x^{k+1} \sum_{i=0}^{j+1} F_i(x) \ln^i\frac{1}{x}  & \textrm{if} & k\in\N , \\[2mm]
\displaystyle \frac{G(0,0)}{j+1}\ln^{j+1}\frac{1}{x}  + \int_0^1 \frac{G(0,\varepsilon)-G(0,0)}{\varepsilon} \ln^j\frac{1}{\varepsilon} \, d\varepsilon  && \\[1mm]
\displaystyle \qquad\qquad\qquad\quad  + \int_0^1 \frac{G(\tau,0)-G(0,0)}{\tau} \ln^j\frac{\tau}{x} \, d\tau + x \sum_{i=0}^{j+1} \tilde F_i(x) \ln^i\frac{1}{x}  & \textrm{if} & k=-1 , \\[2mm]
\displaystyle \sum_{i=0}^{-k-2} \frac{x^{k+1+i}}{i!} \int_0^1 \varepsilon^{-k-2-i}\, \partial_1^iG(0,\varepsilon) \ln^j\frac{1}{\varepsilon} \, d\varepsilon + \sum_{i=0}^{j+1} F_i(x) \ln^p\frac{1}{x}  & \textrm{if} & k\leq -2 ,
\end{array}\right.
$$
for every $x\in(0,1)$, for some $\tilde F_0, \ldots, \tilde F_{j+1}\in\C^\infty(\R)$. 
The function $I_k^j[G](x)$ can be written as a sum (over $p$) of formal series as $x\rightarrow 0^+$, multiplied by $\ln^i\frac{1}{x}$, where the coefficients are distributions.

When $G$ is only continuous, we get only the first term in the asymptotics as $x\rightarrow 0^+$. When $G$ is only of class $C^1$ and $k=-1$, we get only the three first terms (i.e., we replace the sum involving the $\tilde F_i$'s with a remainder term). 

Moreover, for every $i\in\N$, we have
\begin{equation}\label{Fj+1}
F_{j+1}^{(i)}(0) = \left\{ \begin{array}{lll}
\displaystyle \frac{1}{j+1}\frac{1}{(k+1+2i)!} \binom{k+1+2i}{i} \partial_1^i \partial_2^{k+1+i}G(0,0)  & \textrm{if} & k\in\N , \\[4mm]
\displaystyle \frac{1}{j+1}\frac{1}{(-k-1+2i)!} \binom{-k-1+2i}{i} \partial_1^{-k-1+i} \partial_2^{i}G(0,0)  & \textrm{if} & k\in-\N^* .
\end{array}\right.
\end{equation}
\end{proposition}

\begin{remark}\label{rem_prop_general_expansion_even}
If $G$ is even with respect to its second variable (i.e., if $G(\tau,\varepsilon)$ is a function of $(\tau,\varepsilon^2)$), then $\partial_2^iG(\tau,0)=0$ for every $i\in\N$ odd and hence many terms vanish in the above expansion; in particular, \eqref{Fj+1} implies that $F_{j+1}$ has the parity of $k+1$ if $k\in\N$ and is even if $k\in-\N^*$.
\end{remark}

\begin{remark}\label{rem_prop_general_expansion_tau}
If $G$ does not depend on $\tau$ then $\partial_1^iG(\tau,\varepsilon)=0$ for every $i\in\N^*$, and, by \eqref{Fj+1}, $F_{j+1}(x)=F_{j+1}(0)+\mathrm{O}(x^\infty)$ with $F_{j+1}(0)=\frac{1}{j+1}\frac{1}{(k+1)!}\partial_2^{k+1}G(0,0)$ if $k\geq -1$ and $F_{j+1}(0)=0$ if $k\leq -2$.
Moreover, for $k=-1$, the functions $F_i$, $i=0,\ldots,j+1$, are even.
\end{remark}

\begin{proof}
We define 
$$
J_k^j[G](x) = \int_x^1 \tau^k\, G\left(\tau, \frac{x}{\tau} \right) \ln^j\frac{1}{\tau}\, d\tau .
$$
Using the change of variable $u=\frac{x}{\tau}$, we have 
$I_k^j[G](x) = x^{k+1} J_{-k-2}^j[\check G](x)$
where $\check G(\tau,\varepsilon) = G(\varepsilon,\tau)$. Actually, we are going to prove that, given any $G\in C^\infty(\R^2)$, $J_k^j[G](x)$ has the same expansion as $I_k^j[G](x)$, except that $\ln^j\frac{\tau}{x}$ is changed to $\ln^j\frac{1}{\tau}$, and $\ln^j\frac{1}{\varepsilon}$ is changed to $\ln^j\frac{\varepsilon}{x}$. This gives the proposition. Let us then establish the result for $J_k^j[G](x)$.

Using the change of variable $u=\frac{x}{\tau}$, we first note that 
$$
J_k^j[G](x) = \sum_{i=0}^j (-1)^{j-i} \binom{j}{i} x^{k+1} \ln^i\frac{1}{x} \  J_{-k-2}^{j-i}[\check G](x) ,
$$
and thus it suffices to prove the statement for $k\in\N$. 

Second, since $J_0^j[G]$ is continuous and 
$$
J_k^j[G]'(x) = -x^k \ln^j\frac{1}{x}\, G(x,1) + x J_{k-1}^j[\partial_2 G](x) ,
$$
it follows that $J_k^j[G]$ is at least of class $C^k$ when $k\in\N$.

Now, if $G(\tau,\varepsilon) = \mathrm{O}(\tau^{2N}+\varepsilon^{2N})$ for some $N\in\N^*$, then $G(\tau,\varepsilon) = \tau^N G_1(\tau,\varepsilon) + \varepsilon^N G_2(\tau,\varepsilon)$ for some $G_1,G_2\in C^\infty(\R^2)$, and it follows that
$$
J_k^j[G](x) = J_{k+N}^j[G_1](x) + \sum_{i=0}^j (-1)^{j-i} \binom{j}{i} x^{k+1} \ln^i\frac{1}{x}\, J_{N-k-2}^{j-i}[\check G_2](x)
$$
and thus $J_k^j[G]$ is at least of class $C^{N-k-2}$.

Therefore, taking $N$ sufficiently large, it suffices to check the result on monomials $G(\tau,\varepsilon)=\tau^p\varepsilon^q$ with $p,q\in\N$: we have
$J_k^j[G](x) = x^q \int_x^1 \tau^{k+p-q} \ln^j\frac{1}{\tau}\, d\tau$
and the result is then easily established.
The case $k=-1$ is treated separately, writing $G(\tau,u) = G(0,u)+\tau G_1(\tau,u)$ with $G_1$ continuous. 
\end{proof}

\begin{remark}\label{rem_expansion_4}
When $k=-1$ and $j=0$, assuming that $G$ is only of class $C^2$, we have
\begin{multline*}
I_{-1}^0[G](x) = \int_x^1 \frac{1}{\tau} \, G\left( \tau, \frac{x}{\tau} \right) d\tau = G(0,0)\ln\frac{1}{x}  + \int_0^1 \frac{G(0,\varepsilon)-G(0,0)}{\varepsilon}\, d\varepsilon + \int_0^1 \frac{G(\tau,0)-G(0,0)}{\tau}\, d\tau \\
+ \partial_1\partial_2 G(0,0)\, x\ln\frac{1}{x}
+ \bigg( -\partial_1 G(0,0) -\partial_2 G(0,0) + \int_0^1 \frac{ \partial_1 G(0,\varepsilon)-\partial_1 G(0,0)-\varepsilon\partial_1\partial_2 G(0,0) }{\varepsilon^2} \, d\varepsilon \\
+ \int_0^1 \frac{ \partial_2 G(\tau,0)-\partial_2 G(0,0)-\tau\partial_1\partial_2 G(0,0) }{\tau^2} \, d\tau \bigg) x + \mathrm{o}(x)
\end{multline*}
as $x\rightarrow 0^+$, i.e., we have identified the four first terms of the expansion.
\end{remark}

\subsection{Nested integrals} \label{sec_nested}
Let $p\in\N^*$ be arbitrary.
Given any $G\in C^\infty(\R^{p+1})$ and any $(k_1,\ldots,k_p)\in\Z^p$, we define
$$
I_{k_1,\ldots,k_p}[G](x) = \int_x^1 \tau_1^{k_1} \int_{\frac{x}{\tau_1}}^1 \tau_2^{k_2} \int_{\frac{x}{\tau_1\tau_2}}^1 \tau_3^{k_3} \cdots \int_{\frac{x}{\tau_1\cdots\tau_{p-1}}}^1 \tau_p^{k_p} \, G\left(\tau_1,\ldots,\tau_p,\frac{x}{\tau_1\cdots\tau_p}\right) d\tau_p \cdots d\tau_1
$$
for every $x\in(0,1)$. 
Setting $k_{p+1}=-1$, we define $s\in\N^*$ and $(j_1,\ldots,j_s)\in\Z^s$ so that $\{j_1,\ldots,j_s\} = \{ k_1,\ldots, k_{p+1}\}$ and $j_1=\min(k_1,\ldots,k_{p+1}) < j_2< \cdots < j_s=\max(k_1,\ldots,k_{p+1})$. For every $i\in\{1,\ldots,s\}$, we denote by $m_i$ the ``multiplicity" of $j_i$, that is, the number of integers $l\in\{1,\ldots,p+1\}$ such that $k_l=j_i$. In particular, $m_1=\#\mathrm{argmin}(k_1,\ldots,k_{p+1})\in\{1,\ldots,s+1\}$ is the number of elements $k_l$ reaching the minimum $j_1$.

\begin{proposition}\label{prop_equiv_iter}
We have
\begin{equation}\label{prop_equiv_iter_peries}
I_{k_1,\ldots,k_p}[G](x) = \sum_{i=1}^{s} F_i(x) x^{j_i+1} \vert\ln x\vert^{m_i-1} \qquad \forall x\in(0,1),
\end{equation}
i.e., the function $I_{k_1,\ldots,k_p}[G](x)$ can be written as a sum of $s$ formal series multiplied by $\ln^i\frac{1}{x}$, where the coefficients are distributions (thus, depending on integrals and derivatives of $G$). 
Moreover,
\begin{equation*}
I_{k_1,\ldots,k_p}[G](x) = C(G) \, x^{j_1+1} \vert\ln x\vert^{m_1-1} + \mathrm{o}\left( x^{j_1+1} \vert\ln x\vert^{m_1-1} \right)
\end{equation*}
as $x\rightarrow 0^+$, for some $C(G)\in\R$. Actually, $C$ is a Radon measure.

\smallskip
\noindent\textbf{Concentration properties.}
Denoting $G(\tau_1,\ldots,\tau_p,\varepsilon)$, we say that we have ``$\varepsilon$-concentration" whenever $C(G)$ depends only on $G$ restricted to $\varepsilon=0$.
We say that we have ``$\tau$-concentration" if there exist at least one index $i\in\{1,\ldots,p\}$ such that we have concentration on $\tau_i=0$, i.e., $C(G)$ depends only on $G$ restricted to $\tau_i=0$ (in other words, $C$ is Dirac in its $i^\textrm{th}$ variable). In this case, we say that we have $\tau$-concentration at minimal index $i\in\{1,\ldots,p\}$ if we have concentration on $\tau_i=0$ but not on $\tau_j=0$ for any $j<i$, i.e., $C$ is absolutely continuous with respect to its $(i-1)$ first variables and is Dirac in its $i^\textrm{th}$ one.

We have $\varepsilon$-concentration if and only if $k_i\geq -1$ for every $i\in\{1,\ldots,p\}$.
We have $\tau$-concentration if and only if $j_1<0$ (i.e., at least one of the integers $k_1,\ldots,k_p$ is negative); in this case, we have concentration on $\tau_i=0$ at least for all indices $i\in\{1,\ldots,p\}$ such that $k_i=j_1=\min(k_1,\ldots,k_{p+1})$ and, denoting $i_1$ the minimal index of those $k_i$, we do not have concentration on $\tau_i=0$ if $i<i_1$.

The explicit expression of $C(G)$ is complicated in general. It is given below for $p=1$ and $p=2$. 
In the particular case where $k_1=\cdots=k_p=j_1<0$:
\begin{itemize}[parsep=0.5mm,itemsep=0.5mm,topsep=0.5mm]
\item if $j_1\leq -2$ then $m_s=s$ and $C(G) = \frac{1}{(s-1)!} \int_0^1 \varepsilon^{-j_1-2}G(0,\ldots,0,\varepsilon)\, d\varepsilon$;
\item if $j_1=-1$ then $m_s=s+1$ and $C(G) = \frac{1}{s!} G(0,\ldots,0)$.
\end{itemize}
\end{proposition}

To prove Proposition \ref{prop_equiv_iter}, 
setting $G_{\tau_1}(\tau_2,\ldots,\tau_p,x_1)=G(\tau_1,\tau_2,\ldots,\tau_p,x_1)$ for every $x_1>0$ (we will take $x_1=x/\tau_1$), we first note that 
$$
I_{k_1,\ldots,k_p}[G](x) = \int_x^1 \tau_1^{k_1} I_{k_2,\ldots,k_p}[G_{\tau_1}]\Big(\frac{x}{\tau_1}\Big) \, d\tau_1 .
$$
Then, similarly, we have 
$$
I_{k_2,\ldots,k_p}[G_{\tau_1}](x_1) = \int_{x_1}^1 \tau_2^{k_2} I_{k_3,\ldots,k_p}[G_{\tau_1,\tau_2}]\Big(\frac{x_1}{\tau_2}\Big) \, d\tau_2
$$
where $G_{\tau_1,\tau_2}(\tau_3,\ldots,\tau_p,x_2)=G(\tau_1,\ldots,\tau_p,x_2)$ (we will take $x_2=x_1/\tau_2$), and so on 
until we reach 
$$
I_{k_p}[G_{\tau_1,\ldots,\tau_{p-1}}](x_{p-1}) = \int_{x_{p-1}}^1 \tau_p^{k_p} G_{\tau_1,\ldots,\tau_{p-1}}\Big( \tau_p, \frac{x_{p-1}}{\tau_p} \Big) \, d\tau_p .
$$
By applying iteratively \eqref{expansionIkj} in Proposition \ref{prop_general_expansion}, we obtain \eqref{prop_equiv_iter_peries}. 
Besides, recalling that $k_{p+1}=-1$, setting $\varepsilon_i=1$ for $i=1,\ldots,s$, and $\varepsilon_{p+1}=-1$, we establish by induction on $s$ that
\begin{equation}\label{Helem}
\int_x^1 \tau_1^{k_1} \int_{\frac{x}{\tau_1}}^1 \tau_2^{k_2} \int_{\frac{x}{\tau_1\tau_2}}^1 \tau_3^{k_3} \cdots \int_{\frac{x}{\tau_1\cdots\tau_{p-1}}}^1 \tau_p^{k_p} \, d\tau_p \cdots d\tau_1 
= (-1)^s \sum_{i=1}^{p+1} \varepsilon_i \prod_{j=1\atop j\neq i}^{p+1}(k_i-k_j)^{-1} \, x^{k_i+1}
\end{equation}
for any $(k_1,\ldots,k_p)\in\R^p$ (not only integer) if $k_1,\ldots,k_{p+1}$ are paiwise distinct; when $m$ of them are equal, $k_{i_1}=\cdots=k_{i_m}$, the linear combination of $x^{k_{i_1}+1},\ldots,x^{k_{i_m}+1}$ is replaced with the linear combination $x^{k_{i_1}+1} \left( c_0+c_1\vert\ln x\vert+\cdots + c_{m-1} \vert\ln x\vert^{m-1} \right)$, where the coefficients $c_i$ can be explicitly computed by Taylor expansions. 
In particular, if $k_1=\cdots=k_p=j_1$ then \eqref{Helem} is equal to $\frac{1}{p!} \vert\ln x\vert^p$ when $j_1=-1$, and to $\frac{1}{(p-1)! \, \vert j_1+1\vert} x^{j_1+1} \vert\ln x\vert^{p-1}(1+\mathrm{o}(1))$ (as $x\rightarrow 0^+$) when $j_1\leq -2$.

The concentration properties are established by Taylor expansions and by considering monomial functions $G$ as in the proof of Proposition \ref{prop_general_expansion}. 

\paragraph{Case $p=1$.} By Proposition \ref{prop_general_expansion}, the dominating term of $I_{k_1}[G](x)$ as $x\rightarrow 0^+$ is:
\begin{itemize}[parsep=0.1cm,itemsep=0.1cm,topsep=0.1cm]
\item $\int_0^1 \tau_1^{k_1} G(\tau_1,0)\, d\tau_1$ if $k_1\geq 0$ (no $\tau$-concentration, concentration on $\varepsilon=0$);
\item $G(0,0)\ln\frac{1}{x}$ if $k_1=-1$ (concentration on $\tau_1=0$ and on $\varepsilon=0$);
\item $\left( \int_0^1 \varepsilon^{-k_1-2} G(0,\varepsilon)\, d\varepsilon\right) x^{k_1+1}$ if $k_1\leq -2$ (concentration on $\tau_1=0$, no $\varepsilon$-concentration).
\end{itemize}

\paragraph{Case $p=2$.} 
Starting from $I_{k_1,k_2}[G](x)=\int_x^1 \tau_1^{k_1} I_{k_2}[G_{\tau_1}]\big(\frac{x}{\tau_1}\big) \, d\tau_1$, using Proposition \ref{prop_general_expansion}, we obtain that the dominating term of $I_{k_1,k_2}[G](x)$ as $x\rightarrow 0^+$ is:
\begin{align*}
\int_0^1 \tau_1^{k_1} \int_0^1 \tau_2^{k_2} G(\tau_1,\tau_2,0)\, d\tau_2\, d\tau_1  & \qquad \textrm{if }  k_1\geq 0  \textrm{ and }  k_2\geq 0 \\
\left(\int_0^1 \tau_2^{k_2} G(0,\tau_2,0)\, d\tau_2\right)\vert\ln x\vert & \qquad\textrm{if }  k_1=-1  \textrm{ and }  k_2\geq 0 \\
\left( \int_0^1 \varepsilon^{-k_1-2} \int_\varepsilon^1 \tau_2^{k_2} G(0,\tau_2,\varepsilon)\, d\tau_2\, d\varepsilon \right) x^{k_1+1} & \qquad\textrm{if } k_1\leq -2  \textrm{ and } k_2\geq 0 \\
\left( \int_0^1 \tau_1^{k_1} G(\tau_1,0,0)\, d\tau_1 \right) \vert\ln x\vert  &\qquad \textrm{if } k_1\geq 0  \textrm{ and } k_2=-1 \\
\frac{1}{2} G(0,0,0) \vert\ln x\vert^2  & \qquad\textrm{if }  k_1=-1   \textrm{ and } k_2=-1 \\
\left( \int_0^1 \varepsilon^{-k_1-2} \int_\varepsilon^1 \frac{1}{\tau_2} G(0,\tau_2,\varepsilon)\, d\tau_2\, d\varepsilon \right) x^{k_1+1} & \qquad\textrm{if }  k_1\leq -2  \textrm{ and } k_2=-1 \\
\left( \int_0^1 \tau_1^{k_1-k_2-1} \int_0^1 \varepsilon^{-k_2-2} G(\tau_1,0,\varepsilon)\, d\varepsilon\, d\tau_1 \right) x^{k_2+1}  & \qquad\textrm{if } k_1\geq k_2+1 \textrm{ and } k_2\leq -2 \\
\left( \int_0^1 \varepsilon^{-k_1-2} G(0,0,\varepsilon)\, d\varepsilon \right) x^{k_1+1}\vert\ln x\vert  & \qquad\textrm{if } k_1=k_2 \textrm{ and } k_2\leq -2 \\
\left( \int_0^1 u^{k_2-k_1-2} \int_u^1 \varepsilon^{-k_2-2} G(0,u,\varepsilon)\, d\varepsilon\, du \right) x^{k_1+1}  & \qquad\textrm{if } k_1\leq k_2-1  \textrm{ and } k_2\leq -2
\end{align*}

\section{Subanalytic sets and functions}\label{app_subanalytic}
In this appendix we concisely recall the definitions and main properties of subanalytic sets and functions, as well as some desingularization and cell preparation theorems for subanalytic functions, from which we derive some useful results.
This section is based on the references \cite{BierstoneMilman_PMIHES1988, CluckersMillerRolinServi_DMJ2018, CluckersMiller_DMJ2011, Hironaka_1973, Kurdyka_AIF1998, LionRolin_AIF1997, LionRolin_AIF1998, Lojasiewicz_AIF1993, NguyenValette_ASENS2016, Parusinski_2001, Tamm_AM1981, Valette_subanalytic, vandenDries}.

\subsection{Reminders on subanalytic geometry}\label{app_subanalytic_reminders}

\paragraph{Definitions.}
Given any $n\in\N^*$ and any real analytic manifold $M$ of dimension $n$, a subset $X\subset M$ is \emph{locally semianalytic} if for every $x\in M$ there exists an open neighborhood $U$ of $x$ in $M$ such that, in a local analytic chart, $X\cap U$ can be defined by a finite number of equalities and inequalities using real analytic functions on $U$. 

The projection of a locally semianalytic set, even compact, may fail to be locally semianalytic (see \cite{Lojasiewicz_AIF1993}), this is what has motivated the definition of \emph{subanalytic} sets. 

The subset $X$ is \emph{locally subanalytic} if for every $x\in M$ there exists an open neighborhood $U$ of $x$ in $M$ such that, in a local analytic chart, $X\cap U$ is the projection of a relatively compact semianalytic set, i.e., there exists a real analytic manifold $N$ and a relatively compact semianalytic subset $A\subset M\times N$ such that $X\cap U=\pi(A)$ where $\pi:M\times N\rightarrow M$ is the canonical projection (see \cite{BierstoneMilman_PMIHES1988}).

Assuming that $M$ is an analytic submanifold of $\R^n$, the subset $X\subset M$ is \emph{globally semianalytic} (resp., \emph{globally subanalytic}) if its image under the natural embedding $(x_1,\ldots,x_n)\mapsto(1:x_1:\cdots:x_n)$ from $\R^n$ to the projective space $\mathcal{P}^n(\R)$ is a locally semianalytic (resp., locally subanalytic) subset of the analytic manifold $\mathcal{P}^n(\R)$ (see \cite{CluckersMiller_DMJ2011, Kurdyka_AIF1998, LionRolin_AIF1997, Valette_subanalytic}). Here, $(1:x_1:\cdots:x_n)$ is the equivalence class of the collinearity equivalence relation in $\R^{n+1}\setminus\{0\}$. 
One can equivalently use other embeddings provided they have a subanalytic graph (see \cite{Kurdyka_AIF1998}).

Hereafter, when we speak of a globally subanalytic subset $X\subset M$, it is understood that $M$ is an analytic submanifold of $\R^n$.

Of course, globally semianalytic or subanalytic sets are locally semianalytic or subanalytic.
Any relatively compact subset of $M$ is locally subanalytic if and only if it is globally subanalytic. 
For instance, $\N$ is analytic but is not globally subanalytic. The graph of the sine function in $\R^2$ is locally but not globally subanalytic.


The class of (locally or globally) subanalytic sets is stable under locally finite unions and intersections, taking complement, interior, closure, product, image under proper analytic mapping, taking connected components. Moreover, the family of connected components of such a set is locally finite (see \cite{BierstoneMilman_PMIHES1988, Lojasiewicz_AIF1993}).

Given a real analytic finite-dimensional manifold $N$, a mapping $f:X\rightarrow N$ is (locally or globally) subanalytic if its graph is (locally or globally) subanalytic in $M\times N$. 
The composition of two globally subanalytic mappings is globally subanalytic, while the property may fail for locally subanalytic mappings (unless one adds an adequate properness property, see \cite{Kurdyka_AIF1998}). A function $f:X\rightarrow\R$ on a globally subanalytic set $X$ is globally subanalytic if and only if $f$ and $1/f$ (defined on $X\setminus f^{-1}(0)$) are locally subanalytic.

\paragraph{Smoothness and stratification properties.}
Any (locally or globally) subanalytic subset $X$ of a real analytic finite-dimensional manifold $M$ is Whitney stratifiable, with a locally finite number of strata that are connected analytic submanifolds of $M$ and are also (locally or globally) subanalytic sets (see \cite{BierstoneMilman_PMIHES1988, Tamm_AM1981}). 
The topological dimension $\dim X$ of $X$ is defined as the maximal topological dimension of its strata.
The singular set of $X$, which is a closed (locally or globally) subanalytic subset of $X$ of codimension at least one, is the union of strata of dimension less than $\dim X$.

We recall that a stratification of $M$ is a locally finite partition in smooth submanifolds (strata), such that, if the stratum $\S_1$ has a nonempty intersection with the closure $\overline{\S_2}$ of another stratum $\S_2$ then $\S_1\subset\overline{\S_2}$. The strata are glued one to another, at their boundary, according to some rules. The most standard ones are the Whitney (A) and (B) conditions: assuming that $\S_1\subset\overline{\S_2}$, (A) if a sequence of $q_k\in\S_2$ converges to $q\in\S_1$ then $T_q\S_1\subset\lim_k T_{q_k}\S_2$; (B) if two sequences $q_k\in\S_2$ and $q_k'\in\S_1$ converge to the same $q\in\S_1$ then the limit of the chords $[q_k,q_k']$ (in local coordinates) is contained in $\lim_kT_{q_k}\S_2$ provided that both limits exist. We then speak of a Whitney stratification.

Finer stratifications concepts and results can be found, e.g., in \cite{Lojasiewicz_AIF1993, NguyenValette_ASENS2016}.

\paragraph{Uniformization and rectilinearization.}
Let $X$ be a closed (locally or globally) subanalytic subset of a real analytic manifold $M$ of dimension $n$.

According to the \emph{uniformization theorem}, there exists a real analytic manifold $N$ of the same dimension as $X$ and a proper real analytic mapping $\phi:N\rightarrow M$ such that $X=\phi(N)$. 

According to the \emph{rectilinearization theorem}, on every compact subset of $M$ there exists a locally finite covering  on each component of which there exists an analytic mapping $\phi:\R^n\rightarrow M$ such that, locally, $\phi^{-1}(X)$ is a union of quadrants of $\R^n$, where a quadrant is a subset of $\R^n$ defined by $x_i=0$ or $x_i>0$ or $x_i<0$, for $i\in\{1,\ldots,n\}$.

These theorems have first been proved in \cite{Hironaka_1973} by desingularization and resolution of singularities (following the deep celebrated \cite{Hironaka_1964}), and later in \cite{BierstoneMilman_PMIHES1988} by a more direct and elementary approach in the context of what is now called \emph{subanalytic geometry}.

As concerns functions, any 
bounded globally subanalytic (not necessarily continuous) function $f:X\rightarrow\R$ can be desingularized as follows: there exists a real analytic manifold $N$ of the same dimension as $X$ and a proper analytic mapping $\phi:N\rightarrow M$ (which is a local diffeomorphism on a dense subset of $N$ if moreover $X=M$) such that $X=\phi(N)$ and $f\circ\phi$ is locally normal crossings on the components of $N$ where it does not vanish identically, i.e., $f\circ\phi$ can be written in some local charts as $f\circ\phi(x) = \prod_{i=1}^n x_i^{\alpha_i} g(x)$ for some $\alpha_i\in\N$, for $i\in\{1,\ldots,n\}$, and some \emph{unit} subanalytic function $g$ (i.e., bounded and not vanishing in the chart).
Actually, it is always possible to define $\phi$, locally and piecewise, as a finite composition of blowings-up, substitutions of powers and shifts (see \cite{BierstoneMilman_PMIHES1988, Parusinski_2001}).
Unfolding such transforms then yields various versions of rectilinearization or preparation theorems, all of them aiming at writing $f$, locally, in a ``cusp-prepared" (locally fractional normal crossings) form, as the product of a fractional monomial (i.e., as above but with $\alpha_i\in\Q$) with a unit subanalytic function (see \cite{Parusinski_2001}), as explained hereafter.

\paragraph{Subanalytic cell decompositions.}
A \emph{subanalytic cell decomposition} of $\R^n$ is a finite partition of $\R^n$ into so-called \emph{subanalytic cells} that are disjoint globally subanalytic subsets and analytic connected submanifolds of $\R^n$ having a particular cylindrical form. 
There exist various possible cell decompositions (see \cite{CluckersMiller_DMJ2011, LionRolin_AIF1997, NguyenValette_ASENS2016, Valette_subanalytic, vandenDries}). Here, we follow \cite{CluckersMiller_DMJ2011, vandenDries}.

Given any $n\in\N^*$, we denote by $\Pi_n:\R^{n+1}\rightarrow\R^n$ the projection onto the first $n$ coordinates, i.e., taking coordinates $(x,y)\in\R^{n+1}=\R^n\times\R$ with $x\in\R^n$ and $y\in\R$, we have $\Pi_n(x,y)=x$. 

Subanalytic cell decompositions are defined by induction. 
For $n=1$, a subanalytic cell decomposition $\mathcal{C}_1$ of $\R$ is given by a partition on $\R$ in a finite number of singletons $\{a\}$ and open intervals $(b,c)$ with $a\in\R$, $b\in\{-\infty\}\cup\R$ and $c\in\R\cup\{+\infty\}$. Then, by induction, a subanalytic cell decomposition $\mathcal{C}_{n+1}$ of $\R^{n+1}$ is given by a subanalytic cell decomposition $\mathcal{C}_n$ of $\R^n=\Pi_n(\R^{n+1})$ and each subanalytic cell $C\in\mathcal{C}_{n+1}$ of $\R^{n+1}$ is such that $\Pi_n(C)\in\mathcal{C}_n$ (called \emph{basis} of the cell $C$) and either $C$ is of the form
$$
C = \{ (x,y)\in \Pi_n(C)\times\R \ \mid\ y=a(x) \} 
$$
called a \emph{thin cell} in $y$, where $a$ is an analytic and globally subanalytic function on $\Pi_n(C)$, i.e., $C$ is the graph of $a$ above the basis $\Pi_n(C)$, or $C$ is of the form
$$
C = \{ (x,y)\in \Pi_n(C)\times\R \ \mid\  a(x) < y <  b(x)   \} 
$$
called a \emph{fat cell} in $y$, where $a$ and $b$ are analytic and globally subanalytic functions on $\Pi_n(C)$, with possibly $a\equiv-\infty$ or $b\equiv+\infty$, i.e., $C$ is a a cylindrical set between $a$ and $b$ above the basis $\Pi_n(C)$.
Moreover, in the fat case, following \cite[Definition 3.4]{CluckersMillerRolinServi_DMJ2018} or \cite[Definition 3.2 and Definitions 3.4]{CluckersMiller_DMJ2011}, it is always possible to write $C$ in the form 
$$
C = \{ (x,y)\in \Pi_n(C)\times\R \ \mid\  \tilde a(x) < \tilde y < \tilde b(x) \} 
\qquad\textrm{with}\qquad
\tilde y = \varepsilon(y-\zeta(x))^\delta 
$$
for some $\varepsilon,\delta\in\{\pm 1\}$ and for some analytic and globally subanalytic functions $\tilde a$, $\tilde b$ and $\zeta$ on $\Pi_n(C)$ satisfying $0\leq \tilde a(\cdot)<\tilde b(\cdot)\leq 1$, with either $\tilde a(\cdot)>0$ or $\tilde a(\cdot)\equiv 0$, the graph of $\zeta$ being disjoint from $C$. 
There is an infinite number of possible choices for the function $\zeta$, which is called a \emph{center} for $C$. When $0\leq a<b\leq 1$, we can take $\zeta=0$ and $\varepsilon=\delta=1$ (this is what we use in Section \ref{sec_cell_decomp_VF}).

Every subanalytic cell of $\R^n$ is an analytic connected submanifold of $\R^n$ and is also a globally subanalytic subset of $\R^n$.

A subanalytic cell decomposition is said to be compatible with a finite number of given globally subanalytic sets if each of those sets is itself a union of cells. 
The subanalytic cell decomposition theorem states that, given a finite number of arbitrary globally subanalytic subsets of $\R^n$, there exists a (finite) subanalytic cell decomposition of $\R^n$ compatible with those sets (see \cite[Chapter 3, Section 2]{vandenDries} or \cite[Theorem 1.1]{NguyenValette_ASENS2016} or \cite{Valette_subanalytic}).

\paragraph{Subanalytic preparation theorems.}
Preparation theorems for subanalytic functions are far-reaching versions of the celebrated Weierstrass and Malgrange preparation theorems.
As alluded above, there are various existing versions in the literature (see \cite{CluckersMillerRolinServi_DMJ2018, CluckersMiller_DMJ2011, LionRolin_AIF1997, LionRolin_AIF1998, Parusinski_2001, Valette_subanalytic}), with possible variants in the reduced normal form or in the definition of the unit functions. 

The statement that we give hereafter is based on \cite[Theorem 2.4]{CluckersMiller_DMJ2011}, which extends \cite[Th\'eor\`eme 1]{LionRolin_AIF1997} to a finite set of functions, with unit functions defined in \cite[Definition 3.8]{CluckersMillerRolinServi_DMJ2018} (this form of unit function being better suited to our needs; see \cite[Proof of Proposition 3.10]{CluckersMillerRolinServi_DMJ2018} to see how it can be obtained from the more usual unit functions found in \cite{CluckersMiller_DMJ2011, LionRolin_AIF1997}).

Given any $n\in\N$, we consider $\R^{n+1}=\R^n\times\R$ with a system of coordinates $(x,y)$ with $x\in\R^n$ and $y\in\R$. As before, we denote $\Pi_n(x,y)=x$.
Let $X$ be a globally subanalytic subset of $\R^{n+1}$,
let $\mathcal{F}$ be a finite set of globally subanalytic functions on $X$,
and let $\mathcal{X}$ be a finite set of globally subanalytic subsets of $X$.
Then, there exists a (finite) subanalytic cell decomposition of $\R^{n+1}$, compatible with $\mathcal{X}$, such that, for any subanalytic cell $C\subset\R^{n+1}$ of that decomposition:
\begin{itemize}
\item either $C$ is thin in $y$ and for every $f\in\mathcal{F}$ there exists an analytic and globally subanalytic function $a$ on the cell $\Pi_n(C)$ such that $f(x,y)=a(x)$ on $C$;
\item or $C$ is fat in $y$, of the form $C=\{(x,y)\in\Pi_n(C)\times\R\ \mid\ \tilde a(x) < \tilde y < \tilde b(x) \}$ where $\tilde y = \varepsilon(y-\zeta(x))^\delta$ and each function $f\in\mathcal{F}$ can be written in $C$ as
$$
f(x,y) = F(x)\, \tilde y^\alpha \, U \left(  \left( c_i(x) \right)_{1\leq i\leq N}, \left( \frac{\tilde a(x)}{\tilde y} \right)^{1/\ell} , \left( \frac{\tilde y}{\tilde b(x)} \right)^{1/\ell} \right)   
$$
for some $\alpha\in\Q$, $N,\ell\in\N^*$, $\varepsilon,\delta\in\{\pm 1\}$, for some analytic and globally subanalytic functions $F$, $\zeta$, $a$, $b$, $c_1,\ldots,c_N$ on the cell $\Pi_n(C)$, with $0\leq \tilde a(\cdot)<\tilde b(\cdot)\leq 1$ on $\Pi_n(C)$ and either $\tilde a(\cdot)>0$ or $\tilde a(\cdot)\equiv 0$, the graph of $\zeta$ being disjoint from $C$ and the functions $c_1,\ldots,c_N$ taking their values in $[0,1]$, and for some unit (i.e., not vanishing) analytic function $U$ on an open subset of $\R^{N+2}$ containing $[0,1]^{N+2}$. 
\end{itemize}
We say that we have \emph{prepared the functions $f\in\mathcal{F}$ with respect to the variable $y$}. Note that, by permutation, we can prepare $f$ with respect to any of its variables.

\medskip

Now, the above statement can be iterated, as done in \cite[Theorem 3.9]{CluckersMiller_DMJ2011}, leading to the following result. Given any $n,p\in\N$, we consider $\R^{n+p}=\R^n\times\R^p$ with a system of coordinates $(x,y)$ with $x\in\R^n$ and $y=(y_1,\ldots,y_p)\in\R^p$. For every $i\in\{1,\ldots,p\}$, we denote $\Pi_{n+i}(x,y)=(x,y_{\leq i})$ where $y_{\leq i}=(y_1,\ldots,y_i)$, and $\Pi_n(x,y)=x$. We also use the notation $y_{<i}=(y_1,\ldots,y_{i-1})$ for $i>1$.

Let $X$ be a globally subanalytic subset of $\R^{n+p}$,
let $\mathcal{F}$ be a finite set of globally subanalytic functions on $X$,
and let $\mathcal{X}$ be a finite set of globally subanalytic subsets of $X$.
Then, there exists a (finite) subanalytic cell decomposition of $\R^{n+p}$, compatible with $\mathcal{X}$, such that, for any \emph{open} subanalytic cell $C\subset\R^{n+p}$ of that decomposition (here, we give the result only for open cells because this is enough for our needs; for more general results, see \cite[Section 3]{CluckersMiller_DMJ2011}), for every $i\in\{0,\ldots,p\}$, we have
$$
\Pi_{n+i}(C) = \{ (x,y_{\leq i})\in\Pi_{n+i-1}(C)\times\R\ \mid\ a_i(x,y_{<i}) < \tilde y_i < b_i(x,y_{<i}) \}
$$
where $\tilde y_i = \varepsilon_i(y_i-\zeta_i(x,y_{<i}))^{\delta_i}$
and each function $f\in\mathcal{F}$ can be written in $C$ as
$$
f(x,y) = F(x)\, \prod_{i=1}^p \tilde y_i^{\alpha_i} \, U \left(  \left( c_i(x) \right)_{1\leq i\leq N}, \left( \frac{a_i(x,y_{<i})}{\tilde y_i} \right)^{1/\ell_i}_{1\leq i\leq p} , \left( \frac{\tilde y_i}{b_i(x,y_{<i})} \right)^{1/\ell_i}_{1\leq i\leq p} \right)   
$$
for some $\alpha_i\in\Q$, $N,\ell_i\in\N^*$, $\varepsilon_i,\delta_i\in\{\pm 1\}$ ($1\leq i\leq p$), for some analytic and globally subanalytic functions $F$, $c_1,\ldots,c_N$ (defined on $\Pi_n(C)$), $\zeta_i$, $a_i$, $b_i$ (defined on $\Pi_{n+i-1}(C)$) satisfying $0\leq a_i(\cdot)<b_i(\cdot)\leq 1$, the graph of $\zeta_i$ being disjoint from $\Pi_{n+i}(C)$ and the functions $c_1,\ldots,c_N$ taking their values in $[0,1]$, and for some unit (i.e., not vanishing) analytic function $U$ on an open subset of $\R^{N+2p}$ containing $[0,1]^{N+2p}$.

\subsection{A useful result}\label{app_subanalytic_useful}
The following result\footnote{We thank Jean-Philippe Rolin for a useful discussion on this topic.} is, in some sense, a parametric version of Hironaka's uniformization theorem, in the spirit of \cite[Proof of Theorem 6.1]{Parusinski_2001}. 

\begin{lemma}\label{lem_parametric_Hironaka}
Let $P$ be a real analytic compact manifold of dimension $p$ and let $X$ be a globally subanalytic compact subset of $P\times\R^n$ such that $\Pi_P(X)=P$, where $\Pi_P:P\times\R^n\rightarrow P$ is the canonical projection.
We assume that the fibers $X_q=\{x\in\R^n\ \mid\ (q,x)\in X\}$ are bounded and have dimension at most $k$, for any $q\in P$. Then, there exist a real analytic compact manifold $N$ of dimension $k$ and a bounded globally subanalytic mapping $\Phi:P\times N\rightarrow\R^n$ such that $X=\Phi(P\times N)$ and $X_q=\Phi(q,N)$ for every $q\in P$.
\end{lemma}

\begin{proof}
By the subanalytic cell decomposition theorem, 
$X$ is the finite union of (disjoint) relatively compact cells. Let us first assume that $X$ is a single cell, of dimension $p+k$. Assuming that $X$ is sufficiently small (otherwise, do what follows on a finite covering), performing if necessary an orthogonal change of coordinates and reordering variables, it is always possible to choose a system of coordinates $x=(x_1,\ldots,x_n)$ on $\R^n$ such that, denoting by $\Pi_i:P\times\R^n\rightarrow\R^i$ the projection defined by $\Pi_i(q,x)=(q,x_{\leq i})$, where $q\in P$ and $x_{\leq i}=(x_1,\ldots,x_i)$:
\begin{itemize}[parsep=0.5mm,itemsep=0.2mm,topsep=0.5mm]
\item for every $i\in\{k+1,\ldots,n\}$ (if $k<n$), the subanalytic cell $\Pi_i(X)$ is thin and there exists a bounded analytic and globally subanalytic function $a_i$ on $\Pi_{i-1}(X)$ such that $x_i = a_i(q,x_{\leq i-1})$ for every $(q,x)\in X$;
\item for every $i\in\{1,\ldots,k\}$, the subanalytic cell $\Pi_i(X)$ is fat, of dimension $p+i$, and there exist bounded analytic and globally subanalytic functions $a_i$ and $b_i$ on the subanalytic cell $\Pi_i(X)$ such that $a_i(q,x_{\leq i-1}) < x_i < b_i(q,x_{\leq i-1})$ for every $(q,x)\in X$ (with the agreement that $a_1(q)<x_1<b_1(q)$ for $i=1$).
\end{itemize}
Now, setting $\theta=(\theta_1,\ldots,\theta_k)\in[0,1]^k$, we define the parametrization $x_1(q,\theta) = (1-\theta_1) b_1(q) + \theta_1 c_1(q)$, then $x_2(q,\theta) = (1-\theta_2) b_2(q,x_1(q,\theta)) + \theta_2 c_2(q,x_1(q,\theta))$, etc, until $x_k$, and then $x_{k+1},\ldots,x_n$ are then defined accordingly in function of $(q,\theta)$ by composing with $a_i$.
Finally, defining the globally subanalytic function $\phi(q,\theta) = (x_1(q,\theta),\ldots,x_n(q,\theta))$ on $P\times[0,1]^k$, we have $X=\phi(P\times[0,1]^k)$ and $X_q=\phi(q\times[0,1]^k)$.
The lemma is thus obtained when $X$ is a cell. In the general case, $X$ is a finite union of cells and to obtain the global statement of lemma, it suffices to proceed for example as in the proof of \cite[Corollary 4.9 or Theorem 5.1]{BierstoneMilman_PMIHES1988} by ``pasting" copies of $[0,1]^k$.
\end{proof}

{
\small
\bibliographystyle{plain}
\bibliography{bib_sR}
}

\end{document}